\RequirePackage{silence}
\documentclass[11pt]{article}
\pdfoutput=1
\usepackage[ascii]{inputenc}
\usepackage[english]{babel}
\usepackage[theoremfont]{newpxtext}
\usepackage[T1]{fontenc}
\usepackage{amsmath,mathtools,amsthm,thmtools,thm-restate,amssymb}
\usepackage[euler-digits,euler-hat-accent]{eulervm}
\usepackage[bookmarksnumbered,linktocpage,hypertexnames=false,colorlinks=true,linkcolor=blue,urlcolor=blue,citecolor=blue,anchorcolor=green,pdfusetitle]{hyperref}
\usepackage{braket,microtype,fullpage,graphicx,enumitem,float,caption,bm,mathdots,booktabs}
\usepackage{comment}
\usepackage[capitalize]{cleveref}
\usepackage[dvipsnames]{xcolor}
\usepackage{tikz-cd}
\tikzset{down/.style={anchor=south, rotate=-45, inner sep = -.5mm, color=gray}}
\tikzset{up/.style={anchor=south, rotate=45, inner sep=-.5mm, color=gray}}
\usepackage[numbers,sort&compress]{natbib}
\usepackage{footnote}
\makesavenoteenv{tabular}
\makesavenoteenv{table}

\let\counterwithin=\relax

\usepackage{chngcntr}

\newtheorem{thm}{Theorem}[section]\crefname{thm}{Theorem}{Theorems}
\newtheorem{lem}[thm]{Lemma}\crefname{lem}{Lemma}{Lemmas}
\newtheorem{prb}[thm]{Problem}\crefname{prb}{Problem}{Problems}
\newtheorem{cor}[thm]{Corollary}\crefname{cor}{Corollary}{Corollaries}
\newtheorem{dfn}[thm]{Definition}\crefname{dfn}{Definition}{Definitions}
\newtheorem{prp}[thm]{Proposition}\crefname{prp}{Proposition}{Propositions}
\theoremstyle{definition}
\newtheorem{rem}[thm]{Remark}\crefname{rem}{Remark}{Remarks}
\newtheorem{exa}[thm]{Example}\crefname{exa}{Example}{Examples}
\crefname{figure}{Figure}{Figures}

\floatstyle{boxed}\newfloat{Algorithm}{ht}{alg}\crefname{Algorithm}{Algorithm}{Algorithms}
\counterwithin{figure}{section}
\counterwithin{table}{section}
\counterwithin{Algorithm}{section}
\numberwithin{equation}{section}
\allowdisplaybreaks[4]

\DeclareMathOperator*{\argmin}{arg\,min}

\DeclareMathOperator{\tr}{tr}
\DeclareMathOperator{\reg}{reg}
\DeclareMathOperator{\Mat}{Mat}
\DeclareMathOperator{\diag}{diag}
\DeclareMathOperator{\aff}{aff}
\DeclareMathOperator{\adj}{adj}
\DeclareMathOperator{\PD}{PD}
\DeclareMathOperator{\SL}{SL}
\DeclareMathOperator{\ST}{ST}
\DeclareMathOperator{\GL}{GL}
\DeclareMathOperator{\T}{T}

\DeclareMathOperator{\Herm}{Herm}

\DeclareMathOperator{\spec}{spec}

\DeclareMathOperator{\Sym}{Sym}
\DeclareMathOperator{\poly}{poly}
\DeclareMathOperator{\rk}{rk}

\DeclareMathOperator{\conv}{conv}
\DeclareMathOperator{\capacity}{cap}
\DeclareMathOperator{\capa}{cap}

\DeclareMathOperator{\supp}{supp}
\DeclareMathOperator{\Lie}{Lie}
\DeclareMathOperator{\ope}{op}
\DeclareMathOperator{\intopoly}{s}
\DeclareMathOperator{\lcm}{lcm}
\DeclareMathOperator{\Rep}{Rep}

\renewcommand{\Re}{\operatorname{Re}}

\newcommand{\sm}{\sigma_{\min}}
\newcommand{\CC}{\mathbb C}
\newcommand{\RR}{\mathbb R}
\newcommand{\PP}{\mathbb P}
\newcommand{\ZZ}{\mathbb Z}
\newcommand{\NN}{\mathbb N}
\newcommand{\QQ}{\mathbb Q}
\newcommand{\ot}{\otimes}
\newcommand{\op}{\oplus}
\newcommand{\eps}{\varepsilon}
\renewcommand{\vec}{\bm}
\newcommand{\email}[1]{\href{mailto:#1}{#1}}
\newcommand{\abs}[1]{\lvert#1\rvert}
\newcommand{\norm}[1]{\lVert#1\rVert}
\newcommand{\Norm}[1]{\left\lVert#1\right\rVert}
\newcommand{\cO}{\mathcal{O}}
\newcommand{\cN}{\mathcal{N}}
\newcommand{\Id}{\mathbf{I}}
\newcommand{\cS}{\mathcal{S}}
\newcommand{\ncdeg}{\sigma}
\newcommand{\gendeg}{\beta}

\newcommand{\cR}{\mathcal R}

\newcommand{\sgn}{\mathrm{sgn}}
\newcommand{\SU}{\mathrm{SU}}
\newcommand{\cK}{K}

\title{\LARGE Towards a theory of non-commutative optimization:\\ geodesic first and second order methods\\ for moment maps and polytopes}
\author{Peter B\"{u}rgisser%
\thanks{Institut f\"{u}r Mathematik, Technische Universit\"{a}t
  Berlin, \email{pbuerg@math.tu-berlin.de}.
Supported by DFG grant BU 1371 2-2 and by the ERC under the European's Horizon 2020 research and innovation programme (grant agreement no. 787840).}
\and
Cole Franks%
\thanks{Department of Mathematics, Rutgers University and MIT, \email{franks@mit.edu}. Partially supported by Simons Foundation award 32622.}
\and
Ankit Garg%
\thanks{Microsoft Research India, \email{garga@microsoft.com}.}
\and
Rafael Oliveira%
\thanks{Cheriton School of Computer Science, University of Waterloo, \email{rafael@uwaterloo.ca}.  Research done while Rafael was at Princeton, University of Toronto and at the Simons Institute for the Theory of Computing.}
\and
Michael Walter%
\thanks{Korteweg-de Vries Institute for Mathematics, Institute for
  Theoretical Physics, Institute for Logic, Language, and Computation,
  QuSoft, University of Amsterdam, \email{m.walter@uva.nl}. Supported by NWO Veni grant no.~680-47-459 and NWO grant OCENW.KLEIN.267.}
\and
Avi Wigderson%
\thanks{Institute for Advanced Study, Princeton, \email{avi@ias.edu}. Supported by NSF grants CCF-1412958 and CCF-1900460.}
}%

\date{\vspace{-6.5ex}}
\begin{document}
\pagenumbering{gobble}
\maketitle
\vspace{-5pt}
\begin{abstract}
This paper initiates a systematic development of a theory of \emph{non-commutative optimization}.
It aims to unify and generalize a growing body of work from the past few years which developed and analyzed algorithms for natural \emph{geodesically convex} optimization problems on Riemannian manifolds that arise from the symmetries of non-commutative groups.
These algorithms minimize the \emph{moment map} (a non-commutative notion of the usual \emph{gradient}) and test membership in \emph{null cones} and \emph{moment polytopes} (a vast class of polytopes, typically of exponential vertex and facet complexity, which arise from this a priori non-convex, non-linear setting).
This setting captures a diverse set of problems in different areas of computer science, mathematics, and physics.
Several of them were solved efficiently for the first time using non-commutative methods; the corresponding algorithms also lead to solutions of purely structural problems and to many new connections between disparate fields.

In the spirit of standard convex optimization, we develop two general methods in the geodesic setting, a first order and a second order method, which respectively receive first and second order information on the ``derivatives'' of the function to be optimized.
These in particular subsume all past results.
The main technical work 
goes into identifying the key parameters of the underlying group actions which control convergence to the optimum in each of these methods.
These non-commutative analogues of ``smoothness'' are far more complex and require significant algebraic and analytic machinery.
Despite this complexity, the way in which these parameters control convergence in both methods is quite simple and elegant.
We show how to bound these parameters and hence obtain efficient algorithms for null cone membership in several concrete situations.
Our work points to intriguing open problems and suggests further research directions.
\end{abstract}
\clearpage
\tableofcontents
\pagenumbering{arabic}

\section{Introduction}\label{sec:intro}

\subsection{High-level overview}\label{high-level}
Consider a group $G$ that acts by \emph{linear} transformations on the complex Euclidean space $V=\CC^m$.
This partitions~$V$ into \emph{orbits}: For a vector $v\in V$, the orbit ${\cal O}_v$ is simply all vectors of the form $g \cdot v$ to which the action of a group element~$g \in G$
can map~$v$.

The most basic algorithmic question in this setting is as follows.
Given a vector~$v\in V$, compute (or approximate) the smallest $\ell_2$-norm of any vector in the orbit of~$v$, that is, $\inf  \{ \|w\|_2 : w\in {\cal O}_v\}$.
Remarkably, this simple question, for different groups $G$, captures natural important problems in computational complexity, algebra, analysis, and quantum information.
Even when restricted only to \emph{commutative} groups, it already captures all linear programming problems!

Starting with~\cite{garg2016deterministic}, a series of recent works
including~\cite{garg2017algorithmic, burgisser2017alternating, franks2018operator, kwok2018paulsen, AGLOW18, burgisser2018efficient}
designed algorithms and analysis tools to handle this basic and other related optimization problems over \emph{non-commutative} groups $G$.
We note that in all these works, the groups at hand need to be products of at least two copies of rather specific linear groups ($\SL(n)$'s or tori),
to support the algorithms and analysis.
These provided efficient solutions for some applications, and \emph{through algorithms}, the resolution of some purely structural mathematical open problems.
We will mention some of these below.


A great deal of understanding gradually evolved in this sequence of works.
These new algorithms are all essentially iterative methods, progressing from the input vector~$v$ to the desired optimum in small steps, as do convex optimization algorithms.
This seems surprising, as the basic question above is patently \emph{non-convex} for non-commutative groups (in the commutative case, a simple change of variables
discussed below convexifies the problem).
Indeed, neither the domain nor the function to be optimized are convex!
However, in hindsight, a key to all of them are the notions of \emph{geodesic convexity} (which generalizes the familiar Euclidean notion of convexity) and the \emph{moment map}
(which generalizes the familiar Euclidean gradient) in the curved space and new metrics induced by the group action.
A rich duality theory of geometric invariant theory (greatly generalizing LP duality), together with tools from algebraic geometry, representation theory and differential equations
are used in the convergence analysis of these algorithms.

The main objective of this paper is to unify and generalize these works, in a way which naturally extends the familiar first and second order methods of standard convex optimization.
We design geodesic analogs of these methods, which, respectively, have oracle access to first and second order ``derivatives'' of the function being optimized, and apply to \emph{any} (reductive) group action.

Our first order method (which is a non-commutative version of gradient descent) replaces and extends the use of ``alternate minimization'' in most past works, and thus can accommodate more general group actions.
Indeed, the use of ``alternating minimization'' requires the acting group to be a product of two or more (reductive) groups.
For instance, this covers the cases of symmetric and antisymmetric tensors with the standard action of~$\SL(n)$, where alternating minimization does not apply.

Our second order method greatly generalizes the one used for the particular group action corresponding to operator scaling in~\cite{AGLOW18}.
It may be thought of as a geodesic analog of the ``trust region method''~\cite{CGT00} or the ``box-constrained Newton method''~\cite{cohen2017matrix, allen2017much}
applied to a regularized function.
For both methods, in this non-commutative setting, we recover the familiar convergence behavior of the  classical commutative case: to achieve ``proximity'' $\eps$
to the optimum, our first order method converges in~$O(1/\eps)$ iterations and our second order method in~$O(\poly\log (1/\eps))$ iterations.

As in standard convex optimization, the fundamental challenge is to understand the ``constants'' hidden in the big-$O$ notation of each method.
These depend on ``smoothness'' properties of the function optimized, which in turn are determined by the action of the group $G$ on the space $V$ that defines the optimization problem.
\emph{The main technical contributions of the theory we develop are to identify the key parameters which control this dependence, and to bound them for various actions to obtain concrete running time bounds.}
These parameters depend on a combination of algebraic and geometric properties of the group action, in particular the irreducible representations occurring in it.
As mentioned, despite the technical complexity of defining (and bounding) these parameters, the way they control convergence of the algorithms is surprisingly elegant.

We also develop important technical tools which naturally extend ones common in the commutative theory, including regularizers, diameter bounds, numerical stability, and initial starting points, which are key to the design and analysis of the presented (and hopefully future) algorithms in the geodesic setting.

As in previous works, we also address other optimization problems beyond the basic ``norm minimization'' question above, in particular the minimization of the moment map (which turns out to be a dual problem), and the membership problem for \emph{moment polytopes};
a very rich class of polytopes (typically with exponentially many vertices and facets) which arises magically from any such group action.

The paradigm of optimization described above resulted in efficient algorithms for problems from various diverse areas of CS and mathematics.
We mention some of these applications in the following \cref{subsec:unexpected}.
\Cref{subsec:nc_primer} describes the basic setting of non-commutative optimization. 
Next, in \cref{subsubsec:comp prob}, we formally define the problems we are studying and survey what is known about them in the commutative and non-commutative case.
In \cref{subsec:results,subsec:concrete_time}, we describe our contributions and results.
Finally, in \cref{subsec:null cone applications}, we specialize our results to some concrete cases
in which we obtain better complexity results.

\subsection{Some unexpected applications and connections}\label{subsec:unexpected}
We mention here some of the diverse applications of the paradigm of optimization over non-commutative groups, which have been uncovered in previous works and greatly motivate a general algorithmic theory of the nature we explore here:
\begin{enumerate}
\item\label{it:pit} \textbf{Algebraic identities:} Given an arithmetic formula (with inversion gates) in
non-commuting variables, is it identically zero?%
\footnote{For example, Hua's identity states that $(x+y)^{-1} + (x+xy^{-1}x)^{-1}) = x^{-1}$ for noncommuting variables $x, y$.}
This is the word problem for the free skew field.%

\item\label{it:qit} \textbf{Quantum information:} Given density matrices describing local quantum
states of various parties, is there a global pure state consistent with the local states?

\item\label{it:horn} \textbf{Eigenvalues of sums of Hermitian matrices:} Given three real $n$-vectors, do there exist
three Hermitian $n\times n$ matrices $A$, $B$, $C$ with these \emph{prescribed} spectra, such that $A+B=C$?

\item\label{it:brascamp} \textbf{Analytic inequalities:}
Given $m$ linear maps $A_i\colon\RR^n \rightarrow \RR^{n_i}$ and $p_1,\ldots, p_m \ge 0$, does there
exist a finite constant $C$ such that for all integrable functions $f_i\colon \RR^{n_i} \rightarrow \RR_+$ we have
\begin{align*}
 \textstyle \int_{x\in \RR^n} \prod_{i=1}^m  f_i(A_i x) dx \,\, \leq C\, \prod_{i=1}^m  \norm{f_i}_{1/p_i}?
\end{align*}
These inequalities are the celebrated Brascamp-Lieb inequalities, which capture the Cauchy-Schwarz, H\"{o}lder, Loomis-Whitney, Brunn-Minkowski, and many further inequalities.

\item\label{it:MLE}\textbf{Maximum likelihood estimation:}
Given~$m$ samples from the \emph{matrix} or \emph{tensor normal model}, popular statistical models consisting of Gaussians with Kronecker product covariance matrix~$\Psi_1 \ot \cdots \ot \Psi_k$, how can we find the maximum likelihood estimate given these samples?
\end{enumerate}
At first glance, it is far from obvious that solving any of these problems has any relation to \emph{either} optimization or groups.
We will clarify this mystery below, showing not only how symmetries naturally exist in all of them, but also how these help both in formulating them as optimization problems over groups, suggesting natural algorithms (or at least heuristics) for them, and finally in providing tools for analyzing these algorithms.
It perhaps should be stressed that symmetries exist in many examples in which the relevant groups are commutative (e.g., perfect matching in bipartite graphs, matrix scaling, and more generally in linear, geometric, and hyperbolic programming);
however in these cases, standard convex optimization or combinatorial algorithms can be designed and analyzed \emph{without} any reference to these existing symmetries.
Making this connection explicit is an important part of our exposition.

A polynomial time algorithm for Problem~\ref{it:pit} was first given in~\cite{garg2016deterministic}; the works~\cite{ivanyos2017noncommutative,derksen2015,ivanyos2017constructive} later discovered completely different algebraic algorithms.
For Problem~\ref{it:qit}, efficient algorithms were given in~\cite{burgisser2018efficient}; cf.~\cite{verstraete2003normal} and the structural results~\cite{klyachko2004quantum,daftuar2005quantum,christandl2006spectra,christandl2007nonzero,walter2013entanglement,walter2014multipartite,christandl2014eigenvalue}.
For Problem~\ref{it:horn}, the celebrated structural result in~\cite{KT99} and the algorithmic results of~\cite{deloera2006computation,mulmuley2012geometric,burgisser2013positivity} solved the decision problem in polynomial time, while~\cite{franks2018operator} gave an algorithm for the search problem.
For Problem~\ref{it:brascamp}, an efficient algorithm was presented in~\cite{garg2017algorithmic}.
However the algorithms in~\cite{garg2017algorithmic,franks2018operator,burgisser2018efficient} remain exponential time in various input parameters, exemplifying only one aspect of many in which the current theory and understanding is lacking.
For Problem~\ref{it:MLE}, \cite{AKRS:19} connected the problem of finding maximum likelihood estimates (MLE) and the `flip flop' algorithm from statistics to the scaling algorithms of~\cite{garg2016deterministic,burgisser2017alternating} and the geodesic optimization framework of this paper.
The minimal number of samples required for the MLE to (almost surely) exist uniquely was determined in~\cite{DM:20a} for the matrix normal model ($k=2$) and in~\cite{DMW:20} for the general tensor normal model; polynomial-time algorithms that achieve nearly optimal sample complexity MLE are given in \cite{franks21near}.

The unexpected connections revealed in this study are far richer than the mere relevance of optimization and symmetries to such problems.
One type are connections between problems in disparate fields.
For example, the analytic Problem~\ref{it:brascamp} turns out to be a special case of the algebraic Problem~\ref{it:pit}.
Moreover, Problem~\ref{it:pit} has (well-studied) differently looking but equivalent formulations in quantum information theory and in invariant theory, which are automatically solved by the same algorithm.
Another type of connections are of purely structural open problems solved through such geodesic algorithms, reasserting the importance of the computational lens in mathematics.
One was the first dimension-independent bound on the Paulsen problem in operator theory, obtained ingeniously through such an algorithm in~\cite{kwok2018paulsen} (this work was followed by~\cite{hamilton2018paulsen}, which gave a strikingly simpler proof and stronger bounds).
Another was a quantitative bound on the continuity of the best constant~$C$ in Problem~\ref{it:brascamp} (in terms of the input data), important for non-linear variants of such inequalities.
This bound was obtained through the algorithm in~\cite{garg2017algorithmic} and relies on its efficiency; previous methods used compactness arguments that provided no bounds.

We have no doubt that more unexpected applications and connections will follow.
The most extreme and speculative perhaps among such potential applications is to develop a deterministic polynomial-time algorithm for the polynomial identity testing (PIT) problem.
Such an algorithm will imply major algebraic or Boolean lower bounds, namely either separating VP from VNP, or proving that NEXP has no small Boolean circuits~\cite{kabanets2004derandomizing}.
We note that this goal was a central motivation of the initial work in this sequence~\cite{garg2016deterministic}, which provided such a deterministic algorithm for Problem~\ref{it:pit} above, the non-commutative analog of PIT.
The ``real'' PIT problem (in which variables commute) also has a natural group of symmetries acting on it, which does not quite fall into the frameworks developed so far (including the one of this paper).
Yet, the hope of proving lower bounds via optimization methods is a fascinating (and possibly achievable) one.
This agenda of hoping to shed light on the PIT problem by the study of invariant theoretic questions was formulated in the fifth paper of the Geometric Complexity Theory (GCT) series~\cite{mulmuley2012,mulmuley2017geometric}.

\subsection{Non-commutative optimization: a primer}\label{subsec:nc_primer}
We now give an introduction to non-commutative optimization and contrast its geometric structure and convexity properties with the familiar commutative setting.
The basic setting is that of a continuous group
$G$ acting (linearly) on an $m$-dimensional complex vector space~$V \cong \CC^m$.
For this section, and the rest of the introduction, think of~$G$ as either the group of~$n \times n$ complex invertible matrices, denoted $\GL(n)$, or the group of \emph{diagonal} such matrices, denoted~$\T(n)$.
The latter corresponds to the commutative case and the former is a paradigmatic example of the non-commutative case.
An \emph{action} (also called a \emph{representation}) of a group $G$ on a complex vector space $V$ is a group homomorphism~$\pi: G \rightarrow \GL(V)$, that is, an association of an invertible linear map~$\pi(g)$ from $V \to V$ for every group element~$g\in G$ satisfying $\pi(g_1 g_2) = \pi(g_1) \pi(g_2)$ for all~$g_1,g_2\in G$.
Further suppose that $V$ is also equipped with an inner product $\langle \cdot, \cdot \rangle$ and hence a norm $\norm{v}:=\langle v, v\rangle$.%
\footnote{In general, the theory works whenever the group is connected, algebraic and reductive, and our results hold in this generality.
However, for purposes of exposition we only discuss very simple groups in this introduction.
We also suppress some technical details which are spelled out later, e.g., that the representations are rational and map unitary matrices to unitary matrices (both are essentially without loss of generality).}

Given a vector $v\in V$ one can consider the optimization problem of taking the infimum of the norm in the \emph{orbit} of the vector $v$ under the action of~$G$.
More formally, define the \emph{capacity} of~$v$ by%
\footnote{For notational convenience, we suppress the dependence of~$\capacity(v)$ on the group~$G$ and representation~$\pi$ (likewise for the null cone and the moment polytopes defined below).}
\begin{align*}
  \capacity(v) := \inf_{g \in G} \norm{\pi(g) v}.
\end{align*}
This notion generalizes the matrix and operator capacities developed in~\cite{GurYianilos, gurvits2004classical} (to see this, carry out the optimization over one of the two group variables) as well as the polynomial capacity of~\cite{gurvits2006hyperbolic}.
It turns out that this simple-looking optimization problem is already very general in the commutative case and, in the non-commutative case, captures \emph{all} examples discussed in \cref{subsec:unexpected}.

Let us first consider the commutative case, $G = \T(n)$ acting on $V$.
In this simple case, \emph{all} actions~$\pi$ have a very simple form.
We give two equivalent descriptions, first of how any representation~$\pi$ splits into one-dimensional irreducible representations, and another describing~$\pi$ as a natural scaling action on~$n$-variate polynomials with~$m$ monomials.

The irreducible representations are given by an orthonormal basis $v_1,\ldots, v_m$ of $V$ such that the~$v_j$ are simultaneous eigenvectors of all the matrices~$\pi(g)$.
That is, for all $g = \diag(g_1,\ldots, g_n) \in \T(n)$ and~$j \in [m]$,
\begin{equation}\label{eqn:action on weight vectors}
 \pi(g) v_j = \lambda_j(g) v_j ,\quad \text{where} \quad \lambda_j(g) = \textstyle\prod_{i=1}^n g_i^{\omega_{j,i}}
\end{equation}
for fixed integer vectors~$\omega_1,\ldots, \omega_m \in \ZZ^n$, which are called \emph{weights} and encode the simultaneous eigenvalues, and completely determine the action.
Below we also refer to the weights of representation~$\pi$ of~$\GL(n)$, defined as the weights of~$\pi$ restricted to $\T(n)$.

A natural way to view all these actions is as follows.
The natural action of $\T(n)$ on $\CC^n$ by matrix-vector multiplication, induces an action of $\T(n)$ on $n$-variate polynomials $V = \CC[x_1, x_2, \dots ,x_n]$:
simply, any group element $g = \diag(g_1,\ldots, g_n)$ ``scales'' each $x_i$ to $g_ix_i$.
Note that any monomial~$x^\omega = \prod_{i=1}^n x_i^{\omega(i)}$ (where $\omega$ is the integer vector of exponents) is an eigenvector of this action, with an eigenvalue $\lambda(g) = \prod_{i=1}^n g_i^{\omega(i)}$.

Now fix $m$ integer vectors $\omega_j$ as above.
Consider the linear space of $n$-variate Laurent polynomials (i.e., polynomials where the variables can have negative exponents, too) with the following $m$ monomials:
$v_j = x^{\omega_j}=\prod_{i=1}^n x_i^{\omega_{j,i}}$.
The action on any polynomial $v =\sum_{j=1}^m c_j v_j$ is precisely the one described above, scaling each coefficient by the eigenvalue of its monomial.
The norm $\|v\|$ of a polynomial is the sum of the square moduli of its coefficients.
Now let us calculate the capacity of this action.
For any $v =\sum_{j=1}^m c_j v_j$,
\begin{align}\label{eq:abelian cap}
\textstyle
  \capacity^2(v) = \inf_{g_1,\ldots, g_n \in \CC^*} \sum_{j=1}^m |c_j|^2 \prod_{i=1}^n |g_i|^{2 \omega_{j,i}} =  \inf_{x \in \RR^n} \sum_{j=1}^m |c_j|^2 e^{ x \cdot \omega_j},
\end{align}
where we used the change of variables $x_i = \log |g_i|^2$, which makes the problem convex (in fact, log-convex)!
This class of optimization problems (of optimizing norm in the orbit of a commutative group) is known as \emph{geometric programming} and is well-studied in the optimization literature (see, e.g., Chapter~4.5 in~\cite{boyd2004convex}).
Hence for non-commutative groups, one can view $\capacity(v)$ as \emph{non-commutative geometric programming}.%
\footnote{The term \emph{non-commutative optimization} as we use it here should not be confused with polynomial optimization in non-commuting variables, which is also referred to as \emph{non-commutative polynomial optimization} (see, e.g., \cite{pironio2010convergent}).}
Is there a similar change of variables that makes the problem convex in the non-commutative case?
It does not seem so.
However, the non-commutative case also satisfies a notion of convexity, known as geodesic convexity, which we will study next.

\subsubsection{Geodesic convexity}
Geodesic convexity generalizes the notion of convexity in the Euclidean space to arbitrary Riemannian manifolds.
We will not go into the notion of geodesic convexity in this generality but just mention what it amounts to in our concrete setting of norm optimization for~$G=\GL(n)$.

It turns out the appropriate way to define geodesic convexity in this case is as follows.
Fix an action~$\pi$ of~$\GL(n)$ and a vector~$v$.
Then $\log \norm{\pi(e^{tH} g) v}$ is convex in the real parameter $t$ for every Hermitian matrix~$H$ and~$g \in \GL(n)$.
This notion of convexity is quite similar to the notion of Euclidean convexity, where a function is convex iff it is convex along all lines.
However, it is far from obvious how to import optimization techniques from the Euclidean setting to work in this non-commutative geodesic setting.
An essential ingredient we describe next is  the geodesic notion of a gradient, called the \emph{moment map}.

\subsubsection{Moment map}
The moment map is by definition the gradient of the function~$\log \norm{\pi(g) v}$ (understood as a function of~$v$),  at the identity element of the group, $g=I$.
It captures how the norm of the vector~$v$ changes when we act on it by infinitesimal perturbations of the identity.

Again, we start with the commutative case $G = \T(n)$ acting on the multivariate Laurent polynomials.
For a (``direction'') vector $h\in \RR^n$ and a real (``perturbation'') parameter $t$, let~$e^{th} = \diag\left(e^{t h_1},\ldots, e^{t h_n} \right)$.
Then, for~$G = \T(n)$, the  moment map is the function $\mu\colon V\setminus \{0\} \rightarrow \RR^n$, defined by the following property:
\begin{align*}
\mu(v)\cdot h  = \partial_{t=0} \left[\log \Norm{\pi(\diag(e^{t h}) v} \right],
\end{align*}
for all $h \in \RR^n$.
That is, the directional derivative in direction $h$ is given by the dot product $ \mu(v) \cdot h $.
Here one can see that the moment map matches the notion of Euclidean gradient.
For the action of~$\T(n)$ in \cref{eqn:action on weight vectors},
\begin{align}\label{eqn:ankit602}
\mu(v) = \nabla_{x = 0} \log \Big(\sum_{j=1}^m |c_j|^2 e^{ x \cdot \omega_j} \Big) = \frac{\sum_{j=1}^m |c_j|^2 \omega_j}{\sum_{j=1}^m |c_j|^2}.
\end{align}
Note that the gradient $\mu(v)$ at any point $v$ is a convex combination of the weights!
Viewing $v$ as a polynomial, the gradient thus belongs to the so-called \emph{Newton polytope} of $v$, namely the convex hull of the exponent vectors of its monomials!
Conversely, every point in that polytope is a gradient of some polynomial $v$ with these monomials.
We will soon return to this curious fact!

We now proceed to the non-commutative case, focusing on $G = \GL(n)$.
Denote by $\Herm(n)$ the set of $n \times n$ complex Hermitian matrices.
Here ``directions'' will be parametrized by $H \in \Herm(n)$.
For the case of $G = \GL(n)$, the moment map is the function $\mu\colon V\setminus \{0\} \rightarrow \Herm(n)$
defined (in complete analogy to the commutative case above) by the following property that
\begin{align*}
\tr[\mu(v) H] = \partial_{t = 0} \left[\log \Norm{\pi(e^{t H}) v} \right]
\end{align*}
for all $H \in \Herm(n)$.
That is, the directional derivative in direction $H$ is given by $\tr[\mu(v) H]$.

\begin{rem}The reason we are restricting to directions in~$\RR^n$ in the~$\T(n)$ case and to directions in~$\Herm_n$ in the~$\GL(n)$ case is that imaginary and skew-Hermitian directions, respectively, do not change the norm.
\end{rem}

In the commutative case, \cref{eqn:ankit602} is a convex combination of the weights~$\omega_j$.
Thus, the image of~$\mu$ is the convex hull of the weights -- a convex polytope.
This brings us to moment polytopes.

\subsubsection{Moment polytopes}\label{subsubsec:moment polytopes}
One can ask whether the above fact is true for actions of ~$\GL(n)$ i.e., is the set~$\{\mu(v) : v\in V\setminus \{0\}\}$ convex?
This turns out to be blatantly false.
Consider the action of $\GL(n)$ on $\CC^n$ by matrix-vector multiplication.
The moment map in this setting is~$\mu(v) = v v^{\dagger}/\norm{v}^2$, and its image is clearly not convex.
However, there is still something deep and non-trivial that can be said.
Given a Hermitian matrix~$H \in \Herm(n)$, define its \emph{spectrum} to be the vector of its eigenvalues arranged in non-increasing order.
That is, $\spec(H) := (\lambda_1,\ldots, \lambda_n)$, where $\lambda_1 \ge \cdots \ge \lambda_n$ are the eigenvalues of~$H$.
Amazingly, the set of spectra of moment map images, that is,
\begin{align}\label{eq:big moment polytope}
\Delta := \left\{ \spec\bigl(\mu(v)\bigr) : 0 \neq v\in V \right\},
\end{align}
is a convex polytope for every representation~$\pi$~\cite{ness1984stratification,Kostant73, atiyah1982convexity, GS82_convexity, kirwan1984convexity}!
These polytopes are called \emph{moment polytopes}.

Let us mention two important examples of moment polytopes.
The examples are for actions of products of $\GL(n)$'s but the above definitions generalize almost immediately.

\begin{exa}[Horn's problem]\label{exa:horn}
Let $G = \GL(n) \times \GL(n) \times \GL(n)$ act on $V = \Mat(n) \oplus \Mat(n)$ as follows:
$\pi(g_1,g_2, g_3) (X, Y) := (g_1 X g_3^{-1}, g_2 Y g_3^{-1})$.
The moment map in this case is
\begin{align*}
  \mu(X,Y) = \frac{(  X X^{\dagger}, Y Y^{\dagger}, -(X^\dagger X + Y^\dagger Y))}{\norm{X}_F^2 + \norm{Y}_F^2}.
\end{align*}
Using that $XX^\dagger$ and $X^\dagger X$ are PSD and isospectral, we obtain the following moment polytopes, which characterize the eigenvalues of sums of Hermitian matrices, i.e., Horn's problem (see, e.g.,~\cite{fulton2000eigenvalues,berline2016horn}):
\begin{align*}
  \Delta = \bigl\{\left( \spec(A), \spec(B), \spec(-A - B) \right) \;:\; A, B \in \Mat(n), \, A \geq 0, \, B \geq 0, \, \tr A + \tr B = 1 \bigr\}
\end{align*}
These polytopes are known as the \emph{Horn polytopes} and correspond to Problem~\ref{it:horn} in \cref{subsec:unexpected}.
They have been characterized mathematically in~\cite{klyachko1998stable,KT99,belkale2006eigenvalue,ressayre2010geometric} and algorithmically in~\cite{deloera2006computation,mulmuley2012geometric,burgisser2013positivity}.
\end{exa}

The preceding is one of the simplest examples of a moment polytope associated with the representation of a quiver;
in this case, the so-called \emph{star quiver} with two edges, see~\cref{fig:quivers},~(a).
We briefly give the relevant definitions and refer to~\cite{derksen2017introduction} for more details.

\begin{exa}[Quivers]\label{exa:quivers}
A \emph{quiver} is a directed graph~$Q$ with loops and parallel edges allowed (see \cref{fig:quivers} for two examples).
Formally, we have a set $Q_0$ of vertices, a set $Q_1$ of arrows and two maps $h,t\colon Q_1 \to Q_0$
assigning to an arrow $a\in Q_1$ its head $ha$ and tail $ta$.
A representation of the quiver~$Q$ (not to be confused with a group representation!) assigns a vector space~$W_x=\CC^{n_x}$
to each vertex~$x\in Q_0$
and a linear map~$X_a\colon W_{ta}\to W_{ha}$ to each edge $a\in Q_1$.  
Thus, if we fix a dimension vector~$\vec n=(n_x)$,
the representations form a vector space~$V:=\Rep(Q,{\vec n}):= \bigoplus_{a\in Q_1} \Mat(n_{ha},n_{ta})$,
called the representation space of $\vec n$-dimensional representations of~$Q$.
This space carries a natural action of the group~$G = \prod_{x\in Q_0} \GL(n_x)$ defined as follows.
Let $X=(X_a) \in V$ and $g=(g_x) \in G$. Then we define $Y := \pi(g)X$ by setting
$Y_a := g_{ha} X_a g_{ta}^{-1}$ for $a\in Q_1$.
We will call this action of~$G$ on~$V$ the \emph{group representation} associated with the quiver~$Q$ and dimension vector~$\vec n$.
Regarding their moment polytopes, an important slice corresponding to semi-invariants has been
characterized in~\cite{king1994moduli,derksen2000semi,schofield2001semi,ressayre2012git}
and the polytopes have been described completely in~\cite{baldoni2018horn,baldoni2019horn}.
\end{exa}

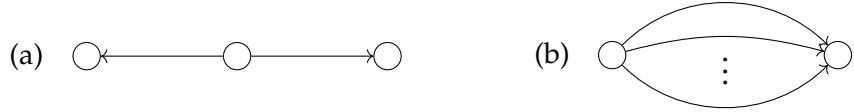
\begin{figure}
\centering
\raisebox{.57cm}{\begin{tikzpicture}
    \node at (-0.8,0) {(a)};
    \node[shape=circle,draw=black] (A) at (0,0) {};
    \node[shape=circle,draw=black] (B) at (2,0) {};
    \node[shape=circle,draw=black] (C) at (4,0) {};
    \path [->] (B) edge (A);
    \path [->] (B) edge (C);
\end{tikzpicture}}
\qquad
\qquad
\begin{tikzpicture}
    \node at (-0.8,0) {(b)};
    \node[shape=circle,draw=black] (A) at (0,0) {};
    \node[shape=circle,draw=black] (B) at (3,0) {};
    \path [->] (A) edge[bend left=45] (B);
    \path [->] (A) edge[bend left=15] (B);
    \path [->] (A) edge[bend right=45] (B);
    \node at (1.5,-0.1) {$\vdots$};
\end{tikzpicture}
\caption{Two examples of quivers.
(a) The star quiver with two edges. The associated moment polytopes model Horn's problem (\cref{exa:horn}).
(b) The generalized Kronecker quiver with~$k$ parallel edges, which corresponds to a variant of the left-right action (cf.~\cref{exa:kronecker_quiver}).}\label{fig:quivers}
\end{figure}

\begin{exa}[Tensor action]\label{exa:tensor}
$G = \GL(n_1) \times \GL(n_2) \times \GL(n_3)$ acts on $V = \CC^{n_1} \otimes \CC^{n_2} \otimes \CC^{n_3}$ as follows:
$\pi(g_1,g_2, g_3) v := (g_1 \otimes g_2 \otimes g_3) v$.
We can think of vectors $v\in V$ as tripartite quantum states with local dimensions $n_1,n_2,n_3$.
Then the moment map for this group action captures precisely the notion of \emph{quantum marginals}.
That is, $\mu(v) = (\rho_1, \rho_2, \rho_3)$, where $\rho_k = \tr_{k^c}(vv^\dagger)$ denotes the reduced density matrix describing the state of
the $k$\textsuperscript{th} particle.
This corresponds to Problem~\ref{it:qit} in \cref{subsec:unexpected}.
The moment polytopes in this case are known as \emph{Kronecker polytopes}, since they can be equivalently described in terms of the
Kronecker coefficients of the symmetric group.
These polytopes have been studied in~\cite{klyachko2004quantum,daftuar2005quantum,christandl2006spectra,christandl2007nonzero,walter2013entanglement,walter2014multipartite,christandl2014eigenvalue,vergne2014inequalities,burgisser2018efficient}.

The same action on \emph{tuples} of tensors corresponds to the maximum likelihood problem for the tensor normal model (Problem~\ref{it:MLE} in \cref{subsec:unexpected}, see also \cref{exa:operator_scaling}).
\end{exa}

There is a more refined notion of a moment polytope.
One can look at the collection of spectra of moment maps of vectors in the orbit of a particular vector~$v \in V$.
Its closure,
\begin{align*}
  \Delta(v) := \overline{\left\{ \spec\bigl(\mu(w)\bigr): w \in \cO_v \right\}}
\end{align*}
is a convex polytope as well, called the \emph{moment polytope of~$v$}~\cite{ness1984stratification, brion1987image}!
It can equivalently be defined as the spectra of moment map images of the orbit's closure in projective space.

\subsubsection{Null cone}\label{subsubsec:null_cone}
Fix a representation $\pi$ of a group $G$ on a vector space $V$ (recall $G$ is $\T(n)$ or $\GL(n)$ for the introduction).
The \emph{null cone} for this group action is defined as the set of vectors $v$ such that $\capacity(v) = 0$:
\begin{align*}
\cN := \{v \in V: \capacity(v) = 0\}
\end{align*}
In other words, $v$ is in the null cone if and only if~$0$ lies in the orbit-closure of $v$.
It is of importance in invariant theory due to the results of Hilbert and Mumford~\cite{Hil,Mum65} which state that the null cone is the algebraic variety defined by non-constant homogeneous invariant polynomials of the group action (see, e.g., the excellent textbooks~\cite{derksen2015computational, sturmfels2008algorithms}).

Let us see what the null cone for the action of $\T(n)$ in \cref{eqn:action on weight vectors} is.
Recall from \cref{eq:abelian cap}, the formulation for $\capacity(v)$.
It is easy to see that $\capacity(v) = 0$ iff there exists $x \in \RR^n$ such that $ x \cdot \omega_j  < 0$ for all $j \in \supp(v)$, where $\supp(v) = \{j \in [m]: c_j \neq 0\}$ for $v = \sum_{j=1}^m c_j v_j$.
Thus the property of $v$ being in the null cone is captured by a simple linear program defined by $\supp(v)$ and the weights $\omega_j$'s.
Hence the null cone membership problem for non-commutative group actions can be thought of as \emph{non-commutative linear programming}.

We know by Farkas' lemma that there exists~$x \in \RR^n$ such that~$x \cdot \omega_j < 0$ for all~$j \in \supp(v)$ iff~$0$ does not lie in~$\conv\{\omega_j : j \in \supp(v)\}$.
In other words, $\capacity(v) = 0$ iff~$0 \notin \Delta(v)$.
Is this true in the non-commutative world?
It is!
This is the Kempf-Ness theorem~\cite{kempf1979length} and it is a consequence of the geodesic convexity of the function $g \rightarrow \log\norm{\pi(g) v}$.
The Kempf-Ness theorem can be thought of as a \emph{non-commutative duality theory} paralleling the linear programming duality given by Farkas' lemma (which corresponds to the commutative world).
Let us now mention an example of an interesting null cone in the non-commutative case.


\begin{exa}[Operator scaling, or left-right action]\label{exa:operator_scaling}
The action of the group $G = \SL(n) \times \SL(n)$
on $\Mat(n)^k$ via
$\pi(g,h) (X_1, \ldots, X_k) := (g X_1 h^T, \ldots, g X_k h^T)$
is called the \emph{left-right} action.
(Recall $\SL(n)$ denotes the group of $n \times n$ matrices with determinant~1.)
The null cone for this action captures \emph{non-commutative singularity} and the non-commutative rational identity testing problem (Problem~\ref{it:pit}), as well as the maximum likelihood problem for the matrix normal model (Problem~\ref{it:MLE} in \cref{subsec:unexpected}).
As such, the left-right action has been crucial in getting deterministic polynomial time algorithms for this problem~\cite{ivanyos2017noncommutative,garg2016deterministic,derksen2015,ivanyos2017constructive}.
The commutative analogue of Problem~\ref{it:pit} is the famous polynomial identity testing (PIT) problem, for which designing a deterministic polynomial time algorithm remains a major open question in derandomization and complexity theory.
\end{exa}

\begin{exa}[Generalized Kronecker quivers]\label{exa:kronecker_quiver}
The action of $G = \GL(n) \times \GL(n)$ on $k$-tuples of matrices~$(X_1, \dots, X_k)$ via
$\pi(g,h) (X_1, \ldots, X_k) := (g X_1 h^{-1}, \ldots, g X_k h^{-1})$
is sometimes also referred to as the left-right action.
It can be obtained from action of \cref{exa:operator_scaling} via the isomorphism~$h \mapsto (h^{-1})^T$ of~$\GL(n)$.
This action is associated to a quiver, namely the \emph{generalized Kronecker quiver}; see \cref{fig:quivers},~(b).
\end{exa}

\begin{exa}[Simultaneous conjugation]\label{exa:SC}
The action of the group $G = \GL(n)$ on $k$-tuples of matrices in~$V = \left(\Mat(n) \right)^d$
by $\pi(g) (X_1,\ldots, X_k) := (g X_1 g^{-1},\ldots, g X_k g^{-1})$ is associated to the quiver with a single vertex and~$k$ self-loops,
briefly called {\em $k$-loop quiver.}
\end{exa}

\subsection{Computational problems and state of the art}\label{subsubsec:comp prob}
In this section, we describe the main computational questions that are of interest for the optimization problems discussed in the previous section and then discuss what is known about them in the commutative and non-commutative worlds.

\begin{prb}[Null cone membership]\label{prb:exact_null_cone}
Given $(\pi,v)$, determine if $v$ is in the null cone, i.e., if $\capacity(v)=0$.
Equivalently, test if $0 \notin \Delta(v)$.
\end{prb}

\noindent
The null cone membership problem for~$\GL(n)$ is interesting only when the action~$\pi(g)$ is given by rational functions in the~$g_{i,j}$ rather than polynomials.
This is completely analogous to the commutative case (e.g., the convex hull of weights~$\omega_j$ with positive entries never contains the origin).
In the important case that~$\pi$ is homogeneous, the null cone membership problem is interesting precisely when the total degree is zero, so that scalar multiples of the identity matrix act trivially.
Thus, in this case the null cone membership problem for~$G=\GL(n)$ is equivalent to the one for~$G=\SL(n)$.
We will come back to this perspective in \cref{subsec:concrete_time}.

\begin{prb}[Scaling]\label{prb:null_cone}
Given $(\pi, v, \eps)$ such that $0 \in \Delta(v)$, output a group element $g \in G$ such that $\|\spec(\mu(g)v)\|_2 = \|\mu(\pi(g) v) \|_F \leq \eps$.
\end{prb}

\noindent
In particular, the following promise problem can be reduced to \cref{prb:null_cone}:
Given $(\pi, v, \eps)$, decide whether~$0 \not\in \Delta(v)$ under the promise that either~$0 \in \Delta(v)$ or $0$ is $\eps$-far from $\Delta(v)$.
In fact, there always exists~$\eps>0$, depending only on the group action, such that this promise is satisfied!
Thus, the null cone membership problem can always be reduced to the scaling problem (see \cref{cor:null_cone_margin} below).

We develop theory in \cref{subsec:duality} showing that an efficient agorithm to minimize the norm on an orbit closure of a vector~$v$ (i.e., approximate the capacity of~$v$) under the promise that~$0 \in \Delta(v)$ results in an efficient algorithm for the scaling problem and hence for the null cone membership problem.
This motivates our next computational problem.

\begin{prb}[Norm minimization]\label{prb:capacity}
Given $(\pi,v,\eps)$ such that $\capa(v) > 0$, output a group element $g \in G$ such that $\log \norm{\pi(g) v} - \log \capacity(v) \leq \eps$.
\end{prb}

\noindent
We also consider the moment polytope membership problem for an arbitrary point $p \in \QQ^n$.

\begin{prb}[Moment polytope membership]\label{prb:exact_moment_polytope}
Given $(\pi,v,p)$, determine if $p \in \Delta(v)$.
\end{prb}

\noindent
The moment polytope membership problem is more general than the null cone membership problem, but there is a reduction from the former to the latter via the ``shifting trick'' in the next subsection.
This forms the basis of our algorithms for the moment polytope membership problem.
As in the case of the null cone, we consider a scaling version of the moment polytope membership problem.

\begin{prb}[$p$-scaling]\label{prb:moment_polytope}
Given $(\pi,v,p,\eps)$ such that $p \in \Delta(v)$, output an element $g \in G$ such that $\lVert \spec(\mu(\pi(g)v)) - p\rVert_2 \leq \eps$.
\end{prb}

\noindent
The above problem has been referred to as \emph{nonuniform scaling}~\cite{burgisser2018efficient} or, for operators, matrices and tensors, as \emph{scaling with specified or prescribed marginals}~\cite{franks2018operator}.
The following problem can be reduced to \cref{prb:moment_polytope}:
Given $(\pi,v,p,\eps)$, decide whether~$p \in \Delta(v)$ under the promise that either~$p \in \Delta(v)$ or~$p$ is~$\eps$-far from $\Delta(v)$.
We later combine the shifting trick with our duality theory to show that there is a value~$\eps>0$ with bitsize polynomial in the input size such that this is promise is always satisfied.
Thus, the moment polytope membership problem can be reduced to~$p$-scaling (see \cref{cor:moment polytope margin} in \cref{subsec:moment polytopes theory} and \cref{lem:concrete polytope margin} in \cref{subsec:concrete_time}).

There are multiple interesting input models for these problems.
One could explicitly describe the weights $\omega_1,\ldots, \omega_m$ for an action of $\T(n)$ (\cref{eqn:action on weight vectors}) and then describe $v$ as $\sum_{j=1}^m c_j v_j$ by describing the $c_j$'s.
The analogous description in the non-commutative world would be to describe the irreducible representations occuring in $V$.
Alternately, one could give black box access to the function $\norm{\pi(g) v}$, or to the moment map $\mu(\pi(g) v)$, etc.
Sometimes $\pi$ can be a non-uniform input as well, such as a fixed family of representations like the simultaneous left-right action \cref{exa:operator_scaling} as done in~\cite{garg2016deterministic}.
The inputs $p$ and $\eps$ will be given in their binary descriptions but we will see that some of the algorithms run in time polynomial in their unary descriptions.

\begin{rem}[Running time in terms of $\eps$]
By standard considerations about the bit complexity of the facets of the moment polytope, it can be shown that polynomial time algorithms for the scaling problems (\cref{prb:null_cone,prb:moment_polytope}) result in polynomial time algorithms for the exact versions (\cref{prb:exact_null_cone,prb:exact_moment_polytope}, respectively).
Polynomial time requires, in particular, $\poly(\log(1/\eps))$ dependence on $\eps$; a~$\poly(1/\eps)$ dependence is only known to suffice in special cases.
\end{rem}

We finally define the orbit closure intersection problem, which is concerned with the fundamental equivalence problem of geometric invariant theory~\cite{Mum65}.
It is a natural extension of the null cone membership problem.

\begin{restatable}[Orbit closure intersection problem]{prb}{oci}\label{prb:OCI}
Given a representation $\pi\colon G \to \GL(V)$ and vectors $v_1,v_2\in V$, decide whether the orbit closures intersect, i.e., whether
$\overline{Gv_1} \cap \overline{Gv_2} \ne \varnothing$.
\end{restatable}


\subsubsection{Commutative groups and geometric programming}
In the commutative case, the preceding problems are reformulations of well-studied optimization problems and much is known about them computationally.
To see this, consider the action of $\T(n)$ as in \cref{eqn:action on weight vectors}, and a vector $v = \sum_{j=1}^m c_j v_j$.
It follows from \cref{subsubsec:null_cone} that~$v$ is in the null cone iff~$0 \notin \Delta(v) = \conv\{\omega_j : c_j \neq 0\}$.
Recall from \cref{eq:abelian cap}, the formulation for~$\capacity(v)$.
Since this formulation is convex, it follows that, given $\omega_1,\ldots, \omega_m \in \ZZ^n$ (recall this is the description of $\pi$) and $c_1,\ldots, c_m \in \QQ[i]$
(each entry described in binary), there is a polynomial-time algorithm for the null cone membership problem via linear programming~\cite{khachiyan1979polynomial, karmarkar1984new}.
The same is true for the moment polytope membership problem.
The capacity optimization problem is an instance of \emph{(unconstrained) geometric programming}. 
The recent paper~\cite{BLNW:20} describes interior-point methods for this, which run in polynomial time.
Before~\cite{BLNW:20}, it was hard to find a reference for the existence of a polynomial time algorithm for geometric programming with precise complexity bounds;
however, is was known that polynomial time can be achieved using the ellipsoid algorithm as done for the same problem in slightly different settings in the papers~\cite{gurvits2004combinatorial,singh2014entropy,straszak2017computing}.
There has been work in the oracle setting as well, in which one has oracle access to the function $\norm{\pi(g)v}$.
The advantage of the oracle setting is that one can handle exponentially large representations of~$\T(n)$ when it is not possible to describe all the weights explicitly.
A very general result of this form is proved in~\cite{straszak2017computing}.
While not explicitly mentioned in~\cite{straszak2017computing}, their techniques can also be used to design polynomial time algorithms for \emph{commutative} null cone
and moment polytope membership in the oracle setting.
Thus, in the commutative case, \cref{prb:exact_null_cone,prb:capacity,prb:exact_moment_polytope} are well-understood.

Note that our first-order algorithm can (in the commutative case) be used to design a polynomial time algorithm for linear programming when the \emph{weight margin} (a complexity measure we define below) is inverse polynomially large.
We have not seen this algorithm explicitly analyzed in the literature but it is quite possible that it is known to experts.
There are many other such algorithms with similar guarantees, e.g., the perceptron algorithm.
However apart from the runtime dependence on the weight margin there do not seem to be other similarities between our algorithm and these algorithms.

\subsubsection{Non-commutative actions}\label{subsubsec:noncommutative state of the art}
Comparatively very little is known in the non-commutative case.
In the special case, where the group is fixed, 
polynomial time algorithms for the null-cone problem and the orbit intersection problem were given by the use of quantifier elimination
(which is impractical) and, more recently, by Mulmuley in~\cite[Theorem 8.5]{mulmuley2017geometric} through a purely algebraic approach. 
Moreover, prior to our work, no general algorithms were known for the moment polytope problem in this setting.
For instance, this applies to the settings of $V=\Sym^d\CC^n$ or $V=\bigwedge^d\CC^n$ with the natural action by $\SL(n)$, where $n$ is fixed.

If the group is not fixed, the only two non-trivial actions for which there are known polynomial-time algorithms for null cone membership (\cref{prb:exact_null_cone}) are
the \emph{simultaneous conjugation} (\cref{exa:SC}) and the \emph{left-right} action (\cref{exa:kronecker_quiver}).
A deterministic polynomial time algorithm for the orbit closure intersection problem (\cref{{prb:OCI}})
for simultaneous conjugation 
was first exhibited by Forbes and Shpilka~\cite{ForbesConj}, based on~\cite{RazShp}.
(This was triggered by the conference version of Mulmuley~\cite{mulmuley2017geometric} at FOCS 2012.)
Another algorithm was described by Derksen and Makam~\cite{DerksenConj}.


For the left-right action, 
deterministic polynomial time algorithms for the orbit closure intersection problem were independently described by
Derksen and Makam~\cite{DerksenConj} and  Allen-Zhu et al.~\cite{AGLOW18}.
These algorithms significantly differ in their nature: the one in~\cite{DerksenConj} uses purely algebraic operations (and works in arbitrary characteristic),
while the one in ~\cite{AGLOW18} is numerical and based on a second order geodesic method
(which we generalize in the the present paper in \cref{sec:second_order}).

Approximate algorithms for null cone membership have been designed for the \emph{tensor action} of products of $\SL(n)$'s~\cite{burgisser2017alternating}.
However the running time is exponential in the binary description of $\eps$ (i.e., polynomial in $1/\eps$).
This is the reason the algorithm does not lead to a polynomial time algorithm for the exact null cone membership problem for the tensor action.

Moment polytope membership is already interesting for the polytope~$\Delta$ in~\eqref{eq:big moment polytope}, the moment polytope of the entire representation~$V$ (not restricted to any orbit closure).
Even here, efficient algorithms are only known in very special cases, such as for the Horn polytope (\cref{exa:horn})~\cite{deloera2006computation,mulmuley2012geometric,burgisser2013positivity}.
The structural results in~\cite{berenstein2000coadjoint,ressayre2010geometric,vergne2014inequalities} characterize~$\Delta$ in terms of linear inequalities (it is known that in general there are exponentially many).
Mathematically, this is related to the asymptotic vanishing of certain representation-theoretic multiplicities~\cite{brion1987image,christandl2012computing,baldoni2018computation} whose non-vanishing is in general NP-hard to decide~\cite{ikenmeyer2017vanishing}.
\cite{burgisser2017membership} proved that the membership problem for~$\Delta$ is in NP~$\cap$~coNP.
As~$\Delta$ and~$\Delta(v)$ coincide for generic~$v\in V$, this problem captures the moment polytope membership problem (\cref{prb:exact_moment_polytope}) for almost all vectors (all except those in a set of measure zero).

The study of \cref{prb:exact_moment_polytope} in the noncommutative case focused on \emph{Brascamp-Lieb polytopes} (which are affine slices of moment polytopes).
\cite{garg2017algorithmic} solved the moment polytope membership problem in time depending polynomially on the \emph{unary} complexity of the target point.
In~\cite{burgisser2018efficient}, efficient algorithms were designed for the $p$-scaling problem (\cref{prb:moment_polytope}) for tensor actions, extending the earlier work of~\cite{franks2018operator} for the simultaneous left-right action.
The running times of both algorithms are $\poly(1/\eps)$; for this reason both algorithms result in moment polytope algorithms depending exponentially on the binary bitsize of~$p$, as in~\cite{garg2017algorithmic}.

Regarding the approximate computation of the capacity (\cref{prb:capacity}), efficient algorithms were previously known only for the simultaneous left-right action.
\cite{garg2016deterministic} gave an algorithm to approximate the capacity in time polynomial in all of the input description except $\eps$, on which it had dependence $\poly(1/\eps)$.
The paper~\cite{AGLOW18} gave an algorithm that depended polynomially on the input description; it has running time dependence $\poly(\log(1/\eps))$ on the error parameter~$\eps$.

To explain the way we generalize all works described above, it is useful to classify them into two categories according to the algorithmic techniques used.
One is that of \emph{alternating minimization} (which can be thought of as block-coordinate gradient descent that can be iterated exactly, i.e., roughly speaking as a first order method).
However, alternating minimization is limited in applicability to  `multilinear' actions of products of $\T(n)$'s or $\GL(n)$'s, where the action is linear in each component so that it is easy to optimize over one component when fixing all the others.
This is true for all the actions described above and hence explains the applicability of alternating minimization.
In fact, in all the above examples, one can even get a closed-form expression for the group element that has to be applied in each alternating step.
However, many group actions of interest -- from classical problems in invariant theory about symmetric forms to the important variant of Problem~\ref{it:qit} in \cref{subsec:unexpected} for fermions -- cannot be reduced to multilinear actions, and hence alternating minimization algorithms do not apply.
The second category are geodesic analogues of \emph{box-constrained Newton's methods} (second order).
Recently,~\cite{AGLOW18} designed an algorithm tailored towards the specific case of the simultaneous left-right action (\cref{exa:operator_scaling}), but no second order algorithms were known for other group actions.

In this paper, we develop new techniques that overcome these limitations.
Specifically, we provide both first and second order algorithms (geodesic variants of gradient descent and box-constrained Newton's method) that apply in great generality and identify the main structural parameters that control the running time of these algorithm.
In the remainder of this introduction, we describe our contributions in more detail.

\subsection{Algorithmic and structural results}\label{subsec:results}
We describe here our algorithmic and structural contributions to the theory of non-commutative optimization.
For simplicity, we focus on the case $G=\GL(n)$ even though our results apply to any
complex reductive group (given as a symmetric subgroup of some $\GL(n)$, see \cref{subsec:groups}.

In \cref{subsubsec:essential_params}, we describe the main parameters that govern the running time of our algorithms.
In \cref{subsubsec:uniform_first_order}, we describe the first order algorithm for $\capacity(v)$ and the structural results we prove for its analysis.
In \cref{subsubsec:non-uniform_first_order}, we describe a first order algorithm for the problem of membership in moment polytopes and the relevant structural results.
In \cref{subsubsec:second_order}, we describe the second order algorithm for $\capacity(v)$ and the techniques and ideas used in its analysis.
Our focus here is on the number of iterations of our algorithms when given oracle access to the function optimized,
the gradient (moment map), and (for the second order algorithm only) the Hessian as well.
Later on, \cref{subsec:concrete_time} will discuss the complexity of the algorithms in terms of the input bit-size.

\subsubsection{Essential parameters and structural results}\label{subsubsec:essential_params}
In this section, we define the essential parameters related to the group action which, in addition to dictating the running times of our first and second order methods,
control the relationships between the null cone, the norm of the moment map, and the capacity, i.e., between \cref{prb:exact_null_cone,prb:null_cone,prb:capacity}.

We saw in \cref{subsec:nc_primer} that for all actions of $\T(n)$ on a vector space $V$, one can find a basis of~$V$ consisting of simultaneous eigenvectors of the matrices $\pi(g)$, $g \in \T(n)$.
While this is in general impossible for non-commutative groups, one can still decompose $V$ into building blocks known as irreducible subspaces (or subrepresentations), as will be discussed in further detail in \cref{subsec:rep theory}.

For $\GL(n)$, these are uniquely characterized by nonincreasing sequences~$\lambda \in \ZZ^n$; such sequences~$\lambda$ are in bijection with irreducible representations~$\pi_\lambda\colon \GL(n)\to \GL(V_\lambda)$.
We say that $\lambda$ \emph{occurs in~$\pi$} if one of its irreducible subspaces is of type~$\lambda$.
If all the $\lambda$ occuring in $\pi$ have nonnegative entries, then the entries of the matrix $\pi(g)$ are polynomials in the entries of $g$.
Such representations $\pi$ are called \emph{polynomial}, and if all $\lambda$ occuring in $\pi$ have sum exactly (resp.
at most) $d$, then $\pi$ is said to be a \emph{homogeneous polynomial representation of degree (resp.
at most) $d$}.
We elaborate further on the representation theory of $\GL(n)$ in \cref{subsec:concrete_time}, \cref{subsec:rep theory}, and \cref{subsec:gt-basis}.

Now we can define the complexity measure which captures the smoothness of the optimization problems of interest.
Later on in \cref{subsec:smooth robust log norm} we discuss how to think of the following measure as a \emph{norm} of the Lie algebra representation $\Pi$, hence the name \emph{weight norm}.

\begin{dfn}[Complexity measure I: weight norm]\label{dfn:weight_norm}
We define the \emph{weight norm}~$N(\pi)$ of an action~$\pi$ of $\GL(n)$ by
$N(\pi) := \max \{ \norm{\lambda}_2 : \lambda \text{ occurs in } \pi \}$, where $\norm{\cdot}_2$ denotes the Euclidean norm.%
\end{dfn}
Another use of the weight norm is to provide a bounding ball for the moment polytope.
As shown in \cref{lem:bound on gradient}, the moment polytope is contained in a Euclidean ball of radius $N(\pi)$.
The weight norm is in turn controlled by the degree of a polynomial representation.
More specifically, if $\pi$ is a polynomial representation of $\GL(n)$ of degree at most $d$, then $N(\pi)\leq d$;
see \cref{lem:homog poly norm bound}.

We now describe our second measure of complexity which will govern the running time bound for our second order algorithm.
This parameter, which will be discussed further in \cref{subsec:duality}, also features in \cref{thm:intro_duality} concerning quantitative non-commutative duality.

\begin{dfn}[Complexity measure II: weight margin]\label{dfn:comp-measure-wm2}
The \emph{weight margin} $\gamma(\pi)$ of an action~$\pi$ of~$\GL(n)$ is the minimum Euclidean distance between the origin and the convex hull of any subset of the weights of~$\pi$ that does not contain the origin.
\end{dfn}

\noindent
Our running time bound will depend inversely on the weight margin.
Two interesting examples with large (inverse polynomial) weight margin are the left-right action (\cref{exa:operator_scaling}) and simultaneous conjugation.
The existing second order algorithm for the left-right action relied on the large weight margin of the action~\cite{AGLOW18}.
It is interesting that the simultaneous conjugation action (\cref{exa:SC}), the sole other interesting example of an action of a non-commutative group for which there are efficient algorithms for the null cone membership problem~\cite{RazShp,ForbesConj,DerksenConj} (which have nothing to do with the weight margin), also happens to have large weight margin!
On the other hand, the only generally applicable lower bound on the weight margin is~$N(\pi)^{1-n}n^{-1}$ (see \cref{prp:general-margin}),
and indeed this exponential behavior is seen for the somewhat intractable $3$-tensor action (\cref{exa:tensor}),
which has weight margin at most~$2^{-n/3}$ and weight norm~$\sqrt{3}$ (\cite{kravtsov2007combinatorial,reichenbach-franks}).
For the convenience of the reader, we arrange in a tabular form the above information about the weight margin
and weight norm for various paradigmatic group actions in \cref{table:margin_norm} (using a definition of the weight margin and weight norm,
given later in the paper, that naturally generalizes the one given above for $\GL(n)$):

\begin{table}[ht!]
\centering
\begin{tabular}{@{}lll@{}}
\toprule
\textbf{Group action} & \textbf{Weight margin $\gamma(\pi)$} & \textbf{Weight norm $N(\pi)$} \\
\midrule
Matrix scaling%
\footnote{This commutative example is modelled as follows: $G=\ST(n) \times \ST(n)$ acts on $\Mat(n)$ by $\pi(A,B) M = A M B$, where $\ST(n)$ is the group of diagonal $n \times n$ matrices with unit determinant.}
& $\geq n^{-3/2}$;~\cite{LSW} and Cor.~\ref{exa:opt-g-s} & $\sqrt2$ (Ex.~\ref{exa:weight norm mat op}) \\
Simultan.~left-right action (\cref{exa:operator_scaling}) & $\geq n^{-3/2}$;~\cite{gurvits2004classical} and~Prop.~\ref{pro:WM-LR} & $\sqrt2$ (Ex.~\ref{exa:weight norm mat op}) \\
Quivers (\cref{exa:quivers}) & $\geq (\sum_x n(x))^{-3/2}$ (Thm.~\ref{prp:weight norm margin quivers}) & $\sqrt2$ (Thm.~\ref{prp:weight norm margin quivers}) \\
Simultaneous conjugation (\cref{exa:SC}) & $\geq n^{-3/2}$ (Thm.~\ref{prp:weight norm margin quivers}) & $\sqrt2$ (Thm.~\ref{prp:weight norm margin quivers}) \\
3-tensor action (\cref{exa:tensor}) & $\leq 2^{-n/3}$;\;\cite{kravtsov2007combinatorial,reichenbach-franks}& $\sqrt3$ (Thm.~\ref{exa:weight norm tensors}) \\
Polynomial $\GL(n)$-action of degree $d$ & $\geq d^{-n} dn^{-1}$ (Thm.~\ref{prp:general-margin}) & $\leq d$ (Ex.~\ref{lem:homog poly norm bound}) \\
Polynomial $\SL(n)$-action of degree $d$ & $\geq (nd)^{-n} dn^{-1}$ (Thm.~\ref{prp:general-margin}) & $\leq d$ (Ex.~\ref{lem:homog poly norm bound}) \\
\bottomrule
\end{tabular}
\caption{Weight margin and norm for various representations (see \cref{sec:norm margin bounds} for more).}
\label{table:margin_norm}
\end{table}

As the moment map is the gradient of the geodesically convex function $\log\norm{v}$, it stands to reason that as $\mu(v)$ tends to zero,
$\norm{v}$ tends to the capacity $\capa(v)$.
However, in order to use this relationship to obtain efficient algorithms, we need this to hold in a precise quantitative sense.
To this end, in \cref{subsec:duality} we show the following fundamental relation between the capacity and the norm of the moment map,
which is quantitative strengthening of the Kempf-Ness theorem~\cite{kempf1979length}.

\begin{thm}[Non-commutative duality]\label{thm:intro_duality}
For $v \in V\setminus\{0\}$ we have
\begin{align*}
1 - \frac{\norm{\mu(v)}_F}{\gamma(\pi)} \leq \frac{\capacity^2(v)}{\norm{v}^2} \leq 1 - \frac{\norm{\mu(v)}^2_F}{4N(\pi)^2}.
\end{align*}
\end{thm}

\noindent
Equipped with these inequalities, it is easy to relate \cref{prb:null_cone,prb:capacity}.
\begin{cor}\label{cor:cap_vs_moment}
An output $g$ for the norm minimization problem on input $(\pi, v, \eps)$ is a valid output for the scaling problem
on input~$(\pi, v, N(\pi)\sqrt{8\eps})$.
If~$\eps/\gamma(\pi) < \frac{1}{2}$ then an output $g$ for the scaling problem on input~$(\pi, v, \eps)$
is a valid output for the norm minimization problem on input~$(\pi, v,\frac{2 \log(2)\eps}{\gamma(\pi)})$.
\end{cor}

\noindent
Because $0 \in \Delta(v)$ if and only if $\capa(v) > 0$, \cref{thm:intro_duality,cor:cap_vs_moment} immediately yield
the accuracy to which we must solve the scaling problem or norm minimization problem to solve the null cone membership problem:

\begin{cor}\label{cor:null_cone_margin}
It holds that~$0 \in \Delta(v)$ if and only if~$\Delta(v)$ contains a point of norm smaller than~$\gamma(\pi)$.
In particular, solving the scaling problem with input~$(\pi,v,\gamma(\pi)/2)$ or the norm minimization problem
with~$(\pi,v,\frac18 (\gamma(\pi) / 2N(\pi))^2)$ suffices to solve the null cone membership problem for~$(\pi,v)$.
\end{cor}

\noindent
\cref{cor:moment polytope margin} in \cref{subsec:moment polytopes theory} provides an analogue of the above corollary
for the moment polytope membership problem.

\subsubsection{First order methods: structural results and algorithms}\label{subsubsec:uniform_first_order}
As discussed above, in order to approximately compute the capacity in the commutative case, one can just run a Euclidean gradient descent on the convex formulation
in \cref{eq:abelian cap}. We will see that the gradient descent method naturally generalizes to the non-commutative setting.
It is worth mentioning that there are several excellent sources of the analysis of gradient descent algorithms for geodesically convex functions
(in the general setting of Riemannian manifolds and not just the group setting that we are interested in);
see e.g.,~\cite{udriste1994convex, absil2009optimization, zhang2016first, zhang2016riemannian, sato2017riemannian, zhang2018towards} and references therein.
In this paper, our contribution is mostly in understanding the geometric properties (such as smoothness) of the optimization problems that we are concerned with,
which allow us to carry out the classical analysis of Euclidean gradient descent in our setting
and to obtain quantitative convergence rates, which are not present in previous work.

The natural analogue of gradient descent for the optimization problem $\capacity(v)$ is the following:
start with $g_0 = I$ and repeat, for~$T$ iterations and a suitable step size~$\eta$:
\begin{align*}
  g_{t+1} = e^{-\eta \mu(\pi(g_t)v)} g_t.
\end{align*}
Finally, return the group element $g$ among $g_0,\ldots, g_{T-1}$, which minimizes~$\norm{\mu(\pi(g) v)}_F$.
This algorithm is described in \cref{alg:gconvex_gradient_uniform}.
A natural geometric parameter which governs the complexity (number of iterations $T$, step size $\eta$) of gradient descent is the \emph{smoothness} of the function to be optimized.
The smoothness parameter for actions of~$\T(n)$ in \cref{eqn:action on weight vectors} can be shown to be $O(\max_{j \in [m]} \norm{\omega_j}^2_2)$ (see, e.g.,~\cite{straszak2017computing}), which is the square of the weight norm defined in \cref{dfn:weight_norm} for this action.
We prove in \cref{sec:convexity} that, in general, the function $\log \norm{\pi(g) v}$ is geodesically smooth, with a smoothness parameter which, analogously to the commutative case, is on the order of the square of the weight norm. We now state the running time for our geodesic gradient descent algorithm for \cref{prb:null_cone}.

\begin{thm}[First order algorithm for scaling]\label{thm:intro_uniform_grad_descent}
Fix a representation~$\pi:\GL(n) \to \GL(V)$ and a unit vector $v \in V$ such that $\capacity(v)>0$ (i.e., $v$ is not in the null cone).
Then \cref{alg:gconvex_gradient_uniform} with a number of iterations at most
\begin{align*}
  T = O\left(\frac{N(\pi)^2}{\eps^2} \bigl| \log \capacity(v) \bigr| \right)
\end{align*}
outputs a group element $g\in \GL(n)$ satisfying $\norm{\mu(\pi(g) v)}_F \leq \eps$.
\end{thm}

\noindent
This is proved in \cref{sec:first-order algorithm}, where it is stated as \cref{thm:uniform_grad_descent} for general groups.
\Cref{thm:concrete} in \cref{subsec:concrete_time} states concrete running time bounds in terms of the bit complexity of the input.

The analysis of \cref{thm:intro_uniform_grad_descent} relies on the smoothness of the function~$F_v(g):=\log \norm{\pi(g) v}$, which implies that
\begin{align*}
F_v(e^H g) \leq F_v(g) + \tr\bigl[\mu\bigl(\pi(g)v\bigr) H\bigr] + N(\pi)^2 \lVert H\rVert_F^2,
\end{align*}
for all $g \in \GL(n)$ and for all Hermitian $H \in \Herm(n)$
(see Cor.~\ref{cor:log_norm_quadratic_ub}).

\subsubsection{First order method for moment polytope membership}\label{subsubsec:non-uniform_first_order}
Next, we describe our first order algorithm for the $p$-scaling problem.
\Cref{thm:intro_uniform_grad_descent} solves the problem of minimizing the moment map (equivalent to capacity computation), hence can be used to determine if~$0 \in \Delta(v)$.
Can we reduce the general moment polytope membership problem, $p \in \Delta(v)$, to this case?
This is straightforward in the commutative case, $G=\T(n)$.
It follows from the reasoning in \cref{subsubsec:null_cone} that, for $p \in \RR^n$, we have~$p \notin \Delta(v)$ iff
\begin{align}\label{eq:cap p}
\textstyle\capacity_p(v)^2 := \inf_{x \in \RR^n} \sum_{j=1}^m |c_j|^2 e^{ x \cdot ( \omega_j - p)} = 0.
\end{align}
Thus, all we need to do is shift the relevant vectors by~$p$.
Is there an analog of this trick in the non-commutative world?
There is!
It is called, unsurprisingly, the \emph{shifting trick}~\cite{brion1987image}.
Let us describe it here.
A nice property about \cref{eq:cap p} is that (recall \cref{eqn:ankit602}) $\nabla_{x=0} \log \left( \sum_{j=1}^m |c_j|^2 e^{ x \cdot ( \omega_j - p)}  \right) = \mu(v) - p$.
How do we shift the moment map in the case of~$\GL(n)$?
It relies on the following two elementary properties of the moment map:
\begin{enumerate}
\item The moment map of the tensor product $\pi$ of two representations $\pi_1, \pi_2$ of $\GL(n)$, which is defined as $\pi(g) (v \otimes w) := (\pi_1(g) v) \otimes (\pi_2(g) w)$,
satisfies~$\mu(v \otimes w) = \mu(v) + \mu(w)$.
\item There is a vector $v_\lambda$ (known as a \emph{highest weight vector}) in the vector space~$V_\lambda$ of the irreducible action $\pi_\lambda$ such that $\mu(v_\lambda) = \diag(\lambda)$.
\end{enumerate}

\noindent
Now suppose $p \in \QQ^n$ and let $\ell > 0$ be the least integer such that~$\lambda:=\ell p \in \ZZ^n$.
Let $\lambda^* := (-\lambda_n,\ldots, -\lambda_1)$.
Then one can see that the tensor product action of $\GL(n)$ on the space $\Sym^\ell(V) \otimes V_{\lambda^*}$ satisfies~$\frac1\ell\mu\left(v^{\otimes \ell} \otimes v_{\lambda^*}\right) = \mu(v) + \diag(\lambda^*)/\ell = \mu(v) - \Lambda$, where~$\Lambda$ is the diagonal matrix with entries~$\Lambda_{i,i} = p_{n-i+1}$, which has spectrum~$p$.
We have managed to shift the moment map!
So we are led to the following optimization problem,
\begin{align*}
 \capacity_p(v)^{\ell} := \inf_{g \in G} \|(\pi(g) v)^{\otimes \ell} \otimes \left(\pi_{\lambda^*}(g) v_{\lambda^*} \right)\|.
\end{align*}

In the noncommutative case, the relation between this \emph{$p$-capacity} and the moment polytope is slightly more subtle.
While $\capacity_p(v)>0$ always guarantees that $p\in\Delta(v)$, these two conditions are in general \emph{not} equivalent (unless $p=0$, when~$\capacity_p(v)$ reduces to $\capacity(v)$).
However, what holds is that $p\in\Delta(v)$ if and only if~$\capacity_p(\pi(g)v) > 0$ for \emph{generic}~$g\in G$.
We can thus solve the $p$-scaling problem by first applying a random group element and then applying an optimization algorithm to approximate~$\capacity_p(v)$.

We now outline our optimization algorithm for $\capacity_p(v)$.
The optimization problem defining~$\capacity_p(v)$ is defined in terms of actions on a space of exponential dimension.
However, it turns out that the gradients can be explicitly computed and the geodesic gradient descent can be described explicitly as follows:
start with $g_0 = I$ and repeat, for $T$ iterations and suitable step size~$\eta$:
\begin{align*}
g_{t+1} = e^{- \eta \: \left(\mu(\pi(g_t) v) - Q_t \Lambda Q_t^{\dagger}\right)} g_t,
\end{align*}
where $g_t = Q_t R_t$ is the QR decomposition of $g_t$.
Finally return group element $g$ among~$g_0,\ldots, g_{T-1}$, which minimizes $\norm{\mu(\pi(g_t) v) - Q_t \Lambda Q_t^{\dagger}}_F$.
This algorithm is stated precisely as \cref{alg:nonuniform_gradient}.

\begin{thm}[First order algorithm for $p$-scaling]\label{thm:intro_non-uniform_grad_descent}
Fix a representation $\pi:\GL(n) \to \GL(V)$, a unit vector~$v\in V$, and a target point $p \in \QQ^n$ such that $\capacity_p(v)>0$.
Let $N^2 := N(\pi)^2 + \norm{p}_2$.
Then \cref{alg:nonuniform_gradient} with a number of iterations at most
\begin{align*}
  T = O\left(\frac{N^2}{\eps^2} \bigl| \log \capacity_p(v)\bigr| \right)
\end{align*}
outputs a group element $g\in \GL(n)$ satisfying
$\norm{\spec(\mu(\pi(g) v)) - p}_2 \leq \eps$.
\end{thm}

\noindent
This is proved in \cref{subsec:moment polytopes algo}, where it is stated as \cref{thm:nonuniform_grad_descent} for general groups.
A precise calculation of the smoothness of the function~$g \mapsto \log \norm{\pi(g) v} + \frac{1}{\ell} \log \norm{\pi_{\lambda^*}(g) v_{\lambda^*}}$ (which underlies the $p$-capacity) features crucially in our analysis.

As described above, \cref{thm:intro_non-uniform_grad_descent} preceded by a randomization step can be used to solve the $p$-scaling problem.
\Cref{thm:rand_alg} in \cref{subsec:concrete_time} describes the performance of such a randomized algorithm for $G = \GL(n)$.

\subsubsection{Second order methods: structural results and algorithms}\label{subsubsec:second_order}
Here we discuss our second order algorithm for \cref{prb:capacity}, the approximate norm minimization problem.
As mentioned in \cref{subsubsec:comp prob}, the paper~\cite{AGLOW18} (following the algorithms developed in~\cite{allen2017much,cohen2017matrix} for the commutative Euclidean case) developed a second order polynomial-time algorithm for approximating the capacity for the simultaneous left-right action (\cref{exa:operator_scaling}) with running time polynomial in the bit description of the approximation parameter~$\eps$.
In \cref{sec:second_order}, we generalize this algorithm to arbitrary groups and actions (\cref{alg:gconvex second order}).
It repeatedly optimizes quadratic Taylor expansions of the objective in a small neighbourhood.
Such algorithms also go by the name ``trust-region methods'' in the Euclidean optimization literature~\cite{CGT00}.
The running time of our algorithm will depend inversely on the weight margin defined in \cref{dfn:comp-measure-wm2}.

\begin{thm}[Second-order algorithm for norm minimization]\label{thm:intro_second_order} Fix a representation $\pi:\GL(n) \to \GL(V)$ and a unit vector~$v \in V$ such that~$\capacity(v)>0$.
Put
$C := |\log \capacity(v)|$,
$\gamma := \gamma(\pi)$ 
and $N := N(\pi)$. 
Then \cref{alg:gconvex second order} for a suitably regularized objective function outputs~$g \in G$ satisfying~$\log \|\pi(g) v\| \leq \log \capacity(v) + \eps$ with a number of iterations at most
\begin{align*}
T = O\left( \frac{ N \sqrt n}{\gamma} \left(C + \log  \frac{n}{\eps} \right)\log\frac C\eps \right).
\end{align*}
\end{thm}

\noindent
This is proved in \cref{sec:second_order}, where it is restated precisely as \cref{thm:main}.
\Cref{thm:concretesecond} in \cref{subsec:concrete_time} specializes \cref{thm:intro_second_order} to the group $\SL(n)$ by obtaining running time bounds in terms of the bit complexity of the input.

The two main structural parameters which govern the runtime of \cref{alg:gconvex second order} in general are the \emph{robustness} (controlled by the weight norm) and a \emph{diameter bound} (controlled by the weight margin).
The robustness of a function bounds third derivatives in terms of second derivatives, similarly to the well-known notion of self concordance (however, in contrast to the latter, the robustness is not scale-invariant).
As a consequence of the robustness, we show that the function~$F_v(g) = \log \norm{\pi(g) v}$ is sandwiched between two quadratic expansions in a small neighbourhood:
\begin{align*}
  F_v(g) + \partial_{t=0} F_v(e^{tH} g) + \frac1{2e} \partial^2_{t=0} F_v(e^{tH} g)
\leq F_v(e^H g)
\leq F_v(g) + \partial_{t=0} F_v(e^{tH} g) + \frac{e}2 \partial^2_{t=0} F_v(e^{tH} g)
\end{align*}
for every $g\in \GL(n)$ and $H\in \Herm(n)$ such that~$\lVert H\rVert_F\leq 1/(4 N(\pi))$
(see \cref{sec:convexity}).

Another ingredient in the analysis of \cref{alg:gconvex second order} is
to prove the existence of ``well-conditioned'' approximate minimizers, i.e., $g_\star \in G$, with small condition number satisfying
$\log \norm{\pi(g_\star) v} \leq \log \capacity(v) + \eps$.
The bound on the condition numbers of approximate minimizers helps us ensure that the algorithm's trajectory always lies in a compact region with the use of appropriate regularizers.
As in~\cite{AGLOW18}, we obtain this ``diameter bound'' by designing a suitable gradient flow and bounding the (continuous) time it takes for it to converge (\cref{prp:conditions})!
A crucial ingredient of this analysis is our \cref{thm:intro_duality} relating capacity and norm of the moment map.

Our gradient flow approach, which can be traced back to works in symplectic geometry~\cite{kirwan1984cohomology}, is the only one we know for proving diameter bounds in the non-commutative case.
In contrast, in the commutative case several different methods are available (see, e.g.,~\cite{gurvits2004combinatorial,singh2014entropy,straszak2017computing,BLNW:20}).
It is an important open problem to develop alternative methods for diameter bounds in the non-commutative case, which will also lead to improved running time bounds for algorithms like \cref{alg:gconvex second order}.

\subsection{Explicit time complexity bounds for \texorpdfstring{$\SL(n)$ and $\GL(n)$}{SL(n) and GL(n)}}\label{subsec:concrete_time}
Moving beyond the number of oracle calls, we now describe the running time of our algorithms in terms of the bitsize needed to describe the vector~$v$
and the action~$\pi$.
For concreteness, we restrict here to \emph{homogeneous, polynomial} actions of~$\GL(n)$, i.e., those for which there is a degree~$d$
such that entries of~$\pi(g)$ are homogeneous polynomials of degree~$d$ in the matrix entries~$g_{i,j}$.
This important class includes the setting studied by Hilbert in his seminal paper~\cite{Hil}.
The results in this section extend readily to (nonhomogeneous) representations of groups that are products of~$\GL(n)$'s (or $\SL(n)$'s),
a setting which captures all of the interesting examples discussed so far
(tensor scaling, left-right action, simultaneous conjugation action, etc).
This is elaborated in \cref{sec:bds-degree-capacity,sec:explicit algos}.
Up to isomorphism, the irreducible polynomial representations of~$\GL(n)$ can be specified by \emph{partitions} of length at most~$n$, or nonincreasing vectors in $\ZZ_{\geq 0}^n$; the partition corresponding to an irreducible representation is called its \emph{highest weight}.
If~$\lambda$ is a partition of (sums to)~$d$ then the corresponding representation is homogeneous of degree~$d$.

We must specify our input in such a way that the group action and moment map can be efficiently computed.
To this end, if~$\lambda$ is a partition, we take~$\pi_\lambda\colon \GL(n) \to \GL(m_\lambda)$ to be the irreducible representation of highest weight~$\lambda$ such that the standard basis of~$\CC^{m_\lambda}$ is a \emph{Gelfand-Tsetlin basis}.
The Gelfand-Tsetlin basis, described in \cref{subsec:gt-basis}, is a well-studied basis for irreducible representations in which the entries of~$\pi_{\lambda}$ are polynomials with rational coefficients that we can effectively bound.

A list of partitions~$\lambda^1, \dots, \lambda^s$ specifies the representation~$\pi\colon\GL(n)\to\GL(m)$ on~$V=\CC^m$ given by~$\pi := \oplus_{i=1}^s \pi_{\lambda^i}$, where $m := \sum_{i=1}^s m_{\lambda^i}$.
Up to isomorphism, every finite-dimensional polynomial representation~$\pi$ of~$\GL(n)$ can be obtained this way.
If $\pi$ is such a representation, the input size $\langle \pi \rangle$ of $\pi$ is defined to be
$\langle \lambda^1 \rangle +\dots + \langle \lambda^s\rangle$ where $\langle \lambda^i \rangle$
is the total binary size of the entries of~$\lambda^i$.

For a vector~$v \in \CC^m$ with coordinates in~$\QQ + i \QQ$, $\langle v\rangle$ refers to the total binary size of its entries.
In~\cite{burgisser2000computational,burgisser2017membership} it is shown that, for~$\pi$ and~$v$ specified as above and for~$g \in \Mat(n, \QQ + i \QQ)$ specified in binary, the group action~$\pi(g)v$ and moment map~$\mu(v)$ (and hence the gradient of the objective functions in our first and second order algorithms) can be computed in polynomial time. The same will follow for the Hessians of our objective functions.

We now define instances for the problems discussed in \cref{subsubsec:comp prob} for~$\GL(n)$ and~$\SL(n)$.
We will assume that~$\pi$ is polynomial and homogeneous of degree~$d$.
We may assume that any target spectrum~$p$ for the moment polytope membership and $p$-scaling problems has nonnegative, rational entries adding to~$d$,
because every rational element of~$\Delta(v)$ necessarily has this property.
For the scaling (equivalently, norm minimization) and null cone membership problems, we consider the restriction of~$\pi$ to the smaller group~$\SL(n)$.
This is without loss of generality because, unless~$d=0$, the capacity for homogeneous actions of~$\GL(n)$ is always zero (see discussion below \cref{prb:exact_null_cone}).
In fact, the scaling problem for~$\SL(n)$ is equivalent to the $p$-scaling problem for~$\GL(n)$, for~$p$ a suitable multiple of the all-ones vector.
This captures many natural scaling problems.
\begin{enumerate}
\item \label{it:nullcone}A tuple~$(\pi, v)$ is called an \emph{instance of the null cone membership problem for~$\SL(n)$} if
\begin{itemize}
\item $\pi\colon \GL(n) \to \GL(m)$ is a homogeneous, polynomial representation of~$\GL(n)$ of degree~$d>0$, specified by a list of partitions,
\item $v \in V = \CC^m$ is a Gaussian integer vector, i.e., its entries are in~$\ZZ+i\ZZ$.
\end{itemize}
\item A tuple~$(\pi, v, \eps)$ is called an \emph{instance of the scaling problem for~$\SL(n)$} if~$(\pi, v)$ is an instance of the null cone membership problem for~$\SL(n)$ and~$\eps > 0$ is a rational number.
\item A tuple~$(\pi, v, p)$ is an \emph{instance of the moment polytope membership problem for~$\GL(n)$} if~$(\pi, v)$ is an instance of the null cone membership problem for~$\SL(n)$ and~$p \in \QQ^n$ is a vector with entries~$p_1\geq\dots\geq p_d\geq0$ adding to~$d$.
\item A tuple~$(\pi, v, p, \eps)$ is an \emph{instance of the~$p$-scaling problem for~$\GL(n)$} if~$(\pi, v, p)$ is an instance of the moment polytope membership problem for~$\GL(n)$ and~$\eps> 0$ is rational number.
\end{enumerate}

\begin{rem}[Degree versus dimension]\label{rem:translate}
We may assume that for our input representations~$\pi = \oplus_{i=1}^s \pi_{\lambda^i}$ we have~$\lambda^i_n = 0$ for some~$i \in [s]$;
this is without loss of generality because simultaneously translating each~$\lambda^i$ by an integer multiple of the all-ones vector simply shifts the entire moment polytope in~$\RR^n$ by the same vector.
If some~$\lambda^i_n = 0$, then we have $d \leq m$ (see \cref{re:dlem}), which ensures that our bounds in the coming theorems are polynomial in~$\langle v \rangle, \langle \pi \rangle$.
\end{rem}

\noindent
In \cref{sec:bds-degree-capacity} we prove general capacity lower bounds for vectors of bounded bit complexity via Cayley's Omega process,
following \cite{burgisser2017alternating}. Based on this, 
we prove in \cref{sec:explicit algos} that \cref{thm:intro_uniform_grad_descent} implies the following time bound for the scaling problem.

\begin{restatable}[First order algorithm for scaling in terms of input size]{thm}{concrete}\label{thm:concrete}
Let~$(\pi, v, \eps)$ be an instance of the scaling problem for~$\SL(n)$ such that~$0 \in \Delta(v)$
and every component of~$v$ is 
bounded in absolute value by~$M$.
Let~$d$ denote the degree and~$m$ the dimension of~$\pi$.
Then, \cref{alg:gconvex_gradient_uniform} with a number of iterations at most
\begin{align*}
  T = O\left(\frac{d^3}{\eps^2} mn^3\log(Mmnd) \right)
\end{align*}
returns a group element~$g\in\SL(n)$ such that~$\Norm{\mu(\pi(g) v)}_F \leq \eps$.
In particular, there is a~$\poly(\langle\pi\rangle, \langle v\rangle, \eps^{-1})$ time algorithm
to solve the scaling problem (\cref{prb:null_cone}) for~$\SL(n)$.
\end{restatable}

\noindent
\Cref{thm:concrete}, as well as all the other results below, are proved in \cref{sec:explicit algos}. 
(There we also cover the more general situation of polynomial representations of products of $\SL(n)$'s;
moreover, for our results in \cref{subsec:explicit-no-GT} it is not necessary that the representations are given in a Gelfand-Tsetlin basis.)

We also show a concrete version of \cref{thm:intro_second_order} for the norm minimization problem.

\begin{restatable}[Second order algorithm for norm minimization in terms of input size]{thm}{concretesecond}\label{thm:concretesecond}
Let~$(\pi, v, \eps)$ be an instance of the scaling problem for~$\SL(n)$ such that~$0 \in \Delta(v)$
and every coordinate of~$v$ is 
bounded in absolute value by~$M$.
Let~$d$ denote the degree, $m$ the dimension, and $\gamma$ the weight margin of~$\pi$.
Then, \cref{alg:gconvex second order} applied to a suitably regularized objective function and a number of iterations at most
\begin{align*}
  T = O\left( \frac{d\sqrt{n}}{\gamma} \left(mn^3d\log(Mmnd) + \log\frac1\eps\right)\log\left(\frac{mnd\log M}{\eps}\right)\right)
\end{align*}
returns a group element~$g\in\SL(n)$ such that $\log \norm{\pi(g) v} \leq \log\capacity(v) + \eps$.
In particular, there is an algorithm to solve the norm minimization problem (\cref{prb:capacity}) for~$\SL(n)$
in time~$\poly(\braket{\pi}, \braket{v}, \gamma^{-1}, \log(\eps^{-1}))$, which is at most $\poly(\braket{\pi}, \braket{v}^n, \log(\eps^{-1}))$.
\end{restatable}

\noindent
\Cref{cor:null_cone_margin} implies that both the first and second order algorithm result in a null cone membership algorithm with polynomial dependence on $\gamma^{-1}$; the tradeoffs are discussed in \cref{rem:tradeoff} in \cref{sec:first-order algorithm}.

\begin{restatable}[Algorithm for null cone membership problem in terms of input size]{cor}{concretenullcone}\label{cor:concretenullcone}
There is an algorithm to solve the null cone membership problem (\cref{prb:exact_null_cone}) for~$\SL(n)$
in time $\poly(\braket{\pi}, \braket{v}, \gamma^{-1})$,
which is at most $\poly(\braket{\pi}, \braket{v}^n)$.
\end{restatable}

\noindent
In the important setting when the group is fixed (i.e.,~$n$ is constant), the above corollary asserts that our second order algorithm solves
the null cone problem for~$\SL(n)$ in deterministic polynomial time.
Prior to this result, the only known polynomial time algorithms for this class of null cone problems were given by the use of quantifier elimination
(which is impractical) and, more recently, by Mulmuley in~\cite[Theorem 8.5]{mulmuley2017geometric} through a purely algebraic approach.
Mulmuley constructs a circuit which encodes a generating set of invariants for the ring of invariants of the corresponding action,
and then invokes previous results on polynomial identity testing to obtain an algorithm for the null cone problem.

Finally, we give a randomized algorithm for the $p$-scaling problem based on \cref{thm:intro_non-uniform_grad_descent}.
Here it is natural to consider the full group~$\GL(n)$ rather than~$\SL(n)$ as in the scaling problem.

\begin{restatable}[First-order randomized algorithm for $p$-scaling in terms of input size]{thm}{randalg}\label{thm:rand_alg}
Let $(\pi, v, p, \eps)$ be an instance of the moment polytope problem for~$\GL(n)$ such that~$p \in \Delta(v)$
and every component of~$v$ is 
bounded in absolute value by~$M$.
Let~$d$ denote the degree and~$m$ the dimension of~$\pi$.
Then, with probability at least~$1/2$, \cref{alg:moment_polytope} with a number of iterations at most
\begin{align*}
  T = O\left( \frac{d^3}{\eps^2} mn^5 \log(Mmnd)\right)
\end{align*}
returns a group element~$g\in\GL(n)$ such that~$\norm{\spec(\mu(\pi(g) v) - p}_2 \leq \eps$.
In particular, there is a randomized algorithm to solve the $p$-scaling problem (\cref{prb:moment_polytope}) for~$\GL(n)$ in time~$\poly(\braket{\pi},\braket{v},\braket{p},\eps^{-1})$ and using~$\poly(\braket{\pi},\braket{v})$ bits of randomness.
\end{restatable}

\noindent
Similarly as for the null cone problem, our first order algorithm for~$p$-scaling also implies a randomized algorithm for moment polytope membership.

\begin{restatable}[Randomized algorithm for moment polytope membership in terms of input size]{cor}{concretemomentpolytope}\label{cor:concretemomentpolytope}
There is a randomized algorithm to solve the moment polytope membership problem (\cref{prb:exact_moment_polytope})
for~$\GL(n)$ in time~$\poly(\braket{\pi},\braket{v}^n,2^{n\braket{p}})$ and using~$\poly(\braket{\pi},\braket{v})$ bits of randomness.
\end{restatable}

\noindent
We remark that our algorithms readily extend to representations of~$G= \GL(n_1)\times \dots \times \GL(n_k)$
and~$G = \SL(n_1)\times \dots \times \SL(n_k)$ at the expense of running times polynomial in~$n_1\cdots{}n_k$.

\subsection{Applications to some concrete representations}\label{subsec:null cone applications}
As a consequence of our developments, we obtain a uniform approach to polynomial time algorithms for the null cone membership problem.
While the existence of such algorithms was already known in several cases (at least implicitly),
we identify a few new situations with polynomial time algorithms.
This is based on \cref{cor:prod-concretenullcone-no-GT} and \cref{cor:prod-concretenullcone},
which extend \Cref{cor:concretenullcone} to products of general or special linear groups,
and states that the running times of the first and second order algorithm have an inverse polynomial dependence
on the weight margin.

\subsubsection{Tensor action}
The tensor action of $\SL(n_0) \times \SL(n) \times \SL(n)$ on
$\CC^{n_0}\ot\CC^n\ot\CC^n$ (\cref{exa:tensor})
is of great interest in connection with the tensor rank problem and the complexity of matrix multiplication~\cite{strassen:87,burgisser2013algebraic}.
Already the case $n_0=3$ of tensors with ``three slices'' is poorly understood, while certain invariants in this setting are essential for proving
the best known lower bounds on the (border) rank of matrix multiplication tensors~\cite{strassen:83,blaser:99,landsberg-ottaviani:15}.
We obtain, for the first time, the existence of polynomial time algorithms for the null cone membership problem
in this setting (and more generally, when~$n_0$ is assumed to be fixed), see \cref{cor:tensor3rd}.
In fact, we show that the inverse weight margin of the tensor action is polynomially upper bounded in~$n$, provided~$n_0$ is fixed.

\subsubsection{Symmetric and antisymmetric tensor actions}
The natural actions of $\SL(n)$ on the spaces $V=\Sym^d\CC^n$ and $V=\bigwedge^d\CC^n$ of symmetric and antisymmetric tensors, respectively, are perhaps the simplest examples where the alternating minimization algorithms of prior work do not apply (cf.~\cref{subsubsec:noncommutative state of the art}).
In contrast, our general algorithms apply and can for constant~$n$ solve the null cone membership problem in polynomial time.
These actions are of great interest in quantum physics, since symmetric and antisymmetric tensors describe bosonic and fermionic particles, respectively.
In particular, the moment polytope membership problem captures the one-body marginal problem for bosons and fermions; the latter is also known as the pure state $N$-representability problem in quantum chemistry~\cite{altunbulak2008pauli}.

\subsubsection{Quiver representations with GL-action}
Quiver representations, which we introduced in~\cref{exa:quivers} (see~\cite{derksen2017introduction} for more details),
provide interesting families of examples for illustrating our results. 
Let $Q$ be a quiver, ${\vec n} = (n(x))$ be a dimension vector, and consider the action $\pi$ of
$\GL(\vec n) = \prod_{x\in Q_0} \GL(n(x))$ on $\Rep(Q,\vec n)$.
We obtain a polynomial time algorithm for the null cone membership for this action
via the lower bound on the weight margin in~\cref{prp:weight norm margin quivers}(2),
see \cref{cor:NCM-quivers}(1).
The reason for the large weight margin turns out to be the total unimodularity of the weight matrix of $\pi$ (\cref{prp:unimod}).

We note that, alternatively, polynomial time algorithms can also be derived from those for the conjugation action
via a reduction to the conjugation action (\cref{exa:SC}),
which is implicit in Le Bruyn and Procesi~\cite{lebruyn-procesi}.
(This reduction even works for the more general orbit closure intersection problem.)
We also point out that, while alternating minimization would work for the last two examples,
our continuous algorithms are still meaningful here since they provide a path for numerical algorithms to the orbit closure intersection problem
via the second order algorithm, as it was carried out in~\cite{AGLOW18} for the left-right action.

Beyond the null-cone problem for quivers, the moment polytope membership problem for the $\GL(n)$-action on quivers captures the validity of the Brascamp-Lieb inequalities in analysis (Problem~\ref{it:brascamp} in \cref{subsec:unexpected})~\cite{garg2017algorithmic}.
A very special case amounts to the placement of a tuple of $m$ vectors in $\RR^n$ into $p$-isotropic position.
See~\cite{garg2017algorithmic, AM13, franks2018operator} for further discussion.
Applied to this problem, our first-order algorithm places vectors in $p$-isotropic position to precision~$\eps$ in time $\poly(n,m,\eps^{-1})$, matching the guarantees from alternating minimization in~\cite{garg2017algorithmic}.
Similar algorithms can be found in the independent work~\cite{artstein2020radial}.

\subsubsection{Quiver representations with SL-action}
When we restrict the action of the quiver representation
to the subgroup~$\SL(\vec n) := \prod_{x\in Q_0} \SL(n(x))$,
call it $\pi_0$, the situation of quivers becomes more intricate.
Assuming that the quiver $Q$ has a bounded number of vertices,
\cref{cor:NCM-quivers}(2) shows the existence of polynomial time algorithms for the null cone membership.
This result is new to our best knowledge.
The reason behind is that the inverse weight margin $\gamma(\pi_0)^{-1}$ is polynomially upper bounded in $n :=\sum_x n(x)$,
see~\cref{prp:weight norm margin quivers}.
We note that this bound cannot be significantly improved since the weight margin of $\pi_0$ indeed can become exponentially small
if we allow the number of vertices of~$Q$ to grow~\cite{reichenbach-franks}.

If in addition, the quiver is {\em acylic}, there is a highly nontrivial reduction to the left-right action,
which is implicit in Derksen and Weyman~\cite{derksen2000semi}.
(This again works for the orbit closure intersection problem.)
Combined with polynomial time algorithms for the  left-right action,
this again yields polynomial time algorithms for the  null cone membership in this setting.
However, we emphasize that our numeric algorithms also work for quivers having cycles!

\subsection{Organization of the paper}
In \cref{sec:preliminaries}, we discuss preliminaries from group and representation theory.
In \cref{sec:convexity}, we present our main structural results about the geometry of non-commutative optimization including smoothness, robustness,
non-commutative duality, and gradient flow.
\Cref{sec:first-order algorithm,sec:second_order} contain the description and analysis of our first and second order algorithms, respectively,
for null cone and moment polytope membership and capacity computation.
\Cref{sec:norm margin bounds} contains useful bounds on weight norms and weight margins
with applications to representations of quivers and the tensor action.
In \cref{sec:bds-degree-capacity} we first recall known bounds on the degrees of generators of rings of invariants.
Then we derive bounds on the size of coefficients of such generators via Cayley's Omega process, which lead to general capacity lower bounds.
In \cref{sec:explicit algos}, we apply the capacity bounds to prove explicit running time bounds for our algorithms.
We also derive such bounds when the input is given in a Gelfand-Tsetlin basis; for this we review its construction
and prove an upper bound on the coefficient size for such bases.
We conclude in \cref{sec:conclusion} with a discussion of intriguing
open problems.

\section{Preliminaries}\label{sec:preliminaries}
In this section, we fix our basic notation and conventions and explain our basic group and representation theoretic setup.
Throughout this article, we will be working with representations of continuous groups on finite-dimensional vector spaces.
To make our article more accessible, we spell out explicitly all definitions in the important case when~$G=\GL(n)$ (see \cref{tab:summary gl} below).
Thus, \cref{subsec:groups,subsec:rep theory} can be skipped on a first reading.

\subsection{Notation and conventions}
Throughout the paper, $\log$ denotes the natural logarithm.
We abbreviate~$[m] := \{1,\dots,m\}$ for~$m\in\NN$.

All vector spaces are assumed to be finite-dimensional.
If $V$ is a complex vector space, let~$L(V)$ denote the space of linear maps from $V$ to $V$, and~$\GL(V) \subseteq L(V)$ the group of invertible linear maps from $V$ to $V$.
The identity operator in $L(V)$ is denoted by~$I$.
Now assume that $V$ is equipped with a Hermitian inner product~$\braket{\cdot,\cdot}$ (by convention linear in the second argument); we write $\norm{\cdot}$ for the corresponding norm.
We caution that even when~$V=\CC^m$ the inner product need \emph{not} be the standard one.
Then~$A^\dagger$ denotes the adjoint of an operator $A\in L(V)$.
Moreover, $U(V) \subseteq L(V)$ denotes the group of unitary operators (i.e., $U^\dagger U = U U^\dagger = I$), $\Herm(V) \subseteq L(V)$ the real vector space of Hermitian operators (i.e., $A = A^\dagger$),
and~$\PD(V) \subseteq L(V)$ the convex cone of positive definite operators in~$L(V)$.
Given an operator~$X\in L(V)$, we write~$\norm{X}_F := (\tr X^\dagger X)^{1/2}$ for the Frobenius norm
and $\norm{X}_{\ope} := \max_{\lVert v\rVert=1} \lVert X v\rVert$ for the operator norm.

In this paper we work with matrix subgroups of~$\GL(n)$, so it is useful to make distinct notation for each of the notions in the previous paragaph for~$\CC^n$ with its standard inner product.
For~$v, w \in \CC^n$ we define $\langle v , w \rangle := \sum_{i=1}^n \overline{v_i} w_i$ and $\norm{v}_2 := \langle v,  v\rangle^{1/2} = (\sum_{i=1}^n \lvert v_i\rvert^2 )^{1/2}$.
Let $\Mat(n) \cong L(\CC^n)$ denote the space of complex $n\times n$ matrices,
and denote by~$\Herm(n)\subseteq\Mat(n)$ the real space of Hermitian matrices,
by~$\PD(n) \subseteq \Mat(n)$ the convex cone of positive definite Hermitian matrices,
by~$\GL(n)\subseteq\Mat(n)$ the \emph{general linear group} consisting of the invertible matrices,
and by~$U(n)$ the \emph{unitary group} consisting of unitary matrices.
For~$A \in \Mat(n)$, $A^\dagger$ denotes the conjugate transpose of $A$, and we also use~$I$ for the identity matrix in $\Mat(n)$.
We also write~$\Mat(n,R)$ for the~$n\times n$-matrices over a commutative ring~$R$ (e.g., the integers or Gaussian integers).

\subsection{Groups}\label{subsec:groups}
We now define the groups that our algorithms deal with in full generality and explain some of their main structural properties
required in the analysis of the algorithms.
A useful general reference for this material is~\cite{wallach2017geometric}.

\begin{table}
\begin{tabular}{@{}ll@{}}
\toprule
Notation or concept & definition for $\GL(n)$ \\
\midrule
$G$ & $\GL(n)$, invertible $n\times n$-matrices \\
$K \subseteq G$ & $U(n)$, unitary $n\times n$-matrices \\
$B \subseteq G$ & upper triangular invertible $n\times n$-matrices \\
$N \subseteq B$ & upper triangular $n\times n$-matrices with 1s on diagonal \\
$P$ & $\PD(n)$, positive definite $n\times n$-matrices \\
$T_G$ & $\T(n)$, diagonal invertible $n\times n$-matrices \\
$T_K$ & diagonal unitary $n\times n$-matrices \\
\midrule
$G = KP$ & polar decomposition \\
$G = KB$ & QR decomposition \\
\midrule
$\Lie(G)$ & $\Mat(n)$, complex $n\times n$-matrices \\
$\Lie(K)$ & $\Herm(n)$, Hermitian $n\times n$-matrices \\
$\Lie(T_G)$ & complex diagonal $n\times n$-matrices \\
$\Lie(T_K)$ & purely imaginary diagonal $n\times n$-matrices \\
\midrule
$i\Lie(T_K)$ & real diagonal $n\times n$-matrices, usually identified with $\RR^n$ \\
$C(G) \subseteq i\Lie(T_K)$ & $C(n) := \{ p \in \RR^n : p_1 \geq \dots \geq p_n \}$ \\
$\intopoly\colon i\Lie(T_K)\to C(G)$ & $\spec$, sorted eigenvalues of a Hermitian $n\times n$-matrix \\
$p^*$ for $p\in C(G)$ & $p^* = (-p_n,\dots,-p_1)$ \\
$\omega$ weight & $\omega \in \ZZ^n$, corresponding to irreducible representation of $\T(n)$ given \\
  & by $\diag(x) \mapsto \prod_{j=1}^n x_j^{\omega_j}$ \\
$\lambda$ highest weight & $\lambda \in \ZZ^n \cap C(n) = \{ \lambda \in \ZZ^n : \lambda_1 \geq \dots \geq \lambda_n \}$ \\
\midrule
$\pi\colon G \to \GL(V)$ & arbitrary representation \\
$\pi_\lambda\colon G \to \GL(V_\lambda)$ & irreducible representation with highest weight~$\lambda$ \\
\bottomrule
\end{tabular}
\caption{Summary for $\GL(n)$}\label{tab:summary gl}
\end{table}

A subgroup $G$ of $\GL(n)$ is called \emph{symmetric}%
\footnote{Not to be confused with the \emph{symmetric group}, which is the group of permutations of a finite set.}
if it is Zariski-closed and it holds that~$g^\dagger\in G$ for every element~$g\in G$.
Here, Zariski-closed means that $G$ is a subset of $\GL(n)$ defined by polynomial equations in the matrix entries~$g_{i,j}$.
For example, $\SL(n) = \{ g \in \GL(n) : \det(g)=1 \}$.
Symmetric subgroups are linearly reductive, and it can be shown that, up to conjugation, any complex reductive algebraic group is of this form~\cite[\S 3.1-3.2]{wallach2017geometric}.
We will also demand that~$G$ is connected (in the standard topology, which is induced by any of the matrix norms defined above).

Next, define $K := G \cap U(n)$, the set of unitary matrices in~$G$, which forms a maximal compact subgroup of $G$.
Define $\Lie(K) := \{ X \in \Mat(n) : \forall t\in\RR\: e^{tX} \in K \}$, and likewise for~$G$.
Then $\Lie(K)$ is a real Lie algebra and $\Lie(G)$ is a complex Lie algebra.
This means that they are real and complex vector spaces, respectively, and that they are closed with respect to the Lie bracket~$[X,Y]=XY-YX$.
Moreover, $\Lie(K)$ is a subset of the space $i \Herm(n)$ of \emph{skew-Hermitian matrices},
so $i\Lie(K) \subseteq \Herm(n)$ and we have~$\Lie(G) = \Lie(K) \op i \Lie(K)$,
which means the Lie algebra of~$G$ is the \emph{complexification} of that of~$K$.
In fact, $i\Lie(K) = \Lie(G) \cap \Herm(n)$, and so
\begin{equation}\label{eq:def-P}
  P := \exp(i Lie(K)) = G \cap \PD(n).
\end{equation}
Moreover, any $g\in G$ has a unique \emph{(right) polar decomposition}~$g = k e^H$, where $k\in K$ and~$H\in i\Lie(K)$.%
\footnote{In fact, the map $K \times i\Lie(K) \to G$, $(k,H) \mapsto k e^H$ is a diffeomorphism.}

It is a crucial property of the functions we are optimizing that they are invariant under left multiplication by $K$, as we discuss in detail in \cref{sec:convexity}.
Hence these functions can be viewed as being defined on the set~$K\backslash G := \{ Kg : g\in G\}$ of right cosets of the subgroup $K\subseteq G$.
There are two natural ways of identifying~$K \backslash G$ with $P$.
First, we can use the natural map $P\to K \backslash G$
which sends $p \mapsto Kp$ and is a diffeomorphism by the polar decomposition.
In particular, $K\backslash G = \{ Kp : p \in P \}$.
Second, we can consider the surjective ``squaring map'':
\begin{equation}\label{eq:KG-P}
 G\to P,\, g\mapsto  g^\dagger g.
\end{equation}
Its fibers are exactly the cosets $Kg$, so this induces another identification of $K \backslash G$ with $P$.
It is known that~$K\backslash G$ is naturally a \emph{symmetric space} with non-positive sectional curvature; see, e.g.,~\cite{helgason1979differential} or \cite[Chap.~8]{petersen:16}.
The geodesics in $K\backslash G$ take the form $K e^{t H}g$, which under our second identification with $P$ take the form~$g^\dagger e^{2tH} g$;
the latter are geodesics with respect to a commonly studied Riemannian metric on~$\PD(n)$~\cite{farautanalysis}.
In \cref{subsec:geodesic_convexity_defns} we explain this in more detail and in \cref{subsec:smooth robust log norm} we show that the optimization problem of interest satisfies convexity properties along such geodesics.

Let $T_K$ be a maximal connected commutative subgroup of~$K$.
Then the group $T_G := \exp(\Lie(T_K) + i\Lie(T_K)) = T_K \exp(i\Lie(T_K))$ is a maximal connected commutative subgroup of~$G$.
Throughout this paper, we will assume for concreteness that $T_K$ consists of diagonal matrices.%
\footnote{This can always be achieved for the classical Lie groups~\cite{fulton2013representation}. However, it is not necessary and used only in \cref{sec:norm margin bounds}.}
Then, $\Lie(T_G) = \{ X \in \Lie(G) : X \text{ is diagonal} \}$ and similarly for $\Lie(T_K)$.
We record the following generalization of the singular value decomposition, known as the \emph{Cartan decomposition}:
\begin{align}\label{eq:cartan}
  G = K T_G K = K \exp(i\Lie(T_K)) K
\end{align}
The group~$T_G$ is itself a symmetric subgroup of~$\GL(n)$ and so the theory of this paper is also applicable to~$T_G$.
In fact, $T_G$ is always isomorphic to the group of diagonal $r\times r$ matrices in~$\GL(r)$ for some integer~$r$, called the \emph{rank} of~$G$.
As discussed in the introduction, this commutative case corresponds to geometric programming and enjoys a global convexity property that is simpler than the non-commutative case.
At the same time, the subgroup~$T_G$ plays an important role in the representation theory of~$G$, as we explain in \cref{subsec:rep theory}.

Finally, let $B \subseteq G$ be a \emph{Borel subgroup}, i.e., a maximal connected solvable subgroup, that contains~$T_G$.
Here, solvable means that if we inductively define $B^{(k)} := \{ g h g^{-1} h^{-1} : g, h \in B^{(k-1)} \}$, with $B^{(0)} := B$, then we eventually reach the trivial subgroup, i.e., $B^{(k)} = \{I\}$ for some~$k$.
Moreover, define~$N = \{ b \in B :  (b - I)^n = 0 \}$.
Then we have the following generalization of the QR decomposition, known as the \emph{Iwasawa decomposition}:
\begin{align}\label{eq:iwasawa}
  G = K B = K T_G N = K \exp(i\Lie(T_K)) N
\end{align}
Indeed, if $G = \GL(n)$ then $B$~consists of the upper-triangular invertible $n\times n$-matrices and $N$~consists of the upper-triangular matrices with all ones on the diagonal.
The final decomposition in~\eqref{eq:iwasawa} amounts to writing an invertible matrix as a product of a unitary matrix, a diagonal matrix with positive diagonal entries, and an upper triangular matrix with all ones on the diagonal.
As a consequence:
Any element $H\in \Lie(K)$ can be decomposed as~$H = D + R + R^\dagger$, where~$D \in \Lie(T_K)$ and~$R \in \Lie(N)$, and this decomposition is orthogonal with respect to the Hilbert-Schmidt inner product (i.e., $\norm H_F^2 = \norm D_F^2 + 2 \norm R_F^2$).

\subsection{Representations}\label{subsec:rep theory}
In this section, we briefly discuss the basics of representation theory.
The point which is crucial for us is that every representation can be associated with discrete data (a set of integer vectors)
and the properties of these vectors will govern the running time of our algorithms.

Let $G \subseteq \GL(n)$ be a symmetric subgroup as defined in \cref{subsec:groups}.
Let $\pi\colon G\to\GL(V)$ be a rational representation of~$G$.
That is, $\pi$ is a group homomorphism, i.e., $\pi(gh) = \pi(g)\pi(h)$ for all~$g,h\in G$, and in any basis of $V$ the matrix entries of $\pi(g) \in \GL(V)$ are polynomials
in the matrix entries~$g_{i,j}$ and in~$\det(g)^{-1}$.
There always exists a $K$-invariant inner product on~$V$; we let~$\langle \cdot, \cdot \rangle$ denote such an inner product.
That is, $\langle \cdot, \cdot \rangle$ is an inner product such that $\pi(K) \subseteq U(V)$.
Even though we will often work with~$V = \CC^m$, this inner product need \emph{not} be the standard inner product.
For instance, if~$\Sym^2(\CC^2)$ is identified with~$\CC^3$ by the monomial basis, then the standard inner product is not invariant under the action of~$U(2)$ on~$\Sym^2(\CC^2)$.

We now define a number of objects associated to the representation $\pi$.
Consider the complex-linear map~$\Pi\colon \Lie(G) \rightarrow L(V) :=\Lie(\GL(V))$ given by
\begin{align}\label{eq:lie alg rep}
\Pi(H) := \partial_{t=0} \pi\bigl(e^{tH}\bigr).
\end{align}
This is the \emph{Lie algebra representation} corresponding to~$\pi$; in particular, the identity~$\Pi([X,Y])=[\Pi(X),\Pi(Y)]$ holds for all $X,Y\in\Lie(G)$.
It holds that $e^{\Pi(X)} = \pi(e^X)$ for every~$X\in\Lie(G)$.
Furthermore, $\Pi(\Lie(K)) \subseteq i\Herm(V)$, so $\Pi(i\Lie(K)) \subseteq \Herm(V)$, and $\Pi(X^\dagger) = \Pi(X)^\dagger$ for every~$X\in\Lie(G)$.
A representation is called \emph{trivial} if $\pi(g)=I$ for every $g\in G$.

A representation is called \emph{irreducible} if it contains no invariant subspace other than $\{0\}$ and~$V$ itself, i.e., there exists no subspace~$\{0\} \subsetneq W\subsetneq V$
such that~$\pi(G) W \subseteq W$.
Any representation of~$G$ can be decomposed into a direct sum of irreducible representations.
This means that there exist irreducible representations $\pi_k\colon G\to\GL(m_k)$, $\sum_k m_k = m$, and a unitary $u\in U(V)$, such that
$u^\dagger \pi(g) u = \bigoplus_k \pi_k(g) $ for all~$g\in G$.
That is, up to a base change, the representation $\pi$ can be decomposed into diagonal blocks, each of which corresponds to an irreducible representation.

If we restrict the representation to the commutative subgroup~$T_G$ then this decomposition is particularly simple, since it amounts to a joint diagonalization of
the pairwise commuting operators~$\{ \pi(h) : h \in T_G \}$ or $\{ \Pi(H) : H \in \Lie(T_G) \}$.
Thus, there exists a decomposition $V = \bigoplus_{\omega\in\Omega(\pi)} V_\omega$, labeled by a set~$\Omega(\pi) \subseteq i\Lie(T_K)$, such that
\begin{align}
\label{eq:torus on weight vec}
  \pi(e^H) v_\omega = e^{\tr[H \omega]} v_\omega \qquad\text{and}\qquad
  \Pi(H) v_\omega = \tr[H \omega] v_\omega
\end{align}
for all $v_\omega \in V_\omega$ and $H\in \Lie(T_G)$.
The vectors $\omega$ are called \emph{weights}, the spaces~$V_\omega$ are called \emph{weight spaces}, and its elements~$v_\omega$ are called \emph{weight vectors}.
Each $\CC v_\omega$ is a one-dimensional irreducible representation of~$T_G$.
Since $i\Lie(T_K)$ consists of real diagonal matrices, the weights $\omega\in i\Lie(T_K)$ are real diagonal matrices.
It is known that the set of all possible weights is a lattice (isomorphic to~$\ZZ^r$, where $r$ is the rank as defined above).

Often, it will be convenient to view a weight~$\omega$ as the vector in $\RR^n$ given by its diagonal entries.
(Its Frobenius norm~$\lVert\omega\rVert_F$ then just equals the Euclidean norm $\norm\omega_2$.)
We then view $\Omega(\pi)$ as a sublattice of $\RR^n$.
We define the \emph{weight matrix}~$M(\pi)$ of the representation~$\pi$ as the matrix of format~$\lvert\Omega(\pi)\rvert \times n$,
whose rows are given by the elements of~$\Omega(\pi)$ in a chosen ordering.
(The ordering will not be relevant, but note that each weight appears only once, irrespective of its multiplicity in~$\pi$).

For the group $\GL(n)$, the weights are precisely the integer diagonal matrices, due to the fact that the exponential function is $2\pi i$-periodic.
For instance, the weight lattice of the group $\GL(n)$ consists of the diagonal matrices with integer entries and it is isomorphic to $\ZZ^n$.
The weight lattice of the subgroup $\SL(n)$ is given by the orthogonal projection of~$\ZZ^n$ onto the orthogonal complement of the all-ones vector, which is a sublattice of the trace-zero diagonal matrices with entries in $\frac{1}{n}\ZZ$, and isomorphic to $\ZZ^{n-1}$.
More specifically, if $\pi$ is a rational representation of $\GL(n)$, the set of weights~$\Omega(\pi_0)$ of its restriction~$\pi_0$ to $\SL(n)$ can be expressed in terms of $\Omega(\pi)$ as follows:
\begin{equation}\label{eq:Omega_pi_0}
 \Omega(\pi_0) = \big\{ \omega - \frac{\tr[\omega]}n \Id_n \mid \omega \in \Omega(\pi) \big\} .
\end{equation}
If the representation~$\pi\colon\GL(n)\to\GL(V)$ is homogeneous of degree~$d$, then $\tr[\omega]=d$ for any weight~$\omega\in\Omega(\pi)$.
Accordingly, the weight matrix $M(\pi_0)$ of its restriction $\pi_0$ to $\SL(n)$ can be obtained from $M(\pi)$ by adding a rank one matrix as follows:
\begin{equation}\label{eq:Mpi_0}
 M(\pi_0) = M(\pi) - \frac{d}{n} 1_{\lvert\Omega(\pi)\rvert} 1_n^T,
\end{equation}
where $1_k$ denotes the all-ones vector in $\RR^k$ for $k\in\NN$.

For a general symmetric subgroup~$G \subseteq \GL(n)$, the irreducible representations are labeled by a subset of the weights of~$T_G$, called the \emph{highest weights}.
We denote the irreducible representation with highest weight~$\lambda$ by~$\pi_\lambda\colon G \to \GL(V_\lambda)$.
The space $V_\lambda$ contains a one-dimensional invariant subspace for the Borel subgroup~$B$, spanned by a (unique up to phase) unit vector~$v_\lambda$.
The vector~$v_\lambda$ is called a highest weight vector.
It is a weight vector of weight~$\lambda$ and $N$-invariant.
The latter means that
\begin{align}
\label{eq:unipotent on highest weight vec}
  \pi(b)v_\lambda = v_\lambda \qquad\text{and}\qquad \pi(R)v_\lambda = 0
\end{align}
for all~$b\in N$ and $R\in\Lie(N)$.
In general, the highest weights of the irreducible representations that appear in~$\pi$ form a subset of the set of weights~$\Omega(\pi)$.
The set of all possible highest weights spans a convex cone known as the \emph{positive Weyl chamber}, denoted~$C(G) \subseteq i\Lie(T_K)$.
There is an involution $p \mapsto p^*$ on $C(G)$ such that, for every highest weight~$\lambda$, $\lambda^*$ is the highest weight of the dual representation,
i.e., $\pi_{\lambda^*} \cong \pi_\lambda^*$.
For any $H\in i\Lie(K)$, the intersection $\{ k H k^\dagger : k \in K \} \cap C(G)$ is a single point, which we denote by~$\intopoly(H)$.
The function~$\intopoly$ generalizes the function $\spec$ taking a matrix to its spectrum, only with one technical difference:
the image of~$\intopoly$ is a matrix in~$C(G)$ rather than~$\RR^n$ as for~$\spec$;
for instance, if $G=GL(n)$, then $\intopoly(H) = \diag(\spec(H))$.
We will repeatedly use the fact that~$\intopoly(-p^*) = p$ for any $p\in C(G)$.

\section{Geometry of non-commutative optimization}\label{sec:convexity}
In this section, we first define the main optimization problem of interest, which is a norm minimization problem over group orbits.
We then discuss the geometric properties of the objective function. While it is well-known that this function is in some sense log-convex,
we will prove stronger convexity properties that will be instrumental for the algorithms discussed in the sequel.
Throughout the paper, we work in the setup introduced in \cref{sec:preliminaries},
with $\pi\colon G\to\GL(V)$ a representation of a symmetric subgroup~$G \subseteq \GL(n)$ and $\langle \cdot, \cdot \rangle$
a $K$-invariant inner product on $V$.

\subsection{Capacity and moment map}
The norm minimization problem is concerned with solving the following optimization problem, the value of which we call the capacity.

\begin{dfn}[Capacity]\label{dfn:capacity}
The \emph{capacity} of a vector~$v\in V$ is defined as the infimum of the norm on its $G$-orbit.
Formally,
\begin{align*}
  \capacity(v)
:= \inf_{g \in G} \lVert \pi(g) v\rVert
= \min_{w \in \overline{\pi(G) v}} \lVert w\rVert.
\end{align*}
\end{dfn}

In the second formula, the closure can be taken with respect to the standard topology (i.e., the one defined by the norm).
The capacity is manifestly $G$-invariant, i.e., $\capacity(\pi(g)v) = \capacity(v)$ for all~$v\in V$ and~$g\in G$.

In geometric invariant theory, vectors are called \emph{unstable} if $\capacity(v)=0$ and otherwise \emph{semistable}.
The set of unstable vectors forms the so-called \emph{null cone} (this is a cone in the sense of algebraic geometry, namely closed under multiplication by arbitrary scalars).
These are important in geometric invariant theory, defined by Mumford~\cite{Mum65}, and go back to ideas introduced by Hilbert in his work on invariant theory~\cite{Hil}.

We are interested in the log-convexity properties of the objective function, so we define, for~$0\neq v\in V$, the following function, also known as the \emph{Kempf-Ness function}:
\begin{align}\label{eq:kempf ness function}
  F_v\colon G \to [0,\infty), \quad g \mapsto \log \norm{ \pi(g) v } = \frac12\log \norm{ \pi(g) v }^2.
\end{align}
It is useful to observe that this function is right-$G$-equivariant and left-$K$-invariant in the sense that
\begin{align}\label{eq:left right variances}
  F_v(kgh) &=
\log \norm{\pi(kgh) v } =
\log \norm{ \pi(g) \pi(h) v } =
F_{\pi(h) v}(g),
\end{align}
for all $k\in K$, $g,h\in G$, and $0\neq v\in V$.
Here we used that~$\pi(K) \subseteq U(V)$, so group elements in~$K$ do not change the norm.
The equivariance property shows that it suffices to study the local properties of~$F_v$ in a neighborhood of the identity element~$g=I$.
The invariance property implies that we can also think of~$F_v$ as a function on the quotient space $K\backslash G = \{ Kg : g\in G\} = \{ Ke^H : H \in i\Lie(K) \} $.
Therefore, the following definition captures the gradient of~$F_v$ at~$g=I$:

\begin{dfn}[Moment map]\label{dfn:moment_map}
The \emph{moment map} is the function $\mu: V \setminus \{0\} \rightarrow i\Lie(K)$ defined by the property that, for all $H\in i\Lie(K)$,
\begin{align*}
  \tr\bigl[\mu(v) H\bigr]
= \partial_{t=0} F_v(e^{tH})
= \frac{\braket{v, \Pi(H) v}}{\norm{ v }^2}.
\end{align*}
\end{dfn}
\noindent
Here, we recall that~$\Pi$ is the Lie algebra representation defined in \cref{eq:lie alg rep}.
Since it is linear in $H$, the function $\mu(v)$ is well-defined, and since $\Pi$ is efficiently computable \cite{burgisser2017membership}, so is $\mu(v)$.
It is also a moment map in the sense of symplectic geometry (for the $K$-action on the projective space over~$V$), which will have
important implications in \cref{subsec:moment polytopes theory}.
We note that $\mu(\lambda v) = \mu(v)$ for $\lambda \in \CC^* = \CC\setminus \{0\}$.

\begin{rem}
In the literature, the moment map is often defined as a function to the dual of~$\Lie(K)$ or of~$i\Lie(K)$.
For us it is convenient to identify the dual with~$i\Lie(K)$ so that we can think of the moment map concretely as computing gradient vectors rather than derivatives, which are naturally covectors.
\end{rem}

\noindent
If $G$ is commutative (i.e., $G = T_G$ and $K = T_K$) then one can write down a more concrete formula for the capacity and the moment map.
Write $v = \sum_{\omega\in\Omega(\pi)} v_\omega$, with~$v_\omega$ contained in the weight space~$V_\omega$ (cf.~\cref{subsec:rep theory}), and let $\Gamma = \{ \omega \in \Omega(\pi) : v_\omega \neq 0 \}$ denote the \emph{support} of~$v$.
Then, \cref{eq:torus on weight vec} implies that the squared capacity is given by the unconstrained geometric program
\begin{align}\label{eq:commutative capacity}
  \capacity^2(v)
= \inf_{H \in i\Lie(T_K)} \sum_{\omega\in\Gamma} \norm{v_\omega}^2 e^{2 \tr[\omega H]}.
\end{align}
This is a convex program over the real vector space $i \Lie(T_K)$. It follows from Farkas lemma that the capacity is zero if and only if the convex hull of $\Gamma$ does \emph{not} contain the origin.
Moreover, since weight spaces are pairwise orthogonal, we find that the moment map is given by the convex combination
\begin{align}\label{eq:commutative moment map}
  \mu(v)
= \sum_{\omega\in\Gamma} \frac{\norm{v_\omega}^2}{\norm{v}^2} \omega,
\end{align}
which is a point in the convex hull of the support of~$v$.
We will return to these observations later.

\subsection{Geodesic convexity}\label{subsec:geodesic_convexity_defns}
The group $G$ and its quotient $K\backslash G = \{ Kg : g\in G\}$ are not Euclidean spaces, but rather manifolds with an interesting topology and curvature.
Therefore, the usual notions of convexity do not apply.
However, it is well-known that the Kempf-Ness function is convex along certain curves, which may be interpreted as geodesics.
We now show how to appropriately generalize the definitions of convex optimization to this scenario.
Next, we prove some quantitative results that have not been discussed in the literature, but which will be crucial to our algorithms.

Recall that to compute capacities, we will want to optimize left-$K$-invariant functions on $G$, namely the Kempf-Ness functions in \cref{eq:kempf ness function}.

We now provide some background on the geometry of $K\backslash G$.
While these details are not required for following the rest of the paper, they help to arrive at a deeper understanding of the mathematical setting.
For the following, see~\cite{helgason1979differential} or \cite[Chap.~8]{petersen:16}.
The group $G$ is a Lie group and hence carries the structure of a smooth manifold.
Clearly, we have a transitive right action of $G$ on $K\backslash G = \{ Kg : g\in G\}$ given by $(Kg,h) \mapsto Kgh$, and the stabilizer group of the coset~$K$ equals~$K$.
One can then define on the quotient $K\backslash G$ the structure of a smooth manifold in a natural way such that the canonical map $G\to K\backslash G$ is a submersion (its derivative is an orthogonal projection on tangent spaces).
Its derivative at the identity~$I$ allows the tangent space of $K\backslash G$ at $K$ to be identified with $\Lie(G)/\Lie(K)$.
The inner product $\langle X_1,X_2\rangle = \Re \tr(X_1^\dagger X_2)$ on $\Lie(G)\subseteq\Mat(n)$
is invariant under left and right multiplication with elements of~$K$: we have
$\langle k X_1k', k X_2 k'\rangle = \langle X_1,X_2\rangle $ for $k,k'\in K$.
Since we identify $i\Lie(K)$ with the quotient space $\Lie(G)/\Lie(K)$, this defines an inner product on the tangent space of~$K\backslash G$ at~$K$.
We define inner products on all tangent spaces of $K\backslash G$ by requiring the right multiplications
$K\backslash G \to K\backslash G, Kg \mapsto Kgh$ to be isometric for all $h\in G$.
(This is well-defined since the inner product on $\Lie(G)$ we started with is invariant under left and right multiplication with elements of~$K$.)
In this way, $K\backslash G$ becomes a Riemannian manifold with a right $G$-invariant metric.%
\footnote{$K\backslash G$ is even a symmetric space with respect to the involution $g\mapsto (g^{-1})^\dagger$.
Moreover, it has non-positive sectional curvature; see, e.g.,~\cite{helgason1979differential} or \cite[Chap.~8]{petersen:16}.}
One can show that for $H\in i\Lie(K)$ and $g\in G$,
\begin{align}\label{eq:exp-geodes}
  \gamma\colon\RR\to K\backslash G, \, t \mapsto Ke^{tH} g
\end{align}
is the general form of a geodesic in $K\backslash G$. 
For the corresponding geodesic distance, we obtain
\begin{align*}
 d(K,Kg) = \frac12\norm{\log(g^{\dagger}g)}_F.
\end{align*}
since $Ke^{\frac12\log(g^\dagger g)} = Kg$.
More generally, combining this with the right invariance, for $g, h \in G$ we obtain
\begin{align}\label{eq:geodesic distance}
  d(Kh, Kg)
= d(K, Kgh^{-1})
= \frac12\norm{\log(h^{-\dagger} g^\dagger g h^{-1})}_F.
\end{align}

This can be made more concrete by identifying $K\backslash G$ with $P=\exp(i\Lie(K))$ via the diffeomorphism $K\backslash G \to P,\, Kg\mapsto g^\dagger g$,
see~\cref{eq:KG-P}. Note that its derivative at $K$ gives the map $H\mapsto 2H$, where we identify $T_I P= i\Lie(K)$.
The above definition of the Riemannian metric translates to $\langle X_1,X_2 \rangle_I = \frac14 \tr (X_1 X_2)$, for $X_1,X_2 \in T_I P= i\Lie(K)$
and more generally at $p \in P$,
\begin{align}\label{eq:RM-formula}
 \langle \dot{Y}_1,\dot{Y}_2 \rangle_p
= \frac14 \tr (\dot{Y}_1 p^{-1} \dot{Y}_2 p^{-1}),
\end{align}
where $\dot{Y}_1,\dot{Y}_2 \in T_pP$.
For seeing this, we note that the transitive right action of $K$ on $K\backslash G$ corresponds to the following transitive right action of $K$ on $P$:
$(p, h) \mapsto h^\dagger p h$, where $p \in P$ and $h\in G$.
From~\cref{eq:exp-geodes} we obtain that the geodesics in $P$ take the form
$t\mapsto g^\dagger e^{2tH} g$, for $g\in G$.
In particular, the unique geodesic connecting
the points $p,q$ in $P$ takes the form
\begin{align*}
  \Gamma\colon\RR\to P,\, t \mapsto p^{\frac12} e^{tH} p^{\frac12}, \quad \mbox{ where } H:= \log(p^{-\frac12} q p^{-\frac12}) .
\end{align*}
Using \cref{eq:geodesic distance}, 
it follows that the geodesic distance between $p,q\in P$ is given by
\begin{align}
  d_P(p, q) = \frac12 \, \norm{\log(p^{-\frac12} q p^{-\frac12})}_F .
\end{align}
For the group $G=\GL(n)$, the space $P=\PD(n)$ consists of the positive definite $n\times n$ matrices, and $i\Lie(K)=\Herm(n)$ are the Hermitian $n\times n$ matrices.
Thus we recognize the well-known formulas for the geodesics and (up to a factor of~2) the geodesic distance in $\PD(n)$ in~\cite{bhatia2009positive}.

\begin{rem}
Rather than optimizing the function~\eqref{eq:kempf ness function} on~$G$ or equivalently~$K\backslash G$, as we do in this paper, one could by using the identification $K g \mapsto g^\dagger g$ also develop the formalism of this paper based on the function~$p \mapsto \log \braket{v, \pi(p) v}$ on~$P$.
\end{rem}

After explaining all this background, we now make the following definition for convenience.

\begin{dfn}[Good geodesic]
  A \emph{good geodesic} is a curve $\gamma\colon\RR\to G$ of the form $\gamma(t) = e^{tH} g$ where $H\in i\Lie(K)$ and $g\in G$.
  We say that $\gamma$ has \emph{unit speed} if~$\lVert H\rVert_F=1$.
\end{dfn}


By the above discussion, if $\gamma(t)$ is a good geodesic, then the curve
$t\mapsto Ke^{tH} g$ is a geodesic in $K\backslash G$, and any geodesic in~$K\backslash G$ can be obtained in this way.
These geodesics provide the natural curves in $G$ with respect to which we will momentarily define our generalized notion of convexity, smoothness, and robustness.

\begin{dfn}[Convex, smooth, robust]\label{dfn:convex etc}
Let $F\colon G\to\RR$ be a function that is left-$K$-invariant in the sense that $F(kg) = F(g)$ for all $k\in K$, $g\in G$.
Assume that $F$ is sufficiently differentiable such that all the derivatives below exist.
We say that
\begin{itemize}
  \item $F$ is \emph{(geodesically) convex} if
  \begin{align*}
    \partial^2_t F(\gamma(t)) \geq 0
  \end{align*}
  for every good geodesic~$\gamma(t) = e^{tH} g$ and $t\in\RR$.
  That is, $F$ is convex along all good geodesics (as a function of~$t$).
  \item $F$ is \emph{$L$-smooth} for some $L>0$ if
  \begin{align*}
    \left\lvert \partial^2_t F(\gamma(t)) \right\rvert \leq L \lVert H\rVert_F^2
  \end{align*}
  for every good geodesic~$\gamma(t) = e^{tH} g$ and $t\in\RR$.
  That is, it is $L$-smooth along all good geodesics with unit speed (as a function of~$t$).
  \item $F$ is \emph{$R$-robust} for some $R>0$ if
  \begin{align*}
    \left\lvert \partial^3_t F(\gamma(t)) \right\rvert \leq R \lVert H\rVert_F \, \partial^2_t F(\gamma(t))
  \end{align*}
  for every good geodesic~$\gamma(t) = e^{tH} g$ and $t\in\RR$.
\end{itemize}
\end{dfn}

\noindent
Any robust function is convex (the right-hand side contains the second derivative, not its absolute value).%
\footnote{We note that the notion of robustness is similar but different from the notion of self-concordance (which plays a crucial role in the analysis of Newton's method and interior point methods in the Euclidean world~\cite{nesterov1994interior}) which requires that $\left\lvert \partial^3_t F(\gamma(t)) \right\rvert \leq R \left(\partial^2_t F(\gamma(t)) \right)^{3/2}$ for every good geodesic~$\gamma(t) = e^{tH} g$ and $t\in\RR$ (generalizing the Euclidean definition to the geodesic world). One difference is that self-concordance is scale invariant whereas robustness is not.}
Just like in the Euclidean case, smooth convex functions and robust functions have local models that are useful for optimization.
To state these concisely, it is useful to introduce the following notions:

\begin{dfn}[Geodesic gradient and Hessian]\label{dfn:grad hessian}
Let $F\colon G\to\RR$ be a function that is left-$K$-invariant in the sense that $F(kg) = F(g)$ for all $k\in K$, $g\in G$.
Assume that $F$ is sufficiently differentiable such that all the derivatives below exist.
The \emph{geodesic gradient} at~$g\in G$ is defined as the vector $\nabla F(g) \in i\Lie(K)$ defined by
\begin{align}\label{eq:g grad}
  \tr\bigl[ \nabla F(g) H \bigr] = \partial_{t=0} F(e^{tH} g)
\end{align}
for all~$H\in i\Lie(K)$.
The \emph{geodesic Hessian} at~$g\in G$ is the symmetric tensor~$\nabla^2 F(g) \in \Sym^2(i\Lie(K))$ given by
\begin{align}\label{eq:g hess}
  \tr\bigl[ \nabla^2 F(g) (H \ot H) \bigr] = \partial^2_{t=0} F(e^{tH} g)
\end{align}
for all~$H\in i\Lie(K)$.
\end{dfn}
\noindent
In other words, $\nabla F(g)$ and $\nabla^2F(g)$ are the gradient and Hessian of the function~$f_g\colon i\Lie(K)\to\RR$ defined by~$f_g(H) := F(e^H g)$ at~$H=0$.

Smoothness implies that a function is universally upper-bounded by a quadratic expansion.

\begin{lem}\label{lem:quadratic_ub}
Let $F\colon G\to\RR$ be a convex and $L$-smooth function as defined in \cref{dfn:convex etc}.
Then,
\begin{align*}
  F(g) + \tr\bigl[\nabla F(g) H\bigr] \leq F(e^H g) \leq F(g) + \tr\bigl[\nabla F(g) H\bigr] + \frac L2 \lVert H\rVert_F^2
\end{align*}
\end{lem}
\begin{proof}
Consider the function~$f(t) := F(e^{tH} g)$.
By Taylor's approximation and the mean value theorem, we know that
\begin{align*}
  f(1) = f(0) + f'(0) + \frac12 f''(\zeta)
\end{align*}
for some~$\zeta \in [0,1]$.
By \cref{eq:g grad}, $f'(0) = \tr[\nabla F(g) H]$.
Finally, convexity and $L$-smoothness mean that
\begin{align*}
  0 \leq f''(t) \leq L \lVert H\rVert_F^2
\end{align*}
for all $t\in\RR$.
Thus the claim follows.
\end{proof}

\noindent
Similarly, robustness implies upper \emph{and} lower bounds in terms of local quadratic expansions.

\begin{lem}\label{cor:robust taylor}
Let $F\colon G\to\RR$ be an $R$-robust function as defined in \cref{dfn:convex etc}.
Then,
\begin{align*}
  F(g) + \partial_{t=0} F(e^{tH} g) + \frac1{2e} \partial^2_{t=0} F(e^{tH} g)
\leq F(e^H g)
\leq F(g) + \partial_{t=0} F(e^{tH} g) + \frac{e}2 \partial^2_{t=0} F(e^{tH} g)
\end{align*}
for every $g\in G$ and $H\in i\Lie(K)$ such that~$\lVert H\rVert_F\leq 1/R$.
\end{lem}
\begin{proof}
Consider the function~$f(t) := F(e^{tH} g)$.
Since $F$ is $R$-robust, it holds that~$|f'''(t)| \leq R \lVert H\rVert_F f''(t)$.
Then the claim follows from~\cite[Proposition B.1]{AGLOW18}, which asserts that if $f\colon\RR\to\RR$ satisfies $\lvert f'''(t)\rvert \leq \rho f''(t)$ for all~$t\in\RR$ then
\begin{align*}
  f(0) + f'(0) t + \frac1{2e} f''(0) t^2 \leq f(t) \leq f(0) + f'(0) t + \frac{e}{2} f''(0) t^2
\end{align*}
for all $|t| \leq \frac1\rho$.
\end{proof}

\subsection{Smoothness and robustness of the log-norm function}\label{subsec:smooth robust log norm}
We now return to the log-norm or Kempf-Ness function~$F_v$ defined in \cref{eq:kempf ness function}, the logarithm of the objective function that defines the capacity (\cref{dfn:capacity}).
Note that the moment map from \cref{dfn:moment_map} is nothing but its geodesic gradient at $g=I$.
More generally,
\begin{align}\label{eq:moment map vs gradient}
  \mu(\pi(g) v) = \nabla F_v(g).
\end{align}
We will prove that the left-$K$-invariant function $F_v$ is convex and prove bounds on its smoothness and robustness.
This will, unsurprisingly, depend on the properties of the Lie algebra representation.
In particular, we will see that it depends on the following norm:

\begin{dfn}[Weight norm]\label{dfn:weight norm}
We define the \emph{weight norm} of the representation~$\pi$ by
\begin{align*}
  N(\pi) := \max_{H\in i\Lie(K), \lVert H\rVert_F = 1} \lVert \Pi(H) \rVert_{\ope}.
\end{align*}
That is, the weight norm is an induced norm of the Lie algebra representation~$\Pi$ defined in \cref{eq:lie alg rep}, where we equip the Lie algebra with the Frobenius norm and the linear operators on~$V$ with the usual operator norm.
\end{dfn}

\noindent
The weight norm can be computed explicitly in terms of representation-theoretic data, which justifies the name.
For this, we borrow the following result from~\cite[Proof of Lemma~14]{burgisser2017membership}:

\begin{prp}\label{prp:opnorm}
The weight norm can be computed as
\begin{align*}
  N(\pi)
= \max \{ \lVert\omega\rVert_F : \omega \in \Omega(\pi) \}
= \max \{ \lVert\lambda\rVert_F : \pi_\lambda \subseteq \pi \}.
\end{align*}
In the first formula, we maximize over all weights of the representation~$V$ and in the second formula over all irreducible representations~$\pi_\lambda$ that appear in~$\pi$ (cf.~\cref{subsec:rep theory}).
\end{prp}

\noindent
Next, we show that the moment map (i.e., the gradient of~$F_v$) is universally bounded by the weight norm:

\begin{lem}[Bound on gradient]\label{lem:bound on gradient}
For every~$v\in V \setminus \{0\}$, we have that $\lVert \mu(v) \rVert_F \leq N(\pi)$.
\end{lem}
\begin{proof}
Using \cref{dfn:moment_map} with $H=\mu(v)\in i\Lie(K)$, we obtain
\begin{align*}
  \lVert \mu(v) \rVert_F^2
= \tr[\mu(v) \mu(v)]
= \frac{\braket{v, \Pi(\mu(v)) v}}{\lVert v\rVert^2}
\leq \lVert \Pi(\mu(v)) \rVert_{\ope}
\leq N(\pi) \lVert \mu(v) \rVert_F,
\end{align*}
from which the claim follows.
\end{proof}

\noindent
Now we are ready to prove the desired convexity properties. Moreover, because $\Pi$ is efficiently computable \cite{burgisser2017membership}, the explicit formula below shows that the Hessian is efficiently computable.

\begin{prp}[Convexity and smoothness]\label{prp:smooth}
For any~$v\in V\setminus \{0\}$, the function $F_v$ defined in \cref{eq:kempf ness function} is convex and~$2N(\pi)^2$-smooth. Moreover, the Hessian is given by
\begin{align} 
\tr \nabla^2 F_v (g) (H \ot  H) &= 2 \left(\langle u, \Pi(H)^2 u \rangle  -  \langle u, \Pi(H) u \rangle^2\right).\label{eq:hess-formula}
\end{align}
for all $H \in i \Lie(K)$, where $u = \frac{\pi(g) v}{\norm{\pi(g)v}}$.
\end{prp}

\begin{proof}
Consider a good geodesic $\gamma(t) = e^{tH} g$, so $H\in i\Lie(K)$ and $g\in G$.
Define $\tilde H := \Pi(H)$, $w(t) := \pi(e^{tH}) \pi(g) v = e^{t\tilde H} \pi(g) v$, and $f(t) := F_v(\gamma(t)) = \frac12 \log \lVert w(t) \rVert^2$.
Further, define unit vectors~$u(t) := \frac{w(t)}{\lVert w(t) \rVert}$.
Then, $w'(t) = \tilde H w(t)$, and after a short calculation we obtain that~$u'(t) = \bigl(\tilde H - \braket{u(t), \tilde H u(t)} I\bigr) u(t)$ and
\begin{align}
\nonumber
  f'(t) &= \braket{u(t), \tilde H u(t)}, \\
\label{eq:second derivative}
  f''(t) &= 2 \left( \lVert \tilde H u(t)\rVert^2 - \braket{u(t), \tilde H u(t)}^2 \right).
\end{align}
Evaluating the second derivative at $t = 0$ yields \cref{eq:hess-formula}. By the Cauchy-Schwarz inequality, $f''(t) \geq 0$, which proves convexity.
Moreover,
\begin{align*}
   |f''(t)| = f''(t)
\leq 2 \lVert \tilde H u(t)\rVert^2
\leq 2 \lVert \tilde H \rVert_{\ope}^2
\leq 2 N(\pi)^2 \norm{H}_F^2,
\end{align*}
where we used that the~$u(t)$ are unit vectors and \cref{dfn:weight norm}.
\end{proof}

\noindent
A simple corollary shows that~$F_v$ is universally upper-bounded by a quadratic expansion.


\begin{cor}\label{cor:log_norm_quadratic_ub}
For any~$v\in V\setminus \{0\}$, the function $F_v$ defined in \cref{eq:kempf ness function} satisfies
\begin{align*}
  F_v(g) + \tr\bigl[\mu\bigl(\pi(g)v\bigr) H\bigr] \leq F_v(e^H g) \leq F_v(g) + \tr\bigl[\mu\bigl(\pi(g)v\bigr) H\bigr] + N(\pi)^2 \lVert H\rVert_F^2
\end{align*}
for every $g\in G$ and $H\in i\Lie(K)$.
\end{cor}
\begin{proof}
This follows from \cref{lem:quadratic_ub,prp:smooth,eq:moment map vs gradient}.
\end{proof}

\begin{prp}[Robustness]\label{prp:g-GSOR}
For every~$v\in V\setminus \{0\}$, the function $F_v$ defined in \cref{eq:kempf ness function} is~$4N(\pi)$-robust.
\end{prp}
\begin{proof}
We continue the calculation in the proof of \cref{prp:smooth}.
On the one hand, we can rewrite \cref{eq:second derivative} as
\begin{align*}
  f''(t)
= 2 \braket{u(t), \bigl( \tilde H - \braket{u(t), \tilde H u(t)} I \bigr)^2 u(t)}
= 2 \left\lVert \left( \tilde H - \braket{u(t), \tilde H u(t)} I \right) u(t) \right\rVert^2.
\end{align*}
On the other hand, we obtain by taking another derivative that
\begin{align*}
  f'''(t) &= 4 \braket{u(t), \tilde H^3 u(t)} - 12 \braket{u(t), \tilde H u(t)} \braket{u(t), \tilde H^2 u(t)} + 8 \braket{u(t), \tilde H u(t)}^3 \\
  &= 4 \braket{u(t), \bigl( \tilde H - \braket{u(t), \tilde H u(t)} I \bigr)^3 u(t)}.
\end{align*}
By the Cauchy-Schwarz inequality and the triangle inequality, we obtain
\begin{align*}
  |f'''(t)|
\leq 2 \bigl\lVert \tilde H - \braket{u(t), \tilde H u(t)} I \bigr\rVert_{\ope} f''(t)
\leq 4 \lVert \tilde H \rVert_{\ope} f''(t)
\leq 4 N(\pi) \norm{H}_F f''(t).
\end{align*}
This proves the claim.
\end{proof}


\begin{rem}[Cumulant generating functions, tightness of \cref{prp:g-GSOR}]
The preceding two propositions can also be established by interpreting $f(t)$ as a cumulant generating function.
Without loss of generality, assume that $\pi(g)v$ is a unit vector.
Consider the spectral decomposition $\tilde H = \sum_{\omega\in\Omega} \omega P_\omega$, and define a random variable~$X$ by $\Pr(X=2\omega) = \lVert P_\omega \pi(g) v\rVert^2$.
Then, $ f(t) = \frac{1}{2}\log E[e^{tX}]$ is half the cumulant generating function of~$X$, so we can interpret the $k$\textsuperscript{th} derivative of~$f(t)$ at~$t=0$ as half the $k$\textsuperscript{th} cumulant of~$X$.
For $k=2,3$ the $k$\textsuperscript{th} cumulant is nothing but the $k$\textsuperscript{th} central moment, which when re-expressed in terms of $\tilde H$ yields the claim.
Likewise, the function $n(t)$ in the proof of \cref{prp:norm square robust} below has a pleasant interpretation in terms of a moment generating function. More inequalities between higher order derivatives can be obtained via this connection but it is not clear if they are useful.

The preceding discussion also implies that, for any given~$\eps>0$, there exists a representation~$\pi$ of, e.g., $G=\GL(1)$ that is not $(\frac{1}{2} N(\pi) - \eps)$-robust.
Showing this amounts to finding a distribution on~$[-2\beta, 2\beta] \cap \ZZ$ whose second and third central moments differ by a factor of~$\beta$, which is quite simple.
Thus, \cref{prp:g-GSOR} is tight up to a constant factor.
\end{rem}

A calculation similar to \cref{prp:g-GSOR} shows that the norm-square function is robust.
We state this in the following proposition (which has similar corollaries as the above).

\begin{prp}[Robustness]\label{prp:norm square robust}
For every~$0\neq v\in V$, the function $N_v(g) := \lVert \pi(g) v \rVert^2$ is left-$K$-invariant, convex, and~$2N(\pi)$-robust. Moreover, the Hessian $\nabla^2 N_v$ is given by
\begin{align}
\tr \nabla^2 N_v(g) H \ot H  = 4 \langle w, \Pi(H)^2 w \rangle \label{eq:norm-hess-formula}
\end{align}
for all $H \in i \Lie(K)$, where $w = \pi(g) v$.
\end{prp}
\begin{proof}
Since $\pi(K) \subseteq U(V)$, $N_v$ is clearly left-$K$-invariant.
To prove robustness, fix~$H\in i\Lie(K)$ and~$g\in G$ as before.
Define $\tilde H := \Pi(H)$, $w(t) := e^{t\tilde H} g v$, and $n(t) := \norm{ w(t) }^2$.
Then, $w'(t) = \tilde H w(t)$ and
\begin{align*}
  n'(t) &= 2 \braket{w(t), \tilde H w(t)}, \\
  n''(t) &= 4 \braket{w(t), \tilde H^2 w(t)} = 4 \Norm{\tilde H w(t)}^2, \\
  n'''(t) &= 8 \braket{w(t), \tilde H^3 w(t)} = 8 \braket{\tilde H w(t), \tilde H \tilde H w(t)}
\end{align*}
\cref{eq:norm-hess-formula} follows by evaluating the second derivative at $t = 0$. Thus, by the Cauchy-Schwarz inequality,
\begin{align*}
  \lvert n'''(t)\rvert
\leq 8 \norm{\tilde H}_{\ope} \norm{\tilde H w(t)}^2
= 2 \norm{\tilde H}_{\ope} n''(t)
\leq 2 N(\pi) \norm{H}_F n''(t).
\end{align*}
We conclude that~$N_v(g)$ is $2N(\pi)$-robust.
\end{proof}

\subsection{Non-commutative duality theory}\label{subsec:duality}
As discussed in the preceding section, the log-norm function is geodesically convex in the sense of \cref{dfn:convex etc}.
In particular, it follows that critical points of the norm function are global minima, and it is not hard to see that, within each orbit closure, minima are unique up to the action of~$K$.
These are basic and important results of geometric invariant theory~\cite{Mum65,kempf1979length}.
For example, the well-known Kempf-Ness theorem~\cite{kempf1979length} asserts that
\begin{equation}\label{eq:kempf ness}
 \capacity(v) = \inf_{g \in G} \norm{\pi(g) v} = \!\!\!\!\min_{w \in \overline{\pi(G) v}} \lVert w\rVert > 0
\;\iff\; \inf_{g\in G} \Norm{\mu(\pi(g) v)}_F = \!\!\!\!\min_{w \in \overline{\pi(G) v}} \Norm{\mu(w)}_F = 0,
\end{equation}
From the perspective of optimization, this means that we can think of computing the capacity (i.e., minimizing the norm in an orbit closure) and minimizing the moment map (i.e., mimizing the gradient of the log-norm) as two \emph{dual problems} -- a point of view that was initially taken in~\cite{burgisser2017alternating,burgisser2018efficient}.

In the following, we will prove two results that show that the norm of a vector~$v$ is close to its minimum (the capacity) if and only if the moment map is small.
From the perspective of optimization theory, they relate the primal gap and dual gap of the two optimization problems.
This improves over non-commutative duality theory developed in~\cite{burgisser2017alternating,burgisser2018efficient} and systematizes and generalizes results for matrix and operator scaling~\cite{LSW, garg2016deterministic} to arbitrary group representations.

We first prove the most difficult part, which is to show that if the gradient is small then the norm of~$v$ is close to its minimum.
Before stating our quantitative bound, we recall the definition of weight margin.

\begin{dfn}[Weight margin; precise statement of \cref{dfn:comp-measure-wm2}]\label{dfn:weight margin}
We define the \emph{weight margin} of the representation~$\pi$ by
\begin{align*}
  \gamma(\pi) := \min \bigl\{ d(0, \conv(\Gamma)) \;:\; \Gamma \subseteq \Omega(\pi), \, \conv(\Gamma) \not\ni 0 \bigr\},
\end{align*}
where $d(0,\conv(\Gamma)) := \min \{ \norm{x}_F : x\in\conv(\Gamma) \}$ the minimal distance from the convex hull of~$\Gamma$ to the origin; we recall that $\Omega(\pi)$ denotes the set of weights of the representation~$\pi$.
\end{dfn}

\noindent
We now give a different interpretation of the weight margin.
Namely, $\gamma(\pi)$ can be interpreted as the minimal norm of the moment map for vectors in the null cone \emph{when the group action is restricted to the maximal torus~$T_G$}.
This is stated in the following lemma, which also gives a similar relation for the action of the group~$G$.

\begin{lem}[Weight margin vs.\ moment map]\label{lem:weight margin vs moment map}
We have
\begin{align}\label{eq:weight margin is minimal moment map for torus}
  \gamma(\pi) = \min \bigl\{ \norm{\mu_{T_G}(v)}_F \;:\; v \in V\setminus\{0\}, \, \capacity_{T_G}(v) = 0 \bigr\}.
\end{align}
As a consequence,
\begin{align}\label{eq:weight margin vs moment map}
  \gamma(\pi) \leq \min \bigl\{ \norm{\mu_G(v)}_F \;:\; v \in V\setminus\{0\}, \, \capacity_G(v) = 0 \bigr\}.
\end{align}
Here we use subscripts to distinguish the capacity and moment map for the action of $G$ from the ones for its maximal torus~$T_G$.
\end{lem}
\begin{proof}
Let $v=\sum_{\omega\in\Omega(\pi)} v_\omega \neq 0$ be an arbitrary nonzero vector, with each $v_\omega$ in the weight space~$V_\omega$, and let $\Gamma = \{ \omega \in \Omega(\pi) : v_\omega \neq 0 \}$ denote its support.
Then we know from \cref{eq:commutative capacity} that $\capacity_{T_G}(v)=0$ if and only if $\conv(\Gamma)\not\ni0$.
On the other hand, $\mu_{T_G}(v)$ is an arbitrary point in~$\conv(\Gamma)$ by \cref{eq:commutative moment map}.
This establishes \cref{eq:weight margin is minimal moment map for torus}.

To prove \cref{eq:weight margin vs moment map}, we need to show that $\capacity_G(v)=0$ implies that $\norm{\mu_G(v)}_F \geq \gamma(\pi)$.
By the Hilbert-Mumford criterion, $\capacity_G(v) = 0$ implies that there exists $k \in K$ such that $\capacity_{T_G}(\pi(k)v) = 0$, hence $\norm{\mu_{T_G}(\pi(k) v)}_F \geq \gamma(\pi)$ by \cref{eq:weight margin is minimal moment map for torus}.
On the one hand, it holds for any~$k\in K$ that
\begin{align}\label{eq:mu G vs T}
  \norm{\mu_G(v)}_F
= \norm{\mu_G(\pi(k) v)}_F
\geq \norm{\mu_{T_G}(\pi(k) v)}_F,
\end{align}
using $K$-equivariance of the moment map and that $\mu_{T_G}(\pi(k) v)$ is the orthogonal projection of~$\mu_G(\pi(k) v)$ onto~$i\Lie(T_K) \subseteq i\Lie(K)$.%
\footnote{This is because, by definition of the moment map, $\tr[\mu_G(\pi(k) v) H] = \tr[\mu_{T_G}(\pi(k) v) H]$ for all $H \in i \Lie(T_K)$.}
The claim follows by combining these two estimates.
\end{proof}

\begin{rem}
In the language of moment polytopes, \cref{eq:weight margin is minimal moment map for torus} asserts that the weight margin is the minimal distance between the origin and the moment polytope of unstable vectors, when restricted to the action of the maximal torus~$T_G$ (cf.~\cref{subsec:moment polytopes theory}).

One can define a similar quantity using the right-hand side of \cref{eq:weight margin vs moment map} or, equivalently, in terms of the moment polytopes for the action of~$G$.
For the tensor action, this was called the \emph{gap constant} in~\cite{burgisser2018efficient} (in the general situation, the term \emph{highest weight margin} suggests itself).
It is an interesting question whether the bound in \cref{thm:cap gap} holds with this improved constant.
\end{rem}

We now state and prove our quantitative bound which shows that if the gradient (moment map) is small then the norm of~$v$ is close to its infimum (the capacity).
Our argument is inspired by the proof of the analogous statement in~\cite{LSW} for matrix scaling.

\begin{thm}[Lower bound from \cref{thm:intro_duality}]\label{thm:cap gap}
For all~$v\in V \setminus \{0\}$,
\begin{align*}
  \frac{\capacity^2(v)}{\norm{v}^2} \geq 1 - \frac{\norm{\mu(v)}_F}{\gamma(\pi)}.
\end{align*}
where $\gamma(\pi)$ is the weight margin defined in \cref{dfn:weight margin}.
\end{thm}

\noindent
When $v$ in the null cone, the bound asserts that $\norm{\mu(v)}_F \geq \gamma(\pi)$, which is precisely the content of \cref{eq:weight margin vs moment map}.
On the other hand, when $\norm{\mu(v)}_F < \gamma(\pi)$, it implies that $\capacity(v)>0$ and gives a relative approximation of the capacity.

\begin{proof}
We first reduce the claim for general~$G$ to its maximal torus~$T_G$ and then prove the claim for the latter.
As in \cref{lem:weight margin vs moment map}, we use subscripts to distinguish the capacity and moment map for the two actions.
Without loss of generality, we assume that $\norm v = 1$.

For general $G$, we first use that by the Cartan decomposition from \cref{eq:cartan}, $G = K T_G K$, hence
\begin{align*}
  \capacity_G(v) = \inf_{k\in K} \inf_{t \in T_G} \norm{\pi(t) \pi(k) v} = \inf_{k\in K} \capacity_{T_G}(\pi(k) v);
\end{align*}
we note that $\norm{\pi(k)v}=\norm{v}=1$ for all $k\in K$.
On the other hand, \cref{eq:mu G vs T} implies
\begin{align*}
  \inf_{k\in K} \left( 1 - \frac{\norm{\mu_{T_G}(\pi(k) v)}_F}{\gamma(\pi)} \right) \geq 1 - \frac{\norm{\mu_G(v)}_F}{\gamma(\pi)}.
\end{align*}
Thus, we see that it suffices to prove the desired bound for the maximal torus, i.e.,
\begin{align}\label{eq:torus claim}
  \capacity^2_{T_G}(v) \geq 1 - \frac{\norm{\mu_{T_G}(v)}_F}{\gamma(\pi)}
\end{align}
for any unit vector $v\in V$.
We start by recalling from \cref{eq:commutative capacity,eq:commutative moment map} that for the commutative group $T_G$, the capacity and moment map can be computed as
\begin{align}\label{eq:cap gap cap}
  \capacity_{T_G}^2(v)
&= \inf_{H \in i\Lie(T_K)} \sum_{\omega\in\supp(p)} p_\omega e^{2 \tr[\omega H]}, \\
\label{eq:cap gap mu}
  \mu_{T_G}(v) &= \sum_{\omega\in\supp(p)} p_\omega \omega,
\end{align}
where the probability distribution $p_\omega := \norm{v_\omega}^2$ is defined in terms of the decomposition of $v$ into weight vectors, $v = \sum_{\omega\in\Omega(\pi)} v_\omega$, with $\supp(p) := \{ \omega : p_\omega > 0 \}$ denoting its support.

We may assume that $0\in\conv(\supp(p))$, since otherwise $\capacity_{T_G}(v)=0$ 
and hence \cref{eq:torus claim} follows trivially from \cref{eq:weight margin is minimal moment map for torus} (as discussed above).
In this case, \cref{lem:prob dist decomposition} below shows that there exists~$s\in[0,1]$ and probability distributions $p',p''$ with $\supp(p'),\supp(p'') \subseteq \supp(p)$ such that
\begin{align*}
  p = (1-s) p' + s p'', \quad
  \sum_{\omega} p'_\omega \omega = 0, \quad
  \text{and if $s>0$ then } 0\not\in\conv(\supp(p'')).
\end{align*}
The significance of this decomposition is that the moment map only depends on~$p''$, while we can obtain a good capacity lower bound from~$p'$.
Indeed, \cref{eq:cap gap mu} implies that
\begin{align*}
  \mu_{T_G}(v) = s \sum_\omega p''_\omega \omega,
\end{align*}
hence for $s>0$ we obtain by definition of the weight margin that
\begin{align*}
  \norm{\mu_{T_G}(v)}_F = s \left\|\sum_\omega p''_\omega \omega\right\|_F \geq s \gamma(\pi),
  \quad\text{hence}\quad
  \frac{\norm{\mu_{T_G}(v)}_F}{\gamma(\pi)} \geq s,
\end{align*}
which also holds trivially for $s=0$.
On the other hand, \cref{eq:cap gap cap} leads to the lower bound
\begin{align*}
  \capacity_{T_G}^2(v)
\geq (1-s) \inf_{H \in i\Lie(T_K)} \sum_\omega p'_\omega e^{2 \tr[\omega H]}
\geq (1-s) \inf_{H \in i\Lie(T_K)} e^{2 \tr[\sum_\omega p'_\omega \omega H]}
= 1-s,
\end{align*}
using Jensen's inequality applied to the convex function~$f_H(\omega) = e^{2\tr[\omega H]}$ (the second inequality is in fact an equality).
Together, we obtain precisely what we wanted to show, i.e., \cref{eq:torus claim}.
\end{proof}

\begin{lem}\label{lem:prob dist decomposition}
Let $p$ be a probability distribution on $\Omega(\pi)$ such that $0 \in \conv(\supp(p))$.
Then there exists~$s\in[0,1]$ and probability distributions $p'$ and $p''$ with $\supp(p'),\supp(p'') \subseteq \supp(p)$ such that
\begin{align*}
  p = (1-s) p' + s p'', \quad
  \sum_{\omega} p'_\omega \omega = 0, \quad
  \text{and if $s>0$ then } 0\not\in\conv(\supp(p'')).
\end{align*}
\end{lem}
\begin{proof}
We prove this by induction on the size of~$\supp(p)$.
If $\abs{\supp(p)}=1$, the statement is clear: we necessarily have $\supp(p)=\{0\}$ and we may take~$s=0$ and~$p'=p$.

If $\abs{\supp(p)}>1$, by assumption there exists a probability distribution~$q$ with $\supp(q)\subseteq\supp(p)$ and~$\sum_\omega q_\omega \omega = 0$.
Let $t$ be the largest number such that $p - t q$ is still a nonnegative vector, i.e.,
\begin{align*}
  t = \min_{\omega\in\supp(q)} \frac{p_\omega}{q_\omega} \in (0,1].
\end{align*}
If $t=1$ then $p=q$ and we are done, since we may again take~$s=0$ and $p'=p$.
Otherwise, $t<1$, so we can define a probability distribution $r = \frac{p - t q}{1-t}$.
Then, $\supp(r)\subsetneq\supp(p)$ and
\begin{align*}
  p = t q + (1-t) r.
\end{align*}
If $0\not\in\conv(\supp(r))$ this yields a decomposition of the desired form with $s=t$, $p'=q$, and~$p''=r$.
On the other hand, if $0\in\conv(\supp(r))$, then by induction there exist $u\in[0,1]$ and probability distributions $r'$ and $r''$ with $\supp(r'), \supp(r'') \subseteq \supp(r) \subseteq \supp(p)$ such that
\begin{align*}
  r = (1-u) r' + u r'', \quad
  \sum_\omega r'_\omega \omega = 0, \quad
  \text{and if $u>0$ then } 0\not\in\conv(\supp(r'')).
\end{align*}
Then,
\begin{align*}
  p
= t q + (1-t) (1-u) r' + (1-t) u r'',
\end{align*}
so we may take $s = (1-t) u$, $p'$ as the probability distribution proportional to $t q + (1-t) (1-u) r'$, and $p'' = r''$ to obtain the desired decomposition.
\end{proof}

\noindent
The next result states that, conversely, if~$\norm{v}$ is close to its infimum then the gradient is small.

\begin{thm}[Upper bound from \cref{thm:intro_duality}]\label{thm:converse gap}
For all~$v\in V \setminus \{0\}$,
\begin{align*}
\frac{\capacity^2(v)}{\norm{v}^2} \leq 1 - \frac{\norm{\mu(v)}^2_F}{4N(\pi)^2}.
\end{align*}
where $N(\pi)$ is the weight norm defined in \cref{dfn:weight norm}.
\end{thm}
\begin{proof}
Consider the function~$F_v(g) = \log \lVert \pi(g) v \rVert$ defined in \cref{eq:kempf ness function}.
If we apply the right-hand inequality in \cref{cor:log_norm_quadratic_ub} with~$g=I$ (the identity element) and~$H = -\frac{\mu(v)}{2N(\pi)^2}$, we obtain
\begin{align*}
F_v\left(e^H\right) - F_v(I)
\leq -\tr\left[\mu\bigl(v\bigr) \frac{\mu(v)}{2N(\pi)^2}\right] + N(\pi)^2 \Norm{\frac{\mu(v)}{2N(\pi)^2}}_F^2
= -\frac{\Norm{\mu(v)}_F^2}{4N(\pi)^2}.
\end{align*}
Since~$F_v\left(e^H\right) - F_v(I) \geq \log \capacity(v) - \log \norm{v}$, we get that
\begin{align*}
\frac{\capacity^2(v)}{\norm{v}^2} \le e^{- \frac{\norm{\mu(v)}^2_F}{2N(\pi)^2}} \leq 1 - \frac{\norm{\mu(v)}^2_F}{4N(\pi)^2},
\end{align*}
where the second inequality follows from the fact that $e^{-x} \le 1 - x/2$ for all $x \in [0,1]$ and $\norm{\mu(v)}_F \le N(\pi)$ (\cref{lem:bound on gradient}).
\end{proof}

\Cref{thm:cap gap,thm:converse gap} together establish \cref{thm:intro_duality} announced in the introduction.
They strengthen the classical Kempf-Ness result, \cref{eq:kempf ness}, which can be obtained as a direct consequence.
Indeed, if $\capacity(v)>0$ then there exists a sequence $g_k\in G$ such that $\norm{\pi(g_k) v} \to \capacity(v)$, so $\mu(\pi(g_k) v) \to 0$ by \cref{thm:converse gap}.
Conversely, if $g_k\in G$ is a sequence such that $\mu(\pi(g_k) v)\to 0$, then $\capacity(v) / \norm{\pi(g_k) v} > 0$ for $k$ sufficiently large by \cref{thm:cap gap}, and so $\capacity(v)>0$.
In both arguments we used that the capacity is $G$-invariant, i.e., $\capacity(v) = \capacity(\pi(g) v)$ for every~$g\in G$.

\subsection{Gradient flow}\label{subsec:flow}
In view of the convexity properties of the log-norm, it is natural to minimize it by using gradient descent.
Indeed, this is the perspective that gives rise to our first-order algorithm presented in \cref{sec:first-order algorithm}.
Since minimizing the log-norm function is dual to minimizing the moment map, it is natural to also study gradient flows for minimizing the moment map.
Kirwan first observed the remarkable properties of the gradient flow for the norm square of the moment map in~\cite{kirwan1984convexity}.
In the context of tensor scaling and quantum marginals, Kirwan's flow was first proposed as an algorithmic tool in~\cite{walter2013entanglement,walter2014multipartite}.
\cite{kwok2018paulsen,AGLOW18} studied the gradient flow for the (unsquared) norm of the moment map in the context of operator scaling.
We will now explain how this analysis can be carried out in complete generality, which also leads to straightforward proofs.

For this, it will be convenient to consider a differently normalized version of the moment map.
Namely, define~$\tilde\mu: V\to i\Lie(K)$ by~$\tilde\mu(v) = \norm{v}^2 \, \mu(v)$.
Then, by \cref{dfn:moment_map},
\begin{align}\label{eq:unnormalized moment map}
  \tr\bigl[\tilde\mu(v) H\bigr] = \braket{v, \Pi(H) v}
\end{align}
for all~$v\in V$ and~$H\in i\Lie(K)$.
From this, we recognize that~$\tilde\mu(v)$ is the gradient
of~$f_v(g)= \frac12\lVert\pi(g) v\rVert^2$ in the same way that~$\mu(v)$
is the gradient of the function $F_v$ defined in \cref{eq:kempf ness function}.
(Recall that the gradient $\nabla f(v)$ of a function~$f\colon V\to\RR$ at~$v\in V$
is defined by~$\Re \braket{\nabla f(v),w} = D_w f(v)$ for all $w\in V$, where $D_w f(v) = \partial_{s=0} f(v+sw)$
denotes the partial derivative in direction~$w$.)
We note that
$\tilde\mu(\lambda v) = \|\lambda\|^2 \mu(v)$ for $\lambda\in\CC^*$.

\begin{rem}
From the point of view of geometric invariant theory, $\tilde\mu$ is a moment map on the vector space~$V$ rather than on projective space~$\PP(V)$.
\end{rem}

\noindent
Consider the gradient of the infinitely differentiable, real-valued function
\begin{align*}
  v\mapsto \norm{\tilde\mu(v)}_F
\end{align*}
defined on the open subset~$U := \{v \in V :  v \neq 0, \mu(v) \neq 0\}$ of~$V$.

We next study the flow corresponding to this gradient in detail.
We note that the first order \cref{alg:gconvex_gradient_general} is a discretization of this gradient flow.
Moreover, this gradient flow will serve us in giving a diameter bound for the second order \cref{alg:gconvex second order}.
We first derive a concrete formula by a slight variation of~\cite[Lemma 6.6]{kirwan1984convexity}.

\begin{lem}\label{lem:grad vs moma}
For $v\in U$ we have
\begin{align*}
  \nabla\norm{\tilde\mu}_F(v) = 2 \frac{\Pi(\tilde\mu(v)) v}{\norm{\tilde\mu(v)}_F} = 2 \frac{\Pi(\mu(v)) v}{\norm{\mu(v)}_F} .
\end{align*}
\end{lem}
\begin{proof}
First we note that for $v\in U$,
\begin{align}\label{eq:nabla first}
  \nabla\norm{\tilde\mu}_F(v) = \frac{\nabla\norm{\tilde\mu}_F^2(v)}{2\norm{\tilde\mu(v)}_F}.
\end{align}
Next, we compute the right-hand side gradient by
\begin{align}\label{eq:rhs grad moment}
  \Re\braket{\nabla\norm{\tilde\mu}_F^2(v), w}
= D_w \norm{\tilde\mu}_F^2(v)
= D_w \tr\bigl[\tilde\mu(v)^2\bigr]
= 2 \tr\bigl[\tilde\mu(v) D_w\tilde\mu(v)\bigr].
\end{align}
However, differentiating both sides of \cref{eq:unnormalized moment map} shows that
\begin{align*}
  \tr\bigl[H D_w\tilde\mu(v)\bigr] = D_w \braket{v, \Pi(H) v} = \braket{w, \Pi(H) v} + \braket{v, \Pi(H) w} = 2\Re\braket{\Pi(H) v, w}
\end{align*}
for every~$H\in i\Lie(K)$.
In particular, this holds for $H=\tilde\mu(v)$.
Plugging this into \cref{eq:rhs grad moment} yields
\begin{align*}
  \Re\braket{\nabla\norm{\tilde\mu}_F^2(v), w} = 4\Re\braket{\Pi(\tilde\mu(v)) v, w}
\end{align*}
for all~$w\in V$, so that~$\nabla\norm{\tilde\mu}_F^2(v) = 4 \Pi(\tilde\mu(v)) v$.
Now the claim follows from \cref{eq:nabla first}.
\end{proof}

\begin{dfn}[Gradient flow]\label{dfn:flow}
For $v\in U$ we consider in $U$ the ordinary differential equation
\begin{align}
v'(t) = -\nabla\norm{\tilde\mu\bigl(v(t)\bigr)}_F, \quad v(0) = v .
\label{eq:grad ode}
\end{align}
We denote by $[0,T_v) \to U,\ t\mapsto v(t)$ its unique solution on its
maximal interval of definition, where $0 < T_v\le \infty$.
\end{dfn}

\noindent
The existence (i.e. the fact that $T_v > 0$) and uniqueness of the solution follow from a standard first-order ODE result (the Picard-Lindel\"of theorem, see, e.g., Theorem~2.2 in~\cite{coddington1955theory}), since the vector field $v\mapsto \nabla\norm{\tilde\mu\bigl(v\bigr)}_F$ is~$C^1$ and hence locally Lipschitz continuous on~$U$.
As a consequence of \cref{lem:grad vs moma}, we can write the flow in \cref{dfn:flow} as
\begin{align}\label{eq:flow with lie alg}
  v'(t) = -2 \Pi\left(\frac{\mu(v(t))}{\norm{\mu(v(t))}_F}\right) v(t)
\in \Pi(\Lie(G)) v(t).
\end{align}
It follows that the flow~$v(t)$ actually stays in the orbit~$\pi(G) v(0)$ at all times~$t\in[0,T_v)$.
We record this and additional useful properties of the flow in the following proposition:

\begin{prp}[Properties of the flow]\label{prp:flow_identities}
For $0\le t <T_v$ we have:
\begin{enumerate}
\item \label{it:grad}$\partial_t \norm{\tilde\mu(v(t))}_F = -\norm{v'(t)}^2$.
\item \label{it:dnorm} $\partial_t \norm{v(t)}^2 = -4 \norm{\tilde\mu(v(t))}_F$.
\item \label{it:d2norm} $\partial^2_t \norm{v(t)}^2 = 4 \norm{v'(t)}^2$.
\item \label{it:orbit} The ordinary differential equation in $G$,
  \begin{align}\label{eq:g(t) ode first}
    g'(t) = -2 \frac{\mu(v(t))}{\norm{\mu(v(t))}_F} g(t), \quad  g(0) = I ,
  \end{align}
has a solution $g:[0,T_v) \to G$, which  satisfies $v(t) = \pi(g(t)) v$.
In particular,
$v(t) \in \pi(G) v$ for $t \in [0,T_v)$.

\item\label{it:limit} Suppose $T_v$ is finite.
Then the limit $v(T_v): = \lim_{t \uparrow T_v} v(t)$ exists and
it satisfies
\begin{align*}
  \norm{v(T_v)} = \capacity(v).
\end{align*}
\end{enumerate}
\end{prp}
\begin{proof}
\cref{it:grad} is true for any gradient flow.
To see \cref{it:dnorm}, note that
\begin{align*}
  \partial_t \norm{v(t)}^2
= 2\braket{v'(t), v(t)}
= -4\Re\frac{\braket{\Pi(\mu(v(t))) v(t), v(t)}}{\norm{\mu(v(t))}_F}
= -4\frac{\tr[\tilde\mu(v(t)) \mu(v(t))]}{\norm{\mu(v(t))}_F}
= -4 \norm{\tilde\mu(v(t))}_F,
\end{align*}
where the second equality is \cref{eq:flow with lie alg} and the third is \cref{eq:unnormalized moment map}.
\cref{it:d2norm} follows from \cref{it:grad} and \cref{it:dnorm}.

\cref{it:orbit} follows because \cref{eq:g(t) ode first} is a linear ODE in the entries of $g(t)$ with continuous coefficients and hence has a solution on $[0,T_v)$. Observe that $\pi(g(t))v$ also solves \cref{eq:flow with lie alg}, and hence $v(t) = \pi(g(t))v$ by the uniqueness of $v(t)$.

Finally, for showing~\cref{it:limit}, we assume $T_v< \infty$.
By \cref{it:dnorm}, $t\mapsto \norm{v(t)}$ is monotonically decreasing,
hence $\norm{v(t)} \leq \norm{v}$ for all $0\le t <T_v$ and the solution is bounded.
In the situation of a finite time and bounded solution,
a standard  ODE argument (e.g., see~\cite[Theorem 4.1 ]{coddington1955theory})
tells us that the limit $v(T_v) := \lim_{t \uparrow T_v} v(t)$ exists, but it does not lie in the domain of definition $U$.

Observe that $\lim_{t \uparrow T_v} \norm{v(t)} = \norm{v(T_v)}$. If $\norm{v(T_v)} > 0$, then because $v(T_v)$ is outside the domain of definition of $U$ we must have $\mu(v(T_v)) = 0$. In particular, $\lim_{t \uparrow T_v} \mu(v(t)) = 0$, which implies $ \norm{v(T_v)} = \lim_{t \uparrow T_v} \norm{v(t)} = \capa(v)$ by \cref{thm:intro_duality}. On the other hand, if $v(T_v) = 0$, then $\lim_{t \uparrow T_v} \norm{v(t)} = 0$ and hence $\capa(v) = 0$ because $v(t) \in \pi(G) v$ for $t < T_v$ by \cref{it:orbit}.\footnote{Though we do not need it here, we will see later in the proof of \cref{prp:conditions} that $T_v$ is, in fact, never finite if $\capa(v) = 0$.} This proves \cref{it:limit}.
\end{proof}

\noindent
Next, we show that the flow converges quickly to an approximate minimizer of the norm-square function.

\begin{thm}[Convergence of gradient flow]\label{thm:flow convergence}
Let~$v(t)$ denote the gradient flow from \cref{dfn:flow} with initial vector~$v=v(0)$.
For every ~$\eps>0$, there is some $T \leq \frac1{4\gamma(\pi)}\log(\lVert v\rVert^2/\eps)$ such that $T < T_v$
and for every $T \leq t < T_v$ we have
\begin{align*}
  \norm{v(t)}^2 \leq \capacity^2(v) + \eps,
\end{align*}
where~$\gamma(\pi)$ is the weight margin defined in \cref{dfn:weight margin}.
\end{thm}

\begin{proof}
By the second fact in \cref{prp:flow_identities}, we have for $0\le t < T_v$
\begin{align*}
  \partial_t \left( \norm{v(t)}^2 - \capacity^2(v) \right) = -4 \norm{\tilde\mu(v(t))}_F.
\end{align*}
Moreover, by \cref{thm:cap gap},
we have~$\capacity^2(v(t)) \geq \norm{v(t)}^2 - \norm{\tilde\mu(v(t))}_F / \gamma(\pi)$, hence
\begin{align*}
  \partial_t \left( \norm{v(t)}^2 - \capacity^2(v) \right)
  \leq - 4\gamma(\pi) \left( \norm{v(t)}^2 - \capacity^2(v(t)) \right)
  = - 4\gamma(\pi) \left( \norm{v(t)}^2 - \capacity^2(v) \right),
\end{align*}
where we used that the capacity is $G$-invariant.
Hence we obtain for $0\le t < T_v$ that
\begin{align*}
  \norm{v(t)}^2 - \capacity^2(v) \leq e^{-4\gamma(\pi) t} \left( \norm{v}^2 - \capacity^2(v) \right) \leq e^{-4\gamma(\pi) t} \lVert v\rVert^2 .
\end{align*}
If $\frac1{4\gamma(\pi)}\log(\lVert v\rVert^2/\eps) < T_v$,
the assertion follows by taking $T := \frac1{4\gamma(\pi)}\log(\lVert v\rVert^2/\eps)$.
Otherwise, $T_v$ is finite and we have
$\norm{v(T_v)} = \capacity(v)$ by \cref{it:limit}.
In this case, any $T< T_v$ sufficiently close to $T_v$ will do.
\end{proof}

\noindent
The next corollary gives an analogous bound for the log-norm.
We will use it in \cref{sec:second_order} to derive a diameter bound for our second order algorithm.

\begin{cor}\label{cor:multiplicative}
Let~$v(t)$ denote the gradient flow from \cref{dfn:flow} with initial vector~$v=v(0)$
and assume $\capacity(v)>0$.
For every~$\eps>0$, there is some~$T\leq \frac1{4\gamma(\pi)}\log\bigl(\lVert v\rVert^2/(2\capacity^2(v)\eps)\bigr)$ such that $T < T_v$
and for every $T \leq t < T_v$, we have
\begin{align*}
  \log\norm{v(t)} \leq \log\capacity(v) + \eps.
\end{align*}
\end{cor}

\begin{proof}
By \cref{thm:flow convergence} and our choice of~$T$, for $T \leq t < T_v$ we have
\begin{align*}
  \norm{v(t)}^2 \leq \capacity^2(v) + 2\capacity^2(v)\eps = (1 + 2\eps) \capacity^2(v),
\end{align*}
so the claim follows by taking logarithms and using the estimate~$\log(1+x)\leq x$.
\end{proof}

\subsection{Moment polytopes}\label{subsec:moment polytopes theory}
In this section, we discuss the optimization problem underlying the moment polytope membership problem.
We first explain the general definition of the moment polytope.
For a nonzero vector~$v \in V$, we define the \emph{moment polytope of~$v$} by
\begin{align*}
  \Delta(v)
:= \overline{\left\{ \mu(w) : w \in G \cdot v \right\}} \cap C(G)
= \overline{\left\{ \intopoly(\mu(w)) : w \in G \cdot v \right\}},
\end{align*}
where $C(G)$ is the positive Weyl chamber
and $s$ generalizes the function taking a matrix to its spectrum
(see the end of~\cref{subsec:rep theory} for the definition).
The equality follows because the moment map is $K$-equivariant,
which means that, $\mu(\pi(k) w) = k \mu(w) k^\dagger$ for all~$w\in V$ and~$k\in K$.
If $G=\GL(n)$ then $K=U(n)$, $C(G)=C(n)$, and $s=\spec$, so the moment polytope is precisely the set of all spectra (eigenvalues ordered non-increasingly)
of moment map images obtained from scalings of~$v$; this is the definition that we gave in \cref{subsubsec:moment polytopes}.
A point in $C(G)$ is called rational if an integer multiple of it is a (highest) weight.
We remark that~$\Delta(v)$ is a moment polytope of the orbit closure of~$v$ in the projective space~$\PP(V)$ in the sense of symplectic geometry.
It is a nontrivial fact that~$\Delta(v)$ is a convex polytope with rational vertices~\cite{ness1984stratification,brion1987image}.

Now let $p\in C(G)$ be an arbitrary rational point and let $\ell > 0$ be an integer such that $\lambda := \ell p$ is a highest weight.
In \cref{subsubsec:non-uniform_first_order} we motivated the following \emph{$p$-capacity},
\begin{align}\label{eq:def p cap}
  \capacity_p(v) := \inf_{g \in G} \, \norm{(\pi(g) v)^{\otimes \ell} \otimes \left(\pi_{\lambda^*}(g) v_{\lambda^*} \right)}^{1/\ell},
\end{align}
where~$\lambda^*$ denotes the highest weight of the dual representation as defined in \cref{subsec:rep theory}.
Clearly,
\begin{align}\label{eq:p-cap vs cap}
  \capacity_p(v) = \capacity(v^{\ot\ell} \ot v_{\lambda^*})^{1/\ell},
\end{align}
where the right-hand side capacity is computed using the representation~$\rho\colon G \to \GL(W)$
on the vector space~$W = \Sym^\ell(V) \ot V_{\lambda^*}$ defined by~$\rho(g) = \pi(g)^{\ot \ell} \ot \pi_{\lambda^*}(g)$.
The relevance of the $p$-capacity is due to the `shifting trick' from~\cite{ness1984stratification,brion1987image}, which shows that $p\in\Delta(v)$
iff $0\in\Delta(w)$ for some vector of the form $w = (\pi(h) v)^{\ot \ell} \ot v_{\lambda^*}$.
Moreover, if the latter condition holds for some $h\in G$ then it holds for generic $h\in G$.
Now, by the Kempf-Ness theorem,~$0\in\Delta(w)$ if and only if~$\capacity(w) > 0$, as explained in \cref{subsubsec:null_cone}.
Since $\capacity(w)$ is nothing but the $p$-capacity of~$\pi(h) v$, we obtain the following important equivalence:
\begin{align}\label{eq:p vs cap_p generic}
  p \in \Delta(v)
  \,\iff\,
  \capacity_p(\pi(h) v) > 0 \text{ for some $h\in G$}
  \,\iff\,
  \capacity_p(\pi(h) v) > 0 \text{ for generic $h\in G$}.
\end{align}
Thus, membership in the moment polytope can be reduced to $p$-capacities by a suitable randomization step ($v \mapsto \pi(h)v$ for random $h$).
We will later state an effective version of this observation that shows, for the case $G = \GL(n)$, how much randomness suffices (\cref{thm:random_capacity}).

In the remainder of this section we will focus on the $p$-capacity.
We first analyze the smoothness and robustness of the logarithm of the objective function underlying the $p$-capacity~\eqref{eq:def p cap}, i.e.,
\begin{equation}\label{eq:def-Fun}
  F_{v,p} \colon G \to\RR, \quad F_{v,p}(g) = \log \| \pi(g) v\| + \frac{1}{\ell} \log \| \pi_{\lambda^*}(g) v_{\lambda^*}\|
\end{equation}
(this is nothing but $F_{v^{\ot\ell} \ot v_{\lambda^*}}$, the log-norm function of the vector~$v^{\ot\ell} \ot v_{\lambda^*}$, divided by~$\ell$).
Clearly,~$F_{v,p}$ can be written as a linear combination of two log-norm functions~\eqref{eq:kempf ness function}:
\begin{align}\label{eq:F=F+F}
  F_{v,\lambda}
= F_v + \frac1\ell F_{v_\lambda^*}.
\end{align}
By \cref{prp:smooth}, $F_v$ is convex and $2 N(\pi)^2$-smooth, while $F_{v_\lambda^*}$ is $2 \norm{\lambda}_F^2$-smooth by \cref{prp:smooth}.
Thus, we find that the smoothness of $F_{v,p}$ can be upper bounded by $2 N(\pi)^2 + 2 \ell \norm{p}_F^2$.
This bound does not lead to efficient algorithms since~$\ell$ depends exponentially on the bitsize of~$p$.
Fortunately, it is excessively pessimistic since it does not use the fact that $v_{\lambda^*}$ is a highest weight vector.
We will now derive a better bound that does not depend on~$\ell$:

\begin{prp}\label{prp:smooth_non_uniform}
  The function $F_{v_\lambda}$ is $2\norm{\lambda}_F$-smooth.
  As a consequence, the function $F_{v,p}$ is $2N^2$-smooth, where~$N^2 := N(\pi)^2 + \norm{p}_F$.
\end{prp}
\begin{proof}
  For $g\in G$ and $H \in i\Lie(K)$, consider the function
  \begin{align*}
    h(t) := F_{v_\lambda}(e^{tH} g) = \log \lVert e^{t\Pi_\lambda(H)} \pi_\lambda(g) v_\lambda \rVert.
  \end{align*}
  We would like to show that
  \begin{align}\label{eq:goal}
    h''(t) \leq 2 \norm{\lambda}_F \norm{H}_F^2
  \end{align}
  for all~$t$.
  It suffices to prove \cref{eq:goal} for~$t=0$, since we can always replace~$g$ by~$e^{tH}g$.
  We will now argue that we can also restrict to $g=I$.
  Using the Iwasawa decomposition to write $g = kb$ for some~$k \in K$ and~$b \in B$, we have~$\pi_\lambda(g) v_\lambda = z \pi_\lambda(k) v_\lambda$
for some~$z\in\CC^*$, because $v_\lambda$ is a highest weight vector.
  Thus:
  \begin{align*}
    h(t)
  &= \log \lVert e^{t\Pi_\lambda(H)} \pi_\lambda(g) v_\lambda \rVert
  = \log \lVert e^{t\Pi_\lambda(H)} \pi_\lambda(k) v_\lambda \rVert + \log |z| \\
  &= \log \lVert \pi_\lambda(k^{-1}) e^{t\Pi_\lambda(H)} \pi_\lambda(k) v_\lambda \rVert + \log |z|
  = \log \lVert e^{t\Pi_\lambda(k^{-1} H k)} v_\lambda \rVert + \log |z|,
  \end{align*}
  where we used that the norm is $K$-invariant.
  The additive constant does not impact derivatives and~$k^{-1} H k \in i\Lie(K)$.
  We may thus assume that $g=I$.
  Then, \cref{eq:second derivative} shows that
  \begin{align*}
    \frac12 h''(0) = \braket{\Pi_\lambda(H) v_\lambda, \Pi_\lambda(H) v_\lambda} - \braket{v_\lambda, \Pi_\lambda(H) v_\lambda}^2.
  \end{align*}
  Since $H\in i\Lie(K)$, we can decompose it as~$H = D + R + R^\dagger$, where~$D \in i\Lie(T_K)$ and~$R \in \Lie(N)$.
  Then we know from \cref{eq:torus on weight vec} that~$\Pi_\lambda(D) v_{\lambda}$ is a real scalar multiple of~$v_{\lambda}$,
from \cref{eq:unipotent on highest weight vec} that~$\Pi_\lambda(R) v_{\lambda} = 0$, and that~$\Pi_\lambda(R^\dagger) = \Pi_\lambda(R)^\dagger$.
  Using this, we can simplify as follows:
  \begin{align*}
    &\braket{\Pi_\lambda(H) v_\lambda, \Pi_\lambda(H) v_\lambda} - \braket{v_\lambda, \Pi_\lambda(H) v_\lambda}^2
  = \braket{\Pi_\lambda(R)^\dagger v_\lambda, \Pi_\lambda(R^\dagger) v_\lambda} \\
  = &\braket{v_\lambda, \Pi_\lambda(R) \Pi_\lambda(R^\dagger) v_\lambda}
    = \braket{v_\lambda, [\Pi_\lambda(R), \Pi_\lambda(R^\dagger)] v_\lambda}
    = \braket{v_\lambda, \Pi_\lambda([R, R^\dagger]) v_\lambda}.
  \end{align*}
  In the second line we used once more that~$\Pi_\lambda(R) v_\lambda = 0$ and that $\Pi_\lambda$ is a Lie algebra representation.
  We obtain
  \begin{align*}
    \frac12 h''(0)
  \leq \lVert\Pi_\lambda([R, R^\dagger])\rVert_{\ope}
  \leq N(\pi_\lambda) \lVert [R, R^\dagger] \rVert_F
  \leq 2 N(\pi_\lambda) \lVert R \rVert_F^2
  \leq 2 \norm{\lambda}_F \lVert R \rVert_F^2
  \leq \norm{\lambda}_F \lVert H \rVert_F
  \end{align*}
  by definition of the weight norm, submultiplicativity of the Frobenius norm, \cref{prp:opnorm}, and, finally,
  $2 \lVert R\rVert_F^2 \leq \rVert H \rVert_F^2$, which holds since the decomposition $H = D + R + R^\dagger$
  is orthogonal with respect to the Hilbert-Schmidt inner product.
  We have thus shown \cref{eq:goal} for $t=0$, concluding the proof.
\end{proof}

\noindent
We now compute the geodesic gradient of the objective function~\eqref{eq:def-Fun}.
By \cref{eq:F=F+F,eq:moment map vs gradient},
\begin{align*}
  \nabla F_{v,p}(g) = \mu(\pi(g)v) + \frac1\ell \mu_{\lambda^*}(\pi_{\lambda^*}(g) v_{\lambda^*}),
\end{align*}
where we write $\mu_\lambda^*$ for the moment map associated with the irreducible representation~$\pi_{\lambda^*}$.
The latter can computed readily.
Write $g = k b$ according to the Iwasawa decomposition~$G = KB$ from \cref{eq:iwasawa}.
Using that $v_{\lambda^*}$ is a $B$-eigenvector and the $K$-equivariance of the moment map, we find that
$\mu_{\lambda^*}(\pi_{\lambda^*}(g) v_{\lambda^*}) = k \lambda^* k^\dagger$.
Thus we obtain the following formula for the gradient of~$F_{v,p}$ at $g = kb$:
\begin{align}\label{eq:Fvp gradient}
  \nabla F_{v,p}(g)
= \mu(\pi(g)v) + k p^* k^\dagger
= \mu(\pi(g)v) - k (-p^*) k^\dagger.
\end{align}
Since~$\intopoly(-k p^* k^\dagger) = \intopoly(-p^*) = p$, we note that the gradient vanishes if and only if
$\intopoly(\mu(\pi(g) v)) = p$, i.e., $\pi(g) v$ maps to the desired point~$p$ in the moment polytope.
We will use this formula in our first-order algorithm for non-uniform scaling (\ref{alg:nonuniform_gradient}).
It also implies that~$F_{v,p}$ is universally upper-bounded by the following quadratic expansion, generalizing~\cref{cor:log_norm_quadratic_ub}.

\begin{cor}\label{cor:Fun_quadratic_ub}
For any~$v \in V\setminus\{0\}$ and rational~$p\in C(G)$, the function~$F_{v,p}$ defined in \cref{eq:def-Fun} satisfies
\begin{align*}
  F_{v,p}(g) + \tr\bigl[\left(\mu(\pi(g)v) + p^*\right) H\bigr] \leq F_{v,p}(e^H g) \leq
  F_{v,p}(g) + \tr\bigl[\left(\mu(\pi(g)v) + p^*\right) H\bigr] + N^2 \lVert H\rVert_F^2
\end{align*}
for every $g\in G$ and $H\in i\Lie(K)$.
\end{cor}
\begin{proof}
This follows from \cref{lem:quadratic_ub,prp:smooth_non_uniform,eq:Fvp gradient}.
\end{proof}

\noindent
Next, we derive noncommutative duality results that generalize \cref{thm:cap gap,thm:converse gap}.
Recall that any point~$p$ in the moment polytope necessarily satisfies~$\norm{p}_F \leq N(\pi)$ by \cref{lem:bound on gradient}.
Thus the condition in the following two results is without loss of generality.

\begin{thm}
For any~$v \in V\setminus\{0\}$ and rational~$p\in C(G)$ with $\norm{p}_F \leq N(\pi)$,
\begin{align*}
  \frac{\capacity_p(v)^2}{\norm{v}^2} \leq 1 - \frac{\norm{\mu(v) + p^*}^2_F}{4N^2},
\end{align*}
where $N^2 := N(\pi)^2 + \norm{p}_F$.
\end{thm}
\begin{proof}
This follows by adapting the proof of \cref{thm:converse gap} to use \cref{cor:Fun_quadratic_ub} in place of \cref{cor:log_norm_quadratic_ub}.
We apply the second inequality in \cref{cor:Fun_quadratic_ub} with $g=I$ and~$H = -\frac{\mu(v) + p^*}{2N^2}$. Then,
\begin{align*}
  F_{v,p}(e^H) - F_{v,p}(I)
\leq \tr\bigl[\left(\mu(v) + p^*\right) H\bigr] + N^2 \lVert H\rVert_F^2
= -\frac{\norm{\mu(v) + p^*}_F^2}{4N^2}
\end{align*}
and we can proceed as in the proof of \cref{thm:converse gap} to see that
\begin{align*}
  \frac{\capacity_p(v)^2}{\norm{v}^2}
\leq e^{2\left( F_{v,p}(e^H) - F_{v,p}(I) \right)}
\leq e^{-\frac{\norm{\mu(v) + p^*}_F^2}{2N^2}}
\leq 1 - \frac{\norm{\mu(v) + p^*}_F^2}{4N^2}.
\end{align*}
For the last inequality, use that $e^{-x} \leq 1 - x/2$ for all $x \in [0,1]$.
\end{proof}

\begin{thm}\label{thm:p-cap gap}
Let~$v\in V \setminus \{0\}$ and let~$p\in C(G)$ be rational with $\norm{p}_F \leq N(\pi)$.
Let $\ell>0$~be an integer such that $\lambda := \ell p$ is a highest weight.
Then,
\begin{align*}
  \frac{\capacity_p^{2\ell}(v)}{\norm{v}^{2\ell}} \geq 1 - \frac{\ell\norm{\mu(v) + p^*}_F}{\gamma(\rho)},
\end{align*}
where $\gamma(\rho)$ is the weight margin of the representation~$\rho\colon G \to \GL(W)$ on~$W = \Sym^\ell(V) \ot V_{\lambda^*}$.
In particular, if~$\norm{s(\mu(v)) - p}_F < \gamma(\rho)/\ell$ then $p\in\Delta(v)$.
\end{thm}
\begin{proof}
Use \cref{eq:p-cap vs cap} to write $\capacity_p(v)^{\ell}$ as the capacity of the vector~$w=v^{\ot\ell} \ot v_{\lambda^*}$
with respect to the representation~$\rho$.
In view of \eqref{eq:Fvp gradient}, the corresponding moment map is given by~$\mu(w) = \ell \mu(v) + \lambda^*$.
Thus, the first claim is a consequence of \cref{thm:cap gap}.
The second claim follows from the first, since~$\norm{s(\mu(v)) - p}_F < \gamma(\rho)/\ell$ means that there exists~$k\in K$
such that~$\norm{\mu(\pi(k)v) + p^*}_F < \gamma(\rho)/\ell$.
Then, $\capacity_p(\pi(k)v)>0$ and so $p\in\Delta(\pi(k)v)=\Delta(v)$ by \cref{eq:p vs cap_p generic}.
\end{proof}

\begin{rem}
When $G$ is commutative or, more generally, when the irreducible representation~$\pi_\lambda$ is one-dimensional, then the bound in \cref{thm:p-cap gap} can be improved as follows:%
\footnote{For $G = \GL(n_1) \times \dots \times \GL(n_k)$, the representation $\pi_\lambda$ is one-dimensional when $p=(p_1 1_{n_1},\dots,p_k 1_{n_k})$, in which case $\pi_\lambda(g_1,\dots,g_k) = \det(g_1)^{\ell p_1} \cdots \det(g_k)^{\ell p_k}$. As such this remark is relevant for the study of semi-invariants.}
\begin{align}
  \frac{\capacity^2_p(v)}{\norm{v}^2}
\geq 1 - \frac{\norm{\mu(v) + p^*}_F}{\gamma_p(\pi)}
\end{align}
where $\gamma_p(\pi)$ is defined like the weight margin (\cref{dfn:weight margin}) but with the origin replaced by the point~$p$.
It would be interesting to determine to which extent this bound holds more generally.
\end{rem}

\Cref{thm:p-cap gap} shows that we can reduce the moment polytope membership problem to the~$p$-scaling problem
for some suitable choice of~$\eps>0$, generalizing \cref{cor:null_cone_margin}.

\begin{cor}\label{cor:moment polytope margin}
Let~$v\in V \setminus \{0\}$ and let~$p\in C(G)$ be rational with $\norm{p}_F \leq N(\pi)$.
Let $\ell>0$~be an integer such that $\lambda := \ell p$ is a highest weight.
Then, $p\in\Delta(v)$ if and only if~$\Delta(v)$ contains a point of distance smaller than~$\gamma(\rho)/\ell$ to~$p$,
where~$\gamma(\rho)$ is the weight margin of the representation~$\rho$ on~$W = \Sym^\ell(V) \ot V_{\lambda^*}$.
In particular, solving the $p$-scaling problem with input~$(\pi,v,p,\gamma(\rho)/2\ell)$ suffices to solve
the moment polytope membership problem for~$(\pi,v,p)$.
\end{cor}

Finally, we show that the $p$-capacity is log-concave in the parameter~$p$.
For this, it is convenient to generalize its definition from rational~$p$ to all of~$C(G)$.
Recall the Iwasawa decomposition~\eqref{eq:iwasawa} in the form $G=K\exp(i\Lie(T_K))N$
(which generalizes the decomposition of a matrix in $\GL(n)$ into a product of a unitary matrix, a diagonal matrix with positive diagonal entries,
and an upper triangular matrix with ones on its diagonal).
For rational~$p=\lambda/\ell$ and $g=k\exp(H)b$, where $k\in K$, $H \in i\Lie(T_K)$, and~$b\in N$, we have
\begin{align*}
  \norm{(\pi(g) v)^{\otimes \ell} \otimes \left(\pi_{\lambda^*}(g) v_{\lambda^*} \right)}
&= \norm{(\pi(\exp(H)b) v)^{\otimes \ell} \otimes \left(\pi_{\lambda^*}(\exp(H)b) v_{\lambda^*} \right)} \\
&= e^{\tr[\lambda^* H]} \norm{\pi(\exp(H)b) v}^{\ell},
\end{align*}
where we first used that the inner product is $K$-invariant and then that $v_{\lambda^*}$ is a highest weight vector
(hence invariant under the action of~$N$, see \cref{eq:unipotent on highest weight vec}) of weight~$\lambda^*$
(hence transforms as~\eqref{eq:torus on weight vec}).
Thus:
\begin{align}\label{eq:p cap general}
  \capacity_p(v)
= \inf_{H\in i\Lie(T_K), b\in N} e^{\tr[p^* H]} \norm{\pi(\exp(H)b) v}.
\end{align}
We will take this formula as the definition of the $p$-capacity for general~$p\in C(G)$.

\begin{prp}[Log-concavity in $p$]\label{prp:log cap concave general}
For $0\neq v \in V$, the function $C(G) \mapsto \RR \cup \{-\infty\}$ given by $p \mapsto \log\capacity_p(v)$ is \emph{concave}.
In particular, $\Delta^+(v) := \{p \in C(G) : \capa_p(v) > 0\}$ is a convex subset of the moment polytope~$\Delta(v) \subseteq C(G)$.
\end{prp}
\begin{proof}
This follows directly from \cref{eq:p cap general}.
Indeed, note that
\begin{align*}
  \log\capacity_p(v) = \inf_{H\in i\Lie(T_K), b\in N} \bigl( \tr[p^* H] - \log \norm{\pi(\exp(H)b) v} \bigr).
\end{align*}
Since $p\mapsto p^*$ is linear and the expression inside the infimum is affine in~$p^*$, the log-capacity is manifestly concave in~$p$.
\end{proof}

\section{First-order algorithms}\label{sec:first-order algorithm}
In this section, we state and analyze a first-order method for scaling and norm minimization, elaborating on the discussion in \cref{subsubsec:uniform_first_order,subsubsec:non-uniform_first_order}.
We first state a general \emph{geodesic gradient descent} algorithm (\cref{alg:gconvex_gradient_general}) and analyze it for arbitrary convex smooth left-$K$-invariant functions as defined in \cref{subsec:geodesic_convexity_defns}.
Our algorithm for the scaling problem (\cref{alg:gconvex_gradient_uniform}) is then obtained by specializing this algorithm to the log-norm function defined in~\cref{eq:kempf ness function}.
This is natural since norm minimization and scaling are dual to each other, as explained in \cref{sec:intro,subsec:duality}.
In \cref{subsec:moment polytopes algo}, we extend our first-order algorithm to the~$p$-scaling problem (\cref{alg:nonuniform_gradient}).

Throughout the section we shall assume that we are in the setting of \cref{sec:preliminaries}:
$G\subseteq\GL(n)$ denotes a symmetric subgroup with maximal compact subgroup~$K$.
Morover, $\pi\colon G\to\GL(V)$ denotes a rational representation on the vector space~$V$
with $K$-invariant inner product~$\langle \cdot, \cdot \rangle$.


\subsection{General first-order optimization algorithm}
We now state our general first-order geodesic optimization algorithm and its analysis.

\begin{Algorithm}[th!]
\textbf{Input}:\vspace{-.2cm}
\begin{itemize}
\item Oracle access to the geodesic gradient $\nabla F$ of a left-$K$-invariant convex function $F\colon G\to\RR$ (see \cref{dfn:convex etc,dfn:grad hessian}),
\item a step size $\eta>0$,
\item a number of iterations~$T$.
\end{itemize}
\textbf{Output:} A group element $g\in G$.
\medskip

\textbf{Algorithm:}\vspace{-.2cm}
\begin{enumerate}
\item Set $g_0 = I$ (identity element of the group $G$).
\item For $t=0,\dots,T-1$: Set $g_{t+1} := e^{- \eta \: \nabla F(g_t)} g_t$.
\item \textbf{Return} $\argmin_{g\in \{g_0, \dots, g_{T-1}\}} \Norm{\nabla F(g)}_F^2$
\end{enumerate}
\vspace{-.2cm}
\caption{Geodesic first-order minimization algorithm (cf.~\cref{thm:geodesic_grad_descent}).}
\label{alg:gconvex_gradient_general}
\end{Algorithm}

\begin{thm}\label{thm:geodesic_grad_descent}
Let $F\colon G\to\RR$ be a left-$K$-invariant function in the sense that $F(kg) = F(g)$ for all~$k\in K$, $g\in G$.
Moreover, suppose that $F$ is geodesically convex and $L$-smooth in the sense of \cref{dfn:convex etc} for some~$L>0$, and that $F_{\inf} := \inf_{g \in G} F(g) > 0$.
For every~$\eps > 0$, \cref{alg:gconvex_gradient_general} with step size~$\eta = 1/L$ and
\begin{align*}
  T \geq \frac{2L}{\eps^2} \left( F(I) - F_{\inf} \right )
\end{align*}
iterations returns a group element $g\in G$ such that $\Norm{\nabla F(g)}_F \leq \eps$.
\end{thm}
\begin{proof}
Suppose to the contrary that $\Norm{\nabla F(g_t)}_F > \eps$ for all $t=0,\dots,T-1$.
Then we find, using \cref{lem:quadratic_ub} with~$H = -\eta \nabla F(g_t)$ for the first inequality, that
\begin{align*}
  F(g_{t+1}) - F(g_t)
&= F(e^{-\eta \nabla F(g_t)} g_t) - F(g_t) \\
&\leq -\eta \tr\!\left[\nabla F(g_t) \nabla F(g_t)\right] + \frac L 2 \eta^2 \lVert \nabla F(g_t) \rVert_F^2 \\
&= \left( \frac L 2 \eta^2 - \eta \right) \norm{\nabla F(g_t)}_F^2
= -\frac 1 {2L} \norm{\nabla F(g_t)}_F^2
< -\frac{\eps^2}{2L}
\end{align*}
for $t=0,\dots,T-1$.
By a telescoping sum, we obtain the upper bound in
\begin{align*}
  F_{\inf} - F(I)
\leq F(g_{T}) - F(g_0)
< -\frac{T \eps^2}{2L},
\end{align*}
hence that $T < \frac{2L}{\eps^2} \left( F(I) - F_{\inf} \right)$.
In view of our choice of~$T$, this is the desired contradiction.
\end{proof}

\subsection{Application to scaling and norm minimization problem}\label{subsec:first_order_specialize}
We now specialize \cref{alg:gconvex_gradient_general} to the function~$g \mapsto \log\norm{\pi(g) v}$.
The resulting algorithm is \cref{alg:gconvex_gradient_uniform}:

\begin{Algorithm}[th!]
\textbf{Input}:\vspace{-.2cm}
\begin{itemize}
\item Oracle access to the moment map restricted to a group orbit, i.e., to the map $g \mapsto \mu(\pi(g)v)$,
\item a number of iterations~$T$.
\end{itemize}
\textbf{Output:} A group element $g\in G$.
\medskip

\textbf{Algorithm:}\vspace{-.2cm}
\begin{enumerate}
\item Set $g_0 = I$. Set a step size $\eta = \frac 1 {2N(\pi)^2}$.
\item For $t=0,\dots,T-1$: Set $g_{t+1} := e^{- \eta \: \mu(\pi(g_t) v)} g_t$.
\item \textbf{Return} $\argmin_{g\in \{g_0, \dots, g_{T-1}\}} \Norm{\mu(\pi(g) v)}_F^2$
\end{enumerate}
\vspace{-.2cm}
\caption{Algorithm for the scaling problem (cf.~\cref{thm:uniform_grad_descent}).}
\label{alg:gconvex_gradient_uniform}
\end{Algorithm}

\noindent
The following theorem gives rigorous guarantees for \cref{alg:gconvex_gradient_uniform} in terms of the capacity of~$v$ and the weight norm of the Lie algebra representation (\cref{dfn:weight norm}).

\begin{thm}[First order algorithm for scaling; general version of \cref{thm:intro_uniform_grad_descent}]\label{thm:uniform_grad_descent}
Let $v\in V$ be a vector with~$\capacity(v)>0$.
For every~$\eps > 0$, \cref{alg:gconvex_gradient_uniform} with 
\begin{align*}
  T \geq \frac{4 N(\pi)^2}{\eps^2} \log\left(\frac{\norm{v}}{\capacity(v)}\right)
\end{align*}
iterations returns a group element $g\in G$ such that $\Norm{\mu(\pi(g) v)}_F \leq \eps$.
\end{thm}
\begin{proof}
Since the gradient of the log-norm function~$F_v$ defined in \cref{eq:kempf ness function} is computed by the moment map (\cref{eq:moment map vs gradient}), we can interpret \cref{alg:gconvex_gradient_uniform} as the specialization of \cref{alg:gconvex_gradient_general} to~$F_v$.
Note that $F_v(I) = \log\norm{v}$ and $F_{\inf} = \log\capacity(v)$.
Moreover, $F_v$ is convex and $2 N(\pi)^2$-smooth by \cref{prp:smooth}.
Thus the claim follows from from \cref{thm:geodesic_grad_descent}.
\end{proof}

\noindent
By \cref{cor:cap_vs_moment}, \cref{thm:uniform_grad_descent} implies that the first order algorithm also computes an approximation to capacity, however the runtime becomes inversely proportional to the weight margin:

\begin{cor}[First order algorithm for norm minimization]\label{cor:first_order_approx_capacity}
Let $v\in V$ be a vector with~$\capacity(v)>0$.
For every $\eps > 0$, \cref{alg:gconvex_gradient_uniform} with step size~$\eta =  1/(2N(\pi)^2)$ and
\begin{align*}
  T \geq \frac{4N(\pi)^2}{\gamma(\pi)^2 \eps^2} \log\left(\frac{\norm{v}}{\capacity(v)}\right)
\end{align*}
iterations returns an $\eps$-approximate minimizer for log-capacity, i.e., a group element~$g\in G$ such that $\log \norm{\pi(g) v} - \log \capacity(v) \leq \eps$.
\end{cor}

\begin{rem}\label{rem:variations and comparison with alt min} 
\cref{alg:gconvex_gradient_uniform} performs gradient steps with a fixed step size.
Instead, one could in each iteration also perform a line search in the direction of the gradient (or search in a torus that contains the gradient).
Then each iteration corresponds to an unconstrained geometric program, which can be solved using standard techniques~\cite{gurvits2004combinatorial,singh2014entropy,straszak2017computing,BLNW:20}.

When the group is a product group, $G = G_1 \times \dots \times G_k$, is also natural to consider alternating minimization schemes where in each iteration the objective function is minimized over a single factor~$G_i$ while keeping all other factors constant.
If the action of a single factor is sufficiently simple (e.g., for multi-linear actions of $G$) the individual minimizations can be solved in closed form.
The well-known Sinkhorn algorithm for matrix scaling, its counterpart by Gurvits for operator scaling, and many other algorithms that have been discussed in the literature fall into this framework.
However, there are many actions of interest where alternating minimization does not apply, see also \cref{subsubsec:noncommutative state of the art}.
\end{rem}

\subsection{Application to \texorpdfstring{$p$-scaling}{p-scaling} and moment polytopes}\label{subsec:moment polytopes algo}
We now explain how to generalize \cref{alg:gconvex_gradient_uniform} to the optimization problem underlying~$p$-scaling.
Since the latter is characterized by the $p$-capacity~\eqref{eq:def p cap}, as explained in \cref{subsec:moment polytopes theory}, this is achieved by replacing the log-norm objective function by its `shifted' variant~\eqref{eq:def-Fun}, namely
\begin{equation*}
  F_{v,p} \colon G \to\RR, \quad F_{v,p}(g) = \log \| \pi(g) v\| + \frac{1}{\ell} \log \| \pi_{\lambda^*}(g) v_{\lambda^*}\|,
\end{equation*}

where we recall that $p = \lambda/\ell$ for a highest weight $\lambda$ and $\ell\in \ZZ_{> 0}$, and $\lambda^*$ is the highest weight of the dual representation. 
For the general definition and its specialization to the case when $G = \GL(n)$ we refer the reader to \cref{tab:summary gl} and \cref{subsec:rep theory}.

We state our first-order optimization algorithm in \cref{alg:nonuniform_gradient}.

%
%

\begin{Algorithm}[th!]
\textbf{Input}:\vspace{-.2cm}
\begin{itemize}
\item Oracle access to the moment map restricted to a group orbit, i.e., to the map $g \mapsto \mu(\pi(g)v)$,
\item a rational point~$p \in C(G)$,
\item a number of iterations~$T$.
\end{itemize}
\textbf{Output:} A group element $g$.
\medskip

\textbf{Algorithm:}\vspace{-.2cm}
\begin{enumerate}
\item Set $g_0 = I$. Set a step size $\eta = \frac1{2N^2}$, where $N^2 := N(\pi)^2 + \norm{p}_F$.
\item For $t=0,\ldots,T-1$: Set $g_{t+1} = e^{- \eta \: \left( \mu(\pi(g_t)v) + k_t p^* k_t^\dagger \right)} g_t$,
where $g_t = k_t b_t$ according to the Iwasawa decomposition~$G = KB$.
(If $G=\GL(n)$ then this is the QR decomposition.)
\item \textbf{Return} $g_t$, where $t = \argmin_{t=0,\dots,T-1} \norm{\mu(\pi(g_t) v) + k_t p^* k_t^{\dagger}}_F $
\end{enumerate}
\vspace{-.2cm}
\caption{Algorithm for the $p$-scaling problem (cf.~\cref{thm:nonuniform_grad_descent})}\label{alg:nonuniform_gradient}
\end{Algorithm}


\noindent
The following theorem gives rigorous guarantees on its performance

\begin{thm}[First order algorithm for $p$-scaling; general version of \cref{thm:intro_non-uniform_grad_descent}]\label{thm:nonuniform_grad_descent}
  Let $p\in C(G)$ and let~$v\in V$ be a vector with~$\capacity_p(v)>0$.
  Set $N^2 := N(\pi)^2 + \norm{p}_F$.
  For every~$\eps > 0$, \cref{alg:nonuniform_gradient} with 
  \begin{align*}
    T \geq \frac{4N^2}{\eps^2} \log\left(\frac{\norm{v}}{\capacity_p(v)}\right)
  \end{align*}
  iterations returns~$g\in G$ such $\lVert \mu(\pi(g) v) - k (-p^*) k^\dagger \rVert_F \leq \eps$, where~$g=kb$ according to~$G=KB$.
  In particular, $\norm{\intopoly(\mu(\pi(g) v)) - p}_F \leq \eps$.
\end{thm}
\begin{proof}
Similarly to the proof of \cref{thm:uniform_grad_descent}, we observe that \cref{alg:nonuniform_gradient} is obtained by specializing \cref{alg:gconvex_gradient_general} to the function~$F_{v,p}$, whose geodesic gradient is given by \cref{eq:Fvp gradient}.
Note that $F_{v,p,}(I) = \log\norm{v}$ and $\inf_{g\in G} F_{v,p}(g) = \log\capacity_p(v)$.
Since $F_{v,p}$ is moreover convex and $2N^2$-smooth by \cref{prp:smooth_non_uniform}, the first claim follows from \cref{thm:geodesic_grad_descent}.
The second claim follows since $\intopoly(-k_t p^* k_t^\dagger) = p$, as mentioned above, and the map~$s$ is a contraction~\cite[Lemma 4.10]{walter2014multipartite}.
\end{proof}

\section{Second-order algorithms}\label{sec:second_order}
In this section, we state our second order algorithm for $\capacity(v)$.
As before 
$G\subseteq\GL(n)$ is a symmetric subgroup with maximal compact subgroup~$K$ and
$\pi\colon G\to\GL(V)$ denotes a rational representation on the vector space~$V$
with $K$-invariant inner product~$\langle \cdot, \cdot \rangle$.

We first give a general algorithm for minimizing geodesically convex functions in \cref{subsec:general_gconvex}.
Next, in \cref{subsec:second_order_specialize}, we specialize to the norm minimization problem and derive our running time bounds using the analysis of gradient flow from \cref{subsec:flow}.
Again, we work in the setup introduced in \cref{sec:preliminaries}, with~$\pi\colon G\to\GL(V)$ a representation of a symmetric subgroup~$G \subseteq \GL(n)$ and~$\norm{\cdot}$ a~$K$-invariant norm on~$V$.

\subsection{General second-order optimization algorithm}\label{subsec:general_gconvex}
Our second-order optimization \cref{alg:gconvex second order} stated below generalizes the algorithm of~\cite{AGLOW18} to find an approximate minimum of an arbitrary geodesically convex left-$K$-invariant function~$F\colon G\to\RR$ (equivalently, on the symmetric space~$K\backslash G \cong P$, cf.~\cref{subsec:groups,subsec:geodesic_convexity_defns}).
The algorithm is analogous to ``box-constrained Newton's method'': progress is made in steps (starting with the identity) and at each step a group element is chosen that optimizes a simpler approximation to $F$ in a bounded region about the group element from the previous step. The size of the region is determined by the robustness of the target function.
If $F$ is convex in the sense of \cref{dfn:convex etc} then the geodesic Hessians~$\nabla^2 F$ are positive definite, so the optimization problem in step~2,~(b)
of \cref{alg:gconvex second order} is an ordinary convex quadratic optimization problem on the real vector space~$i\Lie(K)$, which can be solved using standard methods.

\begin{Algorithm}[th!]
\textbf{Input}:\vspace{-.2cm}
\begin{itemize}
\item Oracle access to the geodesic gradient~$\nabla F$ and Hessian~$\nabla^2 F$ of a left-$K$-invariant convex function $F\colon G\to\RR$ (see \cref{dfn:convex etc,dfn:grad hessian}),
\item a robustness parameter $R\geq1$,
\item a number of iterations $T$.
\end{itemize}

\textbf{Output:}
An element $g \in G$.
\medskip

\textbf{Algorithm:}\vspace{-.2cm}
\begin{enumerate}
\item Set $g_0 = I$.
\item For $t=0,\dots,T-1$:
  \begin{enumerate}
  \item Compute the geodesic gradient $W := \nabla F(g_t)$ and Hessian $Q := \nabla^2 F(g_t)$ at~$g_t$.
  \item Solve the following (Euclidean) convex quadratic optimization problem:
    \begin{align*}
      H_t := \argmin \left\{ \tr[W H] + \frac1{2e} \tr[Q (H \ot H)]  \;:\;  H \in i\Lie(K), \, \lVert H\rVert_F \leq \frac1R \right\}
    \end{align*}
  \item Set $g_{t+1} := e^{H_t/e^2} g_t$.
  \end{enumerate}
\item \textbf{Return} $g_{T}$.
\end{enumerate}
\vspace{-.2cm}
\caption{Geodesic second-order minimization algorithm (cf.~\cref{thm:gconvex}).}
\label{alg:gconvex second order}
\end{Algorithm}

Let us informally discuss the guarantees of \cref{alg:gconvex second order}.
It is an iterative algorithm which outputs a sequence of group elements beginning with the identity.
Given any fixed element $g_\star \in G$, eventually a group element $g \in G$ in this sequence will satisfy $F(g) \leq F(g_\star) + \eps$.
In particular, if $g_\star$ is an $\eps$-minimizer of $F$ then $g$ is a $2\eps$-minimizer.
The number of steps the algorithm requires depends on how ``far away'' $g_\star$ is from the starting point; more precisely, on the maximum geodesic distance between $g_\star$ and any point in the sublevel set $\{g:F(g) \leq F(I)\}$.
We now make these guarantees precise.

\begin{thm}\label{thm:gconvex}
Let $F\colon G\to\RR$ be a function that is left-$K$-invariant in the sense that $F(kg) = F(g)$ for all~$k\in K$, $g\in G$.
Suppose that~$F$ is $R$-robust in the sense of \cref{dfn:convex etc} for some $R\geq1$ (in particular, it is geodesically convex).
Let $g_\star \in G$ be a point such that $F(g_\star) < F(I)$ and let $D\geq1$ be a constant satisfying
\begin{align}\label{eq:D defn}
  D \geq \max_{F(g) \leq F(I)} \frac12\bigl\lVert \log \bigl((g_\star g^{-1})^\dagger ( g_\star g^{-1}) \bigr) \bigr\rVert_F.
\end{align}
Then \cref{alg:gconvex second order} with robustness parameter~$R$ and
\begin{align*}
  T \geq e^2 D R \log\left(\frac{F(I) - F(g_\star)}{\eps} \right)
\end{align*}
iterations returns a group element~$g\in G$ such that~$F(g) \leq F(g_\star) + \eps$.
\end{thm}

\noindent
\Cref{eq:D defn} states that the constant~$D$ bounds the maximum geodesic distance from $g_\star$ to any point in the sublevel set, see \cref{eq:geodesic distance}.
Under the assumption that $F(g_\star) \leq F(I)$, this distance is in turn bounded by the (geodesic) diameter of the sublevel set.
As such, we may think of~$D$ as a \emph{diameter bound}.
Dependence on diameter bounds and robustness is standard among guarantees for box-constrained methods.

\begin{proof}
The proof is a straightforward generalization of the argument in~\cite{AGLOW18}.
We prove the following assertions for~$t \geq 0$:
\begin{enumerate}
\item $F(g_t) \leq F(I)$
\item $F(g_t)- F(g_\star) \leq \left( 1 - \frac{1}{e^2DR} \right)^t \left( F(I) - F(g_\star) \right)$.
\end{enumerate}
These two statements clearly imply the theorem:
If $F(g_\star) < F(I)$ then the second assertion shows that~$F(g_T) - F(g_\star) \leq \eps$ for our choice of~$T$.
Otherwise, the first shows that $F(g_T) \leq F(g_\star)$ for any~$T$.

We now prove the two statements by induction on~$t \geq 0$:
For $t=0$, they are evident since $g_0=I$.
Now suppose they hold up to some~$t$.
Define
\begin{align*}
  H_\star := \frac12 \log \bigl( (g_\star g_t^{-1})^\dagger ( g_\star g_t^{-1}) \bigr).
\end{align*}
Since $F(g_t) \leq F(I)$, by definition of~$D$ we have~$\lVert H_\star \rVert_F \leq D$.
Furthermore,
\begin{align}\label{eq:H* is the g*}
  F(e^{H_\star} g_t)
= F\left(\sqrt{(g_\star g_t^{-1})^\dagger ( g_\star g_t^{-1})} g_t\right)
= F(g_\star g_t^{-1} g_t)
= F(g_\star),
\end{align}
where the second equality follows by the left-$K$-invariance of~$F$ and the polar decomposition~$G=KP$.
Now consider the quadratic approximation from \cref{cor:robust taylor}, which asserts that, for all $\norm{H}_F \leq 1/R$,
\begin{align}\label{eq:q taylor}
  q_-(H) \leq F(e^H g_t) - F(g_t) \leq q_+(H)
\end{align}
where
\begin{align*}
  q_+(H) &:= \partial_{s=0} F(e^{sH} g_t) + \frac{e}2 \partial^2_{s=0} F(e^{sH} g_t), \\
  q_-(H) &:= \partial_{s=0} F(e^{sH} g_t) + \frac1{2e} \partial^2_{s=0} F(e^{sH} g_t).
\end{align*}
Note that~$H_t$ in \cref{alg:gconvex second order} is precisely the minimizer of~$q_-$ subject to~$\lVert H\rVert_F \leq 1/R$.
If we define~$H_\diamond := H_\star / (D R)$ then~$\lVert H_\diamond \rVert_F \leq 1/R$; together with the lower bound in \cref{eq:q taylor}, we find that
\begin{align}\label{eq:upper bound on q_-}
  q_-(H_t) \leq q_-(H_\diamond) \leq F(e^{H_\diamond} g_t) - F(g_t).
\end{align}
Since $F$ is geodesically convex in the sense of \cref{dfn:convex etc}, the function $h(s) := F(e^{s H_\star} g_t)$ is convex in~$s\in\RR$.
In particular, since $DR\geq1$, and using \cref{eq:H* is the g*},
\begin{align}\label{eq:H_diamond progress}
  F(e^{H_\diamond} g_t) - F(g_t)
= h\left(\tfrac1{DR}\right) - h(0)
\leq \frac1{DR} \left( h(1) - h(0) \right)
= \frac1{DR} \left( F(g_\star) - F(g_t) \right).
\end{align}
If we combine \cref{eq:upper bound on q_-,eq:H_diamond progress}, we get
\begin{align}\label{eq:quad_prog}
  q_-(H_t) \leq -\frac1{DR} \left( F(g_t) - F(g_\star) \right).
\end{align}
This shows that our choice of $H_t$ makes significant progress in decreasing the \emph{quadratic approximation} of~$F$.
It remains to show that we actually decrease~$F$ itself.
Here we use that $q_-(H) = e^2 q_+(H/e^2)$ for all~$H$.
Using the upper bound in \cref{eq:q taylor} and noting that $\norm{H_t/e^2}_F \leq 1/R$, we find that
\begin{align}\label{eq:real vs quad prog}
  e^2 \left( F(g_{t+1}) - F(g_t) \right)
  = e^2 \left( F(e^{H_t/e^2} g_t) - F(g_t) \right)
  \leq e^2 q_+(H_t / e^2)
  = q_-(H_t).
\end{align}
We have~$q_-(H_t) \leq 0$ by definition of~$H_t$.
Thus, \cref{eq:real vs quad prog} implies that~$F(g_{t+1}) \leq F(g_t) \leq F(I)$, the latter by the induction hypothesis.
This establishes the first assertion that we wanted to show.
Moreover, if we combine \cref{eq:real vs quad prog,eq:quad_prog} then we obtain
\begin{align*}
  e^2 \left( F(g_{t+1}) - F(g_t) \right) \leq -\frac1{DR} \left( F(g_t) - F(g_\star) \right),
\end{align*}
which can be rearranged as
\begin{align*}
  F(g_{t+1})- F(g_\star) &\leq \left( 1 -\frac1{e^2DR} \right) \left( F(g_t) - F(g_\star) \right).
\end{align*}
Using the induction hypothesis, this establishes the second assertion, concluding the induction.
\end{proof}

\subsection{Application to norm minimization and scaling problem}\label{subsec:second_order_specialize}
We now describe our second order algorithm for the norm minimization problem and analyze its complexity.
The algorithm is simply \cref{alg:gconvex second order} run on an appropriate function~$F$, namely the log-norm function~\eqref{eq:kempf ness} plus a suitable regularizer, which we now define.

\begin{dfn}[Regularizer]
The \emph{regularizer} for the group~$G$ is the function $\reg\colon G\to(0,\infty)$ defined as
\begin{align*}
  \reg(g)
:= \norm{g}_F^2 + \norm{g^{-1}}_F^2
= \tr\bigl[g^\dagger g\bigr] + \tr\bigl[(g^\dagger g)^{-1}\bigr].
\end{align*}
\end{dfn}


\begin{Algorithm}[th!]
\textbf{Input}:\vspace{-.2cm}
\begin{itemize}
\item A representation $\pi: G \to V$ and a vector $v \in V$,
\item robustness and regularization parameters $R\geq 1, \kappa>0$,
\item a real number $\eps > 0$,
\item a number of iterations $T$.
\end{itemize}

\textbf{Output:}
An element $g \in G$.
\medskip

\textbf{Algorithm:}\vspace{-.2cm}
\begin{enumerate}
\item Set $g_0 = I$.
\item For $t=0,\dots,T-1$:
  \begin{enumerate}
  \item Set
  \begin{align*}
    F_{v, \kappa, \eps} (g): = \log\|\pi(g) v\| + \frac\eps\kappa \reg(g).
  \end{align*}
   \item Compute the geodesic gradient $W := \nabla F_{v, \kappa, \eps} (g_t)$ and Hessian $Q := \nabla^2 F_{v, \kappa, \eps} (g_t)$ at~$g_t$ (cf.~\cref{prp:normgrad}).
  \item Solve the following (Euclidean) convex quadratic optimization problem:
    \begin{align*}
      H_t := \argmin \left\{ \tr[W H] + \frac1{2e} \tr[Q (H \ot H)]  \;:\;  H \in i\Lie(K), \, \lVert H\rVert_F \leq \frac1R \right\}
    \end{align*}
  \item Set $g_{t+1} := e^{H_t/e^2} g_t$.
  \end{enumerate}
\item \textbf{Return} $g_{T}$.
\end{enumerate}
\vspace{-.2cm}
\caption{Second order algorithm for norm minimization (cf.~\cref{thm:main}).}
\label{alg:norm second order}
\end{Algorithm}

\noindent
The function~$\reg$ provides a convenient upper bound on the condition number~$\kappa_F(g):= \norm{g}_F \norm{g^{-1}}_F$ of $g \in G$.
Indeed, by the AM-GM inequality, we have
\begin{align*}
  \kappa_F(g)
\leq \frac{1}{2} \left(\norm{g}^2_F + \norm{g^{-1}}^2_F\right)
= \frac{1}{2}\reg(g).
\end{align*}
Note that the singular values of~$g$ are always between~$\reg(g)^{-1/2}$ and~$\reg(g)^{1/2}$.
Furthermore, $\reg$ is minimized at the identity, so we have~$\reg(g) \geq 2n$ for all~$g\in G$.
We first analyze the smoothness and robustness of the function~$\reg$.

\begin{lem}\label{lem:reg robust}
The function $\reg: G\to(0,\infty)$ is left-$K$-invariant and 2-robust.
\end{lem}
\begin{proof}
Since $K\subseteq U(n)$, we have that $(kg)^\dagger kg = g^\dagger g$ for every $k\in K$ and $g\in G$, so $\reg$ is clearly left-$K$-invariant.
To see that it is 2-robust, we prove that each term is individually 2-robust.

Consider the representation $\pi\colon G\to\GL(V)$ obtained by letting $G$ act on $V=\Mat(n)$ by left multiplication, i.e., $\pi(g)M = gM$.
Equip $V$ with the Hilbert-Schmidt inner product, which induces the Frobenius norm, and note that $K$ acts unitarily.
Then, $\tr\left[g^\dagger g\right] = \lVert \pi(g) I\rVert_F^2$, so $g\mapsto \tr\left[g^\dagger g\right]$ is nothing but the norm-square function for this representation and~$v=I$ the identity matrix.
By \cref{prp:norm square robust} we obtain that $g \mapsto \tr\left[g^\dagger g\right]$ is $2N(\pi)$-robust.
Finally, \cref{prp:opnorm} shows that~$N(\pi)=1$.
This can also be calculated directly:
The Lie algebra representation is given by~$\Pi(H)M = HM$ for $H\in \Lie(G)$ and $M\in\Mat(n)$.
Since the Frobenius norm is submultiplicative, it holds that $\norm{\Pi(H) M}_F \leq \norm{H}_F \norm{M}_F$.
It follows that $\norm{\Pi(H)}_{\ope} \leq \norm{H}_F$, and hence~$N(\pi)=1$.

Similarly, the map $g \mapsto \tr\left[(g^\dagger g)^{-1}\right]$ can be interpreted as the norm-square function for the representation $\pi\colon G\to\GL(V)$ defined by $\pi(g)M = M g^{-1}$.
Again, this representation has weight norm~$1$, so we may conclude that $\reg$ is $2$-robust.
\end{proof}

To analyze \cref{alg:norm second order}, we consider the following objective function for given $0\neq v\in V$, $\kappa>0$, and~$\eps>0$:
\begin{align}\label{eq:F regularized}
  F_{v,\kappa,\eps}: G \to \RR, \quad F_{v,\kappa,\eps}(g)
:= F_v(g) + \frac\eps\kappa \reg(g)
= \log \norm{\pi(g)v} + \frac\eps\kappa \reg(g)
\end{align}
Note that \cref{alg:norm second order} is simply \cref{alg:gconvex second order} run on $F_{v, \kappa, \eps}$ for an appropriate choice of $\kappa$ and the robustness parameter $R$. We first show that each iteration is efficient, namely we compute the Hessian and gradient explicitly. That the explicit formulae below imply efficient computability follows again from the efficient computability of $\Pi$, discussed before \cref{prp:smooth}.

\begin{prp}[Gradient and Hessian oracles for norm minimization]\label{prp:normgrad}
Let $0\neq v \in V$.
Then the gradient and Hessian of $F_{v, \kappa, \eps}$ at $g$ are given by
\begin{align*}
\nabla F_{v, \kappa, \eps}(g) &= \mu(u) + \frac{2\eps}{\kappa} \left( g g^\dagger - ( gg^{\dagger})^{-1}\right)\\
\tr \nabla^2 F_{v, \kappa,\eps} (g) (H \ot H)& =  2 \left(\langle u, \Pi(H)^2 u \rangle - \langle u, \Pi(H) u \rangle^2\right) + \frac{4\eps}\kappa \tr \left( (gg^\dagger + (gg^\dagger)^{-1})H^2 \right)
\end{align*}
for all $H \in i \Lie(K)$, where $u = \frac{\pi(g) v}{\norm{\pi(g) v}}$.
\end{prp}
\begin{proof}
Recall that $F_{v, \kappa, \eps}(g) = F_{v}(g) + \frac{\eps}{\kappa} \reg(g)$.
We first calculate $\nabla\reg(g)$ and $\nabla^2 \reg(g)$.
As in the proof of \cref{lem:reg robust}, we will use that $\reg(g)$ is the sum of two norm-squared functions,
for which we already know how to compute the Hessian.
Let $\pi_1, \pi_2:G \to \Mat(n)$ be the actions of $G$ on matrices by $\pi_1(g)M = gM$ and $\pi_2(g) = M g^{-1}$ respectively, so that
$\reg(g) = \|\pi_1(g) I_n\|^2_F + \|\pi_1(g) I_n\|_F^2$.
The gradient of $\varphi_i(g) := \|\pi_i(g) I_n\|^2_F$
is given by
\begin{align*}
  \tr(\nabla \varphi_i(g) H) = 2 \tr \big( (\pi_i(g) I_n)^\dagger \Pi_i(H)\pi_i(g) I_n \big)
\end{align*}
for $H \in i \Lie(K)$.
Using $\Pi_1(H)M = HM$ and $\Pi_2(H) M = - MH$, we get
$\nabla \varphi_1(g) =2 g g^\dagger$  and $\nabla \varphi_2(g) = -2 (gg^\dagger)^{-1}$.
Thus
\begin{align*}
  \nabla \reg (g) = 2(g g^\dagger - ( gg^{\dagger})^{-1}).
\end{align*}
By \cref{prp:norm square robust}, we have
\begin{align*}
  \tr \nabla^2 \reg (g) (H \ot H) = 4 tr g^\dagger H^2 g  + 4 tr (g^{-1})^\dagger g^{-1} H^2 = 4 \tr (gg^\dagger + (gg^\dagger)^{-1})H^2.
\end{align*}
Applying \cref{eq:moment map vs gradient} and \cref{prp:smooth} to calculate the contribution from $F_v$ completes the proof.
\end{proof}

We now establish that \cref{alg:norm second order} runs in few iterations \emph{assuming the existence of a well-conditioned approximate minimizer}.
Afterwards we prove the existence of such a minimizer.

\begin{prp}\label{lem:promise}
Let $v\in V$ be such that $\capa(v) >0$,
$C \geq \log(\norm{v}/\capa(v))$, and let
$\eps,\kappa>0$.
\begin{enumerate}
\item The function $F_{v,\kappa,\eps}$ is left-$K$-invariant and $4N$-robust, where $N:=\max\{N(\pi),1/2\}$.
\item\label{it:promise algo}
Suppose there exists an element~$g_\star\in G$ with $\log\norm{\pi(g_\star) v} \leq \log\capacity(v) + \eps$ and $\reg(g_\star) \leq \kappa$.
Then \cref{alg:norm second order} with robustness parameter~$R=4N$, regularization parameter $\kappa$, and
\begin{align*}
  T \geq 8 e^2  N \sqrt n \left( \log\kappa + \log\left(1 + \frac C\eps\right) \right) \log\left(\frac C\eps\right)
\end{align*}
iterations returns a group element~$g\in G$ such that $\log \norm{\pi(g) v} \leq \log\capacity(v) + 3\eps$.
\end{enumerate}
\end{prp}
\begin{proof}
We abbreviate $F := F_{v,\kappa,\eps}$.
By \cref{prp:g-GSOR,lem:reg robust}, $F$ is left-$K$-invariant and a sum of a $4N(\pi)$-robust function and a $2$-robust function, hence $4N$-robust.
This shows the first claim.

To prove the second claim, we apply \cref{thm:gconvex} to the function~$F$.
The theorem asserts that, if we run \cref{alg:gconvex second order} with robustness parameter~$R=4N$ and a suitable number of iterations, we obtain a group element~$g\in G$ such that $F(g) \leq F(g_\star) + \eps$.
The latter implies that
\begin{align*}
  \log \norm{\pi(g)v}
\leq F(g)
\leq F(g_\star) + \eps
= \log \norm{\pi(g_\star)v} + \frac\eps\kappa \reg(g_\star) + \eps
\leq \log\capacity(v) + 3\eps,
\end{align*}
as desired.
It suffices to bound the number of iterations.
Only the case that~$F(g_\star) < F(I)$ is of interest.
Here, \cref{thm:gconvex} asserts that
\begin{align*}
  T \geq e^2 D R \log\left(\frac{F(I) - F(g_\star)}{\eps} \right)
\end{align*}
iterations suffice.
We first note that, since~$\reg(g)$ is minimal at~$g=I$,
\begin{align}\label{eq:F1Fstar bound}
  F(I) - F(g_\star)
\leq \log\norm{v} - \log\norm{\pi(g_\star) v}
\leq \log\frac{\norm v}{\capacity(v)}
\leq C.
\end{align}
To find a suitable distance bound~$D$, recall that we need $D\geq1$ as well as (cf.~\cref{eq:D defn})
\begin{align}\label{eq:diam goal}
  D \geq \max_{F(g) \leq F(I)} \frac12\bigl\lVert \log \bigl((g_\star g^{-1})^\dagger ( g_\star g^{-1}) \bigr) \bigr\rVert_F.
\end{align}
The condition~$F(g) < F(I)$ implies that
\begin{align*}
  \reg(g)
\leq \reg(I) + \frac\kappa\eps \left(\log\norm{v} - \log\norm{\pi(g)v} \right)
\leq 2n + \frac\kappa\eps \log\frac{\norm v}{\capacity(v)}
\leq \kappa \left(1 + \frac C\eps \right).
\end{align*}
In the last step, we used that~$\kappa\geq\reg(g_\star)\geq2n$ (recall that~$\reg$ is bounded from below by~$2n$).
As mentioned earlier, this implies that the singular values of $g$ are between~$\kappa^{-1/2}(1 + \frac C\eps)^{-1/2}$ and~$\kappa^{1/2}(1 + \frac C\eps)^{1/2}$.
We also have~$\reg(g_\star) \leq \kappa$, which implies that the singular values of~$g_\star$ are between~$\kappa^{-1/2}$ and~$\kappa^{1/2}$.
It follows that the singular values of $g_\star g^{-1}$ are between~$\kappa^{-1}(1 + \frac C\eps)^{-1/2}$ and $\kappa(1 + \frac C\eps)^{1/2}$,
so
\begin{align*}
  \frac12\bigl\lVert \log \bigl((g_\star g^{-1})^\dagger ( g_\star g^{-1}) \bigr) \bigr\rVert_F
\leq \frac{\sqrt n}2\log\left(\kappa^2 \left(1 + \frac C\eps\right) \right).
\end{align*}
Thus, we choose
\begin{align*}
  D := \sqrt{n} \log\left(\kappa^2\left(1 + \frac C\eps\right)\right),
\end{align*}
so that \cref{eq:diam goal} and~$D\geq1$ are satisfied.
Together with \cref{eq:F1Fstar bound}, we find that, indeed,
\begin{align*}
e^2 D R \log\left(\frac{F(I) - F(g_\star)}{\eps} \right)
&\leq e^2 \sqrt n \log\left(\kappa^2\left(1 + \frac C\eps\right)\right) 4 N \log\left(\frac C\eps\right) \\
&\leq 8 e^2  N \sqrt n \left( \log\kappa + \log\left(1 + \frac C\eps\right) \right) \log\left(\frac C\eps\right)
\end{align*}
iterations suffice.
\end{proof}

\noindent
Finally, we show that there exist well-conditioned approximate minimizers.
Our bounds are described in terms of the \emph{weight margin}~$\gamma(\pi)$ defined in \cref{dfn:weight margin}, which is the closest the convex hull of a subset of weights can be to the origin without containing it.
The main mathematical tool in our proof is the gradient flow analyzed in \cref{subsec:flow}.

\begin{prp}[Condition bound]\label{prp:conditions}
Let $v\in V$ be a vector with~$\capacity(v)>0$ and let~$\eps>0$.
Then there exists a group element~$g_\star\in G$ such that~$\log\norm{\pi(g_\star) v} \leq \log\capacity(v) + \eps$ and
\begin{align*}
  \reg(g_\star) \leq 2 n \left( \frac {\lVert v\rVert^2}{2\capacity^2(v)\eps} \right)^{\frac1{\gamma(\pi)}}.
\end{align*}
\end{prp}
\begin{proof}
Let~$v\colon [0,T_v) \to V$ denote the solution of the gradient flow on its maximal domain of definition (\cref{dfn:flow}).
Recall from \cref{it:orbit} of \cref{prp:flow_identities} that $v(t) = \pi(g(t)) v$ for $t\in [0,T_v)$, where~$g\colon [0,T_v) \to G$ is a solution to the ordinary differential equation
\begin{align}\label{eq:g(t) ode}
 g'(t)
= -2 \frac{\mu(\pi(g(t)) v)}{\norm{\mu(\pi(g(t)) v)}_F} g(t)
= -2 \frac{\mu(v(t))}{\norm{\mu(v(t))}_F} g(t),
\quad g(0) = I.
\end{align}
By \cref{cor:multiplicative}, there is some $T< T_v$ with
\begin{align}\label{eq:T bound for *}
  T \leq \frac1{4\gamma(\pi)}\log \frac{\norm{v}^2}{2\capacity^2(v)\eps}
\end{align}
such that, for $g_\star:=g(T)$, we have
\begin{align*}
 \log\norm{\pi(g_\star)v} = \log\norm{v(T)} \leq \log\capacity(v) + \eps.
\end{align*}
It remains to verify the bound on~$\reg(g_\star)$.
We first bound~$\varphi(t) := \tr\bigl[g(t)^\dagger g(t)\bigr]$ by an ODE argument.
By taking derivatives, using \cref{eq:g(t) ode},
\begin{align*}
    \partial_t \tr\bigl[g(t)^\dagger g(t)\bigr]
  &= \tr\bigl[g'(t)^\dagger g(t)\bigr] + \tr\bigl[ g(t)^\dagger g'(t) \bigr]
  =-  \frac4{\norm{\mu(v(t))}_F} \tr\bigl[ g(t)^\dagger \mu(v(t)) g(t) \bigr] \\
  &\leq 4 \frac{\norm{\mu(v(t))}_{\ope}}{\norm{\mu(v(t))}_F} \tr\bigl[ g(t)^\dagger g(t) \bigr]
  \leq 4 \tr\bigl[ g(t)^\dagger g(t) \bigr].
\end{align*}
For the first inequality we used the general fact that~$\lvert\tr[AB]\rvert \leq \norm{A}_{\ope} \tr[B]$ for any Hermitian matrix~$A$ and positive semidefinite matrix~$B$.
Thus, we have shown that $\varphi'(t) \leq 4 \varphi(t)$, which implies that~$\tr\bigl[g(t)^\dagger g(t)\bigr] = \varphi(t) \leq \varphi(0) e^{4t} = n e^{4t}$.
Similarly,
\begin{align*}
    \partial_t \tr\Bigl[\left(g(t)^\dagger g(t)\right)^{-1}\Bigr]
  &= -\tr\Bigl[ \left(g(t)^\dagger g(t)\right)^{-1} \partial_t \left( g(t)^\dagger g(t) \right) \left(g(t)^\dagger g(t)\right)^{-1} \Bigr] \\
  &= -\tr\Bigl[\left(g(t)^\dagger g(t)\right)^{-1} \left( (g'(t))^\dagger g(t) + g(t)^\dagger g'(t) \right) \left(g(t)^\dagger g(t)\right)^{-1} \Bigr] \\
  &= \frac4{\norm{\mu(v(t))}_F} \tr\Bigl[\left(g(t)^\dagger g(t)\right)^{-1} g(t)^\dagger \mu(v(t)) g(t) \left(g(t)^\dagger g(t)\right)^{-1} \Bigr] \\
  &= \frac4{\norm{\mu(v(t))}_F} \tr\Bigl[g(t)^{-1} \mu(v(t)) g(t)^{-\dagger} \Bigr] \\
  &\leq 4 \frac{\norm{\mu(v(t))}_{\ope}}{\norm{\mu(v(t))}_F} \tr\Bigl[g(t)^{-1} g(t)^{-\dagger} \Bigr]
  \leq 4 \tr\Bigl[\left(g(t)^\dagger g(t)\right)^{-1}\Bigr],
\end{align*}
so $\tr\bigl[(g(t)^\dagger g(t))^{-1}\bigr] \leq n e^{4t}$.
Together, evaluating at~$t=T$, and using \cref{eq:T bound for *}, we obtain
\begin{align*}
 \reg(g_\star)
\leq 2 n e^{4T}
\leq 2 n e^{\frac1{\gamma(\pi)}\log\frac{\norm{v}^2}{2\capacity^2(v)\eps}}
= 2 n \left( \frac {\lVert v\rVert^2}{2\capacity^2(v)\eps} \right)^{\frac1{\gamma(\pi)}},
\end{align*}
completing the proof.
\end{proof}

\noindent
We thus obtain the main theorem of this section -- a second order optimization algorithm for minimizing the norm, i.e., approximating the capacity.

\begin{thm}[Second order algorithm for norm minimization; general statement of \cref{thm:intro_second_order}]\label{thm:main}
Let $v\in V$ be a vector with $\capacity(v)>0$. Let~$0<\eps<1/2$ and~$C \geq \log(\norm{v}/\capa(v))$.
Set~$\gamma := \min\{\gamma(\pi),1\}$, $N := \max\{N(\pi),1/2\}$, and $\kappa := 2 n \bigl(e^{2C} / 2\eps\bigr)^{1/\gamma}$.
Then \cref{alg:norm second order} with
\begin{align*}
  T \geq 24 e^2 \frac{N\sqrt n}\gamma \left(\log\frac n\eps + C \right) \log\frac C\eps
\end{align*}
iterations returns a group element~$g\in G$ such that $\log \norm{\pi(g) v} \leq \log\capacity(v) + 3\eps$.
\end{thm}
\begin{proof}
According to \cref{prp:conditions}, there exists $g_\star \in G$ such that $\log\norm{\pi(g_\star)v} \leq \log\capacity(v) + \eps$ and
\begin{align*}
 \reg(g_\star)
\leq 2 n \left( \frac{\lVert v\rVert^2}{2\capacity^2(v)\eps} \right)^{\frac1{\gamma(\pi)}}
\leq 2 n \left( \frac{e^{2C}}{2\eps} \right)^{\frac1{\gamma(\pi)}}
\leq 2 n \left( \frac{e^{2C}}{2\eps} \right)^{\frac1\gamma}
= \kappa,
\end{align*}
where the second inequality holds by the assumption on~$C$. The last inequality follows from~$\gamma(\pi)\geq\gamma$ and~$e^{2C}/2\eps\geq1$, which holds by the assumption~$\eps\leq1/2$.

Now apply \cref{it:promise algo} of \cref{lem:promise} with the element~$g_\star$, which asserts that \cref{alg:gconvex second order} applied to the function~$F_{v,\kappa,\eps}$ and robustness parameter~$R=4N$ returns the desired group element~$g\in G$ in a number of iterations at most
\begin{align*}
8 e^2  N \sqrt n \left( \log 2n + \frac1\gamma \log\frac{e^{2C}}{2\eps} + \log\left(1 + \frac C\eps\right) \right) \log\frac C\eps
&\leq 24 e^2 \frac{N\sqrt n}\gamma \left(\log\frac n\eps + C \right) \log\frac C\eps,
\end{align*}
where the inequality is obtained using~$\gamma \leq 1$ and~$\log(1 + C/\eps) \leq \log (1/\eps) + C$.
\end{proof}

\noindent
By \cref{cor:cap_vs_moment}, the second order algorithm described in \cref{thm:intro_second_order} can also be used to address the scaling problem.

\begin{rem}\label{rem:tradeoff}
Comparing \cref{cor:first_order_approx_capacity} (first order) with \cref{thm:main} (second order),
it is clear that the second order algorithm is better in terms of the dependence on the approximation parameter~$\eps$
(with the dependence on weight margin and norm similar) if the goal is to approximate the capacity.
However, the first order algorithm can be better if the goal is solve the scaling problem (\cref{thm:uniform_grad_descent})
because of the non-dependence on the weight margin in this case.
\end{rem}

\section{Bounds on weight norm and weight margin}\label{sec:norm margin bounds}
In this section we first provide some upper bounds on the weight norm (\cref{dfn:weight norm}).
Then we move on to prove lower bounds on the weight margin (\cref{dfn:weight margin}), which is the key complexity parameter for both our first and second order algorithms.
After presenting general lower bounds on the weight margin (which necessarily are exponentially small), we focus on the situation where~$G$ is a product of general linear groups or special linear groups and the weight matrix is totally unimodular.
Representations of quivers, which have been intensively studied in the literature from the structural viewpoint (e.g., see \cite{derksen2017introduction}),
turn out to have totally unimodular weight matrices.
For these well behaved representations, we are able to prove large lower bounds on the weight margin, which implies that our algorithms work fast: in fact in polynomial time in many situations!
We note that the concept of total unimodularity is of great importance in combinatorial optimization, see~\cite{SchrijverBooks}.

Recall from~\cref{eq:torus on weight vec} the definition of the weights $\omega$ of a representation $\pi$ of a symmetric subgroup $G\subseteq\GL(n)$.
The weights are elements of $i\Lie(T_K)$, hence real diagonal matrices, which we can identify with the vector in $\RR^n$ of its diagonal entries (see~\cref{subsec:rep theory}).
We can thus view $\Omega(\pi)$ as a sublattice of $\RR^n$.
Our cases of main interest are the groups $\GL(n_1)\times\dots\times \GL(n_k)$ and $\SL(n_1)\times\dots\times \SL(n_k)$,
seen as embedded in $\GL(n)$, where $n=n_1+\ldots+n_k$.
In the first case of products of $\GL(n)$s, the weights are integer vectors in $\ZZ^n$.
In the case $G=\SL(n_1)\times\dots\times \SL(n_k)$, the weights are of the form
$(\lambda_1,\ldots,\lambda_k)$, where $\lambda_i \in \frac{1}{n_i}\ZZ^{n_i}$ and the elements of $\lambda_i$ add up to zero.

\subsection{Weight norm}
Recall the characterization of the weight norm in~\cref{prp:opnorm}. Note that the Frobenius norm~$\norm{\omega}_F$ of
a weight $\omega$, interpreted as a vector in~$\RR^n$, equals the Euclidean norm~$\norm{\omega}_2$.
We first state a simple bound on the weight norm of polynomial representations
(see~\cref{dfn:weight norm}).

\begin{lem}[Weight norm of homogeneous representations]\label{lem:homog poly norm bound}
Let $\pi$ be a polynomial representation of a symmetric subgroup~$G$ of~$\GL(n)$ that is homogeneous of degree~$d$.
Then, $N(\pi) \leq d$.
\end{lem}

\begin{proof}
By \cref{prp:opnorm}, it suffices to bound the Euclidean norm of the weights
$\lambda\in\ZZ^{n}$ of $\pi$.
The sum of the elements of $\lambda$ equals $d$ by the homogeneity.
Thus we get $\norm{\lambda}_2 \leq \norm{\lambda}_1 = d$.
\end{proof}


\noindent
We now compute the weight norm for several important applications.

\begin{exa}[Weight norm for matrix and operator scaling]\label{exa:weight norm mat op}
Consider the action of~$G=\GL(n)\times\GL(n)$ on~$\Mat(n)$ by~$\pi(g,h) M := gMh^T$.
Clearly, $\pi$ is polynomial and homogeneous of degree~$2$, so $N(\pi)\leq2$.
In fact,~$N(\pi)=\sqrt2$ since it is an irreducible representation of highest weight~$\lambda=e_1+e_{n+1}$.
The same holds for the simultaneous left-right action on $\Mat(n)^k$, since it is a sum of $k$~copies of this representation.

Matrix scaling corresponds to restricting to~$\ST(n)\times\ST(n)$ and~$k=1$, whereas operator scaling as in \cref{exa:operator_scaling} is obtained by restricting to~$\SL(n)\times\SL(n)$.
Thus, \cref{lem:homog poly norm bound} 
shows that $N(\pi)\leq\sqrt2$ also holds for these representations.
\end{exa}

\begin{exa}[Weight norm for tensor scaling]\label{exa:weight norm tensors}
Consider the action of~$G=\GL(n_1)\times\cdots\times\GL(n_k)$ on~$V=\CC^{n_1}\ot\cdots\ot\CC^{n_k}$
by~$\pi(g_1,\dots,g_k) X := (g_1\ot\cdots\ot g_k) X$ (generalizing \cref{exa:tensor}).
Clearly, $\pi$ is polynomial and homogeneous of degree~$k$, hence~$N(\pi)\leq k$ by \cref{lem:homog poly norm bound}.
In fact, $\pi$ is irreducible with highest weight~$\lambda=e_1+e_{n_1+1}+\dots+e_{n_1+\ldots+n_{k-1}+1}$, hence $N(\pi)=\sqrt{k}$ by \cref{prp:opnorm}.
The same holds for tuples of $k$-tensors.
\end{exa}

\noindent In \cref{prp:weight norm margin quivers} we prove that the weight norm for group representations associated with quivers satisfies~$N(\pi) = \sqrt2$
(unless the representation is trivial).

\subsection{Weight margin and gap}
Recall that the weight matrix $M(\pi)\in\RR^{m\times n}$ of a representation $\pi$ has as its rows the different weights of $\pi$, thus $m=\lvert\Omega(\pi)\rvert$ (see Section~\ref{subsec:rep theory}).
Our main technical tool for deriving lower bounds on the weight margin
of a represention~$\pi$ will be the notion of the \emph{gap}
of its weight matrix~$M(\pi)$, which can be seen as a condition measure of~$M(\pi)$.

We first recall some facts regarding the smallest singular value $\sigma_r(A)$ of a matrix $A\in\RR^{r\times n}$,
where $r\le n$, see~\cite{golub-van-loan:83}.
It may be defined as
\begin{equation}\label{eq:char-sm}
 \sigma_r(A) := \min_{\norm{v}_2 =1} \norm{A^Tv}_2 .
\end{equation}
Clearly, $\sm(A)=0$ if $A$ has rank less than $r$. In fact, it is well known that $\sigma_r(A)$ equals
the distance of~$A$ to the set of matrices of rank less than $r$
(the distance being measured either with respect to the operator or the Frobenius norm.)
Note that \cref{eq:char-sm} implies that $\sigma_r(A) \ge \sigma_r(B)$ for
all invertible $r\times r$ submatrices $B$ of $A$.
We also recall the formula $\sigma_r(A)=\norm{(AA^T)^{-1}}_{\ope}^{-1/2}$, which will be used several times.

\begin{lem}\label{lem:SV}
Let $A \in\RR^{r\times n}$ be of full rank~$r$  
and $y\in\RR^r$. Then the unique $x\in\RR^n$ in the row span of~$A$ such that~$Ax=y$
satisfies $\norm{x}_2 \leq \sigma_r(A)^{-1} \norm{y}_2$.
\end{lem}

\begin{proof}
It is easily checked that~$x = A^T (AA^T)^{-1} y$.
We have
\begin{align*}
  \norm{x}_2^2
  = x^T x
  = y^T (AA^T)^{-T} AA^T (AA^T)^{-1} y
  = y^T (AA^T)^{-T} y
  \leq \norm{(AA^T)^{-T}}_{\ope} \norm{y}_2^2.
\end{align*}
The assertion follows using $\sigma_r(A)=\norm{(AA^T)^{-1}}_{\ope}^{-1/2}$.
\end{proof}

\begin{dfn}[Gap]\label{def:gap}
The \emph{gap}~$\sigma(M)$ of a nonzero matrix $M\in\RR^{m\times n}$
is the minimum of~$\sigma_r(B)$ taken over all {\em invertible} $r\times r$ submatrices~$B$ of $M$,
where $1\le r \le \min\{m,n\}$.
\end{dfn}

\noindent By the above comment we have $\sigma_r(A) \ge \sigma(M)$ for any submatrix $A\in \RR^{r\times n}$ of rank~$r\le n$.

We can lower bound the weight margin of a representation in terms of the gap of its weight matrix.
This bound is tight up to a constant factor, see \cref{rem:tight}.

\begin{prp}[Weight margin lower bound in terms of gap]\label{prp:gapped}
Let $\pi$ be a rational representation of a symmetric subgroup $G\subseteq \GL(n)$.
Then we have $\gamma(\pi) \geq \sigma(M(\pi)) \, n^{-\frac12}$.
\end{prp}

\begin{proof}
Recall that we view $\Omega(\pi) \subseteq\RR^n$. 
Let~$\Gamma\subseteq\Omega(\pi)$ be such that $0$ is not contained in the polytope~$P:=\conv(\Gamma)\subseteq\RR^n$.
We need to show that~$d(0,P) \geq \sigma n^{-\frac12}$ with respect to the Euclidean distance~$d$,
where we have put $\sigma:=\sigma(M(\pi))$.

It is easy to see that there is a face~$F$ of~$P$ with $d(0,P) = d(0,F)$ such that $0\not\in\aff(F)$.
Let~$v$ be the point in~$\aff(F)$ closest to the origin.
Since $d(0,F) \geq d(0,\aff(F)) = \norm{v}_2$, it suffices to prove that $\norm{v}_2 \geq \sigma n^{-\frac12}$.

Since the face~$F$ is the convex hull of a subset of~$\Gamma$, there exist~$\omega_1$, \dots, $\omega_r \in \Gamma$ that are affinely independent and affinely span~$\aff(F)$.
Since $0\not\in\aff(F)$, we have~$\dim\aff(F) = r-1 \leq n-1$ and the~$\omega_1$, \dots, $\omega_r$ are linearly independent.
Thus, the matrix~$A\in\RR^{r\times n}$ with the rows $\omega_1, \dots, \omega_r$ has rank~$r$
and we can bound $\sigma_r(A) \ge\sigma$.
Note that $v$ is in the row span of $A$.
Since $v$ is the point in $\aff(F)$ closest to the origin, we have that~$(\omega_i - v) \cdot v = 0$ and hence~$\omega_i \cdot v = \norm{v}_2^2$ for all~$i\in[r]$.
Thus, $x := \norm{v}_2^{-2} v$ satisfies~$Ax = \vec 1$, where~$\vec 1\in \RR^r$ is the all-ones vector, and is in the row span of~$A$.
With ~\cref{lem:SV} we conclude that, using $r\le n$,
\begin{align*}
  \norm{x}_2 \leq \sigma^{-1} \norm{\vec 1}_2 \leq \sigma^{-1} n^{1/2}.
\end{align*}
Therefore, $\norm{v}_2 = \norm{x}_2^{-1} \geq \sigma n^{-1/2}$, which completes the proof.
\end{proof}

In the following, we present some general methods for proving lower bounds on the gap of specific matrices.
For this purpose, it is useful to introduce some auxiliary notions and quantities.
For a matrix $M\in\RR^{m\times n}$ we denote by $\cS(M)$ the set of its invertible square submatrices.
Recall that the gap $\sigma(M)$ is defined as the minimum of $\sigma_r(A)$ over all $A\in\cS(M)$,
where $r$ denotes the rank of~$A$. 
We define now related quantities by
\begin{equation}
 \alpha(M): = \min_{A\in\cS(M)} |\det(A)|, \quad
 \beta(M): = \max_{A\in\cS(M) \cup\{0\}} |\det(A)| .
\end{equation}
Note that $\alpha(M)$ is only defined if $M\ne 0$, in which case we have
$0 < \alpha(M) \le \beta(M)$.
We also note that $\sigma(M)=\sigma(M^T)$,
$\alpha(M)=\alpha(M^T)$ and $\beta(M) = \beta(M^T)$.

\begin{prp}\label{pro:gap-AB}
For nonzero $M\in\RR^{m\times n}$ we have
$\sigma(M) \ge \frac{\alpha(M)}{\beta(M) \rk(M)}$.
\end{prp}

\begin{proof}
Let $A\in\cS(M)$ be of order $r\le\rk(M)$.
Using the arithmetic geometric mean inequality, we can bound
\begin{equation}\label{eq:AG-SM}
 \sigma_r(A)^{-2} = \norm{(AA^T)^{-1}}_{\ope} \leq \tr\bigl[ (AA^T)^{-1} \bigr]
 = \frac{1}{\det(AA^T)} \sum_{i=1}^r \adj(AA^T)_{ii} ,
\end{equation}
where $\adj$ denotes the adjugate.
We have $\det(AA^T) = |\det(A)|^2 \ge \alpha(M)^2$ by definition.
Moreover,
$\adj(AA^T)_{ii} = \pm \det(BB^T)$,
where the matrix $B\in\RR^{(r-1)\times r}$ is obtained from $A$ by removing the $i$th row.
The Binet-Cauchy formula gives
$\det(BB^T) = \sum_I |\det B'|^2$,
where the sum is over the $r$~many square submatrices $B'$ of $B$ of size $r-1$.
Since $|\det(B')| \le \beta(M)$ by definition, we obtain
$\det(BB^T) \le r \beta(M)^2$.
Altogether, we arrive at
\begin{align*}
 \sigma_r(A)^{-2} \le \frac{\rk(M)^2\beta(M)^2}{\alpha(M)^2}
\end{align*}
and the assertion follows.
\end{proof}

Recall that an integer matrix $M\in\ZZ^{m\times n}$ is called \emph{totally unimodular} if all its subdeterminants are $0$, $1$, or $-1$.
In other words, $\beta(M) \le 1$. Such matrices, if they are nonzero, satisfy $\alpha(M)=\beta(M) =1$, and we obtain
from \cref{pro:gap-AB} the following result (the bound is tight up to a constant factor, see \cref{rem:tight}).

\begin{cor}\label{lem:unimod_sing}
A totally unimodular matrix $M\in\ZZ^{m\times n}$ satisfies~$\sigma(M) \ge \rk(M)^{-1}$. 
\end{cor}

It is easy to provide general bounds on $\alpha(M)$ and $\beta(M)$.

\begin{lem}\label{le:bound-AB}
\begin{enumerate}
\item Suppose that $dM$ has integer entries, where $M\in\QQ^{m\times n}$ is nonzero and $d\in\NN_{\ge 1}$. Then
$\alpha(M) \ge d^{-\min(m,n)}$.

\item We have $\beta(M) \le b^{\min(m,n)}$ if the rows of $M$ have Euclidean norm bounded by $b$.

\end{enumerate}
\end{lem}

\begin{proof}
If $A\in\cS(dM)$ is of format $r\times r$, we have
$\det(d^{-1}A) = d^{-r}\det(A)$ and hence $|\det(d^{-1}A)| \ge d^{-r}$
since $\det(A)\in\ZZ$. The first assertion follows noting $r\le \min(m,n)$.
The second assertion follows with Hadamard's inequality.
\end{proof}

Suppose that $M$ satisfies the assumptions of \cref{le:bound-AB}.
Together with \cref{pro:gap-AB}, this implies $\sigma(M) \ge (bd)^{-n} n^{-1}$.
We can prove a slightly better bound.

\begin{thm}[General weight margin lower bound]\label{prp:general-margin}
\begin{enumerate}
\item Suppose that $M\in\RR^{m\times n}$ has rows that have Euclidean norm bounded by $b$.
Then $\sigma(M) \ge \alpha(M) b^{1-n} n^{-1/2}$.
If $M=M(\pi)$ is the weight matrix of a representation~$\pi$ then $\gamma(\pi) \ge \alpha(M) b^{1-n} n^{-1}$.

\item A rational representation $\pi$ of $\GL(n_1)\times\dots\times\GL(n_k)$ satisfies
$\gamma(\pi) \geq N(\pi)^{1-n} n^{-1}$, where $n = n_1 + \ldots + n_k$.

\item We have $\gamma(\pi_0) \geq R^{-n} N(\pi)^{1-n} n^{-1}$
for a rational representations $\pi_0$ of $\SL(n_1)\times\dots\times\SL(n_k)$,
where $R:=\lcm(n_1,\ldots,n_k)$.
\end{enumerate}
\end{thm}
\begin{proof}
\begin{enumerate}
\item Let $A\in\cS(M)$.  We again use \cref{eq:AG-SM}
as in the proof of \cref{pro:gap-AB}.
Let $B\in\RR^{(r-1)\times r}$ be obtained from $A$ by removing the $i$th row.
But now we bound as follows
\begin{align*}
  \det(BB^T) \ \le\ \Big(\frac{1}{r-1}\tr[BB^T] \Big)^{r-1}  = \Big(\frac{1}{r-1}\|B\|_F^2) \Big)^{r-1}
  \ \le\ b^{2(r-1)},
\end{align*}
where we used the arithmetic geometric mean for the first inequality.
As in the proof of \cref{pro:gap-AB}, we conclude
\begin{align*}
  \sigma_r(A)^{-2} \le \frac{r b^{2(r-1)} }{\alpha(M)^2} \le  \frac{ n b^{2(n-1)}}{\alpha(M)^2}
\end{align*}
and the claimed lower bound for $\sigma(M)$ follows.
We apply \cref{prp:gapped} for the lower bound on~$\gamma(M)$.

\item We have $\alpha(M(\pi))\ge 1$ since the weight matrix $M(\pi)$ has integer entries,
see \cref{le:bound-AB}.
Moreover, the Euclidean norm of the rows of $M(\pi)$ is bounded by $N(\pi)$
(see~\cref{prp:opnorm}). The assertion follows from the first part.

\item Similarly as for \cref{eq:Omega_pi_0},
each weight $(\omega_1,\ldots,\omega_k)$ of $\pi$, where $\omega_i\in\ZZ^{n_i}$,
gives rise to the weight
$(\tilde{\omega}_1,\ldots,\tilde{\omega}_k)$ of the restriction~$\pi_0$, where
$\tilde{\omega}_i = \omega_i - n_i^{-1} \tr[\omega_i]$
and $\tr[\omega_i]$ stands for the sum of the components of $\omega_i$.
Therefore, the matrix $RM(\pi_0)$ has integer entries and we obtain
$\alpha(M(\pi_0)) \ge R^{-n}$ from \cref{le:bound-AB}.
As before, the Euclidean norm of the rows of $M(\pi_0)$ is bounded by $N(\pi)$
(see~\cref{prp:opnorm}). The assertion follows from the first part.
\end{enumerate}
\end{proof}

We remark that the general bound in part three of~\cref{prp:general-margin}
can be significantly improved in several cases,
see \cref{cor:multihom-SL} and \cref{prp:weight norm margin quivers}.

\subsection{Small rank perturbations of totally unimodular weight matrices}\label{se:low-rank}
Recall that an integer matrix is called totally unimodular if all its subdeterminants are $0$, $1$, or $-1$.
For several representations of interest, the weight matrix is totally unimodular,
see \cref{prp:weight norm margin quivers}.
In this case, we immediately obtain the following from \cref{prp:gapped,lem:unimod_sing}.

\begin{cor}[Weight margin lower bound for totally unimodular weight matrices]\label{prp:unimod}
If the weight matrix~$M(\pi)$ of a rational representation of $G\subseteq \GL(n)$ is unimodular, then
$\gamma(\pi) \geq n^{-\frac32}$.
\end{cor}

Let us illustrate this with the example of the simultaneous conjugation $\pi$
of~$\GL(n)$ on $\Mat(n)^k$ (\cref{exa:SC}).

\begin{cor}\label{exa:opt-g-s}
The simultaneous conjugation action $\pi$ from \cref{exa:SC} satisfies $\gamma(\pi) = \Theta (n^{-\frac32})$.
\end{cor}

\begin{proof}
The set of weights equals $\Omega(\pi) = \{e_i - e_j \mid i,j, \in [n] \}\subseteq\ZZ^n$.
By a general principle~\cite[\S 19.3, Example~2]{schrijver1986}),
the corresponding weight matrix $M(\pi)$ is totally unimodular:
indeed, $M(\pi)$ equals the incidence matrix of the complete digraph
on~$n$ nodes (up the removing repeated zero rows) and
incidence matrices of directed graphs are always totally unimodular.
(We again use this argument later on for the more general~\cref{prp:weight norm margin quivers}.)
Thus we get $\gamma(\pi) \geq n^{-\frac32}$ from \cref{prp:unimod}.

In order to obtain an upper bound of this order of magnitude,
we note that a subset of weights of $\pi$ is naturally encoded by
a subset $\Gamma\subseteq [n]^2$.
We are interested in those $\Gamma$ with the property that
the polytope $\conv\{ (e_i-e_j) : (i,j) \in \Gamma\}$
does not contain the origin. This means that for all probability distributions
$x=(x_{ij})$ supported on $\Gamma$, the row and and column marginal distributions differ,
that is, there exists $k$ such that $r_k :=\sum_j x_{kj} \ne c_k :=\sum_i x_{ik}$.
(For this, note that
$\sum_{i,j}x_{ij} (e_i-e_j) =  (r_1-c_1,\ldots,r_n-c_n) $.)
Let us show that
\[ \Gamma_n := \{(i,j) \in [n]^2 : i < j\} \]
has this property, i.e., the convex hull of the corresponding set of weights does not contain the origin.
So let $x$ be a distribution supported on $\Gamma_n$.
If there is $i$ with $x_{in}\ne 0$, then $c_n >0 = r_n$ and we are done.
Otherwise, the $n$th column of~$x$ vanishes. We remove the
$n$th row and column of~$x$ and conclude by induction.

We now define a distribution on $\Gamma_n$ such that the corresponding convex combination of weights $e_i -e_j$ is close to the origin.
We now define the nonnegative numbers
\begin{align*}
  p_i := p(n)_i := i \frac{n-1}{2} -\frac{i(i-1)}{2} \quad \mbox{$1\le i \le n-1$.}
\end{align*}
and consider the probability distribution~$x$ defined by
$x_{i,i+1} := \lambda p_i$ for $1 \leq i \leq n-1$ and $x_{i,j}=0$ if $j\ne i+1$.
Here $\lambda := \lambda(n):=(\sum_i p_i)^{-1} = \Theta(n^{-3})$.
Clearly, the support of~$x$ is contained in $\Gamma_n$.
The marginal distributions of $x$ are given by
$\vec r= \lambda(p_1,\ldots,p_{n-1},0)$ and
$\vec c= \lambda(0,p_1,\ldots,p_{n-1})$, hence
\begin{align*}
  \vec r- \vec c = \lambda ( p_1, p_2-p_1,p_3-p_2,\ldots,p _{n-1} - p_{n-2}, -p_{n-1}).
\end{align*}
Note $p_1=p_{n-1}= (n-1)/2$ and
$p_{i + 1} - p_i = \frac{n - 1}{2} - i$ for $1 \leq i \leq n - 1$. Thus,
\begin{align*}
p_1^2 + p_{n-1}^2 + \sum_{i=1}^{n-2} (p_{i+1} -p_i)^2 = \Theta(n^3) ,
\end{align*}
and so $\|\vec r- \vec c\|_2^2  =\Theta(n^{-\frac32})$, which proves $\gamma(\pi) = O(n^{-\frac32})$.
\end{proof}

\begin{rem}\label{rem:tight} 
\cref{exa:opt-g-s} implies that the bounds in \cref{prp:gapped} and \cref{lem:unimod_sing}
are optimal up to a constant factor.
Indeed, \cref{prp:gapped} applied to $\pi$ describing simultaneous conjugation (\cref{exa:SC}) gives
\begin{align*}
  \gamma(\pi) \ge \sigma(M(\pi)) n^{-\frac12} \ge n^{-1} n^{-\frac12} ,
\end{align*}
where the second inequality is by \cref{lem:unimod_sing}, using that $M(\pi)$ is totally unimodular.
Since $\gamma(\pi) =O(n^{-\frac32})$ by \cref{exa:opt-g-s}, we get
$\sigma(M(\pi)) = \Theta(n^{-1}) = \Theta(\rk(M(\pi)^{-1})$.
This shows the claimed tightness for \cref{lem:unimod_sing}.
Moreover, this bound on $\sigma(M(\pi))$ shows $\gamma(\pi) = \Theta(\sigma(M(\pi)) n^{-\frac12})$,
which shows the claimed tightness for \cref{prp:gapped}.
\end{rem}

\cref{prp:unimod} is of interest for representations $\pi$ of the groups $G=\GL(n_1)\times\cdots\times\GL(n_k)$
with $n := \sum_{i=1}^k n_i$.
Usually, we cannot apply \cref{prp:unimod} to the restriction $\pi_0$ of $\pi$
to $\SL(n_1)\times\cdots\times\SL(n_k)$ since the weights are not integers and hence the weight matrix $M(\pi_0)$
is not unimodular.
However, it turns out that $M(\pi_0)$ is obtained from $M(\pi)$ by adding a matrix of rank~$k$ with entries
in $\frac{1}{R}\ZZ$, where $R=\lcm(n_1,\ldots,n_k)$, compare~\cref{eq:Mpi_0}.

We next develop some general results to lower bound the gap
of small rank perturbations of totally unimodular weight matrices.
This will be done via \cref{pro:gap-AB}, by controlling what happens to the $\alpha$ and $\beta$ parameters
when adding a rank one matrix.

\begin{lem}\label{le:beta-one-rank}
Let $M\in\RR^{m\times n}$, $x\in\RR^m$ and $y\in\RR^n$. Then
\begin{align*}
  \beta(M + xy^T) \ \le \ \beta(M) \big(1 +\min(n,\rk(M)+1) \|x\|_\infty \|y\|_1 \big) .
\end{align*}
\end{lem}

\begin{proof}
For index sets $I\subseteq [m]$ and $J\subseteq [n]$ with the same cardinality~$r$,
we denote by $M_{I,J}$ the submatrix of $M\in\RR^{m\times n}$ obtained by selecting
the rows with index in~$I$ and columns with index in~$J$.
If we put $N:= M +xy^T$, then we have
$N_{I,J} = M_{I,J} + x_I y_J^T$.
The matrix determinant lemma~\cite{harville1998matrix} implies that
\begin{align*}
  \det(N_{I,J}) = \det(M_{I,J}) + y_J^T \adj(M_{I,J}) x_I .
\end{align*}
The entries of the adjugate $\adj(M_{I,J})$ are, up to a sign,
determinants of order $r-1$ submatrices of~$M$.
Therefore,
$|\adj(M_{I,J})_{ij}| \le \beta(M)$ for all $i\in I,j\in J$
and hence
\begin{align*}
  |y_J^T \adj(M_{I,J}) x_I | \ \le \ \beta(M) \|x_ I\|_1 \|y_ J\|_1
 \ \le \ \beta(M) r \|x\|_\infty \|y\|_1 .
\end{align*}
We may assume $r \le \rk(M)+1$ since  otherwise $\adj(M_{I,J}) = 0$.
Using $|\det(M_{I,J})| \le \beta(M)$ the assertion follows.
\end{proof}

By induction, we easily derive from \cref{le:beta-one-rank}
the following result.

\begin{cor}\label{cor:beta-update}
Let $M\in\RR^{m\times n}$, $x_1,\ldots,x_k\in\RR^m$ and $y_1,\ldots,y_k\in\RR^n$. Then
\begin{align*}
  \beta(M + x_1y_1^T + \ldots + x_ky_k^T ) \ \le \ \beta(M) \prod_{i=1}^k \big(1 +\min(n,\rk(M)+i) \|x_i\|_\infty \|y_i\|_1 \big) .
\end{align*}
\end{cor}

We next study the alpha parameter by looking at how the denominators of
$\det A$ for $A\in\cS(M)$ change when adding a rank one matrix to $M$.

\begin{lem}\label{le:alpha-one-rank}
Let $M\in\QQ^{m\times n}$ and $d\in\NN_{\ge 1}$ be such that
$d\det(A)$ is an integer for all $A\in\cS(M)$.
Consider the rank one perturbation
$N:= M + \frac{1}{e} xy^T$
for $x\in\ZZ^m, y\in\ZZ^n$ and $e\in\NN_{\ge 1}$.
Then  $de\det(B)$ is an integer for all $B\in\cS(N)$.
\end{lem}

\begin{proof}
For index sets $I\subseteq [m]$ and $J\subseteq [n]$ with the same cardinality
we have
$N_{I,J} = M_{I,J} + \frac{1}{e} x_I y_J^T$.
The matrix determinant lemma implies that
\begin{align*}
  \det(N_{I,J}) = \det(M_{I,J}) + \frac{1}{e} y_J^T \adj(M_{I,J}) x_I^T .
\end{align*}
By assumption, $d\det(M_{I,J})$ and the entries of $d\adj(M_{I,J})$ are integers.
Hence $de\det(N_{I,J})$ is an integer.
\end{proof}

By induction, we easily derive from \cref{le:alpha-one-rank}
the following result.

\begin{cor}\label{cor:alpha-update}
Let $M\in\QQ^{m\times n}$ and $d\in\NN_{\ge 1}$ be such that
$d\det(A)$ is an integer for all $A\in\cS(M)$
(e.g., this is satisfied for $d=1$ if $A$ has integer entries). 
Further, let
$x_1,\ldots,x_k\in\ZZ^m$, $y_1,\ldots,y_k\in\ZZ^n$, and
$e_1,\ldots,e_k\in\ZZ_{\ge 1}$, and assume
\begin{align*}
  N := M + \frac{1}{e_1}x_1y_1^T + \ldots + \frac{1}{e_k} x_ky_k^T
\end{align*}
is nonzero. Then $de_1\cdots e_k\det(A)$ is an integer for all $A\in\cS(N)$.
In particular, $\alpha(N) \ge (de_1\cdots e_k)^{-1}$.
\end{cor}

We next give a first application of our methods to provide a good lower bound on the weight margin
for certain representations of $\SL(n_1)\times\cdots\times\SL(n_k)$.
This bound is inverse polynomial in the dimension parameters $n_i$, the degrees $d_i$, and
the least common multiple of the $n_i$, for $i=1,\ldots,k$.

We call a representation~$\pi$ of~$G=\GL(n_1)\times\cdots\times\GL(n_k)$
{\em multihomogeneous} of multidegree $(d_1,\ldots,d_k)$ if the matrix entries of $\pi(g_1,\ldots,g_k)$
are homogeneous of degree~$d_i$ in $g_i$.
For instance, the left-right action (\cref{exa:operator_scaling})
is multihomogeneous of multidegree $(1,1)$.

\begin{cor}\label{cor:multihom-SL}
Let $\pi$ be a multihomogeneous representation of~$G=\GL(n_1)\times\cdots\times\GL(n_k)$
of multidegree $(d_1,\ldots,d_k)$ and denote by $\pi_0$ the restriction of $\pi$
to $\SL(n_1)\times\cdots\times\SL(n_k)$.
Put $n:= n_1+\ldots+n_k$, $d:= d_1 +\ldots+d_k$, and
$R := \operatorname{lcm}(n_1, \dots, n_k)$.
If $M(\pi)$ is totally unimodular, then we have
\begin{align*}
  \gamma(\pi_0) \ \geq\ R^{-1} ( 1 +  nd)^{-1} n^{-\frac32} .
\end{align*}
\end{cor}

\begin{proof}
Generalizing~\cref{eq:Mpi_0}, the weight matrix $M(\pi_0)$ can be obtained from
the weight matrix $M(\pi)\in\ZZ^{m\times n}$  as follows:
\begin{align*}
  M(\pi_0) = M(\pi) - 1_m v^T ,
\end{align*}
with the row vector
\begin{align*}
  v^T :=  \big(\frac{d_1}{n_1}1_{n_1},\ldots,\frac{d_k}{n_k}1_{n_k}\big)  ,
\end{align*}
which satisfies $\|\nu\|_1 = d$.
\cref{le:beta-one-rank} implies that
\begin{align*}
  \beta(M(\pi_0)) \le \beta(M(\pi)) \big( 1 + n \|1_m\|_\infty \|\nu\|_1  \big)= 1 + n d ,
\end{align*}
using $\beta(M(\pi))= 1$ since $M(\pi)$ is unimodular.
Note that $R\nu$ has integer entries. Hence, \cref{le:alpha-one-rank} gives that
$\alpha(M(\pi_0)) \ge R^{-1}$.
\cref{pro:gap-AB} now implies
$\sigma(M(\pi_o)) \ge R^{-1} (1 +nd)^{-1} n^{-1}$
and applying \cref{prp:gapped} completes the proof.
\end{proof}

\cref{cor:alpha-update} applied to operator scaling from \cref{exa:operator_scaling}
($k = 2$ and $n_1 = n_2$) implies the bound $\gamma(\pi_0) \geq \Omega(n^{-3.5})$.
However, this bound is not optimal. The improvement below is from~\cite[Lemma 5.2]{LSW};
see also~\cite[Prop.~2.4]{gurvits2004classical}. 

\begin{prp}\label{pro:WM-LR}
The weight margin of operator scaling $\pi_0$ (\cref{exa:operator_scaling})
satisfies $\gamma(\pi_0) =\Theta(n^{-\frac32})$.
\end{prp}

Note that~\cite{LSW} studied matrix scaling (group $\ST(n)\times\ST(n)$)
while~\cite{gurvits2004classical} dealt with operator scaling (group $\SL(n)\times\SL(n)$).
Clearly, both actions have the same weight margin.
We also remark that the set of weights for operator scaling and generalized Kronecker quivers
(\cref{exa:kronecker_quiver}) are related by an isometry and thus
the same weight margin.

For the sake of completeness, we present a comprehensive proof of \cref{pro:WM-LR}.
To a  nonempty subset $\Gamma \subset [n] \times [n]$ we assign the polytope
\begin{equation*}\label{eq:PGamma}
 P(\Gamma):=\conv\{(e_i, e_j): (i,j) \in \Gamma\} .
\end{equation*}
The elements of $P(\Gamma)$ are of the form
$(\vec r, \vec c) :=  (r_1,\ldots,r_n,c_1,\ldots,c_n)= \sum_{i,j}x_{ij} (e_i,e_j)$, 
where $x=[x_{ij}]$ is a nonnegative matrix with $\sum_{i,j} x_{ij} =1$ and
$r_i :=\sum_j x_{ij}$ and $c_j :=\sum_i x_{ij}$ are the row and column sums of $x$, respectively.
The following known result of independent interest provides a dual characterization of the polytope~$P(\Gamma)$.


\begin{prp}[Thm.~2 in \cite{rothblum1989scalings}]\label{prp:PGamma}
The polytope $P(\Gamma)$ 
equals the intersection of the halfspaces
\begin{align*}
  H_{R,C} := \big\{ (\vec r,\vec c) \in\RR^{2n} : \sum_{i\in R} r_i + \sum_{j\in C} c_j \ge 1\big\} , 
\end{align*}
taken over all subsets $R,C\subseteq [n]$
such that $\Gamma \subset (R \times [n]) \cup ([n] \times C)$
\end{prp}

\begin{proof}
We use a construction of a flow network appearing in~\cite{LSW}.
Consider the following flow network on a source $s$, a sink $t$ and two disjoint copies of $[n]$.
For each $(i,j)\in\Gamma$ the network has an infinite capacity edge $(i,j)$ from $i$ in the first copy of $[n]$
to $j$ in the second copy of $[n]$, it has an edge from $s$ to $i$ in the first copy of $[n]$ with capacity~$r_i$ and
it has an edge from $j$ in the second copy of $[n]$ to $t$ with capacity $c_j$.
By the above characterization, we have $(\vec r, \vec c) \in P(\Gamma)$ iff this network has a flow of value $1$.
By the max-flow min-cut theorem, this means that every cut of the network has value at least $1$.
It is easy to check that this exactly means that
$\sum_{i\in R} r_i + \sum_{j\in C} c_j \ge 1$.
\end{proof} 

\begin{proof}[Proof of \cref{pro:WM-LR}]
The weights of $\pi_0$ are
of the form $(e_i -\frac1n 1_n, e_j -\frac1n 1_n)\in \RR^{2n}$,
where $(i,j) \in [n]\times [n]$.
By definition, the weight margin $\gamma(\pi_0)$ is the minimum distance of
$z:=(\frac1n 1_n, \frac1n 1_n)$
to the polytopes $P(\Gamma)$, 
taken over all nonempty subsets
$\Gamma \subseteq [n] \times [n]$ such that $z\not\in P(\Gamma)$.

We first prove a lower bound on the weight margin.
Let $1_S$ denote the indicator vector of $S \subseteq [n]$.
If $z\not\in P(\Gamma)$, then there are $R,C$ as in \cref{prp:PGamma} such that

\begin{align*}
  (1_R,1_C)^T(\vec r,\vec c) \ge 1 \ge 1 -\frac1n \ge  (1_R,1_C)^T z  .
\end{align*}
for all $(\vec r,\vec c) \in P(\Gamma)$. (For the gap we used $z\in\frac1n\ZZ$.)
Therefore, by Cauchy-Schwarz,
\begin{align*}
  \sqrt{2n} \|(\vec r,\vec c) - z \|_2 \ge (1_R,1_C)^T \big( (\vec r,\vec c) -z \big) \ge \frac1n
\end{align*}
and we obtain $\gamma(\pi_0) \ge (n\sqrt{2n})^{-1}$.
Alternatively, we could have referred to~\cite[Prop.~2.7]{BLNW:20}:
by \cref{prp:PGamma}, the $nP(\Gamma)$ are integral polytopes with unary facet complexity~$1$,
hence the facet gap of $nP(\Gamma)$ is lower bounded by $(2n)^{-1/2}$
and it follows again that $\gamma(\pi_0) \ge (n\sqrt{2n})^{-1}$.

In order to obtain an upper bound of this order of magnitude, we assume w.l.o.g.\
that $n=2s$ is even. Consider
$\Gamma = (R \times [n]) \cup ([n] \times C)$, where
$R=\{1,\ldots,s-1\}$ and $C=\{s+1,\ldots, 2s\}$.
Then $z\not\in H_{R,C}$ since $|R| + |C| < n$. In particular, $z\not\in P(\Gamma)$.
We now choose a block diagonal matrix
$x= \begin{pmatrix} y & 0 \\ 0 & z \end{pmatrix}$,
where $y$ and $z$ have equal entries and are of the format $(s-1)\times s$, and $(s+1)\times s$,
respectively.
The column distribution $\vec c$ of $x$ is the uniform one, but the
row distribution $\vec r$ slightly deviates from it and equals
$\vec r=((2(s-1))^{-1}1_{s-1}, 2(s+1))^{-1}1_{s+1})$. We calculate
\begin{align*}
  \|\vec r- n^{-1} 1_n\|^2_2 = (s-1) \Big(\frac{1}{n-2} -\frac1n \Big)^2 + (s+1) \Big(\frac{1}{n+2} -\frac1n \Big)^2 ,
\end{align*}
which is asymptotically equal to $4 n^{-3}$. Finally, we note that
$(\vec r,\vec c) \in P(\Gamma)$ since the support of $x$ equals $\Gamma$ by construction.
\end{proof}

\subsection{Application to quiver representations and tensors}\label{se:quiver-appl}
We next derive lower bounds on the weight margin for representations of quivers (\cref{exa:quivers}).
Compare~\cite{derksen2017introduction}.

\begin{thm}[Weight margin for quivers]\label{prp:weight norm margin quivers}
Let $Q$ be a quiver with vertex set~$Q_0$, set of arrows~$Q_1$, and fix a dimension vector $(n(x))_{x\in Q_0}$.
We consider the corresponding representation $\pi$ of~$G = \prod_{x\in Q_0} \GL(n(x))$ on the space
$V = \bigoplus_{a\in Q_1} \Mat(n(ha), n(ta))$.
Let $k:= |Q_0|$ denote the number of vertices and put $n:=\sum_{x\in Q_0} n(x)$ and
$P :=\prod_{x\in Q_0} n(x)$. Then:
\begin{enumerate}

\item The weight matrix of $\pi$ is totally unimodular.

\item We have $\gamma(\pi) \geq n^{-3/2}$.

\item We have $N(\pi) \le \sqrt2$ (with equality unless $Q$ has only self-loops and $n(x)=1$ for all $x$).

\item
The restriction $\pi_0$ of $\pi$ to the subgroup $\prod_{x\in Q_0} \SL(n(x))$ has a weight margin
lower bounded as
\begin{align*}
  \gamma(\pi_0) \ \ge \ P^{-k} (n+1)^{-k} n^{-3/2} .
\end{align*}
\end{enumerate}
\end{thm}
\begin{proof}
We first describe the weights of the representation~$\pi$.
For $x\in Q_0$ and $i\in [n(x)]$ we have a standard basis vector
$e_{x,i} \in \bigoplus_{x'\in Q_0} \ZZ^{n(x')}$, which has the standard basis vector~$e_i$
in the block corresponding to the vertex~$x$ and has zeros elsewhere.
Let $a\in Q_1$ be an arrow from $x$ to $y$. For $i\in [n(y)]$ and $j \in[n(x)]$
we denote by $E^a_{i,j}$ the vector in~$V$  
obtained by putting the basis matrix~$E_{i,j}$ (with all entries zero except for a one in position $i,j$) in the block
corresponding to the edge~$a$
and putting zero matrices elsewhere. It is easy to see that
$E^a_{i,j}$ is a weight vector of weight~$e_{y,i} - e_{x,j}$.
Moreover, by varying~$a$,~$i$,~$j$, we obtain a basis of $V$ consisting of weight vectors.
Note that
$m:=\dim V = \sum_{a\in Q_1} n(ta) n(ha)$.

Summarizing, we can describe the set of weights of $\pi$ as follows:
\begin{align}\label{eq:quiver weights}
 \Omega(\pi) = \bigl\{ e_{y,i} - e_{x,j} \mid a \in Q_1, x=ta, y=ha, i \in [n(y)], \, j \in [n(x)] \bigr\}.
\end{align}
Recall that the \emph{incidence matrix} of a directed graph with vertex set~$[n]$ 
is defined as the $m\times n$ matrix
with one row~$e_i - e_j$ for each edge~$(i,j)$, where $m$ denotes the number of distinct edges.

Therefore, we can view the weight matrix $M(\pi)\in\ZZ^{m\times n}$ as the incidence matrix of the directed graph
with the vertex set~$\{(x,j) \mid x\in Q_0, j \in [n(x)]\}$ of size $n$,
where we put an edge~$(x,j)\to(y,i)$ whenever there exists an edge~$a\in Q_1$ pointing from~$x$ to~$y$
and, if $x=y$, $i\neq j$.
(Recall that by definition, $M(\pi)$ does not have repeated rows.)
We now use the well-known fact that the incidence matrix of a directed graph is totally unimodular (see, e.g.,~\cite[\S 19.3, Example~2]{schrijver1986}).
This implies the first assertion of the theorem.

The second assertion on the weight margin of $M(\pi)$ follows now with \cref{prp:gapped} and \cref{pro:gap-AB}.

The third statement about the weight norm is obvious from \cref{eq:quiver weights}.

In order to show the fourth assertion, we have to analyze the weight matrix of the restricted representation $\pi_0$.
For $x\in Q_0$, let $1_x\in\bigoplus_{x'\in Q_0} \ZZ^{n(x')}$ denote the vector having $(1,\ldots,1)\in\ZZ^{n(x)}$ in place $x$ and zero elsewhere.
Note that $\|1_x\|_1 = n(x)$.
Upon restriction, the weight $e_{y,i} - e_{x,j}$ of $\pi$ becomes the weight
\begin{align*}
  e_{y,i}   -  \frac{1}{n(y)} 1_y   - e_{x,j} + \frac{1}{n(x)} 1_x
\end{align*}
of $\pi_0$. A moment's thought shows that we can decompose
\begin{align*}
  M(\pi_0) = M(\pi) + \sum_{x\in Q_0} \frac{1}{n(x)} u_x  1_x^T ,
\end{align*}
where the column vectors $u_x\in \ZZ^{m}$ are described as follows:
the component of $u_x$ labeled by $a\in Q_1$, $i \in [n(ta)]$, $j \in [n(ha)]$ equals $1$ if $x=ta$;
it equals $-1$ if $y=ha$; and it equals $0$ otherwise.
We are now in a position to apply the results of \cref{se:low-rank}.
By \cref{cor:beta-update}, we obtain
\begin{equation}\label{eq:beta-quiver}
 \beta(M(\pi_0)) \ \le\ \beta(M(\pi)) \prod_{x\in Q_0} \big(1 + n \|u_x\|_\infty  \left\|\frac{1}{n(x)} 1_x\right\|_1 \big) = (1+n)^k ,
\end{equation}
where we have used that $M(\pi)$ is unimodular.
\cref{cor:alpha-update} implies
$\alpha(M(\pi_0)) \ge P^{-k}$, where we recall that $P =\prod_{x\in Q_0} n(x)$.
Now \cref{pro:gap-AB} implies
$\sigma(M(\pi_0) \ge P^{-k} (1+n)^{-k} n^{-1}$.
Finally, \cref{prp:gapped} shows
$\gamma(M(\pi_0)) \ge P^{-k} (1+n)^{-k} n^{-3/2}$.
\end{proof}

Combining \cref{prp:weight norm margin quivers} with~\cref{cor:prod-concretenullcone-no-GT},
which generalizes~\cref{cor:concretenullcone} to products of $\SL(n)$'s,
we immediately obtain numerical algorithms for the null cone problem for quiver representations,
which run in polynomial time in many cases.

\begin{cor}\label{cor:NCM-quivers}
\begin{enumerate}
\item There is a polynomial time algorithm for solving the null-cone membership problem for quiver representations
of the groups $\GL(n_1)\times\cdots\times\GL(n_k)$.

\item For a fixed number~$k$ of vertices, there is a polynomial time algorithm for solving the null-cone membership problem
for quiver representations of the groups $\SL(n_1)\times\cdots\times\SL(n_k)$.

\end{enumerate}
\end{cor}

We note that \cref{cor:NCM-quivers} covers most examples discussed in the introduction:
\begin{itemize}
\item Simultaneous conjugation of~$\GL(n)$ on $\Mat(n)^k$ (\cref{exa:SC}). 
\item Generalized Kronecker quiver actions of~$\GL(n)^2$ on $\Mat(n)^k$ (\cref{exa:kronecker_quiver}). 
\item Action of~$\GL(n)^3$ on $\Mat(n)^2$ underlying Horn's problem (\cref{exa:horn}). 
\end{itemize}

Our next goal is to provide a lower bound on the weight margin for the tensor action from~\cref{exa:tensor}.
It is already of interest for the space of three slice tensors in $\CC^3\ot\CC^n\ot\CC^n$ with the action of $\SL(3)\times\SL(n)\times\SL(n)$.
The bound will easily follow from the general result below (\cref{th:tensor-action}).


Let $\pi$ be a rational representation of the symmetric subgroup $G\subseteq\GL(n)$ acting on the space~$V$ and let $\pi'$ denote the simultaneous action of $G$ on $n_0$ many copies of $V$.
The weights of $\pi'$ are the same as those of $\pi$, they just appear with higher multiplicity.
Therefore, $\gamma(\pi') = \gamma(\pi)$.
Now we take into account a further action of $\SL(n_0)$, which
allows to mix different copies of $V$.
(Equivalently, we could consider the action of $\ST(n_0)$, which is interesting
for the 3D-generalization of matrix scaling, e.g., see \cite{search-problems-gct}.)
More precisely, we consider the tensor product action $\Id_{n_0}\ot\pi$
of the group $\SL(n_0) \times G$
on the space $\CC^{n_0}\ot V$, for fixed $n_0\ge 1$.
The tensor action (\cref{exa:tensor}) of $\SL(n_0) \times \SL(n_1)\times\SL(n_1)$ arises this way from
the simultaneous left-right action of $G=\SL(n_1)\times\SL(n_1)$ on $\Mat(n_1,n_1)$, cf.~\cref{exa:operator_scaling}.

\begin{thm}\label{th:tensor-action}
Let $\pi$ be a rational representation of the symmetric subgroup $G\subseteq\GL(n)$ acting on the space~$V$
and $n_o\in\NN_{\ge 1}$. We put $m:=\dim V$
and assume $R\in\NN_{\ge 1}$ is such that $R M(\pi)$ has integer entries.
Then the action $\Id_{n_0}\ot\pi$ of the product group $\SL(n_0) \times G$
on the tensor product $\CC^{n_0}\ot V$ has a weight margin bounded below by
\begin{align*}
  \gamma(\Id_{n_0}\ot\pi) \ \ge\
  R^{-1} \beta(M)^{-1} (1 + 2 n_0 \cdot (n+n_0) )^{-n_0} n^{-3/2}.
\end{align*}
\end{thm}

\begin{proof}
The set of weights of the tensor product representation is obtained from $\Omega(\pi)$ as follows:
\begin{equation*}\label{eq:Om-te-pro}
 \Omega(\Id_{n_0}\ot\pi) = \Big\{ (e_i - \frac{1}{n_0} 1_{n_0}, \omega) \mid i \in [n_0], ~ \omega \in \Omega(\pi) \Big\} \subseteq \RR^{n_0}\times\RR^n .
\end{equation*}
Using this, we will concisely express how the weight matrix $\tilde{M} := M(\Id_{n_0}\ot\pi)$ arises from the weight matrix
$M:=M(\pi)\in\RR^{m\times n}$.
For this, we introduce some terminology.
First, we append a zero matrix of format $m\times n_0$ to the left of $M$
to obtain a matrix of format $m\times (n_0+n)$. Then, we append $n_0-1$ further copies of this matrix
on the bottom: the resulting matrix of format $n_0m\times (n_0 +n)$ will be denoted by $M_\mathrm{rep}$.
We have $\beta(M_\mathrm{rep}) = \beta(M)$ since for selecting invertible submatrices of $M_\mathrm{rep}$,
we have to avoid the first $n_0$ zero columns and we have to avoid repeated rows.
Moreover, $M_\mathrm{rep}$ has the same rank as $M$.
We also introduce for $i=1,\ldots,n_0$ the vectors
\begin{align*}
  v_i^T := \big( e_i - \frac{1}{n_0} 1_{n_0},0\big) \in \RR^{n_0}\times \RR^{n} .
\end{align*}
Note that $\|v_i\|_1 = \dfrac{2n_0 -2}{n_0}$.
Moreover, we define the vectors
\begin{align*}
  u_j^T := (0,\ldots,0,1_m,0,\ldots,0) \in (\RR^m)^{n_0} ,
\end{align*}
where the $1_m=(1,\ldots,1)\in\RR^m$ is in the $j$ block.
Note $\|u_j\|_\infty \le 1$.
Using this terminology, we can write the weight matrix of $\Id_{n_0}\ot\pi$ as
\begin{align*}
  \tilde{M} = M_\mathrm{rep} + u_1 v_1^T +\ldots + u_{n_0}v^T_{n_0} .
\end{align*}
By \cref{cor:beta-update}, we obtain
\begin{align*}
  \beta(\tilde{M}) \ \le\ \beta(M_\mathrm{rep}) \prod_{i=1}^{n_0} \left( 1 + (\rk(M_\mathrm{rep}) + i) \cdot \dfrac{2n_0 -2}{n_0} \right)
 \ \le\ \beta(M) (1 + 2 \cdot (n+n_0) )^{n_0}.
\end{align*}
\Cref{cor:alpha-update} implies
$\alpha(M(\pi_0)) \ge (Rn_0^{n_0})^{-1}$,
where we have used that $RM_\mathrm{rep}$ has integer entries.
\cref{prp:gapped} combined with \cref{pro:gap-AB} implies that
\begin{align*}
  \gamma(\tilde{M}) \ \ge\  R^{-1} n_0^{-n_0} \beta(M)^{-1} (1 + 2 \cdot (n+n_0) )^{-n_0} n^{-3/2}
\end{align*}
as claimed.
\end{proof}

As a corollary, we obtain polynomial lower bounds for the 3-tensor action from \cref{exa:tensor} with $n_1 = n_0, n_2=n_3 = n$ and $n_0$ fixed. It is already of interest for the space of
three slice tensors in $\CC^3\ot\CC^n\ot\CC^n$ with the action of $\SL(3)\times\SL(n)\times\SL(n)$. The assumption that $n_0$ grows very slowly is necessary - the work \cite{reichenbach-franks} uses a construction of \cite{kravtsov2007combinatorial} to show that the weight margin is~$O(2^{-n_0/3})$.

\begin{cor}\label{cor:tensor3rd}
Consider the action of $G = \SL(n_0) \times \SL(n) \times \SL(n)$ from~\cref{exa:tensor},
where $n_0$ is fixed.
The weight margin for this representation is lower bounded by an inverse polynomial in~$n$.
Therefore, there is a polynomial time algorithm for null cone membership for the tensor action~$\pi$.
\end{cor}

\begin{proof}
We apply \cref{th:tensor-action} to the simultaneous left-right action $\pi$ of $G=\SL(n_1)\times\SL(n_1)$ on $\Mat(n_1,n_1)$;
cf. \cref{exa:operator_scaling}. This corresponds to the quiver with two nodes, $n_0$ parallel arrows between them,
and dimension vector $(n_1,n_1)$.
From \cref{eq:beta-quiver} we get $\beta(M(\pi)) \le (1 + 2n_1)^2$.
Moreover, $n_1 M(\pi)$ has integer entries so that we can take $R=n_1$.
\cref{th:tensor-action} implies
$\gamma(\Id_{n_0}\ot\pi) = \Omega\big( (n_0^2n_1)^{-n_0} n_1^{-4.5}\big)$.
Note that for fixed $n_0$, this bound is inverse polynomial in~$n_1$.
\cref{cor:prod-concretenullcone-no-GT} 
gives a polynomial time algorithm for the null cone problem for $\pi$.
\end{proof}

\section{Bounds on degree and capacity for products of \texorpdfstring{$\GL(n)$}{GL(n)} and \texorpdfstring{$\SL(n)$}{SL(n)}}\label{sec:bds-degree-capacity}

The main goal of this section is to prove general a priori lower bounds on the capacity and $p$-capacity of vectors
for rational representations when the group $G$ is a product of $\GL(n)$ and $\SL(n)$'s;
cf.~\cref{prp:l2 lower bound} and \cref{cor:nonuniform_capacity_bound}.
(For the $p$-capacity we confine the discussion to representations of $\SL(n)$ for simplicity.)
These capacity bounds are a crucial ingredient of the proof of the explicit bounds on the running time
of our first and second order algorithm, which will be given in \cref{sec:explicit algos}.

The capacity lower bounds rely on general degree bounds, which we recall in~\cref{sec:deg-bd-invar},
refering to the monograph~\cite{derksen2015computational} for more details.
Then we go on to develop general bounds on the coefficient size for systems of generators in the case,
where $G$ is a product of $\GL(n)$ and $\SL(n)$'s, presenting material taken from~\cite{burgisser2017alternating}.
The basic idea here is to construct the Reynolds operator by Cayley's Omega process,
an explicit method which was well known in the 19th century~\cite{cayley:1846}.
This is carried out in \cref{sec:cayley} and \cref{sec:Reynolds}.
By combining the bounds on the degree and coefficent size of invariants we derive in~\cref{sec:capacity-lb}
general capacity lower bounds. 
We also develop general lower bounds on $p$-capacities, which generalize work in~\cite{burgisser2018efficient}
for the tensor action.




\subsection{Degree bounds on invariants}\label{sec:deg-bd-invar}
We fix a rational representation $\pi\colon G\to\GL(V)$ of a symmetric subgroup of $\GL(n)$
throughout the section. It induces an action $\tilde{\pi}$ of $G$ on the polynomial ring $\CC[V]$ defined by
$\tilde{\pi}(g)(f)(v) := f(\pi(g)^{T}v)$, for $f\in\CC[V]$, $g\in G$.\footnote{This action has the same invariants as
the dual action $\pi^*(g)(f)(v) := f(\pi(g)^{-1}v)$ which is often considered.}
We denote by $\CC[V]^G$ the subalgebra of $G$-invariants, which is graded:
$\CC[V]^G = \oplus_{d\in\NN}\CC[V]_d^G$.
It is known that there is a unique $G$-invariant surjective linear projection
$\CC[V]_d\to\CC[V]_d^{G}$, which is called the {\em Reynolds operator}.
Hilbert's finiteness theorem states that $\CC[V]^G$ is a finitely generated $\CC$-algebra.
Understanding $\CC[V]^G$ is closely related to analyzing the null cone $\cN(\pi)\subseteq V$ of~$\pi$,
which is the algebraic variety defined as the zero set of the non-constant homogeneous invariant polynomials, that is,
the zero set of $\oplus_{d >0}\CC[V]_d^G$;
compare~\cref{subsubsec:null_cone}.
In~\cite{derksen2015computational}, two degree bounding quantities are associated with the representation~$\pi$.

\begin{dfn}\label{def:beta-sigma}
$\gendeg(\pi)$ denotes the smallest integer~$d$ such that $\CC[V]^G$ is
generated by the homogeneous invariants of degree at most~$d$. 
Moreover, $\ncdeg(\pi)$ denotes the smallest integer~$d$ such that
the null cone $\cN(\pi)$ is the zero set of the nonconstant homogeneous invariants of degree at most~$d$.
\end{dfn}

It is clear that $\ncdeg(\pi)\le\gendeg(\pi)$.
A deep fact proved by Derksen is an inequality in the other direction, which we
will not use in our developments.
We state it at the end of the section for the sake of completeness (\cref{thm:ncdeg-gendeg}).

The following result is a special case of~\cite[Prop.~4.7.16]{derksen2015computational}.

\begin{prp}\label{prp:ncdeg-degree}
Let $\pi\colon \GL(n_1)\times\dots\times\GL(n_k)\to \GL(m)$ be a rational
representation such that its components $\pi_{ij}(g)$, as  functions of
$\det(g)^{-1}$ and the $n^2$ matrix entries of~$g$,
are polynomials of degree at most $d$.
Then we have
\begin{align*}
  \ncdeg(\pi) \le n_1\cdots n_k \, d^{\dim(G)}.
\end{align*}
The same bound holds for polynomial representations of
$\SL(n_1)\times\dots\times\SL(n_k)$.
\end{prp}

We note that the degree bound in \cref{prp:ncdeg-degree}
is exponential in the dimensions~$n_i$.
The examples given below show that often, better bounds can be obtained.

\begin{exa}\label{exa:simconj-deg-bound}
For the simultaneous conjugation action $\pi$ of~$\GL(n)$ on $\Mat(n)^k$ (\cref{exa:SC})
\cref{prp:ncdeg-degree} gives $\ncdeg(\pi) \le n^{n^2}$, however,
Razmyslov~\cite{razmyslov:74} showed that $\gendeg(\pi) \le n^{2}$.\footnote{Note that
the $\GL(n)$-invariants are the same as the $\SL(n)$-invariants here.}
More generally, for a representation $\pi$ of the group $\GL(n_1)\times\dots\times\GL(n_k)$
on a quiver on $k$ vertices and with dimension vector $(n_1,\ldots,n_k)$ (\cref{exa:quivers}),
we have
$\gendeg(\pi) \le (n_1+\ldots + n_k)^{2}$.
This follows easily from the stated bound for the simultaneous conjugation
by a reduction
contained in~\cite{lebruyn-procesi}.
\end{exa}

\begin{exa}\label{exa:LR-deg-bound}
For the left-right action $\pi$, Derksen and Makam showed in an important work~\cite{derksen2015}
that $\gendeg(\pi) \le n(n -1)$ if $n\ge 2$.
This implies Razmyslov's bound via a simple reduction from representations of the
$k$-loop quiver to the ones of the generalized Kronecker quiver with $k+1$ arrows;
see~\cite{DerksenConj}.
For a representation $\pi$ of the group $\SL(n_1)\times\dots\times\SL(n_k)$
on an {\em acyclic} quiver on $k$ vertices and dimension vector $(n_1,\ldots,n_k)$ (\cref{exa:quivers}),
it is shown in~\cite{DerksenMakam:Degree-Semiinvar}
that, {\em for fixed $k$}, $\gendeg(\pi)$ is upper bounded
by a polynomial in  $n_1,\ldots,n_k,m$, where $m$ denotes the number of arrows.
The proof works by a reduction
going back to Derksen and Weyman~\cite{derksen2000semi}.
The exponential dependence on $k$ cannot be avoided in general,
see~\cite[Prop.~1.5]{DerksenMakam:Degree-Semiinvar}.
\end{exa}

To complete the picture, we finally state Derksen's result relating the quantities
$\ncdeg$ and $\gendeg$.

\begin{thm}[Derksen~\cite{derksen2001polynomial}]\label{thm:ncdeg-gendeg}
We have $\gendeg(\pi) \le \max\big\{ \frac38 \dim(G) \ncdeg(\pi)^2, 2\big\}$.
\end{thm}

This result says that, as far as upper bounds on $\gendeg(\pi)$ are concerned,
it is enough (up to polynomial factors) to focus on~$\ncdeg(\pi)$,
which relates to a geometric object (the null cone).

\begin{rem}
In order to appreciate \cref{thm:ncdeg-gendeg}, it may be worthwile to comment on its highly nontrivial proof.
It is heavily based on the Hochster-Robert's theorem~\cite{hochster-roberts:74},
stating that $\CC[V]^G$ is a Cohen-Macaulay ring.
Another key ingredient is the finding that the Hilbert series
$\sum_d \dim \CC[V]_d^G t^d$ is a rational function of degree at most zero, i.e.,
the degree of its numerator is at most the degree of its denominator.
This was first proved by Kempf~\cite{kempf:79} and improved by Knop~\cite{knop:89}.
The monograph~\cite{derksen2015computational} omits the proof of this degree bound on the Hilbert series.
\end{rem}


\subsection{Cayley's Omega process and Reynolds operator}\label{sec:cayley}
Conceptually, the Reynolds operator is most easily understood for compact groups by averaging with respect to the Haar measure.
This way, via maximally compact subgroups, one can construct the Reynolds operator for any symmetric subgroup of $\GL(n)$, i.e., for any complex reductive group; see \cite[\S3.1]{wallach2017geometric}.
However, we will rely on the construction of the Reynolds operator for~$\SL(n)$ by Cayley's Omega process, an explicit method which was well-known in the 19th century~\cite{cayley:1846}.
To the best of our knowledge, a comparably explicit procedure is not known for Lie groups other than~$\SL(n)$ and~$\GL(n)$.
Cayley's Omega process will be applied to prove coefficient bounds for invariants in~\cref{sec:Reynolds}.

Here is a high-level description of the procedure.
Consider the action $\tilde{\pi}$ on $W=\CC[V]_D$ induced by a homogeneous representation $\pi\colon \GL(n)\to \GL(V)$ of degree~$d$.
Our goal is to compute the homogeneous degree $r$ invariants in $W^G$ of degree $rn$,
where $G=\SL(n)$. We restrict to degree $rn$ because, as one can easily see, $r$ must be divisible by $n$ if nonzero invariants exist.
After fixing a basis of $W$, the action $\pi$ is described by a matrix of polynomials in the $\SL(n)$ variables, $Z_{i,j}$,
which are homogeneous of degree $Dd$.
We seek a basis for the vectors (of coefficients) that are invariant under the action of this symbolic matrix.
The remarkable Omega process converts this polynomial matrix into a constant matrix, whose columns provide such a basis of $W^G$.
It does so by iterated application ($r$ many times)
of a differential operator (called $\Omega$) to each entry of this matrix, until all degrees are reduced to a constant.
In fact, this construction works for any $W$ on which $\GL(n)$ acts (we don't need to assume the situation of the ``plethysm'' $W=\CC[V]_D$).

For the construction of the Omega differential operator, we first develop some facts on linear differential operators on polynomial rings.
A monomial $Z_{i_1}\cdots Z_{i_d}$
in the variables $Z_1,\ldots,Z_m$ gives rise to a differential operator
$\partial_{Z_{i_1}}\cdots\partial_{Z_{i_d}}$.
By linear extension, we assign to a polynomial $p\in\CC[Z_1,\ldots,Z_m]$
a differential operator $\partial_p$, which acts as a linear endomorphism
$\partial_p\colon \CC[Z_1,\ldots,Z_m] \to \CC[Z_1,\ldots,Z_m], q\mapsto \partial_p q$.
We denote by $\partial_{Z^\alpha}$ the operator corresponding to the monomial
$Z^\alpha = Z_1^{\alpha_1}\cdots Z_m^{\alpha_m}$.
(Note that $\partial_1$ acts as identity.)
The following properties are immediately checked.

\begin{lem}\label{le:diff-op}
\begin{enumerate}
\item\label{it:homog} If $p$ is homogeneous, then $\partial_p$ preserves homogeneity and reduces the degree by~$\deg p$.

\item\label{it:lin} $\partial_p$ depends linearly on $p$.

\item\label{it:mult} Multiplication of polynomials corresponds to the composition of operators:
$\partial_{p_1p_2}=\partial_{p_1}\partial_{p_2}$.

\item\label{it:monom} We have
$\partial_{Z^\alpha} Z^\beta =  Z^{\beta-\alpha} \prod_i \beta_i(\beta_i-1) \cdots (\beta_i-\alpha_i +1) $
if $\beta-\alpha$ has nonnegative coordinates and $\partial_{Z^\alpha} Z^\beta=0$ otherwise. If $|\alpha| = |\beta|$, $\partial_{Z^{\alpha}} Z^{\beta}$ is thus $\delta_{\alpha, \beta} \prod_i \alpha_i!$.

\item\label{it:int} If $p$ and $q$ have integral coefficients then $\partial_p q$ has integral coefficients as well.

\item\label{it:conj} Let $p$ be nonzero homogeneous and $\bar{p}$ be obtained from $p$ by replacing its coefficients by their complex conjugates.
Then we have $\partial_p \bar{p} > 0$.
\end{enumerate}
\end{lem}

The lemma implies that $(p,q)\mapsto \partial_p q$ is a nondegenerate bilinear map
on the space of homogeneous forms of degree~$d$ and
$\{(\alpha_1!\cdots\alpha_m!)^{-1}\partial_{Z^\alpha}: |\alpha|=d\}$ is the basis dual to
the monomial basis.

We now assume we have variables $Z_{i,j}$ for $1\le i,j\le n$ and
denote by $\Omega$ the differential operator $\partial_{\det}$ given by the
determinant polynomial. In other words, $\Omega$ sends a polynomial $q$ in $Z$ to
the polynomial
\begin{equation}\label{def:omega}
\Omega(q)  := \sum_{\sigma \in S_n} \sgn(\sigma) \frac{\partial^n q}{\partial Z_{1,\sigma(1)} \cdots \partial Z_{n,\sigma(n)}} .
\end{equation}
Applying $\Omega$ reduces the degree by~$n$, so
if $q$ is homogeneous polynomial of degree $rn$,
then the polynomial $\Omega^r(q)$ obtained by the $r$-fold application of $\Omega$ is a constant.
We also note that $\Omega^r$ maps polynomials with integral coefficients to polynomials with integral coefficients.


In order to state the fundamental invariance property of the Omega operator,
we introduce for a matrix $A\in\Mat(n)$ the left and right multiplication maps
\begin{align*}
 L_A\colon \Mat(n) \to \Mat(n),\, Z\mapsto AZ,\quad
 R_A\colon \Mat(n) \to \Mat(n),\, Z\mapsto ZA .
\end{align*}
We shall view $q$ and $\Omega(q)$ as polynomial functions $\Mat(n) \to\CC$.

\begin{prp}\label{pro:Omega-invar}
\begin{enumerate}
\item\label{it:det} For any polynomial $q$ in $Z$ and $A \in \Mat(n)$, we have
\begin{align*}
\Omega (q \circ L_A) = \det(A) \, (\Omega(q) \circ L_A) ,\quad
\Omega (q \circ R_A) = \det(A) \, (\Omega(q) \circ R_A) .
\end{align*}

\item\label{it:iter-det} If $q$ is homogeneous of degree $rn$, then $\Omega^r(q)\in\CC$.
Moreover,
\begin{align*}
 \Omega^r(q \circ L_A) = \Omega^r(q \circ R_A) = \det(A)^r\, \Omega^r(q)
\end{align*}
for any $A \in \Mat(n)$. In particular,
$\Omega^r(q \circ L_A) = \Omega^r(q \circ R_A) = \Omega^r(q)$ for $A\in\SL(n)$.

\item\label{it:const} $c_{r,n}:=\Omega^{r}\big(\det(Z)^{r} \big)$ is a positive constant.

\item\label{it:coeff} If $q$ has degree $d$ and integer coefficients whose absolute values are bounded by~$R$,
then the coefficients of $\Omega^r(q)$ are integers with absolute value bounded by $R \,(n!\, d^n)^r$.

\end{enumerate}
\end{prp}

\begin{proof}
\cref{it:det}: Define $Y: \CC^n \to \CC^n$ by $Y(Z) =  AZ$. By the chain rule, we have
$\frac{\partial q(AZ)}{\partial Z_{i,j}}
 = \sum_{k=1}^n A_{k,i} \frac{\partial q}{\partial Y_{k,j}}(AZ)$.
Iterating $n$ times, we get
\begin{align*}
  \frac{\partial^n q(AZ)}{\partial Z_{1,j_1} \cdots \partial Z_{n,j_n}}
= \sum_{k\colon [n] \to [n]} \prod_{t=1}^n A_{k(t),t} \; \frac{\partial^n q}{\partial Y_{k(1), j_1} \cdots \partial Y_{k(n), j_n}} (AZ).
\end{align*}
Using this, we obtain
\begin{align*}
  \Omega(q\circ L_A)
&= \sum_{\sigma \in S_n} \sgn(\sigma) \frac{\partial^n q(AZ)}{\partial Z_{1,\sigma(1)} \cdots \partial Z_{n,\sigma(n)}} \\
&= \sum_{\sigma \in S_n} \sgn(\sigma)  \sum_{\kappa \colon [n]\to[n]} \prod_{t=1}^n A_{\kappa(t),t} \;
  \frac{\partial^n q}{\partial Y_{\kappa(1),\sigma(1)} \cdots \partial Y_{\kappa(n),\sigma(n)}} (AZ) \\
&= \sum_{\sigma \in S_n} \sgn(\sigma) \sum_{\kappa \colon [n]\to[n]} \prod_{t=1}^n A_{\kappa(\sigma(t)),t} \;
   \frac{\partial^n q}{\partial Y_{\kappa(1),1} \cdots \partial Y_{\kappa(n),n}} (AZ) \\
\end{align*}
where we substituted $\kappa$ by $\kappa\circ\sigma$ for the last equality.
Interchanging summations, we get
\begin{align*}
\Omega(q \circ L_A)= \sum_{\kappa \colon [n]\to[n]}
 \left( \sum_{\sigma \in S_n} \sgn(\sigma) \prod_{t=1}^n A_{\kappa(\sigma(t)),t} \right)
  \frac{\partial^n q}{\partial Y_{\kappa(1),1} \cdots \partial Y_{\kappa(n),n}}  (AZ).
\end{align*}
For the term in the left parenthesis we verify that
\begin{align*}
  \sum_{\sigma \in S_n} \sgn(\sigma) \prod_{t=1}^n A_{\kappa(\sigma(t)),t}
= \begin{cases}
  \sgn(\kappa) \det(A), &\text{ if $\kappa$ is a permutation}, \\
  0, &\text{ otherwise}.
\end{cases}
\end{align*}
Therefore,
\begin{align*}
  \Omega(q \circ L_A)
= \det(A) \cdot \sum_{\kappa\in S_n} \sgn(\kappa) \frac{\partial^n q}{\partial Y_{\kappa(1),1} \cdots \partial Y_{\kappa(n),n}} (AZ)
= \det(A) \cdot (\Omega(q) \circ L_A) ,
\end{align*}
which shows the assertion for the left multiplication.
The argument for the right multiplication is  analogous.

\cref{it:iter-det}: Part one implies that
$\Omega^r(q \circ L_A) = \det(A)^r\, (\Omega^r(q) \circ L_A)$.
Therefore, since the polynomial on the right-hand side is a constant,
$\Omega^r(q \circ L_A) = \det(A)^r\, \Omega^r(q)$.
The argument for the right multiplication is  analogous.

\cref{it:const}: This follows from~\cref{it:mult,it:conj} of \cref{le:diff-op}.

\cref{it:coeff}: It suffices to prove the assertion for $r=1$.
We write $q = \sum_{\alpha\in\NN^{n\times n}} q_\alpha Z^\alpha$ with $q_\alpha\in\ZZ$
such that $| q_\alpha| \le R$;
$Z^\alpha$ denotes the monomial $\prod_{i,j} Z_{i,j}^{\alpha_{i,j}}$.
By definition of the $\Omega$-operator, we have
\begin{align*}
 \Omega(q) = \sum_\alpha q_\alpha \sum_{\sigma\in S_n}  \sgn(\sigma)
   \left( \prod_{i=1}^n \alpha_{i,\sigma(i)} \right) Z^{\alpha -M_\sigma} ,
\end{align*}
where $M_\sigma$ denotes the permutation matrix of $\sigma\in S_n$;
we use the convention that $Z^{\beta} := 0$ if $\beta\not\in\NN^{n\times n}$.
Therefore,
$\Omega(q) = \sum_\beta p_\beta Z^\beta$,
with the coefficient $p_\beta$ described as
\begin{align*}
 p_\beta = \sum_{\sigma \in S_n}  \sgn(\sigma) q_\alpha \prod_{i=1}^n \alpha_{i,\sigma(i)} ,
\end{align*}
where we have abbreviated $\alpha = \beta + M_\sigma$ in the sum.
Since $q_\alpha=0$ if $\sum_{i,j}\alpha_{i,j}  >d$, we obtain
$|p_\beta| \le n!\,R\, d^n$ as claimed.
\end{proof}

Let $\pi\colon \GL(n)\to\GL(W)$ be a homogeneous, polynomial representation of degree~$D$.
Our goal is to decribe a system of generators for the vector space $W^{\SL(n)}$ of invariants.
It is easy to see that $W^{\SL(n)}=0$ if $n$ does not divide $D$: so let us assume that $D=rn$.

We are now in a position to construct the Reynolds operator for $\pi$. $\Omega^r$ acts on the space of $\CC$-valued polynomials on $\Mat_n(\CC)$, and sends degree $D$ homogeneous polynomials to constants (\cref{it:const} of \cref{pro:Omega-invar}). Thus if we allow $\Omega^r$ to act on each entry in the matrix-valued polynomial $\pi(Z)$ (having chosen some basis of $W$ for concreteness), $\Omega^r$ will send $\pi(Z)$ to a fixed matrix $\cR:W \to W$ which we define to be the Reynolds operator.

 In order to see that this definition is basis-independent, we may instead extend $\Omega^r$ to act on $W$-valued polynomial functions on $\Mat_n(\CC)$, and then define the image of $f$ under $\cR$ to be the image of the $W$-valued polynomial function $Z \mapsto \pi(Z)f$ under $\Omega^r$. The extension can be done by tensoring $\Omega^r$ with the identity function on $W$, because the space of $W$-valued polynomials on $\Mat_n(\CC)$ is simply the tensor product of the space of $\CC$-valued polynomials on $\Mat_n(\CC)$ with $W$. In a basis, this amounts to applying $\Omega^r$ componentwise, but this construction is inherently independent of a choice of basis for $W$.

We now work out how to express $\cR$ once we have chosen a basis $(b_\alpha)$ of $W$. Write
\begin{equation}\label{eq:def-matrix-entries}
 \pi(g)(b_\alpha) = \sum_\beta q_{\beta,\alpha}(g) b_\beta ;
\end{equation}
by our assumption, the matrix coefficients $q_{\beta,\alpha}(g)$ are homogeneous polynomials of degree~$D$
in the entries of~$g$.
We expand a given $f \in W$ in the basis,
$f=\sum_\alpha f_\alpha b_\alpha$ with $f_\alpha\in\CC$,
and note that, for $g\in\GL(n)$,
\begin{equation}\label{eq:f-tilde}
 \pi(g)(f) = \sum_\alpha f_\alpha\, \pi(g)(b_\alpha) =
  \sum_\alpha f_\alpha \sum_\beta  q_{\beta,\alpha}(g) b_\beta
   =  \sum_\beta \tilde{f}_\beta(g) b_\beta ,
\end{equation}
where
$\tilde{f}_\beta(g) := \sum_\alpha  q_{\beta,\alpha}(g) f_\alpha$
is a homogeneous polynomial of degree $D$.
Note that $\tilde{f}_\beta$ depends linearly on $f$.
Thus we define the linear map
$\cR\colon  W \to W$ by
\begin{equation}\label{eq:def-R}
  \cR(f) := \sum_\beta \Omega^{r} (\tilde{f}_\beta) b_\beta .
\end{equation}
It turns out that $\cR$ is equivariant with respect to the action of $\SL(n)$: more specifically,
we have the following result.

\begin{prp}\label{pro:Reynolds}
We have for all $f \in W$ and all $g\in\GL(n)$,
\begin{align*}
 \pi(g)(\cR(f)) = \det(g)^{r} \cR(f) = \cR(\pi(g)(f))   .
\end{align*}
\end{prp}

\begin{proof}
We fix $f\in W$ and $g\in \GL(n)$ and consider the following $W$-valued
function of $h$:
\begin{align*}
   F(h)
:= \pi(gh)(f)
 = \sum_\beta \tilde{f}_\beta(gh) b_\beta
 = \sum_\beta (\tilde{f}_\beta \circ L_g) (h) b_\beta  ,
\end{align*}
where we used~\cref{eq:f-tilde} and $L_g$ denotes the left-multiplication with $g$.
On the other hand,
\begin{align*}
  F(h)  =  \pi(g)(\pi(h)(f))  = \sum_\beta \tilde{f}_\beta(h)\pi(g)(b_\beta) .
\end{align*}
The components of $F(h)$ are homogeneous polynomials of degree~$D=rn$ in $h$.
Applying the differential operator $\Omega^{r}$ componentwise to $F(h)$ (or equivalently, the extension of $\Omega^r$ we defined earlier by tensoring $\Omega^r$ with the identity function),
we obtain with the definition~\cref{eq:def-R} of $\cR$ that,
\begin{align*}
  \sum_\beta \Omega^{r}(\tilde{f}_\beta) \pi(g)(b_\beta) = \pi(g)(\cR(f)) .
\end{align*}
Applying $\Omega^{r}$ componentwise to the first
equation for $F(h)$ and using \cref{it:iter-det} of \cref{pro:Omega-invar},
we get
\begin{align*}
  \sum_\beta \Omega^{r}(\tilde{f}_\beta \circ L_g) b_\beta
 = \det(g)^{r} \sum_\beta \Omega^{r}(\tilde{f}_\beta) b_\beta = \det(g)^{r} \cR(f).
\end{align*}
We have shown that
$\pi(g)(\cR(f)) = \det(g)^{r} \cR(f)$.

For the second equality, denoting the right-multiplication with $g$ by $R_g$, we write
\begin{align*}
  \pi(h)(\pi(g)(f)) = \pi(hg)(f) = \sum_\beta \tilde{f}_\beta(hg) b_\beta
  = \sum_\beta (\tilde{f}_\beta \circ R_g)(h) b_\beta  .
\end{align*}
Applying $\Omega^{r}$ with respect to~$h$ and using \cref{it:iter-det} of \cref{pro:Omega-invar} gives
\begin{align*}
  \cR(\pi(g)(f)) = \sum_\beta \Omega^{r}(\tilde{f}_\beta \circ R_g) b_\beta
 = \det(g)^{r} \sum_\beta \Omega^{r}(\tilde{f}_\beta) b_\beta = \det(g)^{r} \cR(f) ,
\end{align*}
which completes the proof.
\end{proof}

The next proposition shows that $c_{r,n}^{-1} \cR$ is the Reynolds operator for the
restriction of $\pi$ to $\SL(n)$, with the positive normalizing
constant $c_{r,n} :=\Omega^{r}\big(\det(Z)^{r} \big) $,
cf.~\cref{it:const} of \cref{pro:Omega-invar}.

\begin{cor}\label{cor:Reynolds}
Suppose $\pi:\GL(n)\to W$ is a homogeneous polynomial representation of degree $D = rn$. The Reynolds operator is a $G$-invariant surjective linear map $\cR\colon W\to  W^G$,
where $G:=\SL(n)$.
Moreover, we have $\cR(f) = c_{r,n} \, f$ for all $f\in W^G$.
\end{cor}
\begin{proof}
\cref{pro:Reynolds} implies that
$\cR(f) \in W^{G}$ and $\cR(\pi(g)(f)) = \cR(f)$,
for all $f \in W$ and $g\in G$.
Moreover, for a homogeneous $\SL(n)$-invariant $f$ of degree~$D$ we have
$\pi(g)(f) = \det(g)^{r}\, f$ for $g\in\GL(n)$. Thus $\tilde{f}_\beta (g) =(\det g)^r f_\beta$, so that $\Omega^r(\tilde{f}_\beta) =\Omega^r (\det^r) f_\beta$. 
From \cref{eq:def-R} we easily conclude that
$\cR(f) = \Omega^{r}(\det^{r}) \, f$.
This completes the proof.
\end{proof}

%
%
%
%

\subsection{Coefficient bounds on invariants}\label{sec:Reynolds}
We will use now the Omega process to prove (\cref{th:CB-IR}) that if $\pi$ is an $m$-dimensional homogeneous representation
of $\GL(n)$ given by polynomials of degree $d$ and integer coefficients at most $R$, then the space of its invariant polynomials
in degree $D$, $\CC[Y_1,\ldots,Y_m]^{\SL(n)}_D$, is spanned by polynomials with integer coefficients bounded by $(mnRDd)^{Dd}$.  
A similar bound is given for products of SLn's in \cref{th:CB-IR-prod}. 
These upper bounds are  crucial for the capacity lower bounds in the next subsection. 
Specifically, our capacity lower bounds require a bound on the coefficients that grows only exponentially in the degree $D$. 

It is convenient to make the following general definition.

\begin{dfn}[Coefficient size]\label{dfn:coef-size}
Let $\pi\colon\GL(n)\to\GL(m)$ be a homogeneous representation.
The \emph{coefficient size} $R(\pi)$  of~$\pi$ is defined as the least positive
integer~$R$  such that there exists a positive integer $N \le R$ such that
all the scaled matrix entries $N\pi_{i,j}(g)$ are polynomials with integer coefficients
bounded in absolute value by $R$.
We set $R(\pi):=\infty$ if not all coefficients are rational.
\end{dfn}

The following simple lemma is useful for bounding coefficient sizes. We state
it in full generality since it will be of use several times.

\begin{lem}\label{lem:coeff-bd-mult-poly}
Let $p_1,\dots,p_t\in\CC[Y_1,\ldots,Y_m]$ be polynomials such that
each $p_j$ has degree~$d_j$ and coefficients bounded in absolute value by~$R_j$.
Then the product $p := p_1 \cdots p_t$ is a polynomial of degree at most
$d := d_1 + \dots + d_t$ and has coefficients bounded in absolute value by
\begin{align*}
  R_1\cdots R_t \min\{ (m+1)^d, t^d, (1+d/m)^{m(t-1)} \} .
\end{align*}
If the $p_j$ are homogenous, then also the upper bound $R_1 \cdots R_t m^d$ holds.
\end{lem}

\begin{proof}
The coefficient of the monomial $Y^\beta$ in $p$, where $\beta\in\NN^m$,
is bounded in absolute value by
$R_1\cdots R_t\, N$, where $N$ denotes the number of $t$-tuples
$(\alpha^{(1)},\ldots,\alpha^{(t)})$ with $\alpha^{(i)}\in\NN^m$ satisfying the following conditions:
$|\alpha^{(i)}| \le d_i$ and $\alpha^{(1)} +\ldots+\alpha^{(t)} = \beta$.
If the $p_i$ are homogeneous, we have the stronger requirement $|\alpha^{(i)}| = d_i$.
We now bound $N$ in different ways.

Firstly, when omitting the last condition, we obtain
\begin{align*}
  N \ \le\ \binom{m+d_1}{d_1} \cdots \binom{m+d_t}{d_t} \ \le\ (m+1)^d,
\end{align*}
where we have used that $\binom{m+d_i}{d_i} \le (m+1)^{d_i}$.
If the $p_i$ are homogeneous, we have can replace $m$ by $m-1$ in the bound.

Secondly, after omitting the conditions $|\alpha^{(i)}| \le d_i$, we rewrite the
remaining conditions componentwise as
$\alpha^{(1)}_j +\ldots+\alpha^{(t)}_j = \beta_j$
for $j\in [m]$. The number of possibilities is now given by
\begin{align*}
 \binom{t-1 + \beta_1}{\beta_1} \cdots \binom{t-1 + \beta_m}{\beta_m} \ \le\ t^{\beta_1+\ldots\beta_m}
 \ \le\ t^d .
\end{align*}
Alternatively, we can bound
$\binom{t-1 + \beta_j}{\beta_j} = \binom{t-1 + \beta_j}{t-1} \le (\beta_j +1)^{t-1}$.
Using the AM-GM inequality leads to
$\prod_{j=1}^m (\beta_j+1) \leq ( \sum_{j=1}^m (\beta_j + 1)/m )^m \le  (1 + d/m)^m$
and third bound follows.
\end{proof}

For later use, as an immediate consequence of~\cref{lem:coeff-bd-mult-poly}, we state a result on controlling the growth of coefficient sizes under tensor products.

\begin{lem}[Tensor products]\label{lem:rep-bounds}
Let $\rho_1\colon\GL(n)\to\GL(m_1)$, $\rho_2\colon\GL(n)\to\GL(m_2)$, and $\rho\colon\GL(n)\to\GL(m)$ be homogeneous representations.
Then:
\begin{enumerate}
\item\label{it:tensor-reps-bounds}
The tensor product representation $\rho_1 \ot \rho_2 \colon \GL(n)\to\GL(m_1m_2)$
is homogeneous of degree~$d:=d(\rho_1) + d(\rho_2)$
and has coefficient size
\begin{align*}
R( \rho_1 \ot \rho_2) \leq R(\rho_1)R(\rho_2)
	\min\Big\{  n^{2d},\; 2^{d}, \;(1 + d/n^2)^{n^2} \Big\}.
\end{align*}
\item\label{it:sym-k-rep-bounds}
For $\ell\in\NN$, the tensor power representation $\rho^{\ot\ell} = \rho \ot \cdots \ot \rho \colon \GL(n)\to\GL(m^\ell)$
is homogenous of degree~$\ell d(\rho)$ and has the
coefficient size $R(\rho^{\ot\ell}) \leq R(\rho)^\ell \; \ell^{\ell d(\rho)}$.
\end{enumerate}
\end{lem}

We return to the Reynolds operator.
Let $\pi\colon \GL(n)\to\GL(m)$ be a homogeneous representation of degree~$d$.
We consider the induced representation $\tilde{\pi}$ of $\GL(n)$ on $\CC[V]$, where $V=\CC^m$, defined by
\begin{equation}\label{eq:def-ind-op}
 \big(\tilde{\pi}(g)f\big)(v) := f(\pi(g)^{T}v), \quad \mbox{for $g\in \GL(n)$, $f\in\CC[V]$, $v\in V$}
\end{equation}
and focus on the restriction of this action to
the homogeneous part $W:=\CC[V]_{D}$ of degree $D$.
We shall apply the insights of \cref{sec:cayley} to study the
Reynolds operator for this action on $W$.
Let $b_\alpha := Y^\alpha := Y_1^{\alpha_1}\cdots  Y_m^{\alpha_m}$ denote the basis
of monomials of degree~$D =|\alpha| := \alpha_1+\ldots+\alpha_m$.
We define the functions $q_{\beta,\alpha}(g)$ for $\tilde{\pi}$ according to \cref{eq:def-matrix-entries}, that is,
\begin{equation}\label{eq:def-matrix-entries-p}
 \tilde{\pi}(g)(Y^\alpha) = \sum_\beta q_{\beta,\alpha}(g) Y^\beta .
\end{equation}

\begin{prp}\label{le:CB-IR}
Let $N \le R$ be a positive integer such that all $N\pi_{i,j}(g)$ are integral polynomials with coefficients bounded by $R$ in absolute value. That is, $R(\pi) \leq N \leq \infty$ for $R(\pi)$ as in \cref{dfn:coef-size}.
\begin{enumerate}
\item The $N^D q_{\beta,\alpha}(g)$ are homogeneous polynomials of degree~$Dd$ with
integral coefficients that are bounded in absolute value by $(Rm)^D n^{2Dd}$.

\item Assume $Dd$ is divisible by $n$ and let $\cR$ denote the Reynolds operator of the representation $\tilde{\pi}$ on $W$.
Then $N^D\cR(Y^\alpha)$ has integral coefficients bounded in absolute value by $(R m)^D (n^3Dd)^{Dd}$.

\item The vector space $\CC[Y_1,\ldots,Y_m]_D^{\SL(n)}$ is spanned by invariants with integer coefficients having
coefficients with absolute value at most $(R m)^D (n^3Dd)^{Dd}$.

\end{enumerate}
\end{prp}
\begin{proof}
\begin{enumerate}
\item The matrix entries $q_{\beta,\alpha}(g)$ of $\tilde{\pi}$ can be expressed in terms of the
matrix entries $\pi_{i,j}(g)$ as follows:
\begin{equation*}
 \prod_{i=1}^m \Big(\sum_{j=1}^m N\pi_{j,i}(g) Y_j \Big)^{\alpha_i} = N^D\tilde{\pi}(g)(Y^\alpha) =
 \sum_\beta N^D q_{\beta,\alpha}(g) Y^\beta ;
\end{equation*}
the first equality is by the definition~\cref{eq:def-ind-op}
and the second by \cref{eq:def-matrix-entries-p}.
The left product can be expanded as a sum of at most~$m^{D}$ products of the form $\prod_{r=1}^{D} N \pi_{j_r, i_r}(g) Y_{j_r}$,
where $(j_1,\ldots,j_D) \in [m]^D$ and
$(i_1,\ldots,i_D)=(1,\ldots,1,\ldots,m,\ldots,m)$ with $i$ occurring $\alpha_i$ many times.
The matrix entry $q_{\beta,\alpha}(g)$ is obtained by summing the products  $\prod_{r=1}^{D} N \pi_{j_r, i_r}(g)$
over all tuples $(j_1,\ldots,j_D) \in [m]^D$ with the property that
$Y_{j_1}\cdots Y_{j_D}=Y^\beta$; thus
the number of contributing products equals $\binom D \beta \le m^D$.
The first bound of~\cref{lem:coeff-bd-mult-poly},
applied to the homogeneous polynomials $N \pi_{j_r, i_r}(g)$ of degree~$d$ in $n^2$ variables
and coefficients bounded in absolute value by~$R$,
implies that the coefficients of each product have absolute value at most $R^D (n^2)^{Dd}$.
This shows the first assertion.

\item Put $Dd=rn$. 
From \cref{eq:f-tilde} we obtain that
\begin{align*}
 N^D\cR(Y^\alpha) = \sum_\beta \Omega^{r}( N^D q_{\beta,\alpha}) Y^\beta .
\end{align*}
Part one tells us that $N^D q_{\beta,\alpha}(g)$ has integer coefficients
bounded in absolute value by $(Rm)^D n^{2Dd}$.
Using \cref{pro:Omega-invar}(4), we see that $\Omega^{r}(N^D q_{\beta,\alpha})$ is an integer
satisfying
\begin{align*}
 \lvert \Omega^{r}(N^D q_{\beta,\alpha}) \rvert \ \le\ (Rm)^D n^{2Dd} \big(n! (Dd)^n \big)^{r}
  \ \le\ (Rm)^D \big(n^3 Dd\big)^{Dd} .
\end{align*}

\item The assertion is trivial if $n$ does not divide $Dd$. Otherwise the assertion follows
from part two since, by \cref{cor:Reynolds}, the space of invariants is spanned by the $N^D\cR(Y^\alpha)$. \qedhere
\end{enumerate}
\end{proof}


\cref{le:CB-IR} only applies to homogeneous representations.
We extend it now to direct sum of homogeneous representations,
so that it applies to any polynomial representations.
This is the main result of this subsection.


\begin{thm}[Coefficient bounds for invariants]\label{th:CB-IR}
Let $\pi_r\colon\GL(n)\to\GL(m_r)$ be homogenous polynomial representations
of degree $d_r$ on $V_r := \CC^{m_r}$, for $r=1,\ldots,s$.
We put $d:=\max d_r$ and $m:=m_1+\ldots+m_s$.
We further asssume there is 
a positive integer $N \le R$ such that, for all $r$,
the matrix entries of $N\pi_{r}(g)$ are integral polynomials
with absolute value bounded by $R$.
We consider the resulting action $\pi$ on the direct sum
$V=V_1\oplus\ldots\oplus V_s=\CC^m$, which, for a chosen degree~$D$,
induces the action $\tilde{\pi}$ on $\CC[V]_D$.
Let $\cR_D\colon\CC[V]_D \to \CC[V]_D^{\SL(n)}$ denote its Reynolds operator.
Then we have:
\begin{enumerate}

\item $N^D\cR(Y^\alpha)$ has integral coefficients bounded in absolute value by $(R m)^D (n^3Dd)^{Dd}$.

\item The space $\CC[V]_D^{\SL(n)}$ is spanned by invariants with integer coefficients bounded in
absolute value by $(Rm)^{D} (n^3 D d)^{Dd}$.

\end{enumerate}
\end{thm}

\begin{proof}
We are going to show how to build up the Reynolds operator $\cR_D$ from the
Reynolds operators $\cR_i\colon\CC[V_i]_{D_i} \to \CC[V_i]_{D_i}^{\SL(n)}$
of the homogeneous parts. The latter were constructed in~\cref{le:CB-IR}.

We have $\CC[V] = \CC[V_1]\ot\ldots\ot\CC[V_s]$ and hence
\begin{align*}
  \CC[V]_D = \bigoplus_{D_1+\ldots + D_s=D} \CC[V_1]_{D_1}\ot\ldots\ot\CC[V_s]_{D_s} .
\end{align*}
Taking invariants gives
\begin{align*}
  \CC[V]_D^{\SL(n)} = \bigoplus_{D_1+\ldots + D_s=D} \big(\CC[V_1]_{D_1}\ot\ldots\ot\CC[V_s]_{D_s}\big)^{\SL(n)}   .
\end{align*}
Moreover,
\begin{align*}
  \big(\CC[V_1]_{D_1}\ot\ldots\ot\CC[V_s]_{D_s}\big)^{\SL(n)} = \CC[V_1]_{D_1}^{\SL(n)} \ot\ldots\ot\CC[V_s]_{D_s}^{\SL(n)}
\end{align*}
Thus we see that $\cR_D$ is obtained as a direct sum of the tensor products $\cR_1\ot\ldots\ot\cR_s$, over all $D_1,\ldots,D_s$ such that $D_1+\ldots+D_s=D$.

According to~\cref{le:CB-IR}, if $f_i$ is a monomial in $\CC[V_i]_{D_i}^{\SL(n)}$,
then $N^{D_i}f_i$ is mapped by $\cR_i$ to a polynomial with integer coefficients bounded in absolute value by
$(R_i m_i)^{D_i} (n^3D_id_i)^{D_id_i}$.
Noting that the $f_i$ are monomials in disjoint sets of variables, we see for the monomial $f:=f_1\cdots f_s$ that
$N^{D}f$ is mapped by $\cR_1\ot\ldots\ot\cR_s$ to a polynomial with integer coefficients bounded in absolute value by
\begin{align*}
  L:= \prod_{i=1}^s (R_i m_i)^{D_i} (n^3D_id_i)^{D_id_i}  \ \le\ (Rm)^{D} (n^3 D d)^{Dd} .
\end{align*}
This shows the first assertion. The second follows from the first part with~\cref{cor:Reynolds}.
\end{proof}

We extend now our results to polynomial representations of products of $\GL(n)$'s.
The key is the following general observation. Suppose $G\times H$ acts on $W$. Viewing
$G$ and $H$ as subgroups of $G\times H$, this defines actions of $G$ and $H$ on $W$,
respectively. Let $\cR_G$ and $\cR_H$ denote the corresponding Reynolds operators.
Then it is easy to see that that the Reynolds operator of the action of $G\times H$ is
obtained by composing $\cR_G$ and~$\cR_H$.


\begin{thm}[Coefficient bounds for invariants]\label{th:CB-IR-prod}
Suppose we are given a representation $\pi\colon\GL(n_1)\times\cdots\times\GL(n_k)\to\GL(m)$
such that $\pi_\ell\colon\GL(n_i)\to\GL(m)$ defined by
\begin{align*}
  \pi_\ell(g) := \pi(I_{n_1},\ldots,I_{n_{\ell-1}},g,I_{n_{\ell+1}},\ldots,I_{n_k})
\end{align*}
is a polynomial representations of degree $d_\ell$ with the property that,
as in~\cref{th:CB-IR},
each $\pi_\ell$ splits into a direct sum $V:=\CC^m= \oplus_{r=1}^{s_\ell} \CC^{m_{\ell,r}}$ of
homogeneous representations $\pi_{\ell,r}$, as in~\cref{th:CB-IR}.
We further asssume there is a positive integer $N \le R$ such that, for all $\ell,r$,
the matrix entries of $N\pi_{\ell,r}(g)$ are integral polynomials
with absolute value bounded by $R$.
We put $n:=n_1+\ldots+n_k$ and $d:=\max d_\ell$.
Then the space of invariants of the induced action of $G:=\SL(n_1)\times\ldots\SL(n_k)$ on $\CC[V]_D$ 
is spanned by invariants with integer coefficients bounded in absolute value by
$L:=R^{kD} m^{(2k-1)D} (n^3 D d)^{kDd}$.
Hence $\log L$ is polynomially bounded in $k,D,d$, and $\log(Rmn)$.
\end{thm}

\begin{proof}
Let $\cR_\ell\colon\CC[V]_D \to \CC[V]_D^{\SL(n_\ell)}$ denote the Reynolds operator
induced by the action $\pi_\ell$.
\cref{th:CB-IR}(1) tells us $N^D\cR_\ell$ maps a monomial to a polynomial
with integral coefficients bounded in absolute value by $(R m)^D (n^3Dd)^{Dd}$.
We already noted that the Reynolds operator $\cR\colon\CC[V]\to\CC[V]^G$ for the
action induced by $\pi$ is obtained by composing the
$\cR_1,\ldots,\cR_k$ (in any order).
From this we easily see that, if $f$ is a monomial, then
$N^{kD} \cR(f)$ has integral coefficients bounded in absolute value by
\begin{align*}
  \binom{D-1 + m}{D}^{k-1} \Big((R m)^D (n^3Dd)^{Dd}\Big)^k \ \le\
 m^{D(k-1)}  \Big((R m)^D (n^3Dd)^{Dd}\Big)^k ,
\end{align*}
which proves the assertion.
\end{proof}





\cref{th:CB-IR-prod} applies to general polynomial representations of products of general linear groups.
However, in several situations, it is not necessary to rely on
Cayley's Omega process, since concrete descriptions of invariants
can be obtained by other methods, e.g., by the method
of polarization, see~\cite{procesi2007lie}.
Then better estimations on the coefficient size of invariants
can usually be given. We illustrate this with a few examples.

\begin{exa}\label{exa:simconj-coeff-bound}
For the simultaneous conjugation action $\pi$ of~$\GL(n)$ on $\Mat(n)^k$ (\cref{exa:SC}),
it is known~\cite{procesi:76} that the traces
$\tr[X_{i_1}\cdots X_{i_D}]$
of products of $D$ matrices, where $ i_1,\ldots,i_D\in [n]$, span the space
of $\SL(n)\times\SL(n)$-invariants of degree~$D$.
These invariants have coeffients in $\{0,\pm 1\}$.
More generally, for a representation $\pi$ of the group $\GL(n_1)\times\dots\times\GL(n_k)$
on a quiver on $k$ vertices, with dimension vector $(n_1,\ldots,n_k)$ (\cref{exa:quivers}),
it is known that the trace invariants of oriented cycles of length~$D$ span the space
of degree~$D$ invariants~\cite{lebruyn-procesi}. 
These invariants also have coefficents in $\{0,\pm 1\}$.
\end{exa}

\begin{exa}\label{exa:tensor-coeff-bound}
For the tensor action of $G=\SL(n_1)\times\cdots\times\SL(n_k)$
on $V=\CC^{n_1}\otimes\cdots\otimes \CC^{n_k}$ (\cref{exa:tensor}), the following concrete description
is available~\cite{burgisser2013explicit}.
If $n$ divides $D$, we assign to a map $J\colon [D] \to [n]$ the product of signs
of permutations\footnote{Defined to be zero for nonpermutations.}
\begin{align*}
  C(J) := \sgn((J(1),\ldots, J(n)) \cdot \sgn( J(n+1),\dots,J(2n)) \cdot \ldots \cdot\sgn( J(D-n+1),\ldots,J(D)) \in \{0,\pm 1 \}.
\end{align*}
It can be shown~\cite[\S 3.3]{burgisser2017alternating}
that the space of homogeneous degree $D$ invariants on~$V$ is spanned by the following polynomials,
assigned to permutations $\pi_1,\ldots,\pi_k\in S_D$,
\begin{align*}
  P(X) = \sum_{J_1,\ldots,J_k} \prod_{i=1}^k C(J_i \circ\pi_i) \prod_{t=1}^D X(J_1(t),\ldots,J_k(t)) ,
\end{align*}
where the sum is over all collections of maps $J_i\colon [D] \to [n_i]$, for $i\in [k]$.
This is only defined if all $n_i$ divide $D$, otherwise, there is no nonzero invariant.
It is easy to check that the coefficient of the monomial $X^\alpha$ in $P(X)$,
for $\alpha\in\NN^{n_1\times\dots\times n_k}$, is bounded in absolute value by the multinomial coefficient
$D!\prod_{i_1,\ldots,i_k}(\alpha(i_1,\ldots,i_k)!)^{-1} \le (n_1\cdots n_k)^D$.
Thus the vector space $\CC[V]_D^{G}$
is generated by invariants with integer coefficients having absolute value at most $(n_1\cdots n_k)^D$.
The special case $k=2$ covers the left-right action (\cref{exa:operator_scaling}).
\end{exa}

\subsection{Capacity lower bounds}\label{sec:capacity-lb}
Suppose we have a rational representation $\pi\colon\SL(n_1)\times\dots\times\SL(n_k) \to \GL(V)$
and $\langle \cdot, \cdot \rangle$ is an inner product on $V$, which is invariant
under the action of the compact subgroup $\SU(n_1)\times\dots\times \SU(n_k)$.
In \cref{dfn:capacity} we defined the capacity of a vector $v\in V$ as the infimum of
the unitarily invariant norm $\|\pi(g)v\|$, taken over all group elements $g$.
In the following, we choose a basis of $V$ so that we can identify $V=\CC^m$.
However, we will not not require that the invariant inner product equals the standard Hermitian inner product on $\CC^m$,
even though this is the case in almost all the applications we consider.
Since all norms on $\CC^m$ are equivalent, there is a fixed {\em distortion factor} $\cK>0$ such that
$\cK^{-1}\|v\| \le \|v\|_2 \le \cK\|v\|$ for all $v\in\CC^m$.
Here and in \cref{subsec:explicit-no-GT} below we give bounds in terms of a general distortion factor~$K$.
In \cref{subsec:gt-basis} we then specialize to the Gelfand-Tsetlin basis for~$\GL(n)$, for which the distortion factor can be bounded as $K \leq e^{nd \log(nd)}$ (\cref{lem:distortion}).

We are now in a position to prove capacity lower bounds for vectors $v\in\CC^m$ in terms of their
bit complexity: we assume that its components are in $\ZZ[i]$, i.e., Gaussian integers.

\begin{cor}[Capacity lower bound]\label{prp:l2 lower bound}
Let $\pi\colon\GL(n_1)\times\dots\times\GL(n_k) \to \GL(m)$
be a polynomial representation as in \cref{th:CB-IR-prod},
with degrees bounded by~$d$ and bitsize of coefficients bounded by $R$.
We put $n:=n_1+\ldots+n_k$. 
We further assume that $\|v\|_2 \le \cK \|v\|$ for all $v\in\CC^m$,
where $\|v\|$ denotes the unitarily invariant norm.
If $v\in\ZZ[i]^m$ is a vector with $\capa(v)>0$, then the capacity of $v$
with respect to the corresponding action of $\SL(n_1)\times\dots\times\SL(n_k)$
satisfies
\begin{align*}
 -\log \capa(v)
 = O\bigl(    k^2 d n^2 \log (dn) + k\log(Rm\cK) \big) .
\end{align*}
\end{cor}

\begin{proof}
Since $v$ is not in the null cone, there exists an invariant $p\in\CC[Y_1,\ldots,Y_m]$
with $p(v) \ne 0$ and $D:= \deg p\le   \ncdeg(\pi)$, see \cref{def:beta-sigma}.
According to~\cref{prp:ncdeg-degree}, we have
\begin{align*}
  D \ \le\  \ncdeg(\pi) \ \le \ n_1\cdots n_k\, d^{n_1^2+ \ldots + n_k^2 -k} \ \le\ n^k d^{n^2 - k} .
\end{align*}
By \cref{th:CB-IR-prod}, we can assume that $p$ has integer coefficients bounded in
absolute value by $L= (Rm^2)^{kD} (n^3Dd)^{k d D}$. Therefore,
\begin{align*}
  \frac1D \log L = k\log(Rm^2) + kd\log(n^3d) + kd\log D .
\end{align*}
Plugging in the above estimate $\log D \le k \log n + n^2 \log d$ gives
\begin{align*}
  \frac1D \log L = O\big(k \log (Rm) + k^2 dn^2 \log (dn) \big).
\end{align*}
Since $p(v)$ is a nonzero Gaussian integer, we have $|p(v)| \ge 1$.
Using the invariance of $p$ and \cref{le:p(v)-bound} below,
we obtain for any $g\in \SL(n_1)\times\dots\times\SL(n_k)$,
\begin{align*}
  1 \leq |p(v)| = |p(\pi(g)v) | \ \le\ L m^D  \|\pi(g)v\|_{\infty}^D
\end{align*}
Hence,
$1 \le\ L^{\frac1D} m  \|\pi(g)v\|_2  \le \cK L^{\frac1D} m  \|\pi(g)v\|$
and therefore,
$\capa(v)^{-1} \le \cK L^{\frac1D} m$.
By plugging in the above estimate for $\frac1D\log L$, we obtain
\begin{align*}
 -\log \capa(v)
 = O\bigl(k\log(Rm) + k^2 d n^2 \log (dn) + \log\cK\bigr) .
\end{align*}
The stated bound follows.
\end{proof}


\begin{lem}\label{le:p(v)-bound}
Let $p\in\CC[Y_1,\ldots,Y_m]$ be a homogenous polynomial of degree~$D$
with coefficients bounded in absolute value by~$L$.
Then, for $v\in\CC^m$, we have that
$|p(v)| \le L m^D \|v\|_\infty^D$.
\end{lem}

\begin{proof}
If $p=\sum_\alpha p_\alpha Y^\alpha$ with $|p_\alpha| \le L$, then
$|p(v)| \le L \|v\|_\infty^D \binom{m-1+D}{D} \le  L \|v\|_\infty^D m^D$.
\end{proof}

We next prove a lower bound for~$p$-capacities, cf.~\cref{eq:def p cap}.
For simplicity, we restrict ourselves to homogeneous actions of~$\GL(n)$, but we
note that everything extends to products of $\GL(n)$'s in a straightforward way.
The idea is to use the shifting trick, see~\cref{eq:p-cap vs cap}, which expresses the $p$-capacity~$\capa_p(v)$
of a vector $v\in V$ as the capacity of a transformed vector with respect to a larger
representation. We then apply~\cref{prp:l2 lower bound} 
to lower bound this capacity.


As discussed in \cref{sec:intro,subsec:rep theory,subsec:moment polytopes theory}, the moment polytope~$\Delta(v)$
is naturally a subset of~$C(n) = \{ p \in \RR^n : p_1 \geq \dots \geq p_n \}$, see \cref{tab:summary gl}.
In fact, when~$\pi\colon\GL(n) \to \GL(V)$ is a homogeneous polynomial representation of degree~$d$,
then~$\Delta(v)$ is necessarily contained in
\begin{align*}
  \Delta_d(n) := \Big\{ p \in \RR^n \;:\; p_1 \geq \dots \geq p_n \geq 0, \sum_{i=1}^n p_i = d \Big\} \subseteq C(n),
\end{align*}
and so we only consider target points $p \in \Delta_d(n)$.


\begin{thm}[$p$-capacity lower bound]\label{cor:nonuniform_capacity_bound}
Let $\pi\colon\GL(n) \to \GL(m)$ be a  homogeneous representation of degree~$d$
with $R(\pi) < \infty$.
Again we assume that $\|v\|_2 \le \cK \|v\|$ for all $v\in\CC^m$,
where $\|v\|$ denotes the unitarily invariant norm.
Let~$p\in\QQ^n \cap \Delta_d(n)$ and let~$\ell$ be a positive integer such that~$\ell p\in\ZZ^n$.
If $v\in\ZZ[i]^m$ is a vector with $\capacity_p(v)>0$, then we have
\begin{align*}
  -\log\capacity_p(v) = O\bigl( n^3 d \log (\ell nd) + \log (\cK R(\pi) m) \bigr).
\end{align*}
\end{thm}

\begin{proof}
Set $\lambda := \ell p$.
By \cref{eq:p-cap vs cap}, $\capacity_p(v)^\ell = \capacity(v^{\ot \ell} \ot v_{\lambda^*})$,
where the right-hand side capacity is computed in the
$\GL(n)$-representation~$\Sym^\ell(\CC^m) \ot V_{\lambda^*}$.
The latter representation has degree zero, so the capacity does not change when taken over~$\SL(n)$ rather than~$\GL(n)$.
This in turn allows us to replace~$\lambda^*$ by $\mu := \lambda^* + (\lambda_1,\dots,\lambda_1)$,
since shifting by multiples of the all-ones vector
does not change the representation with respect to~$\SL(n)$.
Since~$\mu$ is a partition, it corresponds to a homogeneous polynomial representation of degree~$r:=n\lambda_1 - \ell d$.
Thus, $\capacity_p(v)^\ell = \capacity(v^{\ot \ell} \ot v_\mu)$, where the right-hand side capacity is computed in the
$\SL(n)$-representation~$\Sym^\ell(\CC^m) \ot V_\mu$.

We shall realize~$\pi_\mu$ using Weyl's construction (see~\cite[Chap.~6 ]{fulton2013representation})
as a subrepresentation of~$\tau^{\ot r}\colon\SL(n)\to\GL((\CC^n)^{\ot r})$, where~$\tau\colon\SL(n)\to\GL(n)$ is the defining representation.
We can use the Euclidean norm as the unitarily invariant norm~$\norm\cdot$ on~$(\CC^n)^{\ot r}$.
A unit norm highest weight vector of~$\pi_\mu$ is then given by~$v_\mu := w_\mu / \norm{w_\mu}$, where
\begin{align*}
  w_\mu := \bigotimes_{i=1}^n \left( e_1 \wedge \dots \wedge e_i \right)^{ \ot (\mu_i - \mu_{i+1})}
  \in (\CC^n)^{\ot r}.
\end{align*}
Here,~$\wedge$ means to antisymmetrize \emph{without} averaging, so that~$w_\mu$ is a vector with integer coefficients in
the standard tensor product basis.
Note that~$\norm{w_\mu} \leq \sqrt{n^r}$ because~$w_\mu$ is a linear combination of some subset of standard basis vectors
with coefficients~$\pm1$.
We now have
\begin{align}\label{eq:addnorm}
  \log\capacity_p(v)
= \frac1\ell \log\capacity\biggl(v^{\ot\ell} \ot \frac{w_\mu}{\norm{w_\mu}}\biggr)
\geq \frac1\ell \log\capacity\bigl(v^{\ot\ell} \ot w_\mu\bigr) - \frac{r}{2\ell} \log n.
\end{align}
Clearly, the capacity is unchanged when computed in the larger representation~$\rho:=\pi^{\ot\ell} \ot \tau^{\ot r}$
on~$(\CC^m)^{\ot\ell} \ot (\CC^n)^{\ot r}$.
To lower-bound it, note that $v^{\ot\ell} \ot w_\mu$ is a Gaussian integer vector in the standard tensor product basis of the latter space,
which has dimension~$m^\ell n^r$.
Next, note that~$\rho$ has degree~$d(\rho)=\ell d+r=n\lambda_1 \leq \ell nd$.
Its coefficient size can be upper-bounded by \cref{lem:rep-bounds}: 
\begin{align*}
  R(\rho)
\leq R(\pi^{\ot\ell}) R(\tau^{\ot r}) 2^{n\lambda_1}
\leq R(\pi)^{\ell} \ell^{\ell d} 2^{n\lambda_1} ,
\end{align*}
we note that~$R(\tau^{\ot r})=1$ by direct inspection.
This gives, using $\lambda_1 \le \ell d$,
\begin{align*}
  \frac{1}{\ell} \log R(\rho) \ \le\ \log R(\pi) + d\log \ell + nd \log 2 .
\end{align*}
Finally, using~\cref{prp:l2 lower bound},
we obtain the following capacity lower bound
\begin{align*}
 -\log\capacity\bigl(v^{\ot\ell} \ot w_\mu\bigr)= O\Bigl(
 d(\rho)n^2 \log (d(\rho)n) +  \log\big(R(\rho)m^\ell n^r \cK\big) \Bigr) .
\end{align*}
Hence,
\begin{align*}
 -\frac{1}{\ell} \log \capacity\bigl(v^{\ot\ell} \ot w_\mu\bigr) = O\Bigl(
 \frac{d(\rho)n^2}{\ell} \log (d(\rho)n) + \frac{1}{\ell} R(\rho) + \log m + \frac{r}{\ell} \log n +  \frac{1}{\ell}\log\cK  \Big).
\end{align*}
Using $r\le \ell nd$, $\lambda_1 \le \ell d$, $d(\rho) \le \ell nd$, and
plugging in the above bound for $\frac{1}{\ell}\log R(\rho)$, we obtain
\begin{align*}
  -\frac{1}{\ell} \log \capacity\bigl(v^{\ot\ell} \ot w_\mu\bigr) = O\Bigl(
  n^3 d \log(\ell n d) + \log R(\pi) + \log m + \log\cK \Big) .
\end{align*}
Combining the above with \cref{eq:addnorm} yields the desired bound.
\end{proof}

We discussed in \cref{subsec:moment polytopes theory} that~$p\in \Delta(v)$ if and only if~$\capa_p(\pi(g)v) > 0$ for generic~$g\in G$,
see \cref{eq:p vs cap_p generic}.
This motivates proving a $p$-capacity lower bound for random elements in the orbit of~$v$.
To start, we need an effective version of the equivalence statement.
Such a result appeared in~\cite{burgisser2018efficient} for the tensor action.
The proof extends to the more general setting by applying Derksen's degree bound in greater generality.

\begin{prp}[Effective version of Mumford's theorem]\label{thm:effective_mumford}
Let $\pi\colon\GL(n)\to\GL(V)$ be a homogeneous polynomial representation of degree~$d$ and let~$v\in V$.
Let~$p\in\QQ^n \cap \Delta_d(n)$ and let~$\ell$ be a positive integer such that~$\ell p\in\ZZ^n$.
Then the set of group elements~$g\in \GL(n)$ such that~$\capa_p(\pi(g)v)=0$ is (as a subset of~$\GL(n)$) defined by
the zero set of polynomials of degree at most~$(\ell n d)^{n^2}$ in the matrix entries~$g_{i,j}$.
\end{prp}

\begin{proof}
Recall from the proof of \cref{cor:nonuniform_capacity_bound} that~$\capacity_p(v)^\ell = \capacity(v^{\ot\ell} \ot v_\mu)$,
where the right-hand side capacity is computed in the $\SL(n)$-representation on~$W := \Sym^\ell(\CC^m) \ot V_\mu$,
with~$\mu := \lambda^* + (\lambda_1,\dots,\lambda_1)$.
Note that the latter representation is polynomial of degree~$\lambda_1 n \leq \ell n d$.
Thus, it follows from \cref{prp:ncdeg-degree}
that there exist finitely many homogeneous $\SL(n)$-invariant polynomials~$p_r$ on~$W$ of degree
at most~$D := n (\ell n d)^{n^2-1}$ such that $\capacity_p(v)=0$ if and only if~$p_r(v^{\ot\ell} \ot v_\mu)=0$ for all~$r$.
Thus, $\capa_p(\pi(g)v)=0$ if and only if~$\tilde p_r(g)=0$ for all~$r$, where $\tilde p_r(g) := p_r((\pi(g) v)^{\ot\ell} \ot v_\mu)$
is a polynomial of degree at most~$d\ell D = (\ell n d)^{n^2}$ in the matrix entries~$g_{i,j}$.
\end{proof}

Lastly, we bound the bit complexity of the vertices of the moment polytopes.

\begin{prp}\label{prp:moment_complexity}
Let $\pi\colon\GL(n)\to\GL(V)$ be a homogeneous polynomial representation of degree~$d>0$ and let~$v\in V$.
Then, every vertex~$q$ of~$\Delta(v)$ is rational and~$\ell q \in \ZZ^n$ for some integer~$1 \leq \ell \leq n^{3n/2} d^{n^2-n}$.
\end{prp}

\begin{proof}
The moment polytope~$\Delta(v)$ can be written as the intersection of the positive Weyl chamber~$C(n)$ (see \cref{tab:summary gl}) and finitely many polytopes~$\Delta_i$ with vertices in~$\Omega(\pi)$~\cite{franz2002moment}.
Note that any $x\in\Omega(\pi)$ is integral and satisfies~$\norm{x}_2\leq d$ by \cref{prp:opnorm,lem:homog poly norm bound}.

We claim that $\Delta' := \bigcap_i \Delta_i$ can be defined as the intersection of halfspaces bounded by hyperplanes passing through $n$ affinely independent points ~$x\in\ZZ^n$ with~$\norm{x}_2\leq d$.
Clearly, it is enough to show that each~$\Delta_i$ is an intersection of such halfspaces.
Indeed, if~$\Delta_i$ has maximal dimension, the hyperplanes spanning its faces suffice because the vertices of~$\Delta_i$ lie in~$\Omega(\pi)$.
Otherwise, if~$\Delta_i$ is of dimension~$k<n$, there is a set~$S$ of at most~$n-k$ points from~$\{e_1, \dots, e_n\}$ so that $\Delta_i' := \conv(\Delta_i \cup S)$ is of full dimension.
Then~$\Delta_i$ is a face of~$\Delta_i'$, and so a subset of hyperplanes defining~$\Delta_i'$ defines~$\Delta_i$.

Next, we claim that~$\Delta(v) = C(n) \cap \Delta'$ can be defined as the intersection of halfspaces bounded by affine hyperplanes of the form~$\{x \in \RR^n : x \cdot y  = b\}$, where~$b \in \ZZ$, $y\in\ZZ^n$, and~$\norm{y}_\infty \leq M := n d^{n-1}$.
Clearly, this holds for the halfspaces defining~$C(n)$, so it remains to show the same for~$\Delta'$.
Thus, it suffices to prove that any hyperplane passing through $n$~affinely independent points~$x\in\ZZ^n$ with~$\norm{x}_2\leq d$ can be written in the form above.
The following standard argument found in~\cite[\S 17.1]{schrijver1986} shows that this is indeed the case.
Let~$A$ denote the matrix whose rows are the $n$~points spanning the hyperplane.
If~$0$ is not in the hyperplane, $A$~is invertible.
By Cramer's rule, the unique solution to the equation~$A y = \det(A) \vec 1$ is given by~$y_i = \det(A_i)$, where~$A_i$ is the matrix with the $i$\textsuperscript{th} column replaced by~$\vec 1$.
The vector~$y$ is the desired hyperplane normal and~$b = \det(A)$.
Expanding down the~$i$\textsuperscript{th} column of $A_i$, using Hadamard's bound for each of the~$n$ minors, we find~$\lvert y_i\rvert \leq M$.
If~$0$ is in the hyperplane, then~$\det(A)=0$, so we can take~$y$ to be any nonzero column of the adjugate matrix, which obeys the same bound.
This proves the second claim.

We now apply a similar argument to bound the complexity of the vertices.
Every vertex~$q$ of~$\Delta(v)$ is the intersection of some $n$~of these hyperplanes~$\{x\in\RR^n : x \cdot y_i = b_i\}$ with linearly independent normal vectors~$y_1$,\dots,$y_n$ that satisfy~$\norm{y_i}_2 \leq n^{1/2} \norm{y_i}_\infty \leq n^{1/2} M$.
We may apply the argument of~\cite{schrijver1986} again
to see that~$\ell := |\det(y_1,\dots,y_n)|$ satisfies~$\ell q \in \ZZ^n$ and
\begin{equation*}
  1 \leq \ell \leq \lVert y_1\rVert_2\cdots\lVert y_n\rVert_2
\leq (n^{1/2} M)^n
= n^{3n/2} d^{n^2-n}.\qedhere
\end{equation*}
\end{proof}

Finally, we obtain our capacity lower bound for a random element in the orbit of~$v$.

\begin{thm}[Randomized $p$-capacity lower bound]\label{thm:random_capacity}
Let $\pi\colon\GL(n) \to \GL(m)$ be a  homogeneous representation of degree~$d$
with $R(\pi) < \infty$.
Again we assume that $\|v\|_2\le \cK \|v\|$ for all $v\in\CC^m$,
where $\|v\|$ denotes the unitarily invariant norm.
Let $v\in\ZZ[i]^m$ be a vector, and~$p\in\Delta(v)$ be a point in its moment polytope.
Set~$S := 4n^{3n^3+1} d^{n^4}$. 
If~$g\in\Mat(n)$ is an integer matrix with entries drawn i.i.d.~uniformly at random from~$[S]$,
then with probability at least~$1/2$ we have~$g\in\GL(n)$ and
\begin{align*}
  -\log\capacity_p(\pi(g)v) = O\bigl( n^5d\log(dn) + \log(\cK R(\pi) m)\bigr).
\end{align*}
\end{thm}

\begin{proof}
Recall that $\Delta(v) \subseteq \Delta_d(n)$ is a convex polytope of dimension at most~$n-1$.
Thus, by Caratheodory's theorem, we know that $p$~is contained in the convex hull of some set~$Q$ of at most~$n$ vertices of~$\Delta(v)$.
For every~$q\in Q$, \cref{prp:moment_complexity} shows that there exists a positive integer~$\ell_q \leq n^{3n/2} d^{n^2-n}$
such that~$\ell_q q \in \ZZ^{n}$.
Applying \cref{thm:effective_mumford} with $\ell = \ell_q$, and using that~$q\in\Delta(v)$, there is a polynomial~$f_q$
of degree~$d_q$ at most~$(\ell_q nd)^{n^2}
\leq n^{3n^3} d^{n^4}$
that is not identically zero but that vanishes on all~$g\in\GL(n)$ such that $\capacity_q(\pi(g)v)=0$.

Now observe that our choice of~$S$ ensures that~$S \geq \max(4nd_q,4n)$ for all~$q\in Q$.
If~$g\in\Mat(n)$ is an integer matrix with entries drawn i.i.d.~uniformly at random from~$[S]$, then, by the Schwarz-Zippel lemma,
$\Pr(f_q(g)=0) \leq d_q/S \leq 1/4n$ for each~$q\in Q$.
Furthermore, $\Pr(\det(g)=0)\leq1/4$.
By the union bound, it follows that, with probabilty at least~$1/2$, we have~$\det(g)\neq0$ (i.e., $g\in\GL(n)$)
and~$f_q(g)\neq0$ for all~$q\in Q$.
The latter implies that~$\capacity_q(\pi(g)v)\neq0$.
Since~$v$ has Gaussian integer entries and $g$ is an integer matrix, there exists a positive integer~$N\leq R(\pi)$ such that~$N\pi(g)v$
has Gaussian integer entries.
Applying \cref{cor:nonuniform_capacity_bound}, we obtain for all $q \in Q$:
\begin{align*}
  -\log\capacity_p(\pi(g) v) &= -\log\capacity_p(N\pi(g) v) + \log N\\
&= O\bigl(n^3d\log(\ell_q n d) + \log(\cK R(\pi) m) \bigr)\\
&= O\bigl(n^5d\log(nd) +  \log(\cK R(\pi) m) \bigr) ,
\end{align*}
where we used $\log\ell_q = O(n\log n + n^2 \log d)$ for the last equality.
To complete the proof, observe that by concavity of~$p \mapsto \log\capacity_p(\pi(g)v)$ (\cref{prp:log cap concave general}),
the quantity $\log\capacity_p(\pi(g)v)$ is at least~$\min_{q \in Q}\log\capa_q(\pi(g)v)$.
\end{proof}


\section{Detailed analysis of the algorithms for products of \texorpdfstring{$\GL(n)$}{GL(n)} and \texorpdfstring{$\SL(n)$}{SL(n)}}%
\label{sec:explicit algos}
In this section we specialize our results and design concrete scaling algorithms for polynomial actions of products of $\GL(n)$ and $\SL(n)$'s with running time bounds depending on input size.
In \cref{subsec:explicit-no-GT} we use the capacity lower bound of~\cref{sec:capacity-lb} to derive upper bounds on the number of iterations of our algorithms.
(Together with \cref{rem:precision} they can be turned to running time bounds, not worrying about precise polynomial factors.)
Since these bounds depend on the choice of a basis, we review in~\cref{subsec:gt-basis} the construction of Gelfand-Tsetlin bases, which are canonical bases for representations of $\GL(n)$. We state a general upper bound (\cref{thm:gt-group-bound}) on the coefficient size of representations given in a Gelfand-Tsetlin basis, whose proof is postponed to~\cref{app:coeff bounds}.
This implies a capacity lower bound for representations given in Gelfand-Tsetlin bases (\cref{cor:uniform_capacity_bound}).
Using this, we finally, we state in~\cref{subsec:explicit} explicit running time bounds for our algorithms when the inputs are given in a Gelfand-Tsetlin basis.

\subsection{Explicit running time bounds}\label{subsec:explicit-no-GT}
We apply here the capacity lower bound in~\cref{prp:l2 lower bound} 
to bound the running time of the first-order and second-order algorithms from \cref{sec:first-order algorithm} and \cref{sec:second_order}.
Firstly, let us discuss the issue of precision.

\begin{rem}[Precision]\label{rem:precision}
Each step of \cref{alg:gconvex_gradient_uniform} can be applied in time polynomial in ~$n$, $m$,
and the desired number of bits of precision by~\cite{burgisser2017membership}.
With some work it can be verified that if~$T$ is the desired number of iterations,
then there is a~$\poly(n, m, T)$ number of bits of precision such that \cref{alg:gconvex_gradient_uniform}
with each step calculated to that precision still satisfies \cref{thm:uniform_grad_descent};
see the discussion on bit-complexity in, e.g.,~\cite{burgisser2018efficient,roeleveld2018tensor}.
The same holds for \cref{alg:gconvex second order,thm:main}.
\end{rem}

We now state and prove our explicit bound on the running time of our first order algorithm
for the scaling problem for representations of products of $\SL(n)$'s. 


\begin{thm}[First order algorithm for scaling, explicit bounds]\label{thm:prod-concrete-no-GT}
Let $G:=\SL(n_1)\times\dots\times\SL(n_k)$ and $\pi\colon G \to\GL(m)$ be the restriction of
a polynomial representation given as in~\cref{th:CB-IR-prod},
with degrees bounded by~$d$ and bitsize of coefficients bounded by~$R$.
Put $n:=n_1+\ldots + n_k$.
We assume further that $\cK^{-1}\|v\| \le \|v\|_2 \le \cK \|v\| $ for all $v\in\CC^m$,
where $\|v\|$ denotes the unitarily invariant norm.
Let $v\in\CC^m$ be such that $\capa(v)>0$ and assume that the components of $v$
are Gaussian integers bounded in absolute value by~$M$.
Then, \cref{alg:gconvex_gradient_uniform} with a number of iterations at most
\begin{align*}
  T = O\left(\frac{d^3}{\eps^2} \big( k^2 n^2 \log (dn) + k\log(RmM\cK) \right) ,
\end{align*}
returns a group element~$g\in G$ such that~$\Norm{\mu(\pi(g) v)}_F \leq \eps$.
\end{thm}

\begin{proof}
\cref{thm:uniform_grad_descent} combined with the bound $N(\pi)\le d$ (\cref{lem:homog poly norm bound})
gives the upper bound
\begin{align*}
  T \le \frac{4d^2}{\eps^2} \log \Big(\frac{\|v\|}{\capacity(v)}\Big)
\end{align*}
on the number of iterations of \cref{alg:gconvex_gradient_uniform}.
We have $\|v\| \le \cK \|v\|_2 \le \cK \sqrt{m} M$.
Moreover, the capacity lower bound from~\cref{prp:l2 lower bound} states that
\begin{align*}
  -\log\capacity(v) = O\bigl( k^2 d n^2 \log(dn) + k \log(Rm\cK) \bigr) .
\end{align*}
Combing these, we obtain the stated bound on the maximal number of iterations.
\end{proof}

Next, we bound the running time of our second order algorithm.


\begin{thm}[Second order algorithm for norm minimization, explicit bounds]\label{thm:prod-concretesecond-no-GT}
Assume we are in the setting of \cref{thm:prod-concrete-no-GT}.
Then, \cref{alg:gconvex second order} applied to a suitably regularized objective function and a number of iterations at most
\begin{align*}
  T = O\left( \frac{d\sqrt{n}}{\gamma} \left(
 k^2dn^2\log(dn) + k\log(RmM\cK) + \log\frac{1}{\eps}\right)
  \left( \log\big( kdn\log(RmM\cK) \Big)+ \log\frac{1}{\eps}\right) \right) .
\end{align*}
returns a group element~$g\in G$ such that $\log \norm{\pi(g) v} \leq \log\capacity(v) + \eps$.
\end{thm}

\begin{proof}
\cref{thm:main} combined with the bound $N(\pi)\le d$ from~\cref{lem:homog poly norm bound}
gives the upper bound
\begin{align*}
T = O\left( \frac{d\sqrt n}\gamma \Big(\log\frac{3n}{\eps} + \log\Big(\frac{\|v\|}{\capacity(v)}\Big) \Big)
     \Big( \log\log\Big(\frac{\|v\|}{\capacity(v)}\Big) + \log\frac{3}{\eps} \Big) \right)
\end{align*}
on the number of iterations of \cref{alg:gconvex second order}.
We have $\norm{v} \le \cK \norm{v}_2 \le \cK\sqrt{m} M$, hence $\log \norm{v} \le \log(\cK m M)$.
Moreover, the capacity lower bound from \cref{prp:l2 lower bound} gives
\begin{align*}
 -\log\capacity(v)  =O\bigl( k^2 d n^2 \log(dn) + k \log(Rm\cK) \bigr) .
\end{align*}
The stated bound on the maximal number of iterations follows by combining these estimations.
\end{proof}


\begin{thm}[Algorithm for null cone membership problem, explicit bounds]\label{cor:prod-concretenullcone-no-GT}
Assume we are in the setting of \cref{thm:prod-concrete-no-GT}.
There is an algorithm to solve the null cone membership problem (\cref{prb:exact_null_cone})
for~$\SL(n_1)\times\cdots\times\SL(n_k)$
in time $\poly(k,n,d,\log(RmM\cK), \gamma^{-1})$.
\end{thm}

\begin{proof}
This either follows from \cref{thm:prod-concrete-no-GT} by applying \cref{alg:gconvex_gradient_uniform}
with $\eps=\gamma(\pi)/2$, see \cref{cor:null_cone_margin}, which holds for any symmetric group $G\subseteq\GL(n)$.
Alternatively, the assertion follows from~\cref{thm:prod-concretesecond-no-GT} by applying \cref{alg:gconvex second order}.
For this we choose $\eps=\frac18 (\gamma(\pi) / 2N(\pi))^2$ as in \cref{cor:null_cone_margin}.
\end{proof}

Similar analyses of the $p$-scaling algorithms for moment polytopes presented in~\cref{sec:first-order algorithm} can be done.
To avoid being repetitious, we will only carry this out in~\cref{subsec:explicit} for representations given in a Gelfand-Tsetlin basis.

\subsection{Construction of the Gelfand-Tsetlin basis}\label{subsec:gt-basis}
Up to isomorphism, the irreducible representations~$\pi_\lambda\colon\GL(n)\to \GL(V_\lambda)$ of~$\GL(n)$ are labeled by their highest weight,
which are integers vectors~$\lambda\in\ZZ^n$ with~$\lambda_1\geq\dots\geq\lambda_n$ (cf.~\cref{subsec:rep theory}).

If~$\lambda_n\geq0$ then $\lambda$ can be identified with a \emph{partition}.
In this case, $\pi_\lambda$ is a polynomial representation.
That is, the matrix entries of~$\pi_\lambda(g)$ in any basis of the representation space~$V_\lambda$ are
homogeneous polynomial functions of degree $d=\sum_{i=1}^n\lambda_i$ in the matrix entries~$g_{i,j}$ of the group element~$g\in\GL(n)$.
The restriction to polynomial irreducible representations is essentially without loss of generality.
This is because $\pi_\lambda \ot \det^k \cong \pi_{(\lambda_1+k,\dots,\lambda_n+k)}$, which means that shifting the highest weight
by the all-ones vectors amounts to tensoring with powers of the determinant
(which itself is a one-dimensional irreducible representation of degree~$n$).
Moreover, the irreducible representation of~$\SL(n)$ are parametrized by highest weights~$\lambda$ modulo this shift.

The \emph{Gelfand-Tsetlin basis} is a particularly convenient basis of~$V_\lambda$~\cite{Molev2002GelfandTsetlin}.
Here, the action is given by rational functions with \emph{rational} coefficients.
Moreover, the group action, Lie algebra action, and moment map can be computed in polynomial time when working
in this basis~\cite{burgisser2000computational,burgisser2017membership}.
This makes it ideally suited as an input format for our algorithms.
We caution that the Gelfand-Tsetlin basis is in general \emph{not} orthonormal with respect to the~$U(n)$-invariant inner product
that we use throughout the paper (e.g., to define the capacity).
Thus the~$U(n)$-invariant norm is \emph{not} the same as the Euclidean norm in this basis (see \cref{lem:distortion} below).

We now show how, given~$\lambda$, to construct the representation in the Gelfand-Tsetlin basis.
The basis elements of the target vector space~$V_\lambda$ will be indexed by \emph{patterns},
which are arrays~$\Lambda = (\lambda_{i,j})$ satisfying some additional properties.
They are often depicted as follows:
\begin{align*}
\begin{tikzcd}[ampersand replacement = \&, text depth=-0.5ex, column sep=tiny, row sep=normal]
\lambda_{n,1} \arrow[dr, phantom, "\geq" down] \&\textcolor{gray}{\geq} \& \lambda_{n,2}\arrow[dr, phantom, "\geq" down] \& \textcolor{gray}{\geq}  \&  \dots \arrow[dr, phantom, "\geq" down]\&  \textcolor{gray}{\geq}  \& \lambda_{n,n} \\
  \& \lambda_{n-1, 1} \arrow[ur, phantom, "\geq" up] \arrow[dr, phantom, "\geq" down]  \& \textcolor{gray}{\geq}   \&  \dots \& \textcolor{gray}{\geq}   \& \lambda_{n-1,n-1}\arrow[ur, phantom, "\geq" up]  \&\\
    \&  \& \dots \arrow[dr, phantom, "\geq" down] \&   \& \dots \arrow[ur, phantom, "\geq" up] \&  \&\\
        \&  \&  \& \lambda_{1,1} \arrow[ur, phantom, "\geq" up] \&   \&  \&
\end{tikzcd}
\end{align*}
Formally, we say~$\Lambda = (\lambda_{i,j})_{i\in[n],j\in[i]}$ is a \emph{pattern} associated with~$\lambda$ if
\begin{itemize}
\item The upper row coincides with $\lambda$, i.e., $\lambda_{n,j} = \lambda_j$ for $j \in [n]$, and
\item The following \emph{betweenness conditions} hold for $2\leq i\leq n$ and $1\leq j \leq i-1$:
\begin{align*}
  \lambda_{i,j} \geq \lambda_{i-1,j} \geq \lambda_{i,j+1}
\end{align*}
That is, $\Lambda$ weakly decreases along both southeast and northeast diagonals.
\end{itemize}
Let~$\mathcal P_\lambda$ denote the set of all patterns associated with~$\lambda$.
Now, define $V_\lambda = \CC^{\mathcal P_\lambda}$, the vector space with basis vectors~$\xi_\Lambda$ labeled
by the pattern~$\Lambda\in\mathcal P_\lambda$.
The dimension $m_\lambda := \dim V_\lambda$ is then $\lvert\mathcal P_\lambda\rvert$.
We frequently identify~$V_\lambda$ with~$\CC^{m_\lambda}$ by putting the patterns in decreasing lexicographic order.

\begin{rem}\label{re:dlem}
Let $\lambda_n\ge 0$. Then the cardinality of $\mathcal P_\lambda$ equals the number of semistandard tableaux
of shape~$\lambda$ with entries in $\{1,2,\ldots,n\}$.
It is straightforward to check that $\lambda_1 + \cdots \lambda_n \le |P_\lambda|$.
\end{rem}

We now describe how to define the \emph{Lie algebra} representation~$\Pi_\Lambda\colon\Mat(n)\to L(V_\lambda)$,
where we recall that~$\Mat(n)$ is the Lie algebra of~$\GL(n)$.
Recall that a Lie algebra representation is a linear map that satisfies~$\Pi([X,Y]) = [\Pi(X),\Pi(Y)]$ for all $X$, $Y\in\Mat(n)$.
By linearity, it is enough to define~$\Pi_\lambda$ on the basis matrices~$E_{i,j} \in \Mat(n)$,
which are all zeroes apart from the $i,j$ entry, which is one.
In fact, we need only define~$\Pi$ on $E_{i,i}$ for~$i\in[n]$ as well as on~$E_{i, i+1}$ and~$E_{i+1, i}$ for $i\in[n-1]$.
This is because any other~$E_{i,j}$ can be obtained as an~$(|i - j| - 1)$-fold commutator of~$E_{i, i+1}$'s or $E_{i+1, i}$'s.
Theorem~2.3 of~\cite{Molev2002GelfandTsetlin} asserts that the following defines a Lie algebra representation of~$\Mat(n)$:
\begin{equation}\label{eq:gt alg action}
\begin{aligned}
  \Pi(E_{i,i}) \xi_\Lambda &:= \left( \sum_{j=1}^i \lambda_{i,j} - \sum_{j=1}^{i-1} \lambda_{i-1,j}\right) \xi_{\Lambda}, \\
  \Pi(E_{i,i+1}) \xi_\Lambda &:= -\sum_{j=1}^i \frac{(l_{i,j}-l_{i+1,1}) \cdots (l_{i,j}-l_{i+1,i+1})} {(l_{i,j}-l_{i,1})\cdots\vee\cdots (l_{i,j}-l_{i,i})} \xi_{\Lambda+\delta_{i,j}}, \\
  \Pi(E_{i+1,i}) \xi_\Lambda &:= \sum_{j=1}^i \frac{(l_{i,j}-l_{i-1,1}) \cdots  (l_{i,j}-l_{i-1,i-1})} {(l_{i,j}-l_{i,1})\cdots\vee\cdots(l_{i,j}-l_{i,i})}
\xi_{\Lambda-\delta_{i,j}},
\end{aligned}
\end{equation}
where $l_{i,j} := \lambda_{i,j} - j + 1$.
The arrays~$\Lambda\pm\delta_{i,j}$ are obtained from~$\Lambda$ by replacing~$\lambda_{i,j}$ by~$\lambda_{i,j}\pm1$.
The symbol~$\vee$ indicates that the zero factor in the denominator
is skipped.
We set~$\xi_{\Lambda}:=0$ if the array~$\Lambda$ is not a pattern associated with~$\lambda$.

The representation $\pi_\lambda\colon\GL(n)\to\GL(V_\lambda)$ of the group~$\GL(n)$ can be defined by exponentiation,
i.e., $\pi_\lambda(e^X) := e^{\Pi(X)}$ for~$X\in\Mat(n)$.
The basis $\{ \xi_\Lambda \}_{\Lambda \in \mathcal P_\lambda}$ of~$V_\lambda$ is called the \emph{Gelfand-Tsetlin basis} for~$\pi_\lambda$.
Then, $\pi_\lambda$ is an irreducible representation of~$\GL(n)$ with highest weight~$\lambda$ and highest weight vector~$\xi := \xi_\Lambda$,
associated with the pattern~$\Lambda_{i,j}:=\lambda_j$.
As mentioned earlier, the matrix entries of~$\pi_\lambda(g)$ in the Gelfand-Tsetlin basis are rational functions
with rational coefficients in the matrix entries of~$g$.
When~$\lambda$ is a partition, the matrix entries are in fact \emph{polynomials} with rational coefficients.
In \cref{thm:gt-group-bound} we prove explicit bounds on the coefficients.
These will be obtained by lifting the following bounds on the Lie algebra representation of the basis matrices:

\begin{lem}\label{lem:gt-alg-bound}
Let $\lambda$ be a partition of~$d$ with at most~$n$ parts.
Let~$\Pi\colon\Mat(n)\to\Mat(m_\lambda)$ be the Lie algebra representation in the Gelfand-Tsetlin basis,
where we identify~$V_\lambda \cong \CC^{m_\lambda}$ using the lexicographic order.
Then~$\Pi(E_{i,i})$ is diagonal for all~$i\in[n]$.
Moreover, there exist positive integers~$\beta \leq R$, where~$R = e^{O(n^3\log(\lambda_1+n))}$,
such that the entries of $\beta \Pi(E_{i,j})$ are integers of absolute value at most~$R$ for all $i,j\in[n]$.
\end{lem}

\begin{proof}
We first focus on $\Pi(E_{i,i+1})$, which is defined in \cref{eq:gt alg action}.
Define $\beta_{i,i+1} := \prod_{1\leq j\neq k\leq i} \lvert l_{i,j} - l_{i,k} \rvert$.
Then, $\beta_{i,i+1}$ as well as all entries of~$\beta_{i,i+1} \Pi(E_{i,i+1})$ are integers bounded in absolute value by~$(\lambda_1 + n)^{n^2}$.
Next consider $\Pi(E_{i,j})$ for~$j>i+1$.
Note that~$E_{i,j} = [E_{i,j-1},E_{j-1,j}]$, which implies that~$\Pi(E_{i,j}) = [\Pi(E_{i,j-1}),\Pi(E_{j-1,j})]$ because~$\Pi$ is a Lie algebra representation.
We can thus write~$\Pi(E_{i,j})$ as an iterated commutator of $\Pi(E_{i,i+1})$, $\Pi(E_{i+1,i+2})$, \dots, $\Pi(E_{j-1,j})$.
In particular, this shows that~$\Pi(E_{i,j})$ is strictly upper-triangular.
As an iterated commutator, $\Pi(E_{i,j})$ is a sum of at most~$2^n$ terms
of the form~$\Pi(E_{\sigma(i),\sigma(i)+1}) \cdots \Pi(E_{\sigma(j-1),\sigma(j-1)+1})$, where~$\sigma$ is a permutation of~$\{i,i+1\dots,j-1\}$.
Because the patterns appearing with nonzero coefficient in~$\Pi(E_{k, k + 1})\xi_\Lambda$ differ from~$\xi_\Lambda$
in the $k$\textsuperscript{th} row of~$\Lambda$ and nowhere else, the coefficient of each pattern appearing in
\begin{align*}
  \Pi(E_{\sigma(i), \sigma(i) + 1}) \cdots \Pi(E_{\sigma(j-1), \sigma(j-1) + 1}) \xi_\Lambda
\end{align*}
may be written as a product of entries of~$\Pi(E_{k,k+1})$ for~$k=i,\dots,j-1$.
Define~$\beta := \beta_{1,2}\cdots\beta_{n-1,n}$.
Then, $\beta$ is a common denominator of $\Pi(E_{i,j})$ for all~$j>i$.
Moreover, $\beta$ and all entries of
\begin{align*}
  \beta \Pi(E_{\sigma(i), \sigma(i) + 1}) \cdots\Pi(E_{\sigma(j-1), \sigma(j-1) + 1})
\end{align*}
are integers bounded in absolute value by~$(\lambda_1 + n)^{n^2(n-1)}$.
It follows that, for all $j>i$, $\beta\Pi(E_{i,j})$ is an integer matrix with entries
bounded in absolute value by~$R := 2^n (\lambda_1 + n)^{n^3} = e^{O(n^3\log(\lambda_1+n))}$.
A completely analogous argument establishes the same bound when~$i>j$.

Finally, consider $\Pi(E_{i,i})$.
Using \cref{eq:gt alg action} and the betweenness conditions, each entry of~$\Pi(E_{i,i})$
can be upper bounded by~$\lambda_{i,1} \leq \lambda_1$.
Thus, for all~$i$, the entries of $\beta\Pi(E_{i,i})$ are certainly bounded by~$R$.
This concludes the proof.
\end{proof}

An important object in our setup is the $U(n)$-invariant inner product~$\braket{\cdot,\cdot}$ on~$V_\lambda$,
which is unique up to a positive scalar.
We will fix it up to a phase by demanding that the highest weight vector~$\xi$ has norm~$\norm{\xi}=1$.
Proposition~2.4 of~\cite{Molev2002GelfandTsetlin} shows that the Gelfand-Tsetlin basis is orthogonal with
respect to ~$\braket{\cdot,\cdot}$; moreover,
\begin{equation} \label{eq:gt_inner}
 \|\xi_\Lambda\|^2
 = 
\prod_{k=2}^n \prod_{1\leq i\leq j<k} \frac {(l_{k,i}-l_{k-1,j})!} {(l_{k-1,i}-l_{k-1,j})!}
  \prod_{1\leq i<j\leq k} \frac{(l_{k,i}-l_{k,j}-1)!} {(l_{k-1,i}-l_{k,j}-1)!}.
\end{equation}
Thus, the Gelfand-Tsetlin basis is orthogonal, but \emph{not} necessarily orthonormal.
For a vector $v=\sum_\Lambda c_\Lambda\xi_\Lambda \in V_\lambda$ we can also
define the Euclidean norm by
\begin{align*}
 \|v\|_2^2 = \sum_\Lambda |c_\Lambda|^2 .
\end{align*}
This norm is more directly related to the input size of our explicit algorithms.
The next lemma compares the two norms, showing that the distortion factor $K$ appearing in \cref{prp:l2 lower bound} is at most exponential in the dimension and degree.

\begin{lem}\label{lem:distortion}
Let $\lambda$ be a partition of~$d>0$ with at most~$n$ parts, and let~$v\in V_\lambda = \CC^{\mathcal P_\lambda}$.
Then,
\begin{align*}
  \norm{v}_2 \leq \norm{v} \leq e^{nd\log(nd)} \norm{v}_2,
\end{align*}
where $\norm{\cdot}_2$ denotes the Euclidean norm and $\norm{\cdot}$ denotes the $U(n)$-invariant norm.
\end{lem}

\begin{proof}
It is enough to show that~$1\leq \lVert\xi_\Lambda\rVert \leq e^{O(nd \log (nd)}$ for any~$\Lambda \in \mathcal{P}$.
The lower bound follows from \cref{eq:gt_inner} using the betweenness conditions.
The upper bound can be obtained using \cref{eq:gt_inner} as follows.
We first bound the left-hand product:
\begin{align*}
&\quad \prod_{k=2}^n\prod_{1\leq i\leq j<k} \frac{(l_{k,i}-l_{k-1,j})!} {(l_{k-1,i}-l_{k-1,j})!}
\leq \prod_{k=2}^n\prod_{1\leq i\leq j<k} (l_{k,i}-l_{k-1,j})^{l_{k,i}-l_{k-1,i}}\\
&\leq \prod_{k=2}^n\prod_{1\leq i\leq j<k} (\lambda_1 + n)^{l_{k,i}-l_{k-1,i}}
= \prod_{1\leq i\leq j<n} (\lambda_1 + n)^{l_{n,i}-l_{j,i}}\\
&=\prod_{1\leq i\leq j<n} (\lambda_1 + n)^{\lambda_{n,i}-\lambda_{j,i}}
\leq\prod_{1\leq i\leq j<n} (\lambda_1 + n)^{\lambda_{n,i}}\\
& \leq \prod_{1\leq i<n}
(\lambda_1 + n)^{n\lambda_{n,i}} \leq (\lambda_1+n)^{n d} \leq e^{nd \log (nd)}.
\end{align*}
We can similarly bound the right-hand product:
\begin{align*}
&\quad \prod_{k=2}^n \prod_{1\leq i<j\leq k} \frac{(l_{k,i}-l_{k,j}-1)!} {(l_{k-1,i}-l_{k,j}-1)!}
\leq \prod_{k=2}^n \prod_{1\leq i<j\leq k} (l_{k,i}-l_{k,j}-1)^{l_{k,i}- l_{k-1,i}} \\
&\leq \prod_{k=2}^n \prod_{1\leq i<j\leq k} (\lambda_1+n)^{l_{k,i}-l_{k-1,i}}
\leq e^{nd\log(nd)},
\end{align*}
where the last inequality follows as above.
Together, we obtain that~$\norm{\xi_\Lambda}
\leq e^{nd\log(nd)}$.
\end{proof}

The notion of Gelfand-Tsetlin bases extends to representation of product groups
$G=\GL(n_1)\times\dots\times\GL(n_k)$ in a straightforward way as follows.
For a $k$-tuple $\lambda:=(\lambda^{(1)},\ldots,\lambda^{(k)})$,
where $\lambda^{(i)}\in\ZZ^{n_i}$ has monotonically decreasing entries,
we define the irreducible $G$-representation $V_\lambda := V_{\lambda^{(1)}}\ot\ldots \ot V_{\lambda^{(k)}}$
by taking tensor products, where $V_{\lambda^{(i)}} \simeq\CC^{m_i}$.
The Gelfand-Tsetlin basis of $V_\lambda$ is then defined to consist of the tensor products
$\xi_{\Lambda^{(1)}}\ot\ldots\ot \xi_{\Lambda^{(k)}}$, where the $\Lambda^{(i)}$ run through the set
${\mathcal P}_{\lambda^{(i)}}$ of patterns associated with $\lambda^{(i)}$.
We note that the resulting representation of $G$ on
$\CC^{m_1}\ot\ldots\ot\CC^{m_k}\simeq\CC^{m_1\cdots m_k}$
satisfies the assumptions of~\cref{th:CB-IR-prod}.

We will say that a representation~$\pi\colon\GL(n_1)\times\dots\times\GL(n_k)\to\GL(m)$ is \emph{given in a Gelfand-Tsetlin basis}
if~$\pi$ is a direct sum of irreducible representations each of which is given in a Gelfand-Tsetlin basis as defined above.
That is, $\pi = \oplus_{i=1}^s \pi_{\lambda^i}$ and we identify $V_\lambda = \bigoplus_{i=1}^s V_{\lambda^i} \cong \CC^m$ by using
the Gelfand-Tsetlin basis in each irreducible summand.
Up to isomorphism, any finite-dimensional representation of~$\GL(n_1)\times\dots\times\GL(n_k)$ is of this form.
If we demand that the $\lambda_i$ are partitions, then~$\pi$ is polynomial, and any finite-dimensional polynomial representation
of~$\GL(n_1)\times\dots\times\GL(n_k)$ is of this form.
Any finite-dimensional representation of~$\SL(n_1)\times\dots\times\SL(n_k)$ can be obtained by restricting such a~$\pi$.
We will say that a representation of~$\SL(n)$ is \emph{given in a Gelfand-Tsetlin basis} if it is obtained in this way.
It is clear that \cref{lem:gt-alg-bound,lem:distortion} extend naturally to such representations.
Our algorithms will take their input in a Gelfand-Tsetlin basis.

We now aim to bound the coefficient size for representations in a Gelfand-Tsetlin basis.
Our main tool is the following result which allows lifting coefficient bounds from Lie algebra to Lie group representations.

\begin{restatable}[Lifting coefficient bounds]{prp}{algToGroup}\label{prp:alg-to-group}
Let~$\pi\colon\GL(n) \to \GL(m)$ be a homogeneous polynomial representation of degree~$d$.
Let~$\Pi\colon\Mat(n)\to\Mat(m)$ denote its Lie algebra representation and suppose $\Pi(E_{i,i})$ is diagonal for all~$i\in[n]$.
Let~$\beta \leq R$ be positive integers such that the entries of $\beta \Pi(E_{i,j})$ are integers of absolute value at most~$R$ for all $i,j\in[n]$.
Then, the entries of $\pi(g)$ are polynomials with rational coefficients in the entries of~$g\in\GL(n)$, and~$R(\pi) = R^{2m} e^{O(mn^3d \log(mnd))}$.
\end{restatable}

We postpone the proof of \cref{prp:alg-to-group} to \cref{app:coeff bounds}
and first draw some consequences. 
We first derive the following fundamental bound on the coefficient size of representations in a Gelfand-Tsetlin basis.

\begin{thm}[Coefficient size in Gelfand-Tsetlin basis]\label{thm:gt-group-bound}
Let $\pi\colon\GL(n_1)\times\dots\times\GL(n_k)\to\GL(m)$ be a polynomial representation
of degree~$d$ given in a Gelfand-Tsetlin basis. Put $n:= n_1 + \ldots + n_k$.
Then the entries of $\pi(g)$ are polynomials with rational coefficients in the entries of~$g$,
and~$R= e^{O(m n^3 d \log(mnd))}$, where $R$ is specified as in \cref{th:CB-IR-prod}.
\end{thm}

\begin{proof}
By definition, an irreducible representation $\pi$ of $\GL(n_1)\times\dots\times\GL(n_k)\to\GL(m)$
given in a Gelfand-Tsetlin basis
is a tensor product $\pi=\pi_1\ot\ldots\ot\pi_k$, where
$\pi_i\colon\GL(n_1)\times\dots\times\GL(n_k)\to\GL(m_i)$ are irreducible and given in a Gelfand-Tsetlin basis.
Note $m=m_1\cdots m_k$.
By assumption, $\pi_i$ is homogeneous of degree~$d_i$.
\cref{lem:gt-alg-bound} and \cref{prp:alg-to-group} imply that
$R(\pi_i) = e^{O(m_i n_i^3 d \log(m_in_id))}$.
Using~\cref{lem:coeff-bd-mult-poly}, we obtain that
$R(\pi) \le R(\pi_1)\cdots R(\pi_k) (n_1^2 +\ldots + n_k^2)^d \le R(\pi_1)\cdots R(\pi_k) n^{2d}$,
hence
$R(\pi) =e^{O(m n^3 d \log(mnd))}$, where we used
$\sum_i n_i^3 \le n^3$,
and $m_1+\ldots+m_k =O(m)$.

In general, we have $\pi = \oplus_{i=1}^s \pi_{\lambda_i}$,
where each~$\pi_{\lambda_i}$ is an irreducible representation of degree at most~$d$
in a Gelfand-Tsetlin basis. By the above, we can bound $R(\pi_i)=e^{O(m_i n^3 d \log(m_i n d))}$,
where here $m_i$ denotes the dimension of the representation space of $\pi_i$.
The assertion follows using $R(\pi) \leq \max_{i} R(\pi_{\lambda_i})$
and $m_i \le m=m_1+\ldots +m_s$.
\end{proof}



We can now derive a capacity lower bound holding for any polynomial representation of $\GL(n_1)\times\dots\times\GL(n_k)$
given in a Gelfand Tsetlin basis.

\begin{cor}[Capacity lower bound for Gelfand-Tsetlin basis]\label{cor:uniform_capacity_bound}
Let $\pi\colon\GL(n_1)\times\dots\times\GL(n_k)\to\GL(m)$ be a polynomial representation of degree~$d$ given in a Gelfand-Tsetlin basis.
Put $n:=n_1 +\ldots + n_k$.
Further, let $v\in\ZZ[i]^m$ be a vector such that~$\capacity(v)>0$. Then we have,
\begin{align*}
  -\log\capacity(v) = O\bigl( k^2mn^3 d\log(mnd) \bigr).
\end{align*}
\end{cor}

\begin{proof}
By~\cref{lem:distortion}
we have
$\|w\|_\infty \le \|w\|_2 \le\|w\|$
for $w\in\CC^m$,
where $\|w\|$ denotes the unitarily invariant norm.
\cref{prp:l2 lower bound} with $\cK=1$ implies that
\begin{align*}
 -\log \capa(v)
 = O\bigl(  k^2 d n^2 \log (dn) + k \log m +  k\log R \bigr) .
\end{align*}
The assertion follows by plugging in the upper bound on $R$
from~\cref{thm:gt-group-bound}.
\end{proof}

\subsection{Explicit running time bounds for Gelfand-Tsetlin bases}\label{subsec:explicit}
In this section, we bound the running time of our algorithms on inputs specified as in~\cref{subsec:concrete_time}.
While the specification given there was for  representations of $\GL(n)$, the extension to products is obvious.
Our arguments are along the same lines as in \cref{subsec:explicit-no-GT}, but instead of the general capacity lower bound
of~\cref{prp:l2 lower bound}, we apply here the capacity lower bound for Gelfand-Tsetlin bases (\cref{cor:uniform_capacity_bound}).
We analyze here also the algorithms for $p$-scaling.

Recall that the issue of precision was already discussed in \cref{rem:precision}.
We now state and prove our explicit bound on the running time of our first order algorithm
for the scaling problem for representations of products of $\SL(n)$ groups.
For the special case $k=1$ this was stated as \cref{thm:concrete}.


\begin{thm}[First order algorithm for scaling in terms of input size]\label{thm:prod-concrete}
Let~$(\pi, v, \eps)$ be an instance of the scaling problem for~$G=\SL(n_1)\times\dots\times\SL(n_k)$
such that~$0 \in \Delta(v)$ and every component of~$v$ is 
bounded in absolute value by~$M$.
Put $n:=n_1+\ldots + n_k$ and
let~$d$ denote the degree and~$m$ the dimension of~$\pi$.
Then, \cref{alg:gconvex_gradient_uniform} with a number of iterations at most
\begin{align*}
  T = O\left(\frac{d^3}{\eps^2} k^2mn^3 \log(Mmnd) \right) ,
\end{align*}
returns a group element~$g\in G$ such that~$\Norm{\mu(\pi(g) v)}_F \leq \eps$.
In particular, there is a~$\poly(\langle\pi\rangle, \langle v\rangle, \eps^{-1})$ time algorithm
to solve the scaling problem (\cref{prb:null_cone}) for~$G$.
\end{thm}

\begin{proof}
\cref{thm:uniform_grad_descent} combined with the bound $N(\pi)\le d$ (\cref{lem:homog poly norm bound})
gives the upper bound
\begin{align*}
  T \le \frac{4d^2}{\eps^2} \log \Big(\frac{\|v\|}{\capacity(v)}\Big)
\end{align*}
on the number of iterations of \cref{alg:gconvex_gradient_uniform}.
An obvious extension of~\cref{lem:distortion} to the product group $G$ shows that
\begin{align*}
  \norm{v}_2 \leq \norm{v} \leq e^{nd\log(nd)} \norm{v}_2 .
\end{align*}
This implies that~$\norm{v} \leq e^{nd\log(nd)} \sqrt{m}M$.
Moreover, the capacity lower bound from \cref{cor:uniform_capacity_bound} states that
\begin{align*}
  -\log\capacity(v) = O\bigl( k^2 mn^3 d \log(mnd) \bigr).
\end{align*}
Combing this, we obtain the stated bound on the maximal number of iterations.
To see that this implies a~$\poly(\braket{\pi},\braket{v},\eps^{-1})$-time algorithm for the scaling problem, recall from \cref{rem:translate}
that we can always preprocess so that that~$d \leq m$ and use \cref{rem:precision}.
\end{proof}

\noindent
Next, we bound the running time of our second order algorithm.
For the special case $k=1$ this was stated as \cref{thm:concretesecond}.


\begin{thm}[Second order algorithm for norm minimization in terms of input size]\label{thm:prod-concretesecond}
Let~$(\pi, v, \eps)$ be an instance of the scaling problem for~$G=\SL(n_1)\times\dots\times\SL(n_k)$
such that~$0 \in \Delta(v)$
and every component of~$v$ is 
bounded in absolute value by~$M$.
Put $n:=n_1+\ldots + n_k$,
let~$d$ denote the degree, $m$ the dimension, and $\gamma$ the weight margin of~$\pi$.
Then, \cref{alg:gconvex second order} applied to a suitably regularized objective function and a number of iterations at most
\begin{align*}
  T = O\left( \frac{d\sqrt{n}}{\gamma} \left( k^2mn^3d\log(Mmnd) +\log\frac1\eps\right)
       \log\left(\frac{kmnd\log M}{\eps}\right)\right)
\end{align*}
returns a group element~$g\in G$ such that $\log \norm{\pi(g) v} \leq \log\capacity(v) + \eps$.
In particular, there is an algorithm to solve the norm minimization problem (\cref{prb:capacity}) for~$G$
in time~$\poly(\braket{\pi}, \braket{v}, \gamma^{-1}, \log(\eps^{-1}))$,
which is at most $\poly( \braket{\pi}^n, \braket{v}^n, \log(\eps^{-1}))$.
\end{thm}

\begin{proof}
\cref{thm:main} combined with the bound $N(\pi)\le d$ from~\cref{lem:homog poly norm bound}
gives the upper bound
\begin{align*}
T = O\left( \frac{d\sqrt n}\gamma \Big(\log\frac{3n}{\eps} + \log\Big(\frac{\|v\|}{\capacity(v)}\Big) \Big)
     \Big( \log\log\Big(\frac{\|v\|}{\capacity(v)}\Big) + \log\frac{3}{\eps} \Big) \right)
\end{align*}
on the number of iterations of \cref{alg:gconvex second order}.
As in the proof of \cref{thm:prod-concrete}, we see from an extension of~\cref{lem:distortion} that
$\log \norm{v} =O( nd\log(nd) + \log(mM))$.
Moreover, the capacity lower bound from \cref{cor:uniform_capacity_bound} states that
\begin{align*}
 -\log\capacity(v) = O\bigl( k^2 mn^3 d \log(mnd) \bigr).
\end{align*}
The stated bound on the maximal number of iterations follows by combining these estimations.
To see that this implies a~$\poly(\braket{\pi},\braket{v},\gamma^{-1},\log(\eps^{-1}))$-time algorithm for the scaling problem,
recall from \cref{rem:translate} that we can always preprocess so that that~$d \leq m$ and use \cref{rem:precision}.
The final claim follows using the bound
$\gamma^{-1} \leq n d^n (n_1\cdots n_k)^n$ from~\cref{prp:general-margin},
noting that $n_1\cdots n_k \le \braket{\pi}$.
\end{proof}


As an application, we can analyze the complexity of our algorithms for the null cone membership problem.
For the special case $k=1$ this was stated as \cref{cor:concretenullcone}.

\begin{thm}[Algorithm for null cone membership problem in terms of input size]\label{cor:prod-concretenullcone}
There is an algorithm to solve the null cone membership problem (\cref{prb:exact_null_cone})
for~$\SL(n_1)\times\cdots\times\SL(n_k)$
in time $\poly(\braket{\pi}, \braket{v}, \gamma^{-1})$,
which is at most $\poly(\braket{\pi}^n, \braket{v}^n)$.
\end{thm}

\begin{proof}
This either follows from \cref{thm:prod-concrete} by applying \cref{alg:gconvex_gradient_uniform}
with $\eps=\gamma(\pi)/2$, see \cref{cor:null_cone_margin}, which holds for any symmetric group $G\subseteq\GL(n)$.
Alternatively, the assertion follows from~\cref{thm:prod-concretesecond} by applying \cref{alg:gconvex second order}.
For this we choose $\eps=\frac18 (\gamma(\pi) / 2N(\pi))^2$ as in \cref{cor:null_cone_margin}.
For the final claim we again use
$\gamma^{-1} \leq n d^n (n_1\cdots n_k)^n$ from~\cref{prp:general-margin},
noting that $n_1\cdots n_k \le \braket{\pi}$.
\end{proof}

\noindent
We now consider the $p$-scaling problem, but for simplicity we confine our discussion to representations of~$\GL(n)$.
We start with a simple lemma that we will use in the theorem below.

\begin{lem}\label{lem:entry-bound}
Let $\pi\colon\GL(n)\to\GL(m)$ be a homogeneous polynomial representation of degree~$d$.
Then, for all~$g\in\GL(n)$,
\begin{align*}
\max_{k,l \in [m]} \lvert\pi_{k,l}(g)\rvert \leq n^{2d} R(\pi) \max_{i,j \in [n]} \lvert g_{i,j}\rvert^d
\end{align*}
\end{lem}
\begin{proof}
Each~$\pi_{k,l}(g)$ is a homogeneous polynomial of degree~$d$ in the~$n^2$ matrix entries~$g_{i,j}$.
There are at most~$\binom{d + n^2 - 1}d \leq (n^2)^d = n^{2d}$ monomials in such a polynomial, and each coefficient is bounded by in absolute value by~$R(\pi)$.
Thus the claim follows.
\end{proof}

\noindent
Finally, we state our algorithm for the $p$-scaling problem and bound its running time.

\begin{Algorithm}[th!]
\textbf{Input}:\vspace{-.2cm}
\begin{itemize}
\item A polynomial representation~$\pi\colon \GL(n) \to \GL(m)$ given in a Gelfand-Tsetlin basis that is homogeneous of degree~$d>0$,
\item a vector~$v\in\ZZ[i]^m$,
\item a target point~$p \in \QQ^n \cap \Delta_d(n)$,
\item a number of iterations~$T$.
\end{itemize}

\textbf{Output:} A group element $g$.
\medskip

\textbf{Algorithm:}
\begin{enumerate}
\item Choose $g_0\in\Mat(n)$ as an integer matrix with entries drawn i.i.d.~uniformly at random from~$[S]$, where~$S := 4n^{3n^3+1} d^{n^4}$.
If $\det(g_0)=0$, \textbf{fail}.
\item Let $g_1$ be the output of \cref{alg:nonuniform_gradient} applied to~$\pi(g_0)v$, target point~$p$, and number of iterations~$T$.
\item \textbf{Return} $g_1 g_0$.
\end{enumerate}
\vspace{-.2cm}
\caption{Randomized algorithm for the $p$-scaling problem for~$\GL(n)$ (cf.~\cref{thm:rand_alg}).}\label{alg:moment_polytope}
\end{Algorithm}

\randalg*

\begin{proof}
From \cref{thm:random_capacity}, we know that~$\det(g_0)\neq0$ and
\begin{align*}
  -\log\capacity_p(\pi(g_0)v) = O\bigl(mn^5d\log(mnd) \bigr)
\end{align*}
with probability at least~$1/2$.
Note that $\norm{p}_2 \leq N(\pi) \leq d$ by \cref{lem:bound on gradient,lem:homog poly norm bound}, since~$p \in \Delta(v)$.
Thus, $N^2 := N(\pi)^2 + \norm{p}_2 \leq d^2 + d$, and hence~$N^2=O(d^2)$.
Furthermore, we find using \cref{lem:distortion,thm:gt-group-bound,lem:entry-bound} that
\begin{align*}
  \norm{\pi(g_0)v}
&\leq e^{nd\log(nd)} \norm{\pi(g_0)v}_2
\leq e^{nd\log(nd)} \norm{\pi(g_0)}_F \norm{v}_2
\leq e^{nd\log(nd)} m \max_{k,l} \lvert\pi_{k,l}(g_0)\rvert \norm{v}_2 \\
&\leq e^{nd\log(nd)} m n^{2d} R(\pi) S^d \norm{v}_2
= e^{O(mn^4d\log(Mmnd))}
\end{align*}
The bound on the maximal number of iterations follows from these estimates by using \cref{thm:nonuniform_grad_descent}.
To see that this implies a~$\poly(\braket{\pi},\braket{v},\braket{p},\eps^{-1})$-time algorithm for the scaling problem,
recall from \cref{rem:translate} that we can always preprocess so that that~$d \leq m$ and reason analogously to \cref{rem:precision}.
Since~$\log(S) = O(n^4\log(d))$, we only require~$\poly(\braket{\pi},\braket{v})$ bits of randomness.
\end{proof}

\begin{lem}\label{lem:concrete polytope margin}
Let $\pi\colon\GL(n)\to\GL(m)$ be a homogeneous polynomial representation of degree~$d$.
Let~$p\in C(G)$ be rational with $\norm{p}_2 \leq N(\pi)$ and let~$\ell>0$ be an integer such that~$\ell p$ is a highest weight.
Set~$\eps := (2\ell)^{-n-1} d^{-n} n^{-1}$.
Then, for all $v\in V\setminus\{0\}$, solving the $p$-scaling problem with input~$(\pi,v,p,\eps)$
suffices to solve the moment polytope membership problem for~$(\pi,v,p)$.
\end{lem}

\begin{proof}
In view of \cref{cor:moment polytope margin} it suffices to verify that~$\eps\leq\gamma(\rho)/2\ell$, where $\gamma(\rho)$ refers to the weight margin of the representation~$\rho := \pi^{\ot\ell} \ot \pi_{\lambda^*}$, with~$\lambda := \ell p$.
The weight norm of~$\rho$ can be estimated as $N(\rho) \leq \ell N(\pi) + \norm{\lambda}_2 \leq 2\ell d$.
Thus, the weight margin is at least~$\gamma(\rho) \geq (2\ell d)^{-n} n^{-1}$ by \cref{prp:general-margin}.
This confirms that~$\eps \leq \gamma(\rho)/2\ell$.
\end{proof}

\noindent
If we choose~$\ell$ as the product of the denominators of the entries of~$p$, then~$\ell \leq 2^{\braket{p}}$, hence
\begin{align}\label{eq:eps bitsize}
  \log(\eps^{-1})
= \log\bigl( (2\ell)^{n+1} d^n n \bigr)
= O(n \braket{p}) + n\log(d).
\end{align}
Thus, the bitsize of~$\eps$ is polynomial in the input size provided that~$d\leq m$ (which we can always be achieved by preprocessing).

\concretemomentpolytope*

\begin{proof}
We may assume that~$\norm{p}_2 \leq N(\pi)$, since otherwise~$p\not\in\Delta(v)$.
By preprocessing, we may assume that~$d\leq m$ (\cref{rem:translate}).
Then the corollary follows from \cref{thm:rand_alg,lem:concrete polytope margin} and the preceding discussion.
\end{proof}

\subsection{Lifting coefficient bounds}\label{app:coeff bounds}
The goal of this subsection is to provide the proof of~\cref{prp:alg-to-group}.
It explains how bounds on the Lie algebra representation~$\Pi$ of the basis matrices~$E_{i,j}$ can be used to bound~$R(\pi)$.
We first state a number of elementary lemmas.

\begin{lem}\label{lem:coeff-subs}
Let $p$ be a polynomial of degree~$d$ in variables~$X_1$,\dots,$X_N$ and with integer coefficients bounded in absolute value by~$R$.
Let $\delta\in\{0,\pm1\}^N$.
Then, $p(X+\delta)$ has integer coefficients bounded in absolute value by~$R (d+1)^N 2^d$.
\end{lem}
\begin{proof}
Write $p$ as a sum of monomials, say, $p = \sum_\omega p_\omega X^\omega$.
Then,
$p(X+\delta)
= \sum_\omega p_\omega (X_1 + \delta_1)^{\omega_1} \cdots (X_N + \delta_N)^{\omega_N}
= \sum_\nu c_\nu X^\nu$,
where
\begin{align*}
  c_\nu = \sum_{\omega_1\geq\nu_1,\dots,\omega_N\geq\nu_N} c_{\nu,\omega} p_\omega \binom{\omega_1}{\nu_1} \cdots \binom{\omega_N}{\nu_N}
\end{align*}
and each~$c_{\nu,\omega}$ is a product of entries of~$\delta$, hence in~$\{0,\pm1\}$.
Clearly,
\begin{align*}
  \lvert c_\nu \rvert
\leq R \sum_{\substack{\omega_1\geq\nu_1,\dots,\omega_N\geq\nu_N \\ \omega_1+\dots+\omega_N \leq d}} \binom{\omega_1}{\nu_1} \cdots \binom{\omega_N}{\nu_N}
\leq R \sum_{\substack{\omega_1+\dots+\omega_N \leq d}} 2^d
= R \binom{N+d}d 2^d
\leq R (d+1)^N 2^d,
\end{align*}
which concludes the proof.
\end{proof}

\begin{lem}\label{lem:det-subs}
Let $X = (X_{i,j})_{i,j \in [n]}$ be a symbolic square matrix.
Then, $\det(X+I)$ is a multilinear (nonhomogeneous) polynomial with nonzero coefficients equal to~$\pm1$.
\end{lem}
\begin{proof}
Recall that $\det(X) = \sum_{\sigma\in S_n} (-1)^\sigma \prod_{i=1}^n X_{i,\sigma(i)}$.
Consider one of its monomials and write
\begin{align*}
  \prod_{i=1}^n X_{i,\sigma(i)} = \left( \prod_{i\in F(\sigma)} X_{i,i} \right) \left( \prod_{i\not\in F(\sigma)} X_{i,\sigma(i)} \right),
\end{align*}
where $F(\sigma)$ denotes the fixed point set of the permutation~$\sigma$.
When we substitute~$X\mapsto X+I$, we obtain the polynomial
\begin{align*}
  \left( \prod_{i\in F(\sigma)} (X_{i,i} + 1) \right) \left( \prod_{i\not\in F(\sigma)} X_{i,\sigma(i)} \right)
= \sum_{T \subseteq F(\sigma)} \left( \prod_{i\in T} X_{i,i} \right) \left( \prod_{i\not\in F(\sigma)} X_{i,\sigma(i)} \right)
\end{align*}
It remains to argue that any monomial in the right-hand side uniquely determines the permutation~$\sigma$.
But this holds since~$\sigma$ is determined by its action on the set of non-fixed points, which we can read off from the variables~$X_{i,j}$ with~$i\neq j$.
This concludes the proof.
\end{proof}

\begin{lem}\label{lem:coeff-bd-factor}
Let $f$ and $q_1$, \dots, $q_t$ be polynomials in variables~$X_1$,\dots,$X_N$ with integer coefficients bounded in absolute value by~$R_f$ and by $R_q$, respectively.
Let $d_f\geq2$ be an upper bound to the degree of~$f$ and suppose that the~$q_i$ have positive degree.
Suppose, moreover, that~$\prod_{i=1}^t q_i$ divides~$f$,~$q_1(0)=\dots=q_t(0)=1$, and that the quotient polynomial~$h = f / \prod_{i=1}^t q_i$ has positive degree $d$. Then $h$ has integer coefficients whose absolute values are bounded by~$R_f R_q^d (2 d_f d)^{4Nd}$.
\end{lem}
\begin{proof}
Since $q_i(0) = 1$, we can write $q_i = 1 - \hat q_i$, where~$\hat q_i(0) = 0$.
That is, $\hat q_i$ has no constant term, which implies that, for every $k$, $\hat q_i^k$ contains no monomials of degree less than~$k$.
Now,
\begin{align*}
  h
= \frac f {\prod_{i=1}^t q_i}
= \frac f {\prod_{i=1}^t (1 - \hat q_i)}
= f \prod_{i=1}^t \left( 1 + \hat q_i + \hat q_i^2 + \dots + \hat q_i^d + G_i \right),
\end{align*}
where~$G_i = \sum_{k=d+1}^\infty \hat q_i^k$.
The above holds on an open neighborhood of~$X=0$ and allows us to calculate the coefficients of~$h$ in terms of the right-hand side expansion.
Since $h(x)$ is a polynomial of degree~$d$ and~$G_i(x)$ does not have any terms of degree less than or equal to~$d$, it follows that
\begin{align*}
  h = \left[ f \prod_{i=1}^t \left( 1 + \hat q_i + \hat q_i^2 + \dots + \hat q_i^d \right) \right]_{\leq d},
\end{align*}
where we write $[p]_{\leq d}$ for the sum of homogeneous parts of degree less than or equal to~$d$ of a polynomial~$p$.
Thus, we are left with bounding the coefficients of the homogeneous parts of degree~$\leq d$ of the right hand side above.
Rewriting further,
\begin{align}\label{eq:h via sumsum}
  h
= \sum_{b_1 + \dots + b_t\leq d} \left[ f \prod_{i=1}^t \hat q_i^{b_i} \right]_{\leq d}
= \sum_{\substack{b_1 + \dots + b_t\leq d \\ c + c_1 + \dots + c_t \leq d}} \bigl[ f \bigr]_c \prod_{i=1}^t \bigl[ \hat q_i^{b_i} \bigr]_{c_i},
\end{align}
where we write $[p]_c$ for the homogeneous part of degree~$c$ of a polynomial~$p$.
Since $q_i$ divides $f$, each~$\hat q_i^{b_i}$ is a polynomial of degree at most~$b_i d_f$, and its coefficients are integers bounded in absolute value by~$R_q^{b_i} (1 + \frac{b_i d_f}N)^{N(b_i-1)}$ by \cref{lem:coeff-bd-mult-poly}.
Clearly, $\bigl[ \hat q_i^{b_i} \bigr]_{c_i}$ satisfies the same coefficient bound, but is homogeneous of degree~$c_i$.
Now consider the product~$[f]_c \prod_{i=1}^t [\hat q_i^{b_i}]_{c_i}$.
The right-hand side may be written as a product of at most~$d$ polynomials, since at most~$d$ of the $c_i$ are nonzero and~$[\hat q_i^{b_i}]_0 = \delta_{0,b_i}$.
Thus, \cref{lem:coeff-bd-mult-poly} shows that each summand of \cref{eq:h via sumsum} is a polynomial of degree at most~$d$ and with integer coefficients bounded in absolute value by
\begin{align*}
  R_f \left( \prod_{i=1}^t R_q^{b_i} \left( 1 + \frac{d_f b_i}N \right)^{N(b_i-1)} \right) d^d
\leq R_f R_q^d \left( 1 + \frac{d_f d}N \right)^{Nd} d^d
\leq R_f R_q^d \left( d_f d \right)^{2Nd}
\end{align*}
As~$h$ is a sum of~$\binom{d+t}{d}\binom{d+t+1}{d} \leq (t+1)^d (t+2)^d$
of such polynomials, $h(x)$ has integer coefficients that can be bounded in absolute value by
\begin{align*}
  (t+1)^d (t+2)^d R_f R_q^d \left( d_f d \right)^{2Nd}
\leq R_f R_q^d \left( 2 d_f d \right)^{4Nd}
\end{align*}
where in the last equality we used the fact that $t \leq d_f$.
\end{proof}

\noindent
We now prove \cref{prp:alg-to-group}, which we restate for convenience.

\algToGroup*
\begin{proof}
We first prove that the coefficients are rational, and afterwards analyze the size of the numerators and the common denominator.

For the former, it suffices to prove that the entries of $\pi(g)$ are on a dense subset of~$\GL(n)$ given by rational functions with rational coefficients.
Indeed, by assumption the entries of $\pi(g)$ are polynomials, polynomials are uniquely determined by their values on a dense subset, and a ratio of polynomials with rational coefficients that is a polynomial must be a polynomial with rational coefficients.
We now proceed with this plan and prove that the entries of~$\pi(g)$ are rational functions with rational coefficients on the dense subset of $g\in\GL(n)$ such that all leading principal minors of~$g$ are nonzero.
Indeed, in this case we can write~$g = L D U$, where~$D$ is diagonal and~$L$ and~$U$ are lower and upper triangular, respectively, with ones on the diagonal.
The entries of $L$, $D$, $U$ are given by rational functions in the entries of~$g$~\cite{Householder1966}:
\begin{align}\label{eq:LDU concrete}
  D_{i,i} = \frac {\lvert g\rvert_{[i],[i]}} {\lvert g\rvert_{[i-1],[i-1]}}, \quad
  L_{i,j} = \begin{cases}
    \frac {\lvert g\rvert_{[j-1] \cup \{i\},[j]}} {\lvert g\rvert_{[j],[j]}} & \text{if $i\geq j$} \\
    0 & \text{if $i<j$}
  \end{cases}, \quad
  U_{i,j} = \begin{cases}
    \frac {\lvert g\rvert_{[i],[i-1]\cup\{j\}}} {\lvert g\rvert_{[i],[i]}} & \text{if $i\leq j$} \\
    0 & \text{if $i>j$},
  \end{cases}
\end{align}
where we write~$\lvert g\rvert_{I,J}$ for the minor of~$g$ corresponding to rows~$I\subseteq [n]$ and columns~$J\subseteq [n]$.
We now show that the entries of $\pi(L)$, $\pi(D)$, and~$\pi(U)$ are given by polynomials with rational coefficients in the entries of~$L$, $D$, and $U$, respectively.
Then clearly the entries of~$\pi(g)$ are on a dense subset of~$\GL(n)$ given by rational function with rational coefficients.
Note that our assumption on the~$\Pi(E_{i,i})$ implies that~$\pi(D)$ is a diagonal matrix with entries
\begin{align}\label{eq:form of pi(D)}
  \pi(D)_{k,k} = D_{1,1}^{\Pi_{k,k}(E_{1,1})} \cdots D_{d,d}^{\Pi_{k,k}(E_{d,d})}.
\end{align}
Since $\pi$ is homogeneous and polynomial of degree~$d$, the exponents are necessarily non-negative integers adding to~$d$.
Thus, the nonzero entries of~$\pi(D)$ are monomials of degree~$d$ in the entries of~$D$.
Next, note that~$L - I$ is nilpotent, so $(L - I)^n = 0$.
Thus, $L=\exp(\log(L))$, where
\begin{align}\label{eq:form of log(L)}
  \log(L)
= \log (I - (I - L))
= - \sum_{i = 1}^{n-1} \frac {(I - L)^i} i,
\end{align}
so the entries of $\log(L)$ are polynomials with rational coefficients in the entries of~$L$.
Since~$\log(L)$ is strictly lower triangular, $\Pi(\log(L))$ is nilpotent by basic representation theory (cf.~\cref{eq:gt alg action}).
It follows that~$\Pi(\log(L))^m = 0$ and, hence, each entry of
\begin{align}\label{eq:form of pi(L)}
  \pi(L) = \exp\bigl(\Pi(\log(L))\bigr) = \sum_{j=0}^{m-1} \frac {\Pi(\log(L))^j} {j!}
\end{align}
is a polynomial with rational coefficients in the entries of~$L$.
By symmetry this is true for~$\pi(U)$ as well.
By the argument mentioned above, we conclude that~$\pi(g)$ is a polynomial with rational coefficients in the entries of~$g\in\GL(n)$.

We now analyze how the numerators and common denominator grow in each step of the proof.
Using the formula for the entries~$D_{i,i}$ in \cref{eq:LDU concrete}, we may write $D = M_0 / p_0$ where the common denominator~$p_0$ and every nonzero entry of the matrix~$M_0$ is a product of~$n$ principal minors of~$g$.
Recall from \cref{eq:form of pi(D)} that each nonzero entry of~$\pi(D)$ is a monomial of degree~$d$ in the variables~$D_{i,i}$.
It follows that~$\pi(D)=M_1/p_1$, where~$p_1$ and each nonzero entry of~$M_1$ is a product of at most~$nd$ principal minors of~$g$.
By \cref{lem:coeff-bd-mult-poly}, these are then homogeneous polynomials of degree at most~$n^2d$ and with integer coefficients bounded in absolute value by~$(nd)^{n^2d}$.

We now turn our attention to~$\pi(L)$.
We first bound~$\log(L)$ and then use \cref{eq:form of pi(L)}.
Using the formula for the entries~$L_{i,j}$ in \cref{eq:LDU concrete}, and noting that~$L$ has ones on its diagonal, we may write~$I - L = M_2/p_2$.
Here, the common denominator $p_2 := \prod_{i=1}^n \lvert g\rvert_{[i],[i]}$ is the product of all leading principal minors of~$g$, and each nonzero entry of~$M_2$ is a product of~$n$ minors of~$g$.
For each $i\in[n]$, we may then write~$(I - L)^i = M_{3,i} / p_3$ with common denominator~$p_3=p_2^n$.
By the iterated matrix multiplication, the entries of~$M_{3,i} = p_2^{n-i} M_2^i$ can be written as a sum of~$n^{i-1}$ many terms, each of which is a product of $n^2$~many minors.
By \cref{lem:coeff-bd-mult-poly}, it follows that each nonzero entry of $M_{3,i}$ is a homogeneous polynomial of degree at most~$n^3$ and with integer coefficients bounded in absolute value by~$n^{2n^3+n}$.
Using \cref{eq:form of log(L)}, we find that~$\log(L) = M_4/p_4$ with common denominator $p_4 = n! p_3$ and numerator~$M_4 = -\sum_{i=1}^{n-1} (n!/i) M_{3,i}$.
Thus, the entries of $M_4$ are homogeneous polynomials of degree at most~$n^3$ and have integer coefficients bounded in absolute value by~$n^{2n^3 + 2n + 1}$.

Now recall that the entries of~$\beta \Pi(E_{i,j})$ are integers bounded in absolute value by~$R\geq \beta$.
Let us write~$\Pi(\log(L)) = M_5/p_5$ with numerator~$M_5 = \sum_{i,j} (M_4)_{i,j} \beta \Pi(E_{i,j})$ and common denominator~$p_5 = \beta p_4 = \beta n! p_2^n$.
The entries of~$M_5$ are homogeneous of degree at most~$n^3$ and have integer coefficients bounded in absolute value by~$R n^{2n^3 + 2n + 3}$.
By \cref{lem:coeff-bd-mult-poly}, the same holds for~$p_5$ with plenty of slack.
For $j\in\{0,\dots,m-1\}$, we may then write~$\Pi(\log(L))^j = M_{6,j}/p_6$ with~$M_{6,j} = p_5^{m-j} M_5^j$ and common denominator~$p_6=p_5^m$.
By iterated matrix multiplication, each entry of~$M_{6,j}$ can be written as a sum of at most~$m^m$ terms, each of which is a product of~$m-j$ copies of~$p_5$ and of~$j$ many entries of~$M_5$.
As before, we use \cref{lem:coeff-bd-mult-poly} to find that the entries of~$M_{6,j}$ are homogeneous polynomials of degree at most~$mn^3$ and have integer coefficients bounded in absolute value by~$m^m (R n^{2n^3 + 2n + 3})^m m^{mn^3}=R^m (mn)^{O(mn^3)}$.
Finally, using \cref{eq:form of pi(L)}, we write~$\pi(L) = M_7/p_7$, with $M_7=\sum_{j=0}^{m-1} (m!/j!) M_{6,j}$ and common denominator $p_7=m! p_6$.
Clearly, the entries of~$M_7$ are again homogeneous polynomials of degree at most~$mn^3$ with integer coefficients bounded in absolute value by~$R^m (mn)^{O(mn^3)}$.
By symmetry, the same bounds hold for~$\pi(U)$ as well.
The upshot of all the above is that we can write
\begin{align*}
  \pi(g) = \pi(L) \pi(D) \pi(U) = \frac {M_8} {p_8}
\end{align*}
with common denominator~$p_8 = p_1 p_7^2$.
Using \cref{lem:coeff-bd-mult-poly}, every entry of~$M_8$ is a homogeneous polynomial of degree at most~$n^2d + 2mn^3$ with integer coefficients bounded in absolute value by~%
$R_8 := R^{2m} e^{O(mn^3d \log(mnd))}$.
Since~$\pi$ is polynomial and~$\pi(g) = M_8/p_8$ on a dense subset, it follows that each entry of~$M_8$ must be a multiple of~$p_8$.

To finish the proof, we need to bound the coefficients of~$\pi(g)$ in terms of the numerator and denominator.
This will be achieved by using \cref{lem:coeff-bd-factor}.
To apply the lemma, recall that~$p_8$ is proportional to a product of principal minors.
Indeed, $p_8 = p_1 p_7^2 = \alpha \prod_{i=1}^t q_i$, where $t \leq nd+2mn^2$, each $q_i$ is a principal minor of~$g$, and~$\alpha = (m!)^2 (\beta n!)^{2m}
= R^{2m} e^{O(mn\log(mn))}$.
In particular, $q_1(I) = \dots = q_t(I) = 1$.
Now let~$h(g)$ be an arbitrary entry of~$\alpha \pi(g)$.
Let~$f$ denote the corresponding entry of~$M_8$, so that $h = f / \prod_{i=1}^t q_i$.
Define $\tilde h(X) := h(X+I)$, $\tilde f(X) := f(X+I)$, and $\tilde q_i(X) := q_i(X+I)$, so that $\tilde q_1(0) = \dots = \tilde q_t(0) = 1$.
Clearly, it still holds that $\tilde h = \tilde f / \prod_{i=1}^t \tilde q_i$ is a polynomial of degree~$d$.
Moreover, $\tilde f$ still has degree at most $d_{\tilde f} := n^2d + 2mn^3$.
By \cref{lem:coeff-subs}, the coefficients of~$\tilde f$ are bounded in absolute value by $R_{\tilde f} := R_8 (d_{\tilde f}+1)^{n^2} 2^{d_{\tilde f}}$, while the nonzero coefficients of $\tilde q_i$ remain~$\pm1$ according to \cref{lem:det-subs}.
With these bounds, \cref{lem:coeff-bd-factor} shows that~$\tilde h$ is a polynomial with integer coefficients bounded in absolute value by~$\tilde R := R_{\tilde f} (2 d_{\tilde f} d)^{4n^2d}$.
Using \cref{lem:coeff-subs} one more time, it follows that~$h$ is a polynomial with integer coefficients bounded in absolute value by~$\tilde R (d+1)^{n^2} 2^d = R^{2m} e^{O(mn^3d \log(mnd))}$.
Since $h$ is an arbitrary entry of~$\alpha\pi$ and since~$\alpha = R^{2m} e^{O(mn\log(mn))}$, we finally obtain the claim.
\end{proof}

\section{Conclusion}\label{sec:conclusion}
This paper initiates a systematic development of a theory of \emph{non-commutative} optimization which greatly extends ordinary (Euclidean) convex optimization.
This setting captures natural geodesically convex optimization problems on Riemannian manifolds that arise from the symmetries of non-commutative groups.
It unifies a diverse range of problems -- many non-convex -- in different areas of computer science, mathematics, and physics.
Several of them were solved efficiently for the first time using non-commutative methods.
The corresponding algorithms also lead to solutions of purely structural problems, and to many new connections between disparate fields.
Our work points to intriguing open problems and suggests further research directions.
We believe that extending this theory will be fruitful both from a mathematical and computational point of view.
It provides a meeting place for ideas and techniques from several different areas of research, and promises better algorithms for existing and yet unforeseen applications.
We mention a few concrete challenges:
\begin{enumerate}
\item Is the null cone membership problem for general group actions in P?
A natural intermediate goal is to prove that they are in NP~$\cap$~coNP.
The quantitative duality theory developed in this paper makes such a result plausible.
The same question may be asked about the moment polytope membership problem for general group actions~\cite{burgisser2017membership}.

\item Can we find more general classes of problems or group actions where our algorithms run in polynomial time?
In view of the complexity parameters we have identified, it is of particular interest to understand when the \emph{weight margin} is only inverse polynomial rather than exponentially small.

\item Interestingly, when restricted to the commutative case discussed in \cref{subsubsec:comp prob}, our algorithms' guarantees do not match those of cut methods in the spirit of the ellipsoid algorithm.
Can we extend non-commutative/geodesic optimization to include cut methods as well as interior point methods?
The foundations we lay in extending first and second order methods to the non-commutative case makes one optimistic that similar extensions are possible of other methods in standard convex optimization.

\item Can geodesic optimization lead to new efficient algorithms in combinatorial optimization?
We know that it captures algorithmic problems like bipartite matching (and more generally matroid intersection).
How about perfect matching in general graphs -- is the Edmonds polytope a moment polytope of a natural group action?

\item Can geodesic optimization lead to new efficient algorithms in algebraic complexity and derandomization?
We know that the null cone membership problem captures polynomial identity testing (PIT) in non-commuting variables.
The variety corresponding to classical PIT is however \emph{not} a null cone~\cite{makam2019singular}.
Can our algorithms be extended beyond null cones to membership in more general classes of varieties?
\end{enumerate}

\phantomsection%
\section*{Acknowledgments}%
\addcontentsline{toc}{section}{Acknowledgments}
We would like to thank Mahmut Levent Do\u{g}an, Visu Makam, Harold Nieuwboer, and Philipp Reichenbach for interesting discussions.
We are grateful to Lin Zhang and Darij Grinberg for pointing out various typographical errors in an earlier version of this manuscript.


{\small
\bibliographystyle{alphaurl}
\addcontentsline{toc}{section}{References}
\bibliography{gradflow}
}

\end{document}